\documentclass[10pt, letterpaper,reqno]{amsart}
\usepackage[left=1in,right=1in,bottom=0.5in,top=0.8in]{geometry}
\usepackage{amsfonts}
\usepackage{amsmath, amssymb}
\usepackage{graphicx}
\usepackage[font=small,labelfont=bf]{caption}
\usepackage{epstopdf}
\usepackage{xcolor}
\usepackage{amsthm}
\usepackage{float}
\usepackage{enumitem}
\usepackage{pgfplots}
\usepackage{listings}
\usepackage{longtable}
\usepackage{mathrsfs}
\usepackage{dsfont}
\usepackage{mathtools}
\usepackage{dsfont}
\usepackage{stmaryrd}
\usepackage[pdfpagelabels,hyperindex]{hyperref}
\hypersetup{linkbordercolor=green}

\newcommand*{\rom}[1]{\expandafter\@slowromancap\romannumeral #1@}
 \colorlet{lgray}{white!80!black}
\colorlet{lred}{white!85!red}
\colorlet{lgreen}{white!60!green}
\colorlet{dgreen}{black!30!green}
\colorlet{lpurple}{white!60!purple}
\colorlet{lblue}{white!60!blue}
\definecolor{green}{rgb}{0.1,0.8,0.1}
\definecolor{yellow}{rgb}{1.0,0.85,0.25}
\definecolor{purple}{rgb}{1.0, 0, 1.0}
\definecolor{blue}{rgb}{0, 0, 1.0}
\tikzstyle{unfused}=[lgray, line width=1.5pt, ->]
\tikzstyle{fused}=[lgray, line width=4pt, ->]
\tikzstyle{dual}=[black, line width=1pt, dashed]
\tikzstyle{lightdual}=[black, line width=0.5pt, dashed]
\tikzstyle{cut}=[black, line width=1.0pt]

\theoremstyle{plain}
\newtheorem{theorem}{Theorem}[section]
\newtheorem{lemma}[theorem]{Lemma}
\newtheorem{proposition}[theorem]{Proposition}
\newtheorem{corollary}[theorem]{Corollary}

\theoremstyle{remark}
\newtheorem{definition}[theorem]{Definition}

\colorlet{shadecolor}{gray!20}
\pgfplotsset{compat=1.9}
\usetikzlibrary{shapes.multipart}
\usetikzlibrary{patterns}
\usetikzlibrary{shapes.multipart}
\usetikzlibrary{arrows}
\usetikzlibrary{decorations.markings}
\usepgflibrary{decorations.shapes}
\usetikzlibrary{decorations.shapes}
\usepgflibrary{shapes.symbols}
\usetikzlibrary{shapes.symbols}
\usetikzlibrary{decorations.pathreplacing}
\tikzstyle{fleche}=[>=stealth', postaction={decorate}, thick]
\tikzstyle{axis}=[->, >=stealth', thick, gray]
\tikzstyle{paths}=[>->, >=stealth', thick]
\tikzstyle{path}=[->, >=stealth', thick]
\tikzstyle{grille}=[dotted, gray]
\DeclareFontFamily{U}{mathx}{}
\DeclareFontShape{U}{mathx}{m}{n}{<-> mathx10}{}
\DeclareSymbolFont{mathx}{U}{mathx}{m}{n}
\DeclareMathAccent{\widehat}{0}{mathx}{"70}
\DeclareMathAccent{\widecheck}{0}{mathx}{"71}

\usepackage{romannum}
\usepackage{upgreek}

\usepackage[scr=boondoxo]{mathalfa}

\newcommand{\ber}{\mathrm{Ber}}
\newcommand{\lrw}{\mathsf{LRW}}
\newcommand{\bg}{\boldsymbol{\gamma}}
\newcommand{\bs}{\boldsymbol{\sigma}} 

\renewcommand{\a}{\mathsf{a}}
\renewcommand{\b}{\mathsf{b}}

\newcommand{\TV}[1]{{\left\lVert #1 \right\rVert}_{\normalfont
\text{TV}}}

\numberwithin{equation}{section}

\usepgflibrary{fpu}
\makeatletter
\let\NAT@parse\undefined
\makeatother

\author{Dominik Schmid}
\address
{
Dominik Schmid, 
University of Augsburg, Germany 
}
\email{d.schmid@uni-a.de} 
\author{Zongrui Yang}
\address
{
Zongrui Yang, 
University of California, Berkeley
}
\email{zongrui@berkeley.edu}

\title[Approximation for the stationary measure of the open ASEP]{Sharp approximations for the stationary measures of the open ASEP II: The coexistence line}  
\begin{document}
\pagenumbering{arabic}

\begin{abstract} 
We study the stationary measure of the open asymmetric simple exclusion process (ASEP) on the coexistence line. Our starting point is Bryc's two-layer Gibbs representation, together with the probabilistic interpretation, established in the first paper in this series, of the associated Bryc polynomial in terms of an inhomogeneous birth--death chain. Our main result gives a sharp approximation of the stationary measure. Specifically, we identify the shock location as an explicit observable of the particle configuration and obtain optimal bounds on the total variation distance between its law and the uniform measure on the lattice. Conditional on the shock location, we approximate, in total variation distance, the joint law of the occupation variables away from the shock by Bernoulli product measures with different densities to its left and right. These approximations lead to further applications, including fluctuation limits, moderate deviation bounds for the particle density, and limits for correlation functions under the stationary measure.
\end{abstract}
\maketitle

%---------------------------------------------------------------------------------------------------------------
% Section 1
%
%----------------------------------------------------------------------------------------------------------------
\section{Introduction and main results}

\subsection{Preface}

Shocks are a central phenomenon in interacting particle systems: they separate regions of different particle densities. For the asymmetric simple exclusion process (ASEP), they provide a discrete counterpart to discontinuities in solutions of the Burgers equation, with the two descriptions connected through hydrodynamic limits; see \cite{kipnis1999scaling} for a survey. The structure and fluctuations of shocks have been studied extensively for exclusion processes on the integer lattice \cite{liggett1999stochastic,FF94current,FF94shock}. Related questions also arise for stochastic partial differential equations, including the stochastic Burgers equation and the Kardar--Parisi--Zhang (KPZ) equation; see, for example, \cite{DS26,dunlap2021viscous} on viscous shock fluctuations.

In this paper, we study shocks in the ASEP with open boundaries, or open ASEP. Particles perform biased random walks on a finite segment, subject to the exclusion rule, and may enter or exit at the endpoints at prescribed rates. On the coexistence line, which forms the boundary between the high- and low-density phases, a shock separates regions of low and high particle density. Its location remains random on the scale of the system; see, for example, \cite{liggett1999stochastic,derrida1993exact,schutz1993phase,essler1996representations,mallick1997finite,derrida2003exact}. Confined by the two boundaries, this shock governs the macroscopic structure of the stationary measure. Our goal is to locate it precisely from a stationary particle configuration and to quantify its influence on the surrounding particles.

Several recent works have clarified this picture. Using Askey--Wilson signed measures, the authors of \cite{wang2023askey} determine the macroscopic density profile in the shock region, including the coexistence line. Microscopic local equilibrium on the coexistence line is studied in \cite{SY:openLight}, building on the local convergence results of \cite{BM06}. For the totally asymmetric case, the open TASEP, the authors of \cite{bryc2024limitCoex} identify the shock location directly from the particle configuration as the first location at which its centered height function attains its minimum. They prove that the rescaled shock location is asymptotically uniform and establish fluctuation limits for the height function after centering by the corresponding random shock profile.

The present paper is the second in a series of three on sharp approximations of the stationary measure of the open ASEP. We work under Liggett's condition on the rates. Our approach builds on Bryc's two-layer representation \cite{bryc2024stationary} (based on \cite{barraquand2024stationary}) and the probabilistic interpretation of its polynomial correction factors developed in the first paper \cite{SchmidYangShock}. These factors, called \emph{Bryc polynomials}, can be expressed in terms of inhomogeneous birth--death Markov chains. As in \cite{SchmidYangShock}, our analysis of these chains relies on the regeneration structure of biased simple random walk paths. On the coexistence line, however, this analysis is substantially more involved, as we must control the polynomial weights uniformly in the shock location.

Our main result gives a quantitative description of both the shock location and the stationary configuration around it. We define the shock location as the first location at which the centered particle height function attains its minimum. For a segment of length $N$, we prove matching upper and lower bounds of order $N^{-1}$ on the total variation distance between its law and the uniform measure on the lattice. This extends the shock-location results of \cite{bryc2024limitCoex} from the open TASEP to the open ASEP and strengthens macroscopic asymptotic uniformity to an approximation of the full discrete distribution at the optimal rate.

Conditional on the shock location, we also approximate the joint law of the occupation variables outside a window of radius $r$ around the shock by a Bernoulli product measure with low density to the left and high density to the right. The total variation error is exponentially small in $r$, uniformly in the system size and the admissible shock location. The estimate controls all sites outside the window jointly and therefore applies to arbitrary collections of observed sites, including those whose size grows with the system. To the best of our knowledge, this combination of sharp total variation estimates is new even for the open TASEP.

These approximations lead to further applications to the stationary measure. We obtain the fluctuation limits, moderate deviation bounds for particle counts centered by the random shock profile, and limits for the correlation functions, under the stationary measure of open ASEP on the coexistence line.

\subsection{The open ASEP}\label{subsec:model and results}

\begin{figure}
\centering
\begin{tikzpicture}[scale=0.85]
\def\spiral[#1](#2)(#3:#4:#5){% \spiral[draw options](placement)(end angle:revolutions:final radius)
\pgfmathsetmacro{\domain}{pi*#3/180+#4*2*pi}
\draw [#1,
       shift={(#2)},
       domain=0:\domain,
       variable=\t,
       smooth,
       samples=int(\domain/0.08)] plot ({\t r}: {#5*\t/\domain})
}
\def\particles(#1)(#2){
  \draw[black,thick](-3.9+#1,0.55-0.075+#2) -- (-4.9+#1,0.55-0.075+#2) -- (-4.9+#1,-0.4-0.075+#2) -- (-3.9+#1,-0.4-0.075+#2) -- (-3.9+#1,0.55-0.075+#2);
\node[shape=circle,scale=0.6,fill=red] (Y1) at (-4.15+#1,0.2-0.075+#2) {};
  	\node[shape=circle,scale=0.6,fill=red] (Y2) at (-4.6+#1,0.35-0.075+#2) {};
  	\node[shape=circle,scale=0.6,fill=red] (Y3) at (-4.2+#1,-0.2-0.075+#2) {};
   	\node[shape=circle,scale=0.6,fill=red] (Y4) at (-4.45+#1,0.05-0.075+#2) {};
  	\node[shape=circle,scale=0.6,fill=red] (Y5) at (-4.65+#1,-0.15-0.075+#2) {}; }
\def\annhil(#1)(#2){	  \spiral[black,thick](9.0+#1,0.09+#2)(0:3:0.42);
\draw[black,thick](8.5+#1,0.55+#2) -- (9.5+#1,0.55+#2) -- (9.5+#1,-0.4+#2) -- (8.5+#1,-0.4+#2) -- (8.5+#1,0.55+#2); }
\node[shape=circle,scale=1.5,draw] (B) at (2.3,0){} ;
	\node[shape=circle,scale=1.5,draw] (C) at (4.6,0) {};
	\node[shape=circle,scale=1.2,fill=red] (CB) at (2.3,0) {};
    \node[shape=circle,scale=1.5,draw] (A) at (0,0){} ;
 %   \node[shape=circle,scale=1.5,fill=red] (AB) at (0,0){} ;    
 	\node[shape=circle,scale=1.5,draw] (D) at (6.9,0){} ; 
% 		\node[shape=circle,scale=1.5,fill=red] (CD) at (6.9,0) {};
 	 	\node[shape=circle,scale=1.5,draw] (Z) at (-2.3,0){} ;
	\node[shape=circle,scale=1.2,fill=red] (YZ) at (-2.3,0) {};
   \node[line width=0pt,shape=circle,scale=1.6] (B2) at (2.3,0){};
	\node[line width=0pt,shape=circle,scale=2.5] (D2) at (6.9,0){};
		\node[line width=0pt,shape=circle,scale=2.5] (Z2) at (-2.3,0){};
		
	\node[line width=0pt,shape=circle,scale=2.5] (X10) at (6.8,0){};
		\node[line width=0pt,shape=circle,scale=2.5] (X11) at (-2.2,0){};	
			\node[line width=0pt,shape=circle,scale=2.5] (X12) at (8.4,0){};
		\node[line width=0pt,shape=circle,scale=2.5] (X13) at (-3.8,0){};		
\draw[thick] (Z) to (A);	
	\draw[thick] (A) to (B);	
		\draw[thick] (B) to (C);	
  \draw[thick] (C) to (D);

\particles(0)(0);
\particles(6.9+4.9+1.6)(0);

%\annhil(0)(0.7);
%\annhil(-13.4)(-0.7);

\draw [->,line width=1pt]  (B2) to [bend right,in=135,out=45] (C);
 \draw [->,line width=1pt] (6.9+1.6,-0.475) to [bend right,in=135,out=45] (D);	
   \draw [->,line width=1pt] (B2) to [bend right,in=-135,out=-45] (A);
     \draw [->,line width=1pt] (Z) to [bend right,in=135,out=45] (-3.9,-0.475);
    \node (text1) at (3.5,1){$1$} ;    
	\node (text2) at (1.1,1){$q$} ;   
	\node (text3) at (-2.3-1,1){$\alpha$}; 
	\node (text4) at (6.9+1,1){$\beta$};   
    \node (text3) at (-2.3-1,-1){$\gamma$}; 
	\node (text4) at (6.9+1,-1){$\delta$};   

\node[scale=0.9] (textL) at (-5.2,0){$\rho_{L}$} ;
\node[scale=0.9] (textR) at (9.9,0){$\rho_{R}$} ;

    \node[scale=0.9] (text1) at (-2.2,-0.7){$1$} ;    
    \node[scale=0.9] (text1) at (6.8,-0.7){$N$} ;   
  	
  \draw [->,line width=1pt] (-3.9,0.475) to [bend right,in=135,out=45] (Z);
 %  \draw [->,line width=1pt] (Z) to [bend right,in=135,out=45] (-3.9,-0.475); 
 %  \draw [->,line width=1pt] (6.9+1.6,-0.475) to [bend right,in=135,out=45] (D);	
   \draw [->,line width=1pt] (D) to [bend right,in=135,out=45] (6.9+1.6,0.475);

	\end{tikzpicture}	
\caption{\label{fig:openASEP} Visualization of the open ASEP for parameters $(q,\alpha,\beta,\gamma,\delta)$  on a segment of length $N$.}
\end{figure}
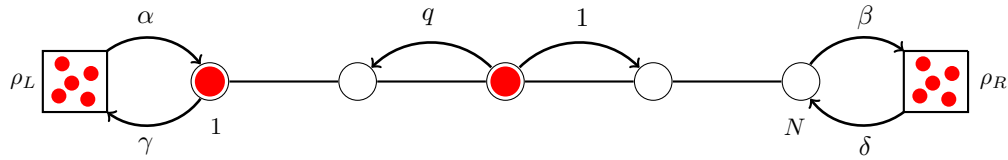

The open ASEP is a continuous-time Markov chain $(\tau(t))_{t \geq 0}$ defined on the state space $\{0,1\}^{N}$ and with respect to parameters
\begin{equation}\label{eq:conditions open ASEP}
\alpha,\beta>0,\quad \gamma,\delta\geq 0,\quad\text{and}\quad  0\leq q<1 . 
\end{equation}
We denote by $\tau=(\tau_1,\tau_2,\dots,\tau_N) \in \{0,1\}^N$ the configurations of the open ASEP. Here, $\tau_i$ indicates whether the $i^{\text{th}}$ site in the configuration $\tau$ is occupied by particle. For all $a<b$, we use the notation  $\llbracket a,b\rrbracket := [a,b] \cap \mathbb Z$. The open ASEP evolves according to the following dynamics: on each site $\llbracket 2,N\rrbracket$ particles attempt independent jumps to the left according to rate $q$ Poisson clocks, and on the sites $\llbracket 1,N-1\rrbracket$ to the right according to rate $1$ Poisson clocks. However, the particle performs the jump if and only if the target site is not already occupied by another particle. In addition to the bulk dynamics, particles can enter and exit at the endpoints of the segment. More precisely, they enter at the left boundary at rate $\alpha$ and exit at rate $\gamma$. Similarly, they enter at the right boundary at rate $\delta$ and exit at rate $\beta$. We refer to Figure~\ref{fig:openASEP} for a visualization of the parameters. Note that under our assumptions in \eqref{eq:conditions open ASEP}, the open ASEP has always a unique stationary measure. In order to describe the boundary effects for the open ASEP, it is convenient to define
\begin{equation}\label{eq:defining ABCD}
\a=\phi_+(\alpha,\gamma),\quad\b=\phi_+(\beta,\delta),\quad \mathsf{c}=\phi_-(\alpha,\gamma), \quad \mathsf{d}=\phi_-(\beta,\delta),
\end{equation}
where we let
\begin{equation} \label{eq:phi}
\phi_{\pm}(x,y):=\frac{1}{2x}\left(1-q-x+y\pm\sqrt{(1-q-x+y)^2+4xy}\right), \quad \mbox{for }\; x>0\mbox{ and }y\geq 0.
\end{equation}
In particular, note that under our assumptions \eqref{eq:conditions open ASEP}, 
\begin{equation}\label{eq:conditions qABCD}
\a,\b\geq0,\quad -1<\mathsf{c},\mathsf{d}\leq 0 . 
\end{equation}
Throughout this article, we will work under \textit{Liggett's condition} that
\begin{equation} \label{eq:Liggett's condition}
\gamma=q(1-\alpha),\quad\delta=q(1-\beta),\quad \alpha,\beta\in(0,1] .
\end{equation}
Note that under \eqref{eq:Liggett's condition}, the parameters $\a,\b,\mathsf{c},\mathsf{d}$ then simplify to
\begin{equation}\label{eq:parameters under Liggett condition}
\a=\frac{1-\alpha}{\alpha},\quad\b=\frac{1-\beta}{\beta},\quad \mathsf{c}=-q,\quad\mathsf{d}=-q.
\end{equation}
We define $\rho_{\textup{L}}:=(1+\a)^{-1}$, respectively $\rho_{\textup{R}}:=\b(1+\b)^{-1}$ as the left side, respectively right side \textit{effective densities}, and denote the unique stationary measure by $\mu_N=\mu^{(q,\a,\b)}_N$ for the open ASEP on a segment of length $N$. The parameters $\a$ and $\b$ allow for a simple partition of the open ASEP into different regimes. When $\a > \max(\b,1)$, we say that the open ASEP belongs to the \textit{low density phase}, for $\b > \max(\a,1)$ the open ASEP belongs to the \textit{high density phase}, and for $\a,\b \leq 1$ to the \textit{maximal current phase}. Intuitively, in the low density phase, the stationary measure of the open ASEP is locally given by a Bernoulli-$\rho_{\textup{L}}$-product measure, in the high density phase by a Bernoulli-$\rho_{\textup{R}}$-product measure, and in the maximal current phase by a Bernoulli-$\frac{1}{2}$-product measure; see Figure~\ref{fig:Phase} for an illustration. 
Let us mention that the parameter space of the open ASEP allows also for a second partition: we say that the open ASEP belongs to the \textit{fan region} if $\a \b < 1$, and to the \textit{shock region} if $\a \b >1$.

In the present paper, we focus on the special case $\a=\b>1$, called the \textit{coexistence line}, where the stationary measure is locally a non-trivial convex combination of Bernoulli product measures; see \cite{SY:openLight}. The coexistence line always belongs to the shock region of the phase diagram and our main goal is investigate the structure of the invariant measure $\mu^{(q,\a,\b)}_N$ in this regime. 

\begin{figure}
    \centering
    \begin{tikzpicture}[scale=0.75]
 \draw[scale = 1,domain=6.7:11,smooth,variable=\x,dotted,thick] plot ({\x},{1/((\x-7)*1/3+2/3)*3+5});
 \draw[->] (5,5) to (5,10.2);
 \draw[->] (5.,5) to (11,5);
   \draw[dashed] (5,8) to (8,8);
  \draw[dashed] (8,8) to (8,5);
   \draw[dashed] (8,8) to (10.2,10.2);
   \node [left] at (5,8) {$1$};
   \node[below] at (8,5) {$1$};
     \node [below] at (11,5) {$\a$};
   \node [left] at (5,10) {$\b$};
 \draw[dashed] (8,4.9) to (8,5.1);
  \draw[dashed] (4.9,8) to (5.1,8);
 \node [below] at (5,5) {$(0,0)$};
    \node [above] at (6.5,8.5) {LD};\node [above] at (8,8.7) {LD};
    %\node[above] at (8,9) {LD};
    \node at (10.6,10) {\small coexistence line }; 
    \node [below] at (10,6) {HD};  \node [below] at (10,8) {HD}; 
      %\node [above] at (10,9) {HD};
 \node [below] at (6.5,6) {MC}; 
 \node at (6.4,10) {\small{fan}};
 \node at (7.6,10) {\small{shock}};
 \end{tikzpicture} 
        \caption{    \label{fig:Phase} Visualization of the different regimes: Low density (LD), high density (HD), and maximal current (MC) phase. Moreover, the curve $\a \b =1$ separates the fan and shock region. In this article, we focus on $\a=\b>1$ called the coexistence line. }
\end{figure}
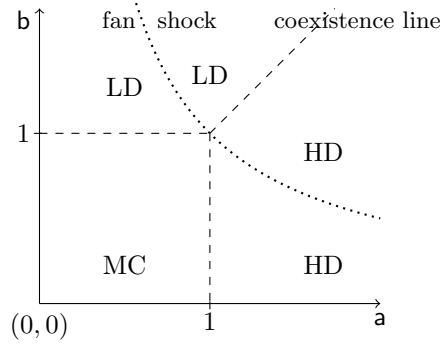

\subsection{Statement of the main result} 

Assume Liggett's condition \eqref{eq:Liggett's condition} holds and that the open ASEP is on the coexistence line $\a=\b>1$.  For each $\tau=(\tau_1,\dots,\tau_N)\in\{0,1\}^N$, we let
\[
    s_0 (\tau)=0,
    \quad
    s_j (\tau)=\sum_{i=1}^j\left(2\tau_i -1\right)\quad\mbox{for }j\in\llbracket1,N\rrbracket,
\]
and define the discrete shock location
\begin{equation}\label{eq:Mtau}
    \mathsf{M}_N =\mathsf{M}_N (\tau)
    :=
    \min\left\{
        j\in\llbracket0,N\rrbracket:
        s_j (\tau)=\min_{\ell\in\llbracket0,N\rrbracket}s_{\ell} (\tau)
    \right\}.
\end{equation}
The following is our main approximation result for the stationary measure of the open ASEP on the coexistence line.

\begin{theorem}\label{thm:ApproxCoex}
Assume $N\geq2$.
Consider the open ASEP on the lattice $\llbracket1,N\rrbracket$ with jump rates $\alpha,\beta,\gamma,\delta$ and $q$ satisfying \eqref{eq:conditions open ASEP}. Assume Liggett's condition \eqref{eq:Liggett's condition} holds and that the open ASEP is on the coexistence line $\a=\b>1$.    
Then there exist constants $c,C,c',C'>0$ depending only on $q$ and $\b$, such that for each $N\geq m\geq1$ and $k\in\llbracket0,N\rrbracket$, 
\begin{equation*}
        \TV{
    \mathbb{P} 
    \left(
        (\tau_i)_{i\in \llbracket1,k-m\rrbracket\cup \llbracket k+m,N\rrbracket}\in\cdot
        \,\middle|\,\mathsf{M}_N=k
    \right)
    -
    \ber_{\llbracket1,k-m\rrbracket} \left(\frac{1}{1+\b}\right)
    \otimes
    \ber_{\llbracket k+m,N\rrbracket} \left(\frac{\b}{1+\b}\right) 
}
\le
Ce^{-cm},
\end{equation*} as well as
\begin{equation*}
\qquad\frac{c'}{N}\leq  \TV{
        \mathbb{P}  (\mathsf M_N\in\cdot)
        -
        \mathrm{Unif}_{\llbracket0,N\rrbracket}
    }
    \leq \frac{C'}{N} , 
\end{equation*} 
where we write $\mathrm{Unif}_{\llbracket0,N\rrbracket}$ to denote the uniform measure on the lattice $\llbracket0,N\rrbracket$. In particular, as $N\to\infty$, $\mathsf M_N/N$ converge weakly to the uniform random variable on the interval $[0,1]$.
\end{theorem}

Let us mention that \cite{bryc2024limitCoex} considers the special case of the open TASEP where $\gamma=\delta=q=0$. For the open ASEP, the macroscopic uniformity of the shock location was shown in \cite{wang2023askey}, and the microscopic approximation in finite intervals away from the shock in \cite{SY:openLight}. Let us mention that combined with Theorem~1.6 from our companion paper \cite{SchmidYangShock}, the above theorem provides sharp approximation results in the entire shock region of the open ASEP.

\subsection{Fluctuations and correlations for the open ASEP on the coexistence line}

As a consequence of our characterization results for the open ASEP on the coexistence line, we obtain a detailed understanding of its stationary measure.
We start with a moderate deviation bound on the number of particles.  

\begin{theorem}\label{thm:moderate deviation Coex}
 Assume Liggett's condition \eqref{eq:Liggett's condition} and that the open ASEP is on the coexistence line $\a=\b>1$. Let $\tau =(\tau_1,\ldots,\tau_N)\sim\mu_N^{(q,\b,\b)}$ and $\mathsf{M}_N =\mathsf{M} (\tau)$ be given by \eqref{eq:Mtau}. Then  
there exist constants \(C,c\in(0,\infty)\), depending only on \(q\) and \(\b\), such that, for every \(N\in\mathbb N\), every nonempty set \(I\subseteq\llbracket1,N\rrbracket\), and every \(x\geq0\),
\[
\mu_N^{(q,\b,\b)}\left(
\left|
\sum_{i\in I}\tau_i
-\rho_-\left|I\cap\llbracket1,\mathsf M_N\rrbracket\right|
-\rho_+\left|I\cap\llbracket\mathsf M_N+1,N\rrbracket\right|
\right|
\geq x\sqrt{|I|}
\right)
\leq Ce^{-cx^2}
\]
where we let $\rho_-=1/(1+\b)$ and $\rho_+=\b/(1+\b)$.
\end{theorem}

Next, we characterize the fluctuations of the number of particles in the stationary open ASEP at the coexistence line, extending the results from \cite{bryc2024limitCoex} to the open ASEP.

\begin{theorem} 
\label{thm:ApproxCoex Fluctuations}
 Assume Liggett's condition \eqref{eq:Liggett's condition} and that the open ASEP is on the coexistence line $\a=\b>1$. Let  $\mathsf{M}_N =\mathsf{M}_N (\tau)$ be given by \eqref{eq:Mtau} and set
\[
h(x)=\sum_{i=1}^{\lfloor Nx\rfloor}\tau_i\quad\mbox{ for }x\in[0,1],
\]
where $\tau =(\tau_1,\ldots,\tau_N)\sim\mu_N^{(q,\b,\b)}$.
Let \(U\sim\mathrm{Unif}[0,1]\), and let
\(B=\{B_x\}_{0\le x\le1}\) and \(B'=\{B'_x\}_{0\le x\le1}\) be two
standard Brownian motions, with \(U,B,B'\) independent.  Then, as $N\to\infty$, 
\[
\left(
    \frac{ \mathsf{M}  }{N},
    \left\{
    \frac{1}{\sqrt N}\left(
        h(x) 
        -\rho_-\left(\lfloor Nx\rfloor\wedge  \mathsf{M}  \right)
        -\rho_+\left(\lfloor Nx\rfloor-  \mathsf{M}  \right)_+\right)
    \right\}_{0\le x\le1}
\right)
\Rightarrow
\left(
    U,
    \left\{
        \sigma
        \left(B_{x\wedge U}+B'_{(x-U)_+}\right)
    \right\}_{0\le x\le1}
\right),
\] 
where $\rho_-=1/(1+\b)$, $\rho_+=\b/(1+\b)$, and $\sigma= \sqrt{\b}/(1+\b)$. Here, the convergence is in the space \([0,1]\times D[0,1]\), where \(D[0,1]\) is equipped with the uniform topology induced by the supremum norm $\|f\|_\infty=\sup_{0\le t\le1}|f(t)|$.
\end{theorem} 

As a third application, we have the following characterization of the correlations in the stationary open ASEP on the coexistence line.

\begin{theorem}\label{thm:asymptotics of correlation functions}
 Assume Liggett's condition \eqref{eq:Liggett's condition} and that the open ASEP is on the coexistence line $\a=\b>1$. Suppose $\tau^{(N)}=(\tau_1^{(N)},\dots,\tau_N^{(N)})\sim\mu_N^{(q,\b,\b)}$.
Fix $m\in\mathbb N$ and \(0\le x_1\le\cdots\le x_m\le1\).
Suppose that for each $N\geq m$, there exist integers  
\(j_{1,N},\ldots,j_{m,N}\) satisfying $1\le j_{1,N}<\cdots<j_{m,N}\le N$ and 
$\lim_{N\to\infty}j_{r,N}/N =x_r$ for each $r\in\llbracket1,m\rrbracket$. 
Then
\[
\lim_{N\to\infty}
\mathbb E 
\left[
    \prod_{r=1}^m \tau_{j_{r,N}}^{(N)}
\right]
=
\sum_{\ell=0}^m
    (x_{\ell+1}-x_\ell)\rho_-^\ell\rho_+^{m-\ell},
\]
where we set \(x_0=0\), \(x_{m+1}=1\), $\rho_-=1/(1+\b)$, and $\rho_+=\b/(1+\b)$.
\end{theorem}

\subsection{Related works}

Shocks for exclusion processes can be studied from several perspectives. Fluctuations of shocks for the ASEP on the integer lattice were investigated by Ferrari and Fontes in \cite{FF94current}; see also
\cite{liggett1999stochastic} for a more comprehensive discussion. As shown in \cite{FF94current}, shocks in the asymmetric simple exclusion process can be identified by second class particles, and hence be studied through the environment seen from a
second-class particle; see also 
\cite{FFK94,martin2025convergence} for similar results for homogeneous Bernoulli-product initial conditions for the ASEP on the integers.

For the open ASEP, the matrix product representation by 
Derrida, Evans, Hakim, and Pasquier
\cite{derrida1993exact} provides an algebraic description of the
stationary measure; see also \cite{uchiyama2004asymmetric} for a solution for general parameters by Uchiyama, Sasamoto and Wadati, which is defined using Askey--Wilson orthogonal polynomials from \cite{askey1985some}.   Using Askey--Wilson signed measures,
Wang et al.~establish in \cite{wang2023askey} 
the macroscopic density profile in the shock region, including
the coexistence line. Corresponding local equilibrium results on the 
coexistence line  can be found in \cite{SY:openLight}.
For special choices of the parameters, Bernoulli shock measures
give particularly explicit descriptions, and their evolution
is related to a reverse duality discovered by
Sch{\"u}tz \cite{schuetz2022reverse}, based on special finite dimensional representations from \cite{mallick1997finite} found by Mallick and Sandow. Moreover, a complementary dynamical
perspective for the open TASEP is provided by the last passage
percolation approach from \cite{schmid2023mixing}, and studied in \cite{ES24} to establish a mixing time of order $N^2$ for the open TASEP on a segment of length $N$ on the coexistence line.

A particularly relevant related work is \cite{bryc2024limitCoex} by Bryc,
Najnudel, and Wang on the stationary measure of the
open TASEP on the coexistence line. This is based on the two-layer representation of the open TASEP by Bryc and Zatitskii in \cite{bryc2024two}, introduced first for integrable models on the strip in \cite{barraquand2024stationary}.
The work \cite{bryc2024limitCoex} identifies the shock location from the
particle configuration through the (first) minimum of its
 height function. The authors prove that this location is
asymptotically uniform and establish Gaussian fluctuation of a stationary path after centering around the shock location. Let us point out that their proof uses the two-layer representation of the stationary
measure from \cite{bryc2024two} together with an extension of Denisov's decomposition, which
splits an auxiliary path at its minimum into two conditioned
random walks, and which is also a crucial input for the approximation results in the present work. 

Complementing the aforementioned results on the macroscopic properties of the stationary measure of the open ASEP, local approximation results for the stationary measure of the open ASEP were achieved in \cite{nestoridi2024approximating,yang2025limits}. While an approximation in finite intervals is immediate by Liggett's characterization of the invariant measure for the ASEP on the integers \cite{liggett1975ergodic}, an approximate of the stationary measure of the open ASEP in the bulk of the segment using homogeneous Bernoulli-product measures first appeared in \cite{nestoridi2024approximating} for the special case of the three parameter ASEP, i.e.~for $\gamma=\delta=0$, relying on the representation from \cite{brak2006combinatorial}. Similarly, in \cite{yang2025limits} the law of the open ASEP is near the boundaries approximated using Askey--Wilson half-line measures. 
%Let us also mention that at the microscopic level, the decay of particle density away from the boundary was determined by Sasamoto in \cite{sasamoto2000density} and correlation studied in \cite{uchiyama2005correlation}.   

\subsection{Results in the shock and fan region}

The present manuscript is the second in a series of three articles on sharp approximations of the stationary measure of the open ASEP. For convenience of the reader, and to distinguish the results in this paper from the approximation results on the stationary measure of the open ASEP in the high density, low density and maximal current phase, let us summaries the main results from \cite{SchmidYangShock} and \cite{SchmidYangFan}.

In \cite{SchmidYangShock}, we characterize the stationary measure of the open ASEP on a segment of length $N$ using its two-layer representation. The main results are two-folded. 
First, we introduce a general framework in order to compare the stationary measure of the open TASEP to the stationary measure of the open ASEP. We consider the polynomial factors from \cite{bryc2024stationary} in the Radon-Nikodym derivative with respect to the two-layer representation of the open TASEP, to which we refer as Bryc polynomials, indexed by compositions of $N$. We provide a probabilistic characterization of the Bryc polynomials. More precisely, we show that the Bryc polynomials can be expressed by the moment generating function of a certain inhomogeneous birth-death chain evaluated after $N$ steps. In return, we study this birth-death chain in depth and control its exponential moments, and hence the coefficients in the Bryc polynomials, for typical compositions. 

Second, we show that the aforementioned characterization of the Bryc polynomials can be used to obtain a sharp quantitative approximation of the stationary measure in terms of homogeneous Bernoulli-product measures for the open ASEP in the high and low density phase of the shock region; see also \cite{SchmidYangOpenKPZ} for further applications to the open KPZ equation.

In \cite{SchmidYangFan}, we provide sharp approximations for the stationary measure of the open ASEP in the fan region. We rely on the Enaud-Derrida solution for the matrix product ansatz from \cite{enaud2004large}, and its two-layer formulation from \cite{hegde2024large}, in order to express the height function of a stationary open ASEP in terms of suitable measures on the space of bi-colored Motzkin paths. This allows us to achieve sharp approximation results on the stationary measure of the open ASEP in the fan region via suitable families of lazy simple random walk path measures.

\subsection{Overview of the proof ideas}

Our main task in the analysis of the stationary measure of the open ASEP on the coexistence line is to control the location of the shock and the distribution of the
particles away from it simultaneously and uniformly in the system size. Our key tool is the two-layer representation for the stationary open ASEP from \cite{bryc2024stationary}, and the probabilistic characterization of its polynomial correction factors from \cite{SchmidYangShock} in terms of the exponential moments of an auxiliary birth-death
chain. While in the high and low density phase the first location $T_N$ of the path minimum can be directly controlled, the analysis on the coexistence line is more delicate as the location $T_N$ will be approximately uniformly distributed on the segment. To analyze the underlying birth-death chain, we consider doubled regeneration times, which account for the $v$-shaped random walk paths. Heuristically, a doubled regeneration level is
visited once by each outward arm of the walk. We prove exponential tail bounds
for the gaps between successive doubled regeneration times and
levels. Our main input at this part is a simultaneous renewal argument  that synchronizes the regeneration levels of independent walks, which is of independent interest. These
estimates allow us to provide moment bounds on the Bryc polynomials uniformly in the length $N$ and the path minimum location $T_N$.

We then pass from the auxiliary path minimum $T_N$ to the particle minimum
$\mathsf M_N$ by showing that $\mathsf M_N-T_N$ has 
exponential tails, uniformly in the position $T_N$. Moreover, conditioning on $T_N=k$, we obtain a total variation approximation for the particle
configuration away from the path minimum. Now in order to show that $\mathsf M_N$ is well-approximated by the uniform distribution on the segment, we prove that sufficiently far away from the boundaries, the conditional distribution of
$\mathsf M_N-T_N$ approaches a probability measure on $\mathbb Z$ that
does not depend on $N$ or $T_N$. The law of $\mathsf M_N$ is therefore well-approximated by the
convolution of the law of $T_N$ with this fixed displacement law.

Finally, our sharp approximation estimates reduce the
fluctuation and correlation results to corresponding statements
for independent Bernoulli variables. In particular, after removing a suitable window
around the shock, the invariant measure is close to a Bernoulli product measure up to an
exponentially small error. We then obtain the applications on the fluctuations, moderate deviations and correlations by suitable concentration estimates, and an invariance principle for 
the two exterior regions around the shock location.

\subsection{Outline of the paper}

In Section~\ref{sec:TwoLayer}, we recall the two-layer representation of the stationary measure of the open ASEP from \cite{bryc2024stationary} and its probabilistic characterization from \cite{SchmidYangShock}. In Section~\ref{sec:DoubleRegen}, we study the doubled regeneration times for bi-infinite random walk paths, and verify exponential tail bounds.  In Section~\ref{sec:PathMeasure}, we use the insights on doubled regeneration times in order to bound the moments of the Bryc polynomials when conditioning on the location of the path minimum. In Section~\ref{sec:macroscopicTN}, we verify that the minimum in the two-layer representation for the open ASEP on the coexistence line is asymptotically uniformly distributed, and provide a sharp approximation in total variation. In Section~\ref{sec: proof of main thm}, we show Theorem~\ref{thm:ApproxCoex} on approximating the stationary measure and the shock location for the open ASEP on the coexistence line. In Section~\ref{sec:Applications}, we prove Theorems~\ref{thm:moderate deviation Coex}, \ref{thm:ApproxCoex Fluctuations}, and~\ref{thm:asymptotics of correlation functions} as applications of our main approximation result.

\subsection*{Acknowledgments} 
D.S. was partially supported by the Packard Foundation via Amol Aggarwal's Packard Fellowships for Science and Engineering.
Z.Y. was partially supported by the Miller Institute for Basic Research in Science at the University of California, Berkeley, Ivan Corwin's NSF grants DMS-1811143, DMS-2246576, and the Fernholz Foundation's ``Summer Minerva Fellows'' program. Both authors were supported by Ivan Corwin's Simons Foundation Grant 929852.

\subsection*{AI declaration} 

This project was initiated in fall 2024, and the proofs were written down by the end of July 2026. The authors used AI to assist with editing the manuscript, performing some computations, and checking arguments. All proof ideas in this article were developed by the authors. 
%---------------------------------------------------------------------------------------------------------------
% Section 2
%
%----------------------------------------------------------------------------------------------------------------
\section{Characterization of the stationary measure of open ASEP}  \label{sec:TwoLayer}

In this section, we review the background of the characterization of the stationary measure of the open ASEP developed in \cite{bryc2024stationary} and our companion paper \cite{SchmidYangShock}. 

\subsection{Basic notations}
A \emph{lazy random walk path} is $\bg=(\gamma_0,\dots,\gamma_N)$ such that $\gamma_i\in\mathbb{Z}$ for all $i\in\llbracket0,N\rrbracket$ and $\gamma_i-\gamma_{i-1}\in\{-1,0,1\}$ for all $i\in\llbracket1,N\rrbracket$. We denote by $\lrw_N$ the space of lazy random walk paths $\bg$ with length $N$ started at $\gamma_0=0$, i.e.
\begin{equation*}
    \lrw_N := \left\{ (\gamma_i)_{i \in \llbracket N\rrbracket} \, \colon \, \gamma_0=0 \text{ and } |\gamma_{i}-\gamma_{i-1}| \in \{0,1\} \text{ for all } i \in \llbracket N\rrbracket\right\} .
\end{equation*} 
For each $\bg\in\lrw_N$, we write
\begin{equation*}
    \min(\bg) = \min\left\{\gamma_i \colon i \in \llbracket 0,N \rrbracket\right\} \quad \text{ and } \quad \max(\bg) = \max\left\{\gamma_i \colon i \in \llbracket 0,N \rrbracket\right\} . 
\end{equation*}
A sequence $\bs=(\sigma_0,\dots,\sigma_r)$, where $r\in\mathbb{Z}_{\geq0}$, is called a composition of $N+1$ if $\sigma_i\in\mathbb{N}$ for $i\in\llbracket0,r\rrbracket$ and $\sigma_0+\dots+\sigma_r=N+1$.  The order in which the elements of $\bs$ appear is significant, which distinguishes it from the notion of a partition.  
For a path $\bg\in\lrw_N$, we define the composition associated to $\bg$, denoted by $\bs(\bg)$, to be the composition $\bs(\bg)=(\sigma_0,\dots,\sigma_r)$ of $N+1$ given by $r=\max(\bg)-\min(\bg)\in\mathbb{Z}_{\geq0}$ and
\begin{equation} \label{eq:Signature}
\sigma_i=\#\{j:\gamma_j=\min(\bg)+i\}\in\mathbb{N}\quad\mbox{ for } i\in\llbracket0,r\rrbracket.
\end{equation}
We note that $\bs(\bg)$ is indeed a composition of $N+1$: since $\bg$ is integer valued and has increments $\{-1,0,1\}$, every integer value in $\llbracket\min(\bg), \max(\bg)\rrbracket$ is visited at least once, and hence $\sigma_i\in\mathbb{N}$ for $i\in\llbracket0,r\rrbracket$; since the level sets $\{j:\gamma_j=\min(\bg)+i\}$ in \eqref{eq:Signature} partition the time indices  $\llbracket0,N\rrbracket$, we have $\sigma_0+\dots+\sigma_r=N+1$.

\subsection{The Bryc polynomial} 

% We call $f(z)$ a rational function (over $\mathbb{R}$) if it can be written as $f(z)=g(z)/h(z)$, where $g(z)$ and $h(z)$ are polynomials in $z$ with coefficients in $\mathbb{R}$, and $h\neq0$.  
% For $q\in(0,1)$ and a rational function $f(z)$, we write $\mathcal{D}_qf$ and $\mathcal{E}_qf $ to denote 
% \[
% (\mathcal{D}_qf)(z) =\frac{f(z)-f(qz)}{z(1-q)}\quad\mbox{ and }\quad (\mathcal{E}_qf)(z) =\frac{f(z)-qf(qz)}{ 1-q } .
% \]
% We view $\mathcal{D}_q$ and $\mathcal{E}_q$ as operators acting on the space of rational functions.
% Note that $\mathcal{D}_q$ is known as the Jackson derivative or the $q$-differentiation operator; see Jackson \cite{jackson1909xi}, and, for example, monographs \cite[Chapter 1]{koekoek2010hypergeometric} and \cite{gasper2011basic}. The operator $\mathcal{E}_q$ can be seen as $\mathcal{D}_q\circ z$, where $z$ denotes the operator to multiply $z$ to a rational function $f(z)$.

The work \cite{bryc2024stationary} defines for every composition $\bs$ a polynomial $z \mapsto W_{\bs}(z)$, to which we refer as the Bryc polynomial associated to $\bs$. It allows to characterize the stationary measure of the open ASEP in terms of the two-layer characterization of the open TASEP from \cite{bryc2024two}; see Theorem~\ref{thm:characterization stationary measure ASEP} for a precise statement. 
% \begin{definition}\label{def:Bryc poly}
% For $q\in[0,1)$ and a composition $\bs=(\sigma_0,\dots,\sigma_r)$ of $N+1$, where $N\in\mathbb{Z}_{\geq0}$, we define a rational function $W_{\bs}(z)$ as follows.
% When $q=0$ we set $W_{\bs}(z)=1$, and when $q\in(0,1)$ we set
%     \begin{equation}\label{eq:original def of Bryc polynomial}
%        W_{\bs}(z)=  \frac{(1-q)^{N+1} }{ (q;q)_{N+1}}  (z;q)_{N+2}  \cdot \left( \underbrace{\mathcal{E}_q\circ\dots \circ\mathcal{E}_q}_{\sigma_0-1}\circ\mathcal{D}_q \right)\circ\dots\circ \left( \underbrace{\mathcal{E}_q\circ\dots \circ\mathcal{E}_q}_{\sigma_r-1}\circ\mathcal{D}_q\right)   \left[\frac{1}{(z;q)_1 }\right].
%     \end{equation}  
% Here, we apply $\sigma_r+\dots+\sigma_0=N+1$ many operators $\mathcal{D}_q$ or $\mathcal{E}_q$ sequentially on $1/(z;q)_1=1/(1-z)$ and finally multiply with the prefactors $(z;q)_{N+2}$ and $(1-q)^{N+1}/ (q;q)_{N+1}$.
% Here and below, for $n\in\mathbb{Z}_{\geq0}\cup\{\infty\}$, we use the $q$-Pochhammer symbol $(z;q)_n=\prod_{j=0}^{n-1}\,(1-z q^j)$. 
% An empty composition of operators is the identity.
% \end{definition}
We now state the probabilistic characterization of the Bryc polynomial from \cite{SchmidYangShock}.

\begin{definition}\label{def:characterization measure}
Let $q\in[0,1)$, $N\in\mathbb{Z}_{\geq0}$, and $\boldsymbol{\mathsf{S}}=(\mathsf{S}(0),\dots,\mathsf{S}(N))\in\{0,1\}^{N+1}$ be a $0$-$1$ sequence. We define a time-inhomogeneous birth-death Markov chain $X^{\boldsymbol{\mathsf{S}}}=\{X^{\boldsymbol{\mathsf{S}}}(i)\}_{i=0}^{N+1}$ indexed by this $0$-$1$ sequence $\boldsymbol{\mathsf{S}}$ through $X^{\boldsymbol{\mathsf{S}}}(0)=X^{\boldsymbol{\mathsf{S}}}(1)=0$ and that for each $i\in\llbracket1,N\rrbracket$ and $j\in\llbracket0,i-1\rrbracket$,
\begin{enumerate}
    \item [(\romannum{1})] if $\mathsf{S}(i)=1$, we let $X^{\boldsymbol{\mathsf{S}}}(i+1)-X^{\boldsymbol{\mathsf{S}}}(i)\in\{-1,0\}$ with 
    \begin{equation}\label{eq:etget1}
    \mathbb{P}\left(X^{\boldsymbol{\mathsf{S}}}(i+1)=j-1 \Big{|}X^{\boldsymbol{\mathsf{S}}}(i)=j\right)=\frac{1-q^{j}}{1-q^{i+1}},
    \end{equation}
    \item [(\romannum{2})] if $\mathsf{S}(i)=0$, we let $X^{\boldsymbol{\mathsf{S}}}(i+1)-X^{\boldsymbol{\mathsf{S}}}(i)\in\{0,1\}$ with 
    \begin{equation}\label{eq:etget2}
    \mathbb{P}\left(X^{\boldsymbol{\mathsf{S}}}(i+1)=j+1 \Big{|}X^{\boldsymbol{\mathsf{S}}}(i)=j\right)=\frac{q^{j+1}-q^{i+1}}{1-q^{i+1}}.
    \end{equation}
\end{enumerate} 
We mention that the right sides of \eqref{eq:etget1} and \eqref{eq:etget2} lie in $[0,1]$ since $j\in\llbracket0,i-1\rrbracket$ and $q\in[0,1)$. 
By an induction argument, we have $X^{\boldsymbol{\mathsf{S}}}(i)\in\llbracket0,i-1\rrbracket$ for all $i\in\llbracket1,N+1\rrbracket$, and hence this Markov chain is well-defined. When $q=0$, we use the convention $0^0=1$.
\end{definition}

\begin{definition}\label{def:associated 01 seq}
    Let $N\in\mathbb{Z}_{\geq0}$ and $\bs=(\sigma_0,\dots,\sigma_r)$, $r\in\mathbb{Z}_{\geq0}$ be a composition of $N+1$. 
We define the $0$-$1$ sequence associated to composition $\bs$ to be $\boldsymbol{\mathsf{S}}^{\bs}=(\mathsf{S}^{\bs}(0),\dots,\mathsf{S}^{\bs}(N))\in\{0,1\}^{N+1}$ such that,  for  $i\in\llbracket0,N\rrbracket$,
\begin{equation}\label{eq:01seq}
  \mathsf{S}^{\bs}(i)=\begin{cases}
        1 &\quad\mbox{ if }i\in\{\sigma_r,\sigma_r+\sigma_{r-1},\dots,\sigma_r+\dots+\sigma_1\},\\  
        0 &\quad\mbox{ otherwise}.
    \end{cases} 
\end{equation}
\end{definition}

The following probabilistic characterization of the Bryc polynomial is taken from \cite{SchmidYangShock}.

\begin{theorem}[c.f. Theorem 1.4 in \cite{SchmidYangShock}]\label{thm:prob interpretation of Bryc polynomials}
Let $q\in[0,1)$, $N\in\mathbb{Z}_{\geq0}$ and $\bs=(\sigma_0,\dots,\sigma_r)$, $r\in\mathbb{Z}_{\geq0}$ be a composition of $N+1$. Then the Bryc polynomial $z \mapsto W_{\bs}(z)$ satisfies (up to a deterministic rescaling)
    \begin{equation}\label{eq:rep of Bryc polynomial}
    W_{\bs}(z)=\mathbb{E}\left[ z^{X^{\boldsymbol{\mathsf{S}}^{\bs}}(N+1)}\right] , 
    \end{equation}
where
$\boldsymbol{\mathsf{S}}^{\bs}$ is the $0$-$1$ sequence associated to the composition $\bs$ as in Definition \ref{def:associated 01 seq}, and $X^{\boldsymbol{\mathsf{S}}^{\bs}}=\{X^{\boldsymbol{\mathsf{S}^{\bs}}}(i)\}_{i=0}^{N+1}$ is the birth-death Markov chain indexed by this $0$-$1$ sequence $\boldsymbol{\mathsf{S}}^{\bs}$ as in Definition \ref{def:characterization measure}. 
\end{theorem}

We refer to  \cite{SchmidYangShock} for a proof of equivalence to the original formulation by Bryc in \cite{bryc2024stationary}.

\subsection{Path characterization of the stationary measure of the open ASEP}

The key role of the Bryc polynomials for the open ASEP will become evident from the following path characterization for the stationary measure of the open ASEP. 
Let $\a,\b\geq0$.  We define the probability measure $\mathbb{P}^{(q,\a,\b)}$ on $\lrw_N$ by 
\begin{equation}\label{eq:PNab}
\mathbb{P}^{(q,\a,\b)}(\bg)=\frac{1}{Z_N^{(q,\a,\b)}}\frac{2^{H(\bg)}\b^{\gamma_N}}{(\a\b)^{\min(\bg)}} W_{\bs(\bg)}(\a\b)\quad\mbox{ for }\bg\in\lrw_N,
\end{equation}
where  
\begin{equation}\label{eq:normalization}
 Z_N^{(q,\a,\b)}=\sum_{\bg\in\lrw_N} \frac{2^{H(\bg)}\b^{\gamma_N}}{(\a\b)^{\min(\bg)}}W_{\bs(\bg)}(\a\b).
\end{equation}
Here, we write $H(\bg)$ to denote the number of horizontal steps in a lazy random walk path $\bg$,
\[
H(\bg)=\#\{i\in\llbracket1,N\rrbracket, \gamma_i=\gamma_{i-1}\} \quad\mbox{ for }\bg\in\lrw_N.
\]
Let $W_{\bs(\bg)}(\a\b)$ be the Bryc polynomial indexed by the composition $\bs(\bg)$ of $N+1$ associated to $\bg$ (as in \eqref{eq:Signature}), evaluated at $z=\a\b$. Since $\min(\bg)\leq\min(0,\gamma_N)$, we know that 
$ 2^{H(\bg)}\b^{\gamma_N} (\a\b)^{-\min(\bg)}  W_{\bs(\bg)}(\a\b)$ (and hence $Z_N^{(q,\a,\b)}$ given by \eqref{eq:normalization})
is a polynomial in $\a$ and $\b$ with non-negative coefficients depending on $q$.  When $\a=0$ or $\b=0$, \eqref{eq:PNab} is well-defined with the convention that $0^0=1$.  We have the following relation to the stationary measure of the open ASEP.

\begin{theorem}[c.f. Theorem 1.1 and (3.2), (3.3) in \cite{bryc2024stationary}]\label{thm:characterization stationary measure ASEP}
Consider the open ASEP on the lattice $\llbracket1,N\rrbracket$ with jump rates $\alpha,\beta,\gamma,\delta$ and $q$ satisfying \eqref{eq:conditions open ASEP}.  Assume that Liggett's condition \eqref{eq:Liggett's condition} holds, and recall that $\a=(1-\alpha)/\alpha$ and $\b=(1-\beta)/\beta$.  We denote the stationary measure of this model by $\mu_N^{(q,\a,\b)}$.
Consider the random lazy random walk path $\bg\in\lrw_N$ with law $\mathbb{P}^{(q,\a,\b)}$ defined as in \eqref{eq:PNab}. Suppose $\kappa_i$, $i\in\llbracket1,N\rrbracket$ are i.i.d. Bernoulli random variables with mean $1/2$ that are independent of $\bg$. 
We define random variables $\tau_i\in\{0,1\}$, $i\in\llbracket1,N\rrbracket$ by  
\[
    \tau_i=\begin{cases}
        1 & \mbox{ if } \gamma_i-\gamma_{i-1}=1,\\
        \kappa_i & \mbox{ if } \gamma_i-\gamma_{i-1}=0 ,\\ 
        0 & \mbox{ if } \gamma_i-\gamma_{i-1}=-1.
    \end{cases}
\]
Then $(\tau_1,\dots,\tau_N)\in\{0,1\}^N$ has the law of the stationary measure $\mu_N^{(q,\a,\b)}$ of the open ASEP.
\end{theorem}

 Next, we record the following useful decomposition of the path measure for the open TASEP from \cite{bryc2024limitCoex}, which we will use at multiple points throughout this article.

\begin{lemma}[cf. Section 3 in \cite{bryc2024limitCoex}, especially Proposition 3.2 and (3.1), (3.2)]\label{lem:decomposition TASEP}
We define the observable $T_N$ on the space $\lrw_N$ by
\[T_N(\bg)=\min\{i\in\llbracket0,N\rrbracket:\gamma_i=\min(\bg)\}\quad\mbox{ for } \bg\in\lrw_N.\]
Assume $\a,\b>0$ and consider the probability measure $\mathbb{P}^{(0,\a,\b)}$ on $\lrw_N$ as in \eqref{eq:PNab}. Then 
\begin{equation}\label{eq: dist of TN}
    \mathbb{P}^{(0,\a,\b)}(T_N = m) = \begin{cases}
\displaystyle\frac{\a}{4C_{N,\a,\b}}w_{\a}^{m-1}w_{\b}^{N-m}p^{(\a)}_{m-1}p^{(\b)}_{N-m}  & \quad\mbox{ if }
m\in\llbracket1,N\rrbracket,\\ 
\displaystyle\frac{1}{C_{N,\a,\b}}w_{\b}^Np^{(\b)}_N  &\quad\mbox{ if } m = 0.
\end{cases}
\end{equation}
Here, we write
\begin{align}
       C_{N,\a,\b} &= \sum_{m=1}^{N}\frac{\a}{4} w_{\a}^{m-1}w_{\b}^{N-m}p_{m-1}^{(\a)}p_{N-m}^{(\b)} + w_{\b}^Np^{(\b)}_N,  \label{eq:C N a b}\\
       w_x &= \frac{1}{4}\left(x+x^{-1}+2\right)\quad\mbox{ for }x>0,\label{eq:def of wx} 
\end{align} 
and we denote by $p_{\ell}^{(x)}$ the probability that a $\nu_x$-random walk started at $0$ does not reach level $-1$ during the first $\ell$ steps, where $x>0$ and $\ell\in\mathbb{Z}_{\geq0}$.
In particular, \(\mathbb{P}^{(0,\a,\b)}(T_N=m)>0\) for every $m\in\llbracket0,N\rrbracket$.

Furthermore, for every $m\in\llbracket0,N\rrbracket$, conditionally on \(\{T_N=m\}\), the path $(\gamma_{m+i}-\gamma_m )_{i\in\llbracket0,N-m\rrbracket}$ has the law of a \(\nu_{\b}\)-random walk started at \(0\), conditioned to stay non-negative up to step \(N-m\). 
If \(m\geq1\), then conditionally on \(\{T_N=m\}\) one also has $\gamma_{m-1}-\gamma_m=1$,
and the reversed left path $(\gamma_{m-1-i}-\gamma_{m-1} )_{i\in\llbracket0,m-1\rrbracket}$
has the law of a \(\nu_{\a}\)-random walk started at \(0\), conditioned to stay non-negative up to step \(m-1\).  For \(m\geq1\), this reversed left path and the right path $(\gamma_{m+i}-\gamma_m )_{i\in\llbracket0,N-m\rrbracket}$ are conditionally independent given \(\{T_N=m\}\).  
\end{lemma}

\section{Doubled regeneration times}\label{sec:DoubleRegen}

As a main input for the approximation results in the coexistence line, we will use the notion of doubled regeneration times for bi-infinite random walk paths. Intuitively, doubled regeneration times ask for a bi-infinite v-shaped path to have a regeneration time at the same level; see Figure~\ref{fig:doubled-regeneration}. We will require sharp exponential tail bounds, similar to the analysis in \cite{DPZ96}, suitably adjusted for doubled regeneration times.  

\subsection{Definition of doubled regeneration times}

We first recall the notion of regeneration times and then define the doubled regeneration times for an arbitrary path. 
 
  \begin{definition}\label{def:regeneration}
Assume $N\in\mathbb{N}\cup\{\infty\}$. Let $\{Y_i\}_{i\in\llbracket0,N\rrbracket}$ be a path on $\mathbb{Z}$ with
$Y_i-Y_{i-1}\in\{-1,0,1\}$ for all $i\in\llbracket1,N\rrbracket$.
We call $n\in\llbracket0,N\rrbracket$ a \emph{regeneration time} for this path if
$Y_i<Y_n$ for all $i<n$ and $Y_i>Y_n$ for all $i>n$.
When $n\in\llbracket0,N\rrbracket$ is a regeneration time, we call $Y_n$ a \emph{regeneration level}.
\end{definition}

\begin{figure}[]
    \centering
    \begin{tikzpicture}[scale=0.75, >=stealth]

        \coordinate (A) at (-6,5);
        \coordinate (B) at (-5,4);
        \coordinate (C) at (-4,5);
        \coordinate (D) at (-3,4);
        \coordinate (E) at (-2,3);
        \coordinate (F) at (-1,2);
        \coordinate (G) at (0,1);
        \coordinate (H) at (1,2);
        \coordinate (I) at (2,1);
        \coordinate (J) at (3,1);
        \coordinate (K) at (4,2);
        \coordinate (L) at (5,3);
        \coordinate (M) at (6,4);
        \coordinate (N) at (7,5);
        \coordinate (O) at (8,4);
        \coordinate (P) at (9,5);

        \draw[->] (-6.5,-0.5) -- (9.8,-0.5)
            node[right] {};
        \draw[->] (-6.5,-0.5) -- (-6.5,6)
            node[above] {};

        \draw[red, dashed, thick] (-6.3,3) -- (9.5,3);
        \node[red, left] at (-6.5,3) {$h$};

        \draw[thick]
            (A) -- (B) -- (C) -- (D) -- (E) --
            (F) -- (G) -- (H) -- (I) -- (J) --
            (K) -- (L) -- (M) -- (N) -- (O) -- (P);

        \fill[red] (E) circle (3pt);
        \fill[red] (L) circle (3pt);

        \draw[dotted] (-2,3) -- (-2,-0.5);
        \draw[dotted] (5,3) -- (5,-0.5);

        \node[below] at (-2,-0.5) {$-m$};
        \node[below] at (5,-0.5) {$n$};
        \node[below] at (0,-0.5) {$0$};

        \node[red, above right] at (5.5,3)
            {$Y_{-m}=Y_n=h$};

    \end{tikzpicture}
    \caption{Visualization of a doubled regeneration time $(m,n)$.
    The path stays strictly below the doubled regeneration
    level $h=Y_{-m}=Y_n$ for $-m<i<n$, and strictly above
    this level for $i<-m$ and $i>n$.}
    \label{fig:doubled-regeneration}
\end{figure}
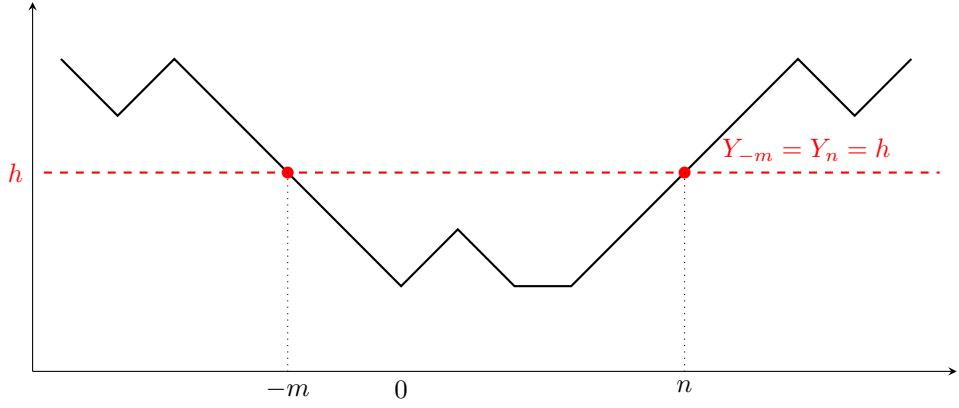

\begin{definition}\label{def:doubled regeneration time}
Assume $M,N\in\mathbb{Z}_{\geq0}\cup\{\infty\}$ and $\{Y_i\}_{i=-M}^{N}$ be an arbitrary path on $\mathbb{Z}$ with each step in $\{-1,0,1\}$. A pair $(m,n)\in\mathbb{N}^2$ is called a \emph{doubled regeneration time} if $Y_{-m}=Y_n$ and
\[
Y_{i}<Y_{-m}=Y_n<Y_j \quad\text{for all } -m<i<n\mbox{ and } j\in\mathbb{Z}\setminus\llbracket-m,n\rrbracket.
\] In this case, we call $Y_{-m}=Y_n$ a \emph{doubled regeneration level}.
\end{definition}

Next, we define the positive and nonnegative $\nu_{\b}$-random walk as well as the bi-infinite $V$-shaped $\nu_{\b}$ random walk. 

We will write $\nu_{\b}$ for $\b>0$ to denote the probability measure on $\{-1,0,1\}$ given by
    \[
    \nu_{\b}(1)=\frac{\b}{\b+\b^{-1}+2},\quad\nu_{\b}(0)=\frac{2}{\b+\b^{-1}+2},\quad \nu_{\b}(-1)=\frac{\b^{-1}}{\b+\b^{-1}+2},
    \]
and $\nu_{\b}^{\otimes k}$ for $k\in\mathbb{N}$ to denote the product measure (of $k$ independent copies of $\nu_{\b}$) on $\{-1,0,1\}^k$. 
A \(\nu_{\b}\)-random walk means a random walk with i.i.d. increments of law \(\nu_{\b}\).

\begin{definition}\label{def:positive lazy random walk}
Assume $\b>1$ and $N\in\mathbb{N}\cup\{\infty\}$. Consider a $\nu_{\b}$-random walk $\{S_i\}_{i=0}^{N}$ started at $S_0=0$. Let
\[ 
E_{+}=\{S_1=1\text{ and }S_i\geq 1\text{ for all }i\in\llbracket1,N\rrbracket\},\qquad
E_{\geq0}=\{ S_i\geq 0\text{ for all }i\in\llbracket0,N\rrbracket\}
\]
Observe that  $E_+$, $E_{\geq0}$ have strictly positive probabilities. 
We call the process $\{S_i\}_{i=0}^{N}$ conditioned on $E_+$ (respectively, $E_{\geq0}$) the \emph{positive (respectively, nonnegative) $\nu_{\b}$-random walk with $N$ steps}. If $N=\infty$, we call the conditioned processes the \emph{positive (respectively, nonnegative) $\nu_{\b}$-random walk with infinitely many steps}.  
\end{definition} 

\begin{definition}\label{def:V shaped RW}
    Assume $\b>1$.
Let $\{R_i\}_{i=0}^{\infty}$ be a nonnegative $\nu_{\b}$-random walk with infinitely many steps started at $R_0=0$. Let $\{L_i\}_{i=0}^{\infty}$ be a positive $\nu_{\b}$-random walk with infinitely many steps started at $L_0=0$. Assume that \(\{L_i\}_{i=0}^{\infty}\) and \(\{R_i\}_{i=0}^{\infty}\) are independent.
We define a bi-infinite process $\{Y_i\}_{i=-\infty}^{\infty}$ by letting $Y_i=R_i$ for each $i\geq0$ and $Y_i=L_{-i}$ for each $i\leq0$. We call $\{Y_i\}_{i=-\infty}^{\infty}$ a (bi-infinite) $V$-shaped $\nu_{\b}$ random walk started at $Y_0=0$. Note that $Y_0=0$, $Y_{-1}=1$, $Y_i\geq0$ for all $i\geq0$, and $Y_i\geq1$ for all $i\leq-1$.
\end{definition}

In the following, a probability measure $\nu$ on $\mathbb{R}_{\geq0}$ is called subexponential if there exist constants $C,c>0$ such that $\nu([y,\infty))\leq Ce^{-cy}$ for all $y\in\mathbb{R}_{\geq0}$. Similarly,  a random variable $X$ on $\mathbb{R}_{\geq0}$ is subexponential if its distribution is subexponential, and for $C,c>0$, we write $X\in\mathsf{Subexp}(C,c)$ if $\mathbb{P}(X\geq y)\leq Ce^{-cy}$ for all $y\in\mathbb{R}_{\geq0}$. The following is our main technical input on regeneration times. 

\begin{proposition}\label{thm:doubled regeneration}
Assume $\b>1$, and let $\{Y_i\}_{i=-\infty}^{\infty}$ be the (bi-infinite) $V$-shaped $\nu_{\b}$ random walk started at $Y_0=0$ as in Definition \ref{def:V shaped RW}. Then, almost surely, there exist infinitely many doubled regeneration times
$(m_i,n_i)$, $i\in\mathbb{N}$ with $0<m_1<m_2<\dots$ and $0<n_1<n_2<\dots$. 
Furthermore, the random variables
\[
    (m_1,n_1,Y_{n_1}),  \quad   (m_2-m_1,n_2-n_1,\,Y_{n_2}-Y_{n_{1}}),\quad       (m_3-m_2,n_3-n_2,\,Y_{n_3}-Y_{n_{2}}),\quad\dots              
\]
are mutually independent; and the random variables
\[
(m_i-m_{i-1},n_i-n_{i-1},\,Y_{n_i}-Y_{n_{i-1}}) \quad\mbox{for } i\geq2,
\]
are i.i.d. Moreover, there exist constants $C,c>0$ depending only on $\b$, such that for every $y\in\mathbb{N}$,
\begin{equation*}
    \begin{split}
        & \mathbb{P}(n_1\geq y)\leq Ce^{-cy},\quad \mathbb{P}(n_2-n_1\geq y)\leq Ce^{-cy}, \\
        &\mathbb{P}(m_1\geq y)\leq Ce^{-cy},  \quad \mathbb{P}(m_2-m_1\geq y)\leq Ce^{-cy},\\
        &\mathbb{P}(Y_{n_1}\geq y)\leq Ce^{-cy},\quad \mathbb{P}(Y_{n_2}-Y_{n_1}\geq y)\leq Ce^{-cy}.
    \end{split}
\end{equation*} 
In particular, the above differences are subexponential random variables.
\end{proposition}

Proposition~\ref{thm:doubled regeneration} will follow from a general result on the simultaneous doubled regeneration times of random sequences with subexponential increments. Despite a comprehensive treatment of regeneration times in the literature, 
to our best knowledge, the precise form of the below statement (that the differences between the simultaneous regeneration times remain subexponential) was not previously shown.

\begin{theorem}\label{thm:find reg times 2}
Let $d\in\mathbb{N}$. For each $i\in\llbracket1,d\rrbracket$, let
$(\Upsilon^{(i)}_j,H^{(i)}_j)_{j=0}^{\infty}$ be a random sequence in
$\mathbb{Z}_{\geq0}\times\mathbb{Z}_{\geq0}$ satisfying
\[
0=\Upsilon^{(i)}_0<\Upsilon^{(i)}_1<\Upsilon^{(i)}_2<\cdots
\quad\mbox{and}\quad
0=H^{(i)}_0<H^{(i)}_1<H^{(i)}_2<\cdots .
\]
We make the following assumptions.
\begin{enumerate}
    \item[(1)] The random sequences
    $(\Upsilon^{(i)}_j,H^{(i)}_j)_{j=0}^{\infty}$, $i\in\llbracket1,d\rrbracket$, are independent.
    \item[(2)] For each fixed $i\in\llbracket1,d\rrbracket$, the increment random vectors
     $(\Upsilon^{(i)}_k-\Upsilon^{(i)}_{k-1},
    H^{(i)}_k-H^{(i)}_{k-1} )$, $k\in\mathbb{N}$ 
    are i.i.d. as $k$ varies. Moreover, $H^{(i)}_1$ has full support on
    $\mathbb{N}$.
    \item[(3)] There exist constants $C_0,c_0\in(0,\infty)$ such that for  
    $i\in\llbracket1,d\rrbracket$,  
    $\mathbb E [e^{c_0\Upsilon^{(i)}_1} ]\leq C_0$
    and
    $\mathbb E [e^{c_0H^{(i)}_1} ]\leq C_0 $. 
\end{enumerate} 
Then
$\cap_{i\in\llbracket1,d\rrbracket}\{H^{(i)}_k:k\in\mathbb{Z}_{\geq0}\}$
has infinitely many elements almost surely. On this almost sure event, write these elements as
$0=H_0<H_1<H_2<\cdots$. For each $i\in\llbracket1,d\rrbracket$ and
$\ell\in\mathbb{Z}_{\geq0}$, let $t_{i,\ell}$ be the unique element of
$\mathbb{Z}_{\geq0}$ such that $H^{(i)}_{t_{i,\ell}}=H_{\ell}$. Then
$0=t_{i,0}<t_{i,1}<t_{i,2}<\cdots$ for each $i$, and the random vectors
\[
\left(
\Upsilon^{(1)}_{t_{1,k}}-\Upsilon^{(1)}_{t_{1,k-1}},\dots,
\Upsilon^{(d)}_{t_{d,k}}-\Upsilon^{(d)}_{t_{d,k-1}},
H_k-H_{k-1}
\right),
\quad k\in\mathbb{N},
\]
are i.i.d. There exist $C,c\in(0,\infty)$ depending only on
$C_0,c_0,d$, and $p=\min_{i\in\llbracket1,d\rrbracket}\mathbb P(H^{(i)}_1=1)$,  such that
\[
\mathbb E\left[e^{c\Upsilon^{(i)}_{t_{i,1}}}\right]\leq C
\quad\mbox{for every }i\in\llbracket1,d\rrbracket,
\quad\mbox{and}\quad
\mathbb E\left[e^{cH_1}\right]\leq C .
\]
\end{theorem}

We will show in the Section~\ref{sec:SimplifyRenewal} that Proposition~\ref{thm:doubled regeneration} follows when assuming Theorem \ref{thm:find reg times 2}, and we give the proof of Theorem \ref{thm:find reg times 2} in Section~\ref{sec:RenewalThm}.

\subsection{Preliminaries on regeneration times for random walk paths.}\label{sec:PrelimsRegeneration}

In preparation for the proofs of Proposition~\ref{thm:doubled regeneration} and Theorem \ref{thm:find reg times 2}, we recall three results from \cite{SchmidYangShock} on (standard) regeneration times for random walk paths. The first one collects several standard properties of subexponential random variables.

\begin{lemma}[c.f. Lemma A.2 in \cite{SchmidYangShock}]\label{lem:basic exponential tails}
The following statements hold:
    \begin{enumerate}
        \item [\textup{(\romannum{1})}] A non-negative random variable $X$ is subexponential if and only if $\mathbb{E}[e^{\theta X}]<\infty$ for some $\theta>0$. Specifically, if $X\in\mathsf{Subexp}(C,c)$, then $\mathbb E[e^{\theta X}]\leq1+ \theta C/(c-\theta)<\infty$ for every $\theta\in(0,c)$; and conversely, if $\mathbb E[e^{\theta X}]\leq M$ for some $\theta,M\in(0,\infty)$, then $X\in\mathsf{Subexp}(M,\theta)$. 
        \item [\textup{(\romannum{2})}] If $0\leq X\leq Y$ and $Y$ is subexponential, then $X$ is also subexponential.
        Specifically, if $0\leq X\leq Y$ and $Y\in\mathsf{Subexp}(C,c)$, then $X\in\mathsf{Subexp}(C,c)$. 
        \item [\textup{(\romannum{3})}] If $X$ and $Y$ are both subexponential, then $X+Y$ is subexponential. Specifically, if $X\in\mathsf{Subexp}(C_1,c_1)$ and $Y\in\mathsf{Subexp}(C_2,c_2)$, then $X+Y\in\mathsf{Subexp}(C_1+C_2,\min(c_1,c_2)/2)$. 
        \item [\textup{(\romannum{4})}] If $X_1,X_2,\dots$ are i.i.d. subexponential variables, and $K$ is a subexponential variable with values in $\mathbb{Z}_{\geq0}$, not necessarily independent of $X_1,X_2,\dots$, then $X_1+\dots+X_K$ is also subexponential. Specifically, if $X_1,X_2,\dots$ are i.i.d., $X_1\in\mathsf{Subexp}(C_0,c_0)$, and $K$ is $\mathbb{Z}_{\geq0}$-valued with $K\in\mathsf{Subexp}(C_K,c_K)$, then there exist $C,c>0$ depending only on $C_0,c_0,C_K,c_K$, such that $X_1+\dots+X_K\in\mathsf{Subexp}(C,c)$.
    \end{enumerate} 
\end{lemma} 

Next, we require a subexponential bound on the differences between regeneration times in a positive $\nu_{\b}$-random walk. 

\begin{proposition}[c.f. Proposition A.4 in \cite{SchmidYangShock}]\label{prop:regeneration}
    Assume $\b>1$. Let $\{Y_i\}_{i=0}^{\infty}$ be the positive $\nu_{\b}$-random walk started at $Y_0=0$ with infinitely many steps as in Definition \ref{def:positive lazy random walk}. Then, almost surely, there exist infinitely many regeneration times for this path given by random numbers $0=n_0<n_1<n_2<\cdots$. Furthermore, the increments $(n_i-n_{i-1},Y_{n_i}-Y_{n_{i-1}})$ for $i\in\mathbb{N}$ are i.i.d., and both coordinates are subexponential. More precisely, there exist constants $C,c>0$, depending only on $\b$, such that for each $y\in\mathbb{N}$, we have
    $\mathbb{P}(n_1\geq y)\leq Ce^{-cy}$ and $\mathbb{P}(Y_{n_1}\geq y)\leq Ce^{-cy}$.
    Moreover, we know that $\mathbb{P}((n_1, Y_{n_1})=(1,1))=\nu_{\b}(1)>0$.
\end{proposition}

Finally, we will use the following lemma, which states a standard monotone coupling for $\nu_{\b}$-random walks. 
\begin{lemma}[c.f. Lemma 3.2 in \cite{SchmidYangShock}]\label{lem:basic hitting monotone coupling}
    Assume $\b>1$, $N\in\mathbb{N}\cup\{\infty\}$ and $h\in\mathbb{Z}_{\geq1}$. Let $\{S_i\}_{i=0}^N$ be the $\nu_{\b}$-random walk started at $S_0=0$. 
    Let $E_{N}^h=\{S_i\geq-h+1\mbox{ for } i\in\llbracket0,N\rrbracket\}$. Then 
    $\mathbb{P}(E_{N}^h)\downarrow\mathbb{P}(E_{\infty}^h)=1-\b^{-2h}$ as $N\to\infty$.
    Denote by $\{S_i'\}_{i=0}^N$ the process $\{S_i\}_{i=0}^N$ conditioned on $E_{N}^{h}$. There exists a coupling such that $S'_i\geq S_i$ almost surely for $i\in\llbracket0,N\rrbracket$.
    Furthermore, fix $\kappa\in(0,(\b-1)/(\b+1))$, then there exist constants $C,c>0$ depending only on $\b$ and $\kappa$ such that for each $N\geq m\geq1$, we have
    $\mathbb{P}(S'_{m}\leq\kappa m)\leq \mathbb{P}(S_{m}\leq\kappa m)\leq Ce^{-cm}$.
\end{lemma}

\subsection{Subexponential doubled regeneration times for bi-infinite  V-shaped random walks}\label{sec:SimplifyRenewal}

We will now show that the differences between doubled regeneration times for bi-infinite  V-shaped random walks are subexponential, assuming that Theorem \ref{thm:find reg times 2} holds.

\begin{proof}[Proof of Proposition~\ref{thm:doubled regeneration} assuming Theorem \ref{thm:find reg times 2}] 
Write $R_i=Y_i$ and $L_i=Y_{-i}$ for $i\in\mathbb Z_{\geq0}$. Then $R$ is a $\nu_\b$-random walk started at $0$ conditioned on
$R_i\geq0$ for all $i\geq0$, and $L$ is a positive $\nu_\b$-random walk with infinitely many steps as in Definition
\ref{def:positive lazy random walk}. The two one-sided paths are independent. Observe  that, for any $h, m,n\in\mathbb N$ satisfying $L_m=R_n=h$, $(m,n)$ is a doubled regeneration time for $Y$ if and only if $m$ is a regeneration time of $L$ and $n$ is a regeneration time of $R$, where regeneration time is as in Definition \ref{def:regeneration}. Let
\[
\zeta=\sup\{i\in\mathbb Z_{\geq0}:R_i=0\}\quad\mbox{and}\quad M=\max_{0\leq i\leq \zeta} R_i .
\]
We claim that $\zeta$ and $M$ are subexponential with constants depending only on $\b$. Indeed, $M\leq\zeta$ and hence it suffices to show $\zeta$ is subexponential. Let $\boldsymbol{S}=\{S_n\}_{n=0}^{\infty}$ be an ordinary $\nu_{\b}$-random walk started at $S_0=0$, and let $E_{\geq0}=\{S_n\geq0\mbox{ for all }n\geq0\}$. Then for $\kappa=(\b-1)/(2(\b+1))$ and $y\in\mathbb N$, by Lemma
\ref{lem:basic hitting monotone coupling},
\[ 
\mathbb{P}(\zeta\geq y)
 =
\frac{\mathbb{P}(E_{\geq0}\cap\{\exists k\geq y:S_k=0\}) }{\mathbb{P}(E_{\geq0})}  
\leq 
\frac{1}{\mathbb{P}(E_{\geq0})}\sum_{k\geq y}\mathbb{P}(S_k=0)    
 \leq
\frac{1}{\mathbb{P}(E_{\geq0})}
\sum_{k\geq y}\mathbb{P}(S_k\leq \kappa k)
\leq C_1e^{-c_1y}, 
\]
for constants $C_1,c_1>0$ depending only on $\b$. Define the shifted post-$\zeta$ path $Z_i=R_{\zeta+i}$ for $i\in\mathbb Z_{\geq0}$. We claim that $L$ and $Z$ are independent positive $\nu_\b$-random walks and are independent of $(\zeta,M)$. Indeed, $L$ is independent of $R$ and hence is independent of $Z,\zeta$, and $M$. Let $E_+=\{S_n\geq1\mbox{ for all }n\in\mathbb N\}$. For each $m\in\mathbb{Z}_{\geq0}$ and  $\gamma\in\lrw_m$ such that
$\mathbb P\left(\zeta=m,\ (R_0,\ldots,R_m)=\gamma\right)>0$, and each cylinder event $F$,
\begin{equation}  \label{eq:jntyjuykyi}
 \mathbb P\left(\zeta=m,\ (R_0,\ldots,R_m)=\gamma,\ Z\in F\right)  
 =
 \mathbb P\left((S_0,\ldots,S_m)=\gamma\right)
      \mathbb P\left(\boldsymbol{S}\in E_+\cap F\right)/ \mathbb P\left(E_{\geq0}\right)  . 
\end{equation}
Now conditioning on $\{\zeta=m,(R_0,\ldots,R_m)=\gamma\}$, the law of $Z$ is the law of $\boldsymbol{S}$ conditioned on $E_+$, and we conclude the claim. 

By Proposition \ref{prop:regeneration}, for the two positive walks $L$ and $Z$, there exist regeneration times
\[
0=\mathfrak m_0<\mathfrak m_1<\mathfrak m_2<\cdots
\quad\mbox{for }L,
\qquad
0=\mathfrak n_0<\mathfrak n_1<\mathfrak n_2<\cdots
\quad\mbox{for }Z,
\]
such that the increment random vectors $(\mathfrak m_j-\mathfrak m_{j-1},L_{\mathfrak m_j}-L_{\mathfrak m_{j-1}})$ and $(\mathfrak n_j-\mathfrak n_{j-1},Z_{\mathfrak n_j}-Z_{\mathfrak n_{j-1}})$ 
are i.i.d. in $j\in\mathbb N$ on each side, the two sides are independent, and both coordinates are subexponential. Let
$H^L_j=L_{\mathfrak m_j}$ and $H^Z_j=Z_{\mathfrak n_j}$ for $j\in\mathbb Z_{\geq0}$, we claim that  
\begin{equation}\label{eq:yfk,ik}
\left(\Upsilon^{(1)}_j,H^{(1)}_j\right)=\left(\mathfrak m_j,H^L_j\right) 
\quad\mbox{and}\quad
\left(\Upsilon^{(2)}_j,H^{(2)}_j\right)=\left(\mathfrak n_j,H^Z_j\right)
\end{equation}
satisfy the hypotheses of Theorem \ref{thm:find reg times 2} for $d=2$. The only hypothesis left to be verified is the full-support one. Let $W$ be a positive infinite $\nu_\b$-random walk and $H_1(W)$ be its first positive regeneration level, then
\[
\mathbb P(H_1(W)=r)\geq \mathbb P\left(W_{2k-1}=W_{2k}=k\mbox{ for } k\in\llbracket1, r-1\rrbracket, W_{2r-1}=r, 
W_i>r\mbox{ for }i>2r-1\right)>0
\quad\mbox{for }r\in\mathbb N ,
\]
where we used Lemma \ref{lem:basic hitting monotone coupling}.

We now apply Theorem \ref{thm:find reg times 2} to \eqref{eq:yfk,ik} and find $0=H_0<H_1<H_2<\cdots$ with
\[
\left\{H^L_j:j\in\mathbb Z_{\geq0}\right\}
\cap
\left\{H^Z_j:j\in\mathbb Z_{\geq0}\right\}=\left\{H_j:j\in\mathbb Z_{\geq0}\right\}.
\]
Let $u_k,v_k$, $k\geq0$ be the unique indices
such that $H^L_{u_k}=H_k=H^Z_{v_k}$. We also know from Theorem \ref{thm:find reg times 2} that 
\[
\Phi_k=
\left(
\mathfrak m_{u_k}-\mathfrak m_{u_{k-1}},
\mathfrak n_{v_k}-\mathfrak n_{v_{k-1}},
H_k-H_{k-1}
\right),
\quad k\in\mathbb N,
\]
are i.i.d. with each coordinate being subexponential with constants depending only on $\b$. Let $K=\min\{k\in\mathbb N:H_k>M\}$, which is subexponential since $K\le M+1$ and $M$ is subexponential. Let $h_i=H_{K+i-1}$, $m_i=\mathfrak m_{u_{K+i-1}}$, and $n_i=\zeta+\mathfrak n_{v_{K+i-1}}$ for $i\in\mathbb{N}$. The monotonicity of the indices $u_k,v_k$ and of the regeneration times $\mathfrak m_j,\mathfrak n_j$ gives $0<m_1<m_2<\cdots$ and $0<n_1<n_2<\cdots$. We claim that $(m_i,n_i)$ is a doubled regeneration time for the bi-infinite $V$-shaped walk $\{Y_i\}_{i=-\infty}^{\infty}$ with level $Y_{-m_i}=Y_{n_i}=h_i$.  
Indeed, $L_{m_i}=h_i$ and $Z_{\mathfrak n_{v_{K+i-1}}}=h_i$, so $R_{n_i}=h_i$. Since $m_i$ is a regeneration time of $L$, it remains only to check that $n_i$ is a regeneration time of $R$. If $0\leq r\leq\zeta$, then $R_r\leq M<h_i$; if $\zeta<r<n_i$, then $R_r=Z_{r-\zeta}<h_i$ because $\mathfrak n_{v_{K+i-1}}$ is a regeneration time of $Z$; and if $r>n_i$, then $R_r=Z_{r-\zeta}>h_i$ for the same reason.

It remains to prove the claim on independence and exponential tails.  The i.i.d.~sequence $(\Phi_k)_{k\in\mathbb N}$ is independent of $(\zeta,M)$.
Let $\mathcal F_k=\sigma(\zeta,M,\Phi_1,\ldots,\Phi_k)$ for $k\in\mathbb Z_{\geq0}$, then $K$ is a stopping time for this filtration, and  hence $(\Phi_{K+j})_{j\in\mathbb N}$ is independent of $\mathcal F_K$ and has the same law as $(\Phi_j)_{j\in\mathbb N}$.
Write $\Phi_k=(A_k,B_k,D_k)$ for $k\in\mathbb N$, then  
\begin{equation}\label{eq:ntujk}
(m_1,n_1,h_1)
=
\left(
\sum_{k=1}^{K}A_k,\,
\zeta+\sum_{k=1}^{K}B_k,\,
\sum_{k=1}^{K}D_k
\right)
\end{equation}
is $\mathcal F_K$-measurable, while for $i\geq2$,
$(m_i-m_{i-1},n_i-n_{i-1},h_i-h_{i-1})
=
(A_{K+i-1},B_{K+i-1},D_{K+i-1})$.
Hence
\[
(m_1,n_1,h_1),\quad
(m_2-m_1,n_2-n_1,h_2-h_1),\quad
(m_3-m_2,n_3-n_2,h_3-h_2),\quad\ldots
\]
are mutually independent, and the vectors after the first are i.i.d. Finally, since $A_1,B_1,D_1$, $\zeta$, and $K$ are subexponential, using Lemma \ref{lem:basic exponential tails} \textup{(\romannum{4})} and then \textup{(\romannum{3})} for $n_1$, we know that   $(m_1,n_1,h_1)$ given by \eqref{eq:ntujk} are subexponential. The increments $m_2-m_1\overset{d}{=}A_1$, $n_2-n_1\overset{d}{=}B_1$, and $h_2-h_1\overset{d}{=}D_1$ are also subexponential.  \end{proof}

\subsection{Simultaneous subexponential regeneration times: Proof of Theorem \ref{thm:find reg times 2}}\label{sec:RenewalThm}

In the following, our goal is to show Theorem \ref{thm:find reg times 2} using generating functions for regeneration times. 

%\begin{proof}[Proof of Theorem \ref{thm:find reg times 2}]
To this end, we denote $\mathcal H^{(i)}=\{H^{(i)}_k:k\in\mathbb{Z}_{\geq0}\}$ for $i\in\llbracket1,d\rrbracket$, and $\mathcal H=\cap_{i=1}^d\mathcal H^{(i)}$. Let $\tau=\inf\{n\in\mathbb{N}:n\in\mathcal H\}$ with the convention that $\tau=\infty$ if this set is empty. Define the power series
\begin{equation}\label{eq:proofanalytic1}
    F(z)=\sum_{n\geq1} \mathbb{P}(\tau=n)z^n,
    \quad
    U(z)=\sum_{n\geq0}\mathbb{P}(n\in\mathcal H)z^n , 
\end{equation}
which both converge absolutely on $\{|z|<1\}$.  
We denote 
\begin{equation}
    A^{(i)}_k=\Upsilon^{(i)}_k-\Upsilon^{(i)}_{k-1} \quad \text{ and } \quad X^{(i)}_k=H^{(i)}_k-H^{(i)}_{k-1}
\end{equation}
 for $i\in\llbracket1,d\rrbracket$ and $k\in\mathbb{N}$. We make the following observation about $F$ and $U$.

\begin{lemma}\label{lem:Generating1}
    We have that 
    \begin{equation}\label{eq:proofanalytic2}
U(z)=1+F(z)U(z)\quad\mbox{for }|z|<1.
\end{equation}
\end{lemma}
\begin{proof}
From the definition in \eqref{eq:proofanalytic1}, it suffices to show
\begin{equation}\label{eq:proofanalytic3}
    \mathbb{P}\left(n\in\mathcal H\right)=\sum_{m=1}^n \mathbb{P}(\tau=m)\mathbb{P}\left(n-m\in\mathcal H\right) 
    \quad\mbox{for } n\in\mathbb{N}.
\end{equation}
Indeed, fix $m\in\llbracket1,n\rrbracket$. For $\boldsymbol{r}=(r_1,\ldots,r_d)\in\llbracket1,m\rrbracket^d$, set
\[
E_{\boldsymbol{r},m}
=
\left\{\tau=m,\ H^{(i)}_{r_i}=m\mbox{ for all }i\in\llbracket1,d\rrbracket\right\}.
\]
The events $E_{\boldsymbol{r},m}$ partition $\{\tau=m\}$, and each $E_{\boldsymbol{r},m}$ is measurable with respect to
$\sigma (X^{(i)}_1,\ldots,X^{(i)}_{r_i}: i\in\llbracket1,d\rrbracket )$.
Conditional on $E_{\boldsymbol{r},m}$, the post-increments
 $ (X^{(i)}_{r_i+j} )_{j\geq1}$, $i\in\llbracket1,d\rrbracket$,
are independent of this sigma-field, are independent across $i$, and have the same respective laws as the original increment sequences. Therefore,
\[
\mathbb{P}\left(n\in\mathcal H, E_{\boldsymbol{r},m}\right)
=
\mathbb{P}(E_{\boldsymbol{r},m})\mathbb{P}\left(n-m\in\mathcal H\right).
\]
Summing first over $\boldsymbol{r}\in\llbracket1,m\rrbracket^d$ and then over $m$ induces \eqref{eq:proofanalytic3}. We conclude the proof of \eqref{eq:proofanalytic2}.
\end{proof}
%this gives
%$\mathbb{P}\left(n\in\mathcal H, \tau=m\right)=\mathbb{P}(\tau=m)\mathbb{P}\left(n-m\in\mathcal H\right)$,
%and summing over $m$ induces \eqref{eq:proofanalytic3}. %The same conditioning argument also shows that, for every $m\in\mathbb N$ with $\mathbb P(\tau=m)>0$, conditionally on $\{\tau=m\}$ the shifted set $(\mathcal H-m)\cap\mathbb Z_{\geq0}$ is independent of $\tau$ and has the same law as $\mathcal H$.  
Next, we define the power series
\[
P_i(z)=\sum_{n\geq1}\mathbb{P}\left(X^{(i)}_1=n\right) z^n\quad\mbox{for } i\in\llbracket1,d\rrbracket,
\]
which converge absolutely on $\{|z|<e^{c_0/2}\}$ since $\sup_{i\in\llbracket1,d\rrbracket}\mathbb E [e^{c_0 X^{(i)}_1} ]\leq C_0$ by our assumptions. As the following lemma shows, the series $P_i$ has a radius of convergence strictly larger than $1$, and is bounded away from $0$ for $z$ sufficiently close to the unit circle.
\begin{lemma}\label{lem:Generating2}
    There exist some $\delta\in(0,1)$, $c_1>0$ depending only on $C_0,c_0,d,p$, such that
\begin{equation}\label{eq:proofanalytic6}
        1-P_i(z)\neq0 
    \mbox{ for }|z|\leq1+\delta,\, z\neq1, \quad\mbox{and}\quad |1-P_i(z)|\geq c_1\mbox{ for } |z|=1+\delta/2.
\end{equation}
\end{lemma}
\begin{proof}
Let $C_2=C_0\sup_{x\geq0} (x+x^2) e^{-c_0x/2}\in(0,\infty)$. Then,
\begin{equation}\label{eq:proofanalytic8}
    \left|P_i'(z)\right|+\left|P_i''(z)\right|
    \leq
    \sum_{n\geq1} \left(n+n^2\right)\cdot \mathbb{P}\left(X^{(i)}_1=n\right) e^{\frac{nc_0}{2}}
    =
    \mathbb{E}\left[\left(X^{(i)}_1+\left(X^{(i)}_1\right)^2\right)  e^{\frac{c_0X^{(i)}_1}{2}}\right]
    \leq C_2  
    \quad  \mbox{for }|z|<e^{\frac{c_0}{2}}.
\end{equation}
Denote $\upsilon_i=P_i'(1)$. Since $X^{(i)}_1\in\mathbb N$, we have $\upsilon_i=\mathbb E X^{(i)}_1\geq1$, and \eqref{eq:proofanalytic8} gives $\upsilon_i\leq C_2$. Choose $c_2\in(0,1)$ so small that $c_2\leq e^{c_0/2}-1$ and $c_2C_2\leq 1$. 
When $|z-1|\leq c_2$, Taylor's formula along the line segment from $1$ to $z$, which is contained in $\{|w|\leq e^{c_0/2}\}$, gives
\begin{equation} \label{eq:proofanalytic9}
        \left|1-P_i(z)\right|=\left|\upsilon_i(z-1)+(z-1)^2\int_0^1(1-t)P_i''(1+t(z-1))\,dt \right|\geq 
    \upsilon_i|z-1|-\frac{C_2}{2}|z-1|^2
    \geq
    \frac12 |z-1| . 
\end{equation}
When $|w|\leq1$ and $|w-1|\geq c_2/2$, we have
\begin{equation}\label{eq:proofanalytic10}
        |1-P_i(w)|
    \geq
    \Re\left(1-P_i(w)\right)                                      
    =
    \sum_{n\geq1}\mathbb{P}\left(X^{(i)}_1=n\right)\left(1-\Re(w^n)\right)                       
    \geq
    \mathbb{P}\left(X^{(i)}_1=1\right)\left(1-\Re(w)\right)                                    
    \geq
    p\frac{|w-1|^2}{2}
    \geq \frac{p c_2^2}{8} .
\end{equation}
Choose $\delta\in(0,1)$ so small that $\delta\leq e^{c_0/2}-1$, $\delta\leq c_2/2$, $C_2\delta\leq pc_2^2/16$, and let 
$c_1=\min(p c_2^2/16,\delta/4)$.  
We now prove \eqref{eq:proofanalytic6} for this choice of $\delta$ and $c_1$.
When $|z-1|\leq c_2$ and $z\neq1$, by \eqref{eq:proofanalytic9}, we have $1-P_i(z)\neq0$; and if additionally $|z|=1+\delta/2$, then $\left|1-P_i(z)\right|\geq \delta/4\geq c_1$.
When $|z-1|\geq c_2$ and $|z|\leq1$, \eqref{eq:proofanalytic10} gives $|1-P_i(z)|\geq 2c_1$. 
When $|z-1|\geq c_2$ and $1<|z|\leq1+\delta$, we let $w=z/|z|$, then $|w|=1$ and 
$|w-1|\geq |z-1|-|z-w|\geq c_2-\delta\geq c_2/2$. Therefore, by \eqref{eq:proofanalytic8} and \eqref{eq:proofanalytic10},
\begin{equation*}
    \begin{split}
        &|1-P_i(z)|\geq |1-P_i(w)|-|P_i(z)-P_i(w)|\geq|1-P_i(w)|- \sum_{n\geq1}\mathbb{P}\left(X^{(i)}_1=n\right)\left(|z|^n-1\right)  \\
        &=  |1-P_i(w)|- \left(P_i(|z|)-P_i(1)\right)\geq|1-P_i(w)|-C_2(|z|-1)\geq p c_2^2/8-C_2\delta\geq p c_2^2/16\geq c_1.
    \end{split}
\end{equation*}
Combining the three cases above, we conclude the proof of \eqref{eq:proofanalytic6}.
\end{proof}

Next, we control the probability that a given index $n$ is contained in $\mathcal{H}$. 

\begin{lemma}\label{lem:Generating3}
    Recall $\upsilon_i =P_i'(1)$.  Then we have that for all $n \in\mathbb{Z}_{\geq0}$, 
    \begin{equation}\label{eq:proofanalytic4}
    \left|\mathbb{P}(n\in\mathcal H)-\prod_{i=1}^d \upsilon_i^{-1} \right|\leq  d\left(\frac{1}{c_1}+ \frac{2}{\delta}\right)\left(1+\frac{\delta}{2}\right)^{-n} . 
\end{equation}
\end{lemma}
\begin{proof}
By the independence of the sequences $\mathcal H^{(i)}$ across different $i$,   $\mathbb{P}(n\in\mathcal H)=\prod_{i=1}^d\mathbb{P}(n\in\mathcal H^{(i)})$. Using   that $|a_1\cdots a_d-b_1\cdots b_d|\leq\sum_{i=1}^d|a_i-b_i|$ for any $a_i,b_i\in[0,1]$, it suffices to show
\begin{equation}\label{eq:proofanalytic5}
    \left|\mathbb{P}\left(n\in\mathcal H^{(i)}\right)- \upsilon_i^{-1}\right|\leq \left(\frac{1}{c_1}+ \frac{2}{\delta}\right)\left(1+\frac{\delta}{2}\right)^{-n}  
    \quad\mbox{for all } n\in\mathbb{Z}_{\geq0},\,i\in\llbracket1,d\rrbracket.
\end{equation} 
Observe that, for $|z|<1$,
\[
\sum_{n\geq0}\mathbb{P}\left(n\in\mathcal H^{(i)}\right)z^n
=
\sum_{k\geq0}\mathbb E\left[z^{H^{(i)}_k}\right]
=
\sum_{k\geq0}P_i(z)^k
=
\frac{1}{1-P_i(z)}.
\]
Here the swap of sums follows from absolute convergence, since
$\sum_{k\geq0}\mathbb E [|z|^{H^{(i)}_k} ]
=
\sum_{k\geq0}P_i(|z|)^k<\infty $.
Let
\[
    G_i(z)=
    \frac{1}{1-P_i(z)}-\frac{1}{\upsilon_i(1-z)}\quad\mbox{for } i\in\llbracket1,d\rrbracket,
\]
By Lemma~\ref{lem:Generating2}, the only possible singularity of $G_i$ in $\{|z|<1+\delta\}$ is $z=1$.  Since $P_i'(1)=\upsilon_i$, by a Taylor expansion of $P_i(z)$ around $1$, $z=1$ is a removable singularity of $G_i(z)$, therefore $G_i(z)$ is analytic in $\{|z|<1+\delta\}$.
By Cauchy's coefficient formula; see \cite[Chapter 2]{stein2010complex}, we know that
\begin{equation}\label{eq:proofanalytic7}
\mathbb{P}\left(n\in\mathcal H^{(i)}\right) -\upsilon_i^{-1}=[z^n]G_i(z)=\frac{1}{2\pi \mathrm{i}}
    \int_{|z|=1+\delta/2}\frac{G_i(z)}{z^{n+1}}\,dz .
\end{equation}
On the circle $|z|=1+\delta/2$, by \eqref{eq:proofanalytic6} in Lemma~\ref{lem:Generating2}, $|G_i(z)|\leq1/c_1+ 2/\delta$. By \eqref{eq:proofanalytic7}, we conclude \eqref{eq:proofanalytic5}.    
\end{proof} 

Next, observe that for $|z|<1$,
\begin{equation}\label{eq:proofanalytic13}
U(z)=\sum_{n\geq0}\mathbb{P}(n\in\mathcal H)z^n=\frac{\prod_{i=1}^d \upsilon_i^{-1}}{1-z}+\sum_{n\geq0}\left(\mathbb{P}(n\in\mathcal H)-\prod_{i=1}^d \upsilon_i^{-1}\right)z^n=: \frac{\prod_{i=1}^d \upsilon_i^{-1}}{1-z}+V(z),
\end{equation}
where $V(z)$ is defined as the series with coefficients $\mathbb{P}(n\in\mathcal H)-\prod_{i=1}^d \upsilon_i^{-1}$ for all $n\in \mathbb Z_{\geq 0}$.
By the estimate \eqref{eq:proofanalytic4}, the power series $V(z)$ converges absolutely in $\{|z|<1+\delta/2\}$, hence the right side above extends $U(z)$ to a meromorphic function on $\{|z|<1+\delta/2\}$ which has a simple pole at $z=1$. Moreover, we have the following estimate on the radius of convergence for $U$, and its values close to the unit circle.  

\begin{lemma}\label{lem:Generating4}
There exist some $\delta_1\in(0,1)$, $c_3>0$ depending only on $C_0,c_0,d,p$, such that
\begin{equation}\label{eq:proofanalytic11}
        U(z)\neq0 
    \mbox{ for }|z| <1+\delta_1,\, z\neq1, \quad\mbox{and}\quad |U(z)|\geq c_3\mbox{ for } |z|=1+\delta_1/2.
\end{equation}     
\end{lemma}
\begin{proof}
Using the estimate \eqref{eq:proofanalytic4}, 
\begin{equation}\label{eq:proofanalytic12}
|V(z)|+|V'(z)|\leq  d\left(\frac{1}{c_1}+ \frac{2}{\delta}\right) \left(\sum_{n\geq0} \left(\frac{1+\delta/4}{1+\delta/2}\right)^n+\sum_{n\geq1}n \frac{(1+\delta/4)^{n-1}}{(1+\delta/2)^n}  \right)
=: C_6     \quad\mbox{for }|z|\leq 1+\delta/4
\end{equation} and constant $C_6>0$.
Choose $c_6\in(0,1)$ so small that $c_6\leq\delta/8$ and $c_6 C_6\leq  (1+C_2)^{-d}/4$. We have 
\begin{equation}\label{eq:proofanalytic14}
        | U(z)|
    \geq
    \frac{\prod_{i=1}^d \upsilon_i^{-1}}{|1-z|}-C_6\geq\frac{(1+C_2)^{-d}}{|1-z|}-C_6
    \geq
    \frac{(1+C_2)^{-d}}{2|1-z|}\quad\mbox{for } 0<|z-1|\leq2c_6,
\end{equation}
where the first step uses \eqref{eq:proofanalytic13} and \eqref{eq:proofanalytic12}, the second step uses $\prod_{i=1}^d \upsilon_i^{-1}\geq(1+C_2)^{-d}$ which follows from $|\upsilon_i|=|P_i'(1)|\leq C_2$ by \eqref{eq:proofanalytic8} in the proof of Lemma~\ref{lem:Generating2}, and the third step uses $0<|z-1|\leq2c_6$ and the choice of $c_6$. 

We next prove the claim that 
\begin{equation}\label{eq:proofanalytic15}
    |U(z)|\geq1/2\quad\mbox{for }|z|\leq1,z\neq1.
\end{equation}
Recall that $F(z)=\sum_{n\geq1} \mathbb{P}(\tau=n)z^n$, which converges absolutely and $|F(z)|\leq1$ on $|z|\leq1$. 
We have 
\[ 
    \Re(1-F(z))                                   
     =
    1-\sum_{n\geq1}\mathbb{P}(\tau=n)\Re(z^n)                      
     =
    \left(1-\sum_{n\geq1}\mathbb{P}(\tau=n)\right) 
      +
      \sum_{n\geq1}\mathbb{P}(\tau=n)\left(1-\Re(z^n)\right)          
     \geq
    \mathbb{P}(\tau=1)\left(1-\Re z\right) .
\]
Note that $\mathbb{P}(\tau=1)\geq p^d>0$ since $\cap_{i=1}^d\{X^{(i)}_1=1\}\subseteq\{\tau=1\}$. Hence $F(z)\neq1$ for all $|z|\leq1$, $z\neq1$.
Recall from Lemma \ref{lem:Generating1} that  $U(z)=1/(1-F(z))$ for $|z|<1$. By \eqref{eq:proofanalytic13},  $U(z)$ also denotes the meromorphic continuation of this function to $\{|z|<1+\delta/2\}$, and $F$ is continuous on $\{|z|\leq1\}$. Hence, for $|z|=1$ and $z\neq1$, letting $r\uparrow1$ in $U(rz)=1/(1-F(rz))$ gives $U(z)=1/(1-F(z))$. Therefore, for all $|z|\leq1$ with $z\neq1$,  as $|F(z)|\leq 1$, we know that $|U(z)|=1/|1-F(z)|\geq1/2$, and we conclude the proof of  \eqref{eq:proofanalytic15}.

We choose $\delta_1\in(0,1)$ so small that $\delta_1\leq\delta/4$, $\delta_1\leq c_6$, and $ (c_6^{-2}+ C_6 )\delta_1\leq1/4$, and we let $c_3= \min(1/4,(1+C_2)^{-d}/(4c_6))$. We now prove \eqref{eq:proofanalytic11} for this $\delta_1$ and $c_3$. 
When $|z|<1$, $U(z)=1/(1-F(z))\neq0$.
When  $0<|z-1|\leq2c_6$, by \eqref{eq:proofanalytic14}, we have $U(z)\neq0$, and if additionally $|z|=1+\delta_1/2$, then $|U(z)|\geq c_3$ by the choice of $c_3$.  When $1\leq |z|<1+\delta_1$ and $|z-1|\geq2c_6$, 
Let $w=z/|z|$, then $|w|=1$ and $|w-1|\geq |z-1|-|z-w|\geq2c_6-\delta_1\geq c_6$.
Along the radial segment from $w$ to $z$, the distance to $1$ is at least $c_6$. Since \eqref{eq:proofanalytic13} gives
$U'(\zeta)= \prod_{i=1}^d\upsilon_i^{-1}\cdot (1-\zeta)^{-2} +V'(\zeta)$
and $\prod_{i=1}^d\upsilon_i^{-1}\leq1$, \eqref{eq:proofanalytic12} gives $|U'|\leq C_6+c_6^{-2}$ on this segment; combining with \eqref{eq:proofanalytic15},     
$ | U(z)|\geq| U(w)|-(C_6+c_6^{-2})|z-w|\geq1/2-(C_6+c_6^{-2})\delta_1\geq1/4\geq c_3 $.
Combining these three cases, we conclude the proof of \eqref{eq:proofanalytic11}.
\end{proof}

\begin{lemma}\label{lem:Generating5}
    We have that $\tau<\infty$ almost surely, and that there exist constants $C_7,c_7\in(0,\infty)$ depending only on $C_0,c_0,d,p$, such that 
$\mathbb{E} [e^{c_7\tau} ]\leq C_7$. 
\end{lemma}
\begin{proof}
By \eqref{eq:proofanalytic13}, we see that
\[\frac{1}{U(z)} =\frac{1-z}{ \prod_{i=1}^d \upsilon_i^{-1}+(1-z)V(z)},
\]
therefore $1/U(z)\to0$ as $z\to1$ since the denominator above converges to a nonzero value. By \eqref{eq:proofanalytic11}, $F(z)=1- 1/U(z)$ is  an analytic function on $\{|z| <1+\delta_1\}$,
since $z=1$ is its removable singularity. We also have $F(1)=1$.
Since $F(z)$ is analytic on $\{|z| <1+\delta_1\}$, its Taylor series at the origin converges absolutely at $z=1$, hence
$\sum_{n\geq1}\mathbb{P}(\tau=n)= F(1)=1$, which implies $\tau<\infty$ almost surely.
On the circle $\{|z|=1+\delta_1/2\}$, by \eqref{eq:proofanalytic11} in Lemma~\ref{lem:Generating4}, $|F(z)|\leq 1+1/c_3$.
By Cauchy's estimate,
\[
 \mathbb{P}(\tau=n)=|[z^n]F(z)|\leq (1+1/c_3)(1+\delta_1/2)^{-n}\quad\mbox{for }n\in\mathbb{N}, 
\] 
where we write $[z^n]F(z)$ for the coefficients in the power series.
 Hence,
\begin{equation*} 
    \mathbb{E}\left[(1+\delta_1/4)^{\tau}\right]
    =
    \sum_{n\geq1}\mathbb{P}(\tau=n)(1+\delta_1/4)^n
    \leq
    (1+1/c_3)\sum_{n\geq1} (1+\delta_1/4)^n (1+\delta_1/2)^{-n}
    <\infty .
\end{equation*}
Let $c_7=\log(1+\delta_1/4)>0$ and $C_7<\infty$ be the right side above, we have $\mathbb{E} [e^{c_7\tau} ]\leq C_7$. 
\end{proof}

We have now all tools to show Theorem~\ref{thm:find reg times 2}.

\begin{proof}[Proof of Theorem~\ref{thm:find reg times 2}]
We begin by constructing $0=H_0<H_1 <\cdots$ which enumerate $\mathcal H$.
Since $\tau=\inf\{n\in\mathbb{N}:n\in\mathcal H\}$ is almost surely finite, the first positive element of $\mathcal H$ is finite. 
The same conditioning argument as in Step 1 also shows that, for every $m\in\mathbb N$ with $\mathbb P(\tau=m)>0$, conditionally on $\{\tau=m\}$ the shifted set $(\mathcal H-m)\cap\mathbb Z_{\geq0}$ is independent of $\tau$ and has the same law as $\mathcal H$.  Repeating this argument inductively gives finite random variables $0=H_0<H_1<H_2<\cdots$  which enumerate $\mathcal H$, and whose gaps $H_k-H_{k-1}$, $k\in\mathbb{N}$ are i.i.d. with $H_1=\tau$. Additionally, 
$\mathbb{E} [e^{c_7H_1} ]=\mathbb{E} [e^{c_7\tau} ]\leq C_7$ using Lemma~\ref{lem:Generating5} for the last inequality. 

We next prove that the increment random vectors
\[
\Phi_k=
\left(
\Upsilon^{(1)}_{t_{1,k}}-\Upsilon^{(1)}_{t_{1,k-1}},\dots,
\Upsilon^{(d)}_{t_{d,k}}-\Upsilon^{(d)}_{t_{d,k-1}},
H_k-H_{k-1}
\right)
\quad\mbox{for } k\in\mathbb{N},
\]
are i.i.d., where we recall that the sequence $0=t_{i,0}<t_{i,1}<t_{i,2}<\cdots$ is determined by $t_{i,\ell}$ being the unique index such that $H^{(i)}_{t_{i,\ell}}=H_\ell$, for $\ell\in\mathbb{Z}_{\geq0}$.  
For
$\boldsymbol{r}=(r_1,\dots,r_d)\in\mathbb{Z}_{\geq0}^d$ and
$m\in\mathbb{Z}_{\geq0}$, denote
\[
D_{\ell,\boldsymbol{r},m}
=
\{H_\ell=m,\ t_{1,\ell}=r_1,\dots,t_{d,\ell}=r_d\},\quad\mbox{and}\quad
\mathcal F_{\boldsymbol{r}}
=
\sigma \left(A^{(i)}_s,X^{(i)}_s:
 i\in\llbracket1,d\rrbracket,\ 1\leq s\leq r_i \right),
\] 
where we recall $A^{(i)}_k=\Upsilon^{(i)}_k-\Upsilon^{(i)}_{k-1}$.
Then the event $D_{\ell,\boldsymbol{r},m}$ is $\mathcal F_{\boldsymbol{r}}$-measurable.  Indeed, this is the event that $H^{(i)}_{r_i}=m$ for every $i$ and that $\cap_{i=1}^d\{H^{(i)}_s:0\leq s\leq r_i\}$ has exactly $\ell+1$ elements and largest element $m$; this condition is fully determined by the increments appearing in $\mathcal F_{\boldsymbol{r}}$.
For this fixed $\boldsymbol{r}$, using the independence of the $d$ increment sequences and the i.i.d. property within each sequence, we know that the post-increment 
\[
\left(A^{(i)}_{r_i+s},X^{(i)}_{r_i+s}:s\in\mathbb{N},\ i\in\llbracket1,d\rrbracket\right)
\]
is independent of $\mathcal F_{\boldsymbol{r}}$ and has the same law as
$(A^{(i)}_s,X^{(i)}_s:s\in\mathbb{N},\ i\in\llbracket1,d\rrbracket)$.
On $D_{\ell,\boldsymbol{r},m}$, the future common renewal levels relative to $m$ and the future index increments relative to $\boldsymbol{r}$ are obtained by applying to this post-increment array exactly the same measurable construction that produces
$(H_k,t_{1,k},\dots,t_{d,k})_{k\geq0}$ from the original increment array. Hence, for every bounded Borel function $g$ on $\mathbb{Z}_{\geq0}^{d+1}$ and every bounded $\mathcal F_{\boldsymbol{r}}$-measurable $W$ which vanishes outside
$D_{\ell,\boldsymbol{r},m}$, we have $\mathbb E [W\,g( \Phi_{\ell+1} ) ]=\mathbb E[W]\,\mathbb E [g(\Phi_1 ) ]$.
We now apply this identity to a specifically chosen $W$. 
Let $f$ be any bounded Borel function of $(\Phi_1,\dots,\Phi_\ell)$, then $W=f(\Phi_1,\dots,\Phi_\ell)\mathbf{1}_{D_{\ell,\boldsymbol{r},m}}$ is $\mathcal F_{\boldsymbol{r}}$-measurable, since the first $\ell$ vectors are determined by the increments up to the indices $r_1,\dots,r_d$. 
Summing $\mathbb E [W\,g( \Phi_{\ell+1} ) ]=\mathbb E[W]\,\mathbb E [g(\Phi_1 ) ]$ over the countable partition $(D_{\ell,\boldsymbol{r},m})_{\boldsymbol{r},m}$ gives
\[
\mathbb E\left[
 f(\Phi_1,\dots,\Phi_\ell)
 g(\Phi_{\ell+1} )
\right]
=
\mathbb E\left[f(\Phi_1,\dots,\Phi_\ell)\right]
\mathbb E\left[g(\Phi_1 )\right].
\]
Iterating this identity shows that $(\Phi_k)_{k\in\mathbb{N}}$ are i.i.d.

Finally, we show that each component of $\Phi_1$ is subexponential. We have $t_{i,1}\leq H_1$ since $H^{(i)}_{t_{i,1}}=H_1$ and each $X^{(i)}_k\geq1$. Therefore, $\mathbb{E} [e^{c_7t_{i,1}} ] \leq \mathbb{E} [e^{c_7H_1} ] \leq C_7$. Since $\mathbb E [e^{c_0 A^{(i)}_1} ]=\mathbb E[ e^{c_0\Upsilon^{(i)}_1} ]\leq C_0$ by the assumption of the lemma,  
$\Upsilon^{(i)}_{t_{i,1}}=\sum_{s=1}^{t_{i,1}}A^{(i)}_s$ is the sum of subexponentially many i.i.d. subexponential variables. Using Lemma \ref{lem:basic exponential tails} \textup{(\romannum{4})} we conclude that $\mathbb E [e^{c_8\Upsilon^{(i)}_{t_{i,1}}} ]\leq C_8$ for some $C_8,c_8\in(0,\infty)$ depending only on $C_0,c_0,C_7,c_7$, and hence depend only on $C_0,c_0,d,p$. Let $c=\min(c_7,c_8)$ and $C=\max(C_7,C_8)$ we conclude the proof. 
\end{proof}

\section{Approximation of the reweighted path measure}\label{sec:PathMeasure}

In this section, we establish an approximation of the path measure for the open ASEP on the coexistence line when we condition on the location of the path minimum. More precisely, take $\a=\b>1$ in \eqref{eq:PNab} and note that
    \[ 
    \mathbb{P}^{(q,\b,\b)}(\bg)=\frac{1}{Z_N^{(q,\b,\b)}} 2^{H(\bg)}\b^{\gamma_N-2\min(\bg)}  W_{\bs(\bg)}\left(\b^2\right)\quad\mbox{ for }\bg\in\lrw_N.
    \]
The following proposition formalizes the approximation of the path measure, which will be our key step towards proving Theorem~\ref{thm:ApproxCoex} on the stationary distribution of the open ASEP on the coexistence line.
\begin{proposition}\label{prop:comparison path q earlier}
    Assume \(q\in(0,1)\) and \(\b>1\). For $N\in\mathbb N$, \(k\in\llbracket0,N\rrbracket\), let $\mathbb P_{N,k}^q(\,\cdot\,)=\mathbb P^{(q,\b,\b)}(\,\cdot\,\mid T_N=k)$
    and
    $\mathbb P_{N,k}^0(\,\cdot\,)=\mathbb P^{(0,\b,\b)}(\,\cdot\,\mid T_N=k)$.
    Then both conditional laws are well-defined.
Moreover, for \(k,m\in\llbracket0,N\rrbracket\), let $\mathsf I_{N,k,m}$ and $\Delta_{N,k,m}(\bg)$ for $\bg\in\lrw_N$ be given by 
\begin{equation}\label{eq:Ianddelta new}
    \begin{split}
         \mathsf I_{N,k,m}& =  
    \{i\in\llbracket0,N-1\rrbracket: |i-k|\ge m\},\\
     \Delta_{N,k,m}(\bg)&=\left(\gamma_{i+1}-\gamma_i\right)_{i\in\mathsf I_{N,k,m}}
    \in \{-1,0,1\}^{\mathsf I_{N,k,m}},
    \end{split}
\end{equation} 
 then there exist constants \(C,c\in(0,\infty)\), depending only on \(q\) and \(\b\), such that, for every \(N\ge1\) and every \(k,m\in\llbracket0,N\rrbracket\),
\[ 
    \TV{
    \mathbb P_{N,k}^q
    \left(
        \Delta_{N,k,m}(\bg)\in\cdot\, \right)
    -
    \mathbb P_{N,k}^0
    \left(
        \Delta_{N,k,m}(\bg)\in\cdot\, \right)
    }
    \le
    Ce^{-cm}.
\]
\end{proposition}

The proof of Proposition \ref{prop:comparison path q earlier} is the main goal of this section, which will appear in Section \ref{subsec:proof of prop comparison of path}. While the overall strategy of proof follows   Sections~4 and 5 in \cite{SchmidYangShock} on  analyzing the probabilistic characterization of the Bryc polynomials, the individual arguments require several adjustments.

\subsection{Comparison with the bi-infinite path measure}

The following comparison lemma allows us to estimate key observables of the finite path measures from the bi-infinite $V$-shaped random walk measure.

\begin{lemma}\label{lem:comparision}
        Assume $\b>1$ and \(N\in\mathbb N\), and recall that for $\bg\in\lrw_N$,
\begin{equation}\label{eq:PN00 and TN} 
\mathbb{P}^{(0,\b,\b)}(\bg) =\frac{1}{Z_N^{(0,\b,\b)}} 2^{H(\bg)}\b^{\gamma_N-2\min(\bg)}  \quad\mbox{ and }\quad
 T_N(\bg)  =\min\{i\in\llbracket0,N\rrbracket:\gamma_i=\min(\bg)\} . 
 \end{equation}
 Let $\{Y_i\}_{i=-\infty}^{\infty}$ be a bi-infinite $V$-shaped $\nu_{\b}$-random walk started at $Y_0=0$ as in Definition \ref{def:V shaped RW}. Then for each $k\in\llbracket0,N\rrbracket$, the path $\boldsymbol{\lambda}$ given by $\lambda_i=Y_{i-k}-Y_{-k}$, $i\in\llbracket0,N\rrbracket$ is in $\lrw_N$, and for every non-negative function $F$ on $\lrw_N$, 
    \[
    \mathbb E^{(0,\b,\b)}\left[F(\bg)\mid T_N=k\right] 
    \le \left(1-\b^{-2}\right)^{-2} \cdot {\mathbb E}\left[F\left(\boldsymbol{\lambda} \right)\right] .
    \]
\end{lemma}

\begin{proof}
Observe that \(\lambda_0=0\) and $\lambda_i-\lambda_{i-1}$ for $i\in\llbracket1,N\rrbracket$, hence \(\boldsymbol{\lambda}\in\lrw_N\).
By Lemma \ref{lem:decomposition TASEP} with \(\a=\b\), under $\mathbb P^{(0,\b,\b)}(\,\cdot\,\mid T_N=k)$, the right piece
$(\gamma_{k+j}-\gamma_k)_{j\in\llbracket0,N-k\rrbracket}$ has the law of a \(\nu_\b\)-random walk started from \(0\), conditioned to
stay non-negative up to time \(N-k\).  If \(k\ge1\), then the left piece $(\gamma_{k-j}-\gamma_k)_{j\in\llbracket0,k\rrbracket}$ has the law of a \(\nu_\b\)-random walk started from \(0\), conditioned on staying strictly positive during $\llbracket1,k\rrbracket$ steps.  The left and right pieces are conditionally independent.  When \(k=0\), there
is no left piece. Let \(\boldsymbol{S}\) be an ordinary \(\nu_\b\)-walk started at \(0\).  We claim that the law of each finite prefix conditioned to stay non-negative up to time \(m\) has Radon--Nikodym density at most \((1-\b^{-2})^{-1}\) with respect to
the law of the same prefix under conditioning to stay non-negative forever. Indeed, for each admissible prefix \(\boldsymbol{w}\), by Lemma \ref{lem:basic hitting monotone coupling},
\[
\frac{
    \mathbb P((S_i)_{i\in\llbracket0,m\rrbracket}=\boldsymbol{w}\mid S_i\ge0 \mbox{ for } i\in\llbracket0,m\rrbracket)
}{
    \mathbb P((S_i)_{i\in\llbracket0,m\rrbracket}=\boldsymbol{w}\mid S_i\ge0\mbox{ for } i\ge0)
}
=
\frac{1-\b^{-2}}{
    \mathbb P(S_i\ge0 \mbox{ for } i\in\llbracket0,m\rrbracket)\,
    \mathbb P_{w_m}(S_i\ge0\mbox{ for } i\ge0)
}
\le \left(1-\b^{-2}\right)^{-1}.
\]
The same argument applies to the left piece after shifting down by $1$.  Multiplying the right and left density bounds completes the proof.  If \(k\in\{0,N\}\), only one factor appear, and the upper bound remains valid.
\end{proof}

\subsection{Regeneration structure for the finite comparison path $\boldsymbol{\lambda}$}

Next, we control the regeneration times of the bi-infinite $V$-shaped $\nu_{\b}$-random walk, showing that their differences are subexponential.

\begin{lemma}\label{lem:regeneration lambda}
Assume $\b>1$ and let $\boldsymbol{Y}=\{Y_i\}_{i=-\infty}^{\infty}$ be a bi-infinite $V$-shaped $\nu_{\b}$-random walk started at $Y_0=0$ as in
Definition \ref{def:V shaped RW}. Fix \(N\in\mathbb Z_{\ge0}\) and \(k\in\llbracket0,N\rrbracket\). Define the path \(\boldsymbol{\lambda}=(\lambda_0,\ldots,\lambda_N)\in\lrw_N\) as in Lemma \ref{lem:comparision}, i.e., $\lambda_i=Y_{i-k}-Y_{-k}$ for $i\in\llbracket0,N\rrbracket$. Then there exist \(\mathbb Z_{\ge0}\)-valued random variables \(\upchi_j\) and \(\mathbb N\)-valued random variables \(G_j\) for \(j\in\mathbb N\), satisfying the following properties.
Set \(\upchi_0=N\).
\begin{enumerate}
    \item[\textup{(\romannum{1})}]
       Let $K=\max (\{j\in\mathbb Z_{\geq0}:\upchi_j>0\}\cup\{0\} )$. Then \(K<\infty\) 
    almost surely. On the event \(\{K\ge1\}\), we have $N\geq \upchi_1>\cdots>\upchi_K\geq1$ and $\mathsf S^{\bs(\boldsymbol{\lambda})}(\upchi_j)=1$ for $j\in\llbracket1,K\rrbracket$, where \(\bs(\boldsymbol{\lambda})\) is the composition associated to \(\boldsymbol{\lambda}\) by \eqref{eq:Signature}, and \(\boldsymbol{\mathsf S}^{\bs(\boldsymbol{\lambda})}\) is the associated \(0\)-\(1\) sequence from Definition \ref{def:associated 01 seq}.

    \item[\textup{(\romannum{2})}]
    The variables \(G_j\), \(j\in\mathbb N\), are independent.
    The variables \(G_j\), \(j\ge2\), are i.i.d. For every
    \(j\in\mathbb N\), \(G_j\) is independent of
    \(\sigma(\upchi_1,\ldots,\upchi_{j-1})\), with the convention that
    this sigma-field is trivial when \(j=1\). There exist constants
    \(C,c>0\), depending only on \(\b\), such that $\mathbb P(G_j\ge y)\le Ce^{-cy}$ for all $j,y\in\mathbb N$.

    \item[\textup{(\romannum{3})}]
    For \(j\in\mathbb N\), set $\mathrm R_j=\upchi_{j-1}-\upchi_j$.
    Then \(\mathrm R_j\in\mathbb Z_{\ge0}\) for all \(j\in\mathbb N\),
    \(\mathrm R_j=0\) for \(j\ge K+2\), and $\mathrm R_1+\cdots+\mathrm R_{K+1}=N$.
    Furthermore, $\mathrm R_j\le G_j$ for $j\in\mathbb N$.
\end{enumerate}
\end{lemma}

\begin{proof}
Using Theorem \ref{thm:doubled regeneration}, we can choose doubled regeneration times $(m_1,n_1),(m_2,n_2),\ldots$ for \(\boldsymbol Y\) and constants $C_1,c_1>0$ depending only on $\b$, satisfying the following properties. We have $0<m_1<m_2<\cdots$ and $0<n_1<n_2<\cdots$. Denote $h_j=Y_{-m_j}=Y_{n_j}$, then 
\begin{equation}\label{eq:etyjneyjne}
(m_1,n_1,h_1),\quad
(m_2-m_1,n_2-n_1,h_2-h_1),\quad
(m_3-m_2,n_3-n_2,h_3-h_2),\quad\ldots
\end{equation}
are mutually independent, with the vectors after the first being i.i.d. and for every $y\in\mathbb{N}$,
\[
\mathbb{P}(n_1\geq y)\leq C_1e^{-c_1y},\quad \mathbb{P}(n_2-n_1\geq y)\leq C_1e^{-c_1y},\quad
\mathbb{P}(m_1\geq y)\leq C_1e^{-c_1y},  \quad \mathbb{P}(m_2-m_1\geq y)\leq C_1e^{-c_1y}.
\]  
Set \(m_0=n_0=0\), and for \(j\in\mathbb N\) set
\[ 
    \upchi_j=(k-m_j+1)_+ +(N-k-n_j+1)_+ \quad \mbox{ and }\quad  G_j=(m_j-m_{j-1})+(n_j-n_{j-1}).
\] 
The variables \(G_j\), \(j\in\mathbb N\), are independent, and the variables \(G_j\), \(j\ge2\), are i.i.d.
Since \(\upchi_i\) is a function of the first $i$ vectors in \eqref{eq:etyjneyjne}, \(G_j\) is independent of \(\sigma(\upchi_1,\ldots,\upchi_{j-1})\). 
Using that $n_1,m_1,n_2-n_1,m_2-m_1$ are subexponential with constants $C_1,c_1>0$ depending only on $\b$, by Lemma \ref{lem:basic exponential tails} \textup{(\romannum{2})} and \textup{(\romannum{3})}, we know that there exist \(C,c>0\), depending only on \(\b\), such that $\mathbb P(G_j\ge y)\le Ce^{-cy}$ for $j,y\in\mathbb N$ in Statement \textup{(\romannum{2})}.  

For Statement \textup{(\romannum{1})}, recall \(K=\max (\{j\in\mathbb Z_{\geq0}:\upchi_j>0\}\cup\{0\})\). Since $m_j,n_j\to\infty$, we have \(\upchi_j=0\) for all sufficiently large \(j\), and hence \(K<\infty\). 
We next show that on \(\{K\ge1\}\), \(1\le \upchi_K<\cdots<\upchi_1\le N\). Indeed, \(\upchi_j\le k+(N-k)=N\) for every \(j\in\mathbb N\) since \(m_j,n_j\ge1\). Since both \(m_j\) and \(n_j\) are strictly increasing, \((\upchi_j)_{j\in\mathbb N}\) is non-increasing. If \(\upchi_j>0\), then \(\upchi_{j+1}<\upchi_j\), because at least one of the two positive-part terms in \(\upchi_j\) is positive and that term strictly decreases when \(m_j,n_j\) are replaced by the strictly larger \(m_{j+1},n_{j+1}\), while the other term cannot increase. Thus, if \(K\ge1\), then \(\upchi_j>0\) for every \(j\in\llbracket1,K\rrbracket\), and \(1\le \upchi_K<\cdots<\upchi_1\le N\). 
We now prove that \(\mathsf S^{\bs(\boldsymbol{\lambda})}(\upchi_j)=1\) for \(j\in\llbracket1,K\rrbracket\).  Recall that $h_j=Y_{-m_j}=Y_{n_j}$ for $j\in\mathbb N$. By the definition of the doubled regeneration times, $h_j>0$ and $Y_{\ell}\ge h_j$ if and only if $\ell\le -m_j$ or $\ell\ge n_j$. Since \(Y_i\ge0\) for all \(i\in\mathbb Z\), \(Y_0=0\), and \(0\in\llbracket-k,N-k\rrbracket\), we have \(\min(\boldsymbol{\lambda})=-Y_{-k}\), and hence \(\lambda_t-\min(\boldsymbol{\lambda})=Y_{t-k}\) for \(t\in\llbracket0,N\rrbracket\). Therefore, for \(j\in\llbracket1,K\rrbracket\), the number of vertices of \(\boldsymbol{\lambda}\in\lrw_N\) whose relative height is at least \(h_j\) is
\begin{equation}\label{eq:nturuyjrtu}
    \begin{split}
        &\#\{t\in\llbracket0,N\rrbracket:Y_{t-k}\ge h_j\}
=
\#(\llbracket-k,N-k\rrbracket\cap(-\infty,-m_j])
+
\#(\llbracket-k,N-k\rrbracket\cap[n_j,\infty))\\
&=(k-m_j+1)_+ +(N-k-n_j+1)_+=\upchi_j .
    \end{split}
\end{equation} 
If \(\bs(\boldsymbol{\lambda})=(\sigma_0,\ldots,\sigma_r)\), then Definition \ref{def:associated 01 seq} gives \(\mathsf S^{\bs(\boldsymbol{\lambda})}(\sigma_u+\cdots+\sigma_r)=1\) for \(u\in\llbracket1,r\rrbracket\). Since \(h_j>0\) and \(\upchi_j>0\), we have \(h_j\in\llbracket1,r\rrbracket\), and by \eqref{eq:Signature} and \eqref{eq:nturuyjrtu},
\(\upchi_j=\sigma_{h_j}+\cdots+\sigma_r\). Hence \(\mathsf S^{\bs(\boldsymbol{\lambda})}(\upchi_j)=1\) for \(j\in\llbracket1,K\rrbracket\).

For Statement \textup{(\romannum{3})}, the bound \(\upchi_1\le N=\upchi_0\) and the monotonicity of \((\upchi_j)_{j\in\mathbb N}\) give \(\mathrm R_j=\upchi_{j-1}-\upchi_j\in\mathbb Z_{\ge0}\) for all \(j\in\mathbb N\). By the definition of \(K\), \(\upchi_j=0\) for \(j\ge K+1\), so \(\mathrm R_j=0\) for \(j\ge K+2\), and telescoping gives
$\mathrm R_1+\cdots+\mathrm R_{K+1}=\upchi_0-\upchi_{K+1}=N$.
It remains to prove \(\mathrm R_j\le G_j\). For \(j=1\),
\[
\begin{aligned}
&\mathrm R_1
 =N-\upchi_1 
 = (k-(k-m_1+1)_+
 + ((N-k)-(N-k-n_1+1)_+ 
 \le (m_1-1)+(n_1-1)\le G_1 .
\end{aligned}
\]
For \(j\ge2\), since the map \(x\mapsto(x)_+\) is \(1\)-Lipschitz and \(m_j\ge m_{j-1}\), \(n_j\ge n_{j-1}\), we have
\[
\mathrm R_{j}=\upchi_{j-1}-\upchi_{j }
\le (m_{j}-m_{j-1})+(n_{j}-n_{j-1})=G_{j},
\] allowing us to conclude.
\end{proof}

\subsection{Moment bounds for the Bryc polynomials}

In this part, our goal is to bound the moments of the involved Bryc polynomials, similar to Section 4 in \cite{SchmidYangShock}, but uniformly in the location of the minimum. 

\begin{proposition}
    \label{prop:moment bound coex}
    Assume $q\in(0,1)$, $\b>1$ and $z>1$. 
    For \(k\in\llbracket0,N\rrbracket\), write $\mathbb P_{N,k}^0(\,\cdot\,)=\mathbb P^{(0,\b,\b)}(\,\cdot\,\mid T_N=k)$ and $\mathbb E_{N,k}^0$ to denote the expectation under $\mathbb P_{N,k}^0$, where we recall that $\mathbb{P}^{(0,\b,\b)}$ and $T_N$ are as in \eqref{eq:PN00 and TN}. Then \[
    \sup_{N\in\mathbb{N}}
    \sup_{k\in\llbracket0,N\rrbracket}
    \sup_{i\in\llbracket0,N+1\rrbracket}
    \mathbb E_{N,k}^0\!\left[
        \mathbb E\!\left[
            z^{X^{\boldsymbol{\mathsf S}^{\bs(\bg)}}(i)}
            \,\middle|\,\bg
        \right]
    \right]
    <\infty
\]   
is a finite constant depending only on $q,\b$ and $z$. Here, the inner expectation is taken only with respect to the transition randomness of the birth-death chain from Definition \(\ref{def:characterization measure}\), conditionally on \(\bg\). 
As a corollary,  let $W_{\bs(\bg)}(z)$ be the Bryc polynomial indexed by the composition $\bs(\bg)$, then for each $m\in\mathbb{N}$, 
\[
    \sup_{N\in\mathbb{N}}
    \sup_{k\in\llbracket0,N\rrbracket} \mathbb E_{N,k}^0\left[W_{\bs(\bg)}(z)^m\right]<\infty
\] 
is a finite constant depending only on $q$, $\b$, $z$, and $m$. 
\end{proposition}

In order to show Proposition~\ref{prop:moment bound coex}, we require some setup. As a first input, we define a Markov chain $\mathscr X$ that will  serve as a simple process in order to bound the birth-death chain $X$ from Theorem~\ref{thm:prob interpretation of Bryc polynomials}.

\begin{definition}\label{def:simpler chain}
    Assume $q\in(0,1)$, $N,x\in\mathbb{Z}_{\geq0}$, and $\boldsymbol{\mathsf{S}}=(\mathsf{S}(0),\dots,\mathsf{S}(N))\in\{0,1\}^{N+1}$ is a $0$-$1$ sequence. Define a Markov chain $\mathscr{X}^{\boldsymbol{\mathsf{S}},x}=\{\mathscr{X}^{\boldsymbol{\mathsf{S}},x}(i)\}_{i=0}^{N+1}$ by $\mathscr{X}^{\boldsymbol{\mathsf{S}},x}(0)=x$, and for each $i\in\llbracket0,N\rrbracket$ and $j\in\mathbb{Z}_{\geq0}$,
    \begin{enumerate}
    \item [\textup{(\romannum{1})}] if $\mathsf{S}(i)=1$, we let $\mathscr{X}^{\boldsymbol{\mathsf{S}},x}(i+1)-\mathscr{X}^{\boldsymbol{\mathsf{S}},x}(i)\in\{-1,0\}$ with 
    \begin{equation}\label{eq:etgets1}
    \mathbb{P}\left(\mathscr{X}^{\boldsymbol{\mathsf{S}},x}(i+1)=j-1 \Big{|}\mathscr{X}^{\boldsymbol{\mathsf{S}},x}(i)=j\right)=(1-q)\cdot {\bf1}_{j\geq1},
    \end{equation}
    \item [\textup{(\romannum{2})}] if $\mathsf{S}(i)=0$, we let $\mathscr{X}^{\boldsymbol{\mathsf{S}},x}(i+1)-\mathscr{X}^{\boldsymbol{\mathsf{S}},x}(i)\in\{0,1\}$ with 
    \begin{equation}\label{eq:etgets2}
    \mathbb{P}\left(\mathscr{X}^{\boldsymbol{\mathsf{S}},x}(i+1)=j+1 \Big{|}\mathscr{X}^{\boldsymbol{\mathsf{S}},x}(i)=j\right)=q^{j+1}.
    \end{equation}
\end{enumerate} 
When $x=0$, we denote the Markov chain $\mathscr{X}^{\boldsymbol{\mathsf{S}},0}$ simply as $\mathscr{X}^{\boldsymbol{\mathsf{S}}}$.
\end{definition}

The next three results are taken from \cite{SchmidYangShock}. We start with a comparision between the birth-death chains $X^{\boldsymbol{\mathsf{S}}}$ and $\mathscr{X}^{\boldsymbol{\mathsf{S}}}$. Here, we write $Y \succeq Y'$ for random variables $Y$ and $Y'$ whenever there exists a coupling $\mathbf P$ such that $\mathbf P (Y \geq Y')=1$.

\begin{lemma}[c.f.~Lemma~4.3 in \cite{SchmidYangShock}]\label{lem:coupling birth death}
    Assume $q\in(0,1)$ and $N\in\mathbb{Z}_{\geq0}$. Then the following statements hold.
    \begin{enumerate}
        \item [\textup{(\romannum{1})}] Let $\boldsymbol{\mathsf{S}},\boldsymbol{\mathsf{S}}'\in\{0,1\}^{N+1}$ satisfy $\mathsf{S}(i)\geq\mathsf{S}'(i)$ for each $i\in\llbracket0,N\rrbracket$. Then there exists a coupling between Markov chains $X^{\boldsymbol{\mathsf{S}}}$ and $X^{\boldsymbol{\mathsf{S}}'}$ such that 
        \[
        X^{\boldsymbol{\mathsf{S}}}(i)\leq X^{\boldsymbol{\mathsf{S}}'}(i)\quad\mbox{ for all }i\in\llbracket0,N+1\rrbracket.
        \]
        \item [\textup{(\romannum{2})}] Let $\boldsymbol{\mathsf{S}}\in\{0,1\}^{N+1}$. Then there exists a coupling between Markov chains $X^{\boldsymbol{\mathsf{S}}}$ and $\mathscr{X}^{\boldsymbol{\mathsf{S}} }$ such that 
        \[
        X^{\boldsymbol{\mathsf{S}}}(i)\leq \mathscr{X}^{\boldsymbol{\mathsf{S}} }(i)\quad\mbox{ for all }i\in\llbracket0,N+1\rrbracket.
        \]
        \item [\textup{(\romannum{3})}] Let $x,x'\in\mathbb{Z}_{\geq0}$ satisfy $x\leq x'$ and $\boldsymbol{\mathsf{S}},\boldsymbol{\mathsf{S}}'\in\{0,1\}^{N+1}$ satisfy $\mathsf{S}(i)\geq\mathsf{S}'(i)$ for each $i\in\llbracket0,N\rrbracket$. Then there exists a coupling between Markov chains $\mathscr{X}^{\boldsymbol{\mathsf{S}},x}$ and $\mathscr{X}^{\boldsymbol{\mathsf{S}}',x'}$ such that 
        \[
        \mathscr{X}^{\boldsymbol{\mathsf{S}},x}(i)\leq \mathscr{X}^{\boldsymbol{\mathsf{S}}',x'}(i)\quad\mbox{ for all }i\in\llbracket0,N+1\rrbracket.
        \]
        \item [\textup{(\romannum{4})}] Let $x,y\in\mathbb{Z}_{\geq0}$ and $\boldsymbol{\mathsf{S}}\in\{0,1\}^{N+1}$. Then there exists a coupling between Markov chains $\mathscr{X}^{\boldsymbol{\mathsf{S}},x}$ and $\mathscr{X}^{\boldsymbol{\mathsf{S}},x+y}$ such that 
        \[
        \mathscr{X}^{\boldsymbol{\mathsf{S}},x+y}(i)-y\leq \mathscr{X}^{\boldsymbol{\mathsf{S}},x}(i)\quad\mbox{ for all }i\in\llbracket0,N+1\rrbracket.
        \]
        As a corollary, let $M\in\llbracket0,N-1\rrbracket$ and $\boldsymbol{\mathsf{S}}'\in\{0,1\}^{N-M}$ be given by $\mathsf{S}'(i)=\mathsf{S}(i+M+1)$ for each $i\in\llbracket0,N-M-1\rrbracket$.
        Then we have
        \[
        \mathscr{X}^{\boldsymbol{\mathsf{S}}} (N+1)-\mathscr{X}^{\boldsymbol{\mathsf{S}}} (M+1)\preceq\mathscr{X}^{\boldsymbol{\mathsf{S}}'} (N -M).
        \]  
    \end{enumerate}
\end{lemma}

The next two statements concern another auxiliary Markov chain $Q^y=\{Q^y(i)\}_{i=0}^\infty$ on $\mathbb{Z}_{\geq0}$, where $Q^y(0)=0$, and for each
$i,j\in\mathbb{Z}_{\geq0}$, we have 
\[
\mathbb{P}\left(Q^y(i+1)=j+1\mid Q^y(i)=j\right)=q^{y+j+1},
\qquad
\mathbb{P}\left(Q^y(i+1)=j\mid Q^y(i)=j\right)=1-q^{y+j+1}.
\] 
This allows us to bound the simpler Markov chain $\mathscr{X}$ from the successive zero-run lengths of the $0$-$1$ index sequence.

\begin{lemma}[c.f. Lemma 4.6 in \cite{SchmidYangShock}]\label{lem:upward jumps U chain}
Assume $q\in(0,1)$. For each $y\in\mathbb{Z}_{\geq0}$, and each $z>1$, there exists a constant $C\in\mathbb{N}$ depending only on $q$ and $z$, such that 
\begin{equation}\label{eq:Q-poly-bound}
    \mathbb{E}\left[z^{Q^y(h)}\right]\leq 1+Cq^yh^{C} \quad\mbox{ for all }  y,h\in\mathbb{Z}_{\geq0}.
\end{equation} 
\end{lemma}

\begin{lemma}[c.f. Lemma 4.7 in \cite{SchmidYangShock}]\label{lem:block moment bound simpler chain arbitrary start}
 Assume $q\in(0,1)$. Let $x,\ell\in\mathbb{Z}_{\geq0}$ and $h_0,\dots,h_{\ell-1}\in\mathbb{Z}_{\geq0}$. Define $Z(0)=x$ and 
\[
Z(i+1)= Z(i) +Q_i^{Z(i)}(h_i)-B_i\,{\bf1}_{\left\{Z(i)+Q_i^{Z(i)}(h_i)>0\right\}}\quad \mbox{ for }i\in\llbracket0,\ell-1\rrbracket,
\]
where, conditionally on $Z(i)$, $Q_i^{Z(i)}$ is an independent copy of the Markov chain from Lemma \ref{lem:upward jumps U chain}, $B_i$, $i\in\mathbb{Z}_{\geq0}$, are i.i.d. $\ber(p)$ random variables independent of all the $Q_i$'s, and $p\in[1-q,1]$. 
Assume $z>1$. There exist constants $C\in\mathbb{N}$ and $\rho\in(0,1)$ depending only on $q$ and $z$, such that 
\begin{equation}\label{eq:block moment bound simpler chain}
\mathbb{E}\left[z^{Z(\ell)}\right]
\leq C\left(1+\rho^\ell z^x+\sum_{i=0}^{\ell-1}\rho^{\ell-i} h_{i} ^C\right).
\end{equation} 
\end{lemma}

With these results at hand, we can now show Proposition \ref{prop:moment bound coex}. While the overall strategy is similar to the proof of Proposition~4.5 in \cite{SchmidYangShock}, the proof requires several adjustments.

\begin{proof}[Proof of Proposition \ref{prop:moment bound coex}] By Theorem \ref{thm:prob interpretation of Bryc polynomials} and conditional Jensen's inequality,
\[
 \mathbb E_{N,k}^0\left[W_{\bs(\bg)}(z)^m\right]= \mathbb E_{N,k}^0\left[\left(\mathbb{E}\left[ z^{X^{\boldsymbol{\mathsf{S}}^{\bs(\bg)}}(N+1)}\Big{|}\bg\right]\right)^m\right]\leq\mathbb E_{N,k}^0\left[ \mathbb{E}\left[z^{mX^{\boldsymbol{\mathsf{S}}^{\bs(\bg)}}(N+1)}\Big{|}\bg\right] \right],
\]
hence the second statement follows from the first statement by using it for $z^m$ in place of $z$.
Using Lemma~\ref{lem:comparision} with $F(\bg)=\mathbb E  [z^{X^{\boldsymbol{\mathsf S}^{\bs(\bg)}}(i)}\, |\,\bg ]$ and  Lemma ~\ref{lem:coupling birth death} \textup{(\romannum{2})}, to prove the first statement, it suffices to show 
that 
\begin{equation}\label{eq:moment bound coex proof1}
\sup_{N\in\mathbb{N}}
    \sup_{k\in\llbracket0,N\rrbracket}
    \sup_{i\in\llbracket0,N+1\rrbracket}{\mathbb E}\left[{\mathbb E}\left[ z^{{\mathscr{X}}^{\boldsymbol{\mathsf S}^{\bs(\boldsymbol{\lambda})}}(i)}\, \Big{|}\boldsymbol{Y} \right]   \right] <\infty
\end{equation}
is a finite constant depending only on $q,\b$ and $z$. Here, $\boldsymbol{\lambda}\in\lrw_N$ is given by $\lambda_j=Y_{j-k}-Y_{-k}$, $j\in\llbracket0,N\rrbracket$, $\boldsymbol{Y}=\{Y_j\}_{j=-\infty}^{\infty}$ is a bi-infinite $V$-shaped $\nu_{\b}$-random walk started at $Y_0=0$, and the inner expectation is taken over the transition randomness of the Markov chain \(\mathscr X\) from Definition \ref{def:simpler chain}, conditionally on \(\boldsymbol{Y}\). When $i=0$, the left side of \eqref{eq:moment bound coex proof1} is $1$. In the rest of the proof, we assume $i\in\llbracket1,N+1\rrbracket$.

Applying Lemma \ref{lem:regeneration lambda}, we choose \(\mathbb Z_{\ge0}\)-valued variables \(\upchi_j\) and \(\mathbb N\)-valued variables \(G_j\) for \(j\in\mathbb N\), and constants $C_1,c_1>0$ depending only on $\b$, satisfying the following conditions. The variables \(G_j\), \(j\in\mathbb N\), are independent, \(G_j\), \(j\ge2\), are i.i.d., and for \(j\geq2\), \(G_j\) is independent of \(\sigma(\upchi_1,\ldots,\upchi_{j-1})\). Let \(\upchi_0=N\); $K=\max\{j\in\mathbb Z_{\geq0}:\upchi_j>0\}<\infty$ almost surely; and $\mathrm R_j=\upchi_{j-1}-\upchi_j\in\mathbb Z_{\ge0}$ for \(j\in\mathbb N\). We also have 
\begin{align}
    & \label{eq:monotone upchi indices 1}\mbox{On }\{K\ge1\}, \mbox{ }N\geq \upchi_1>\cdots>\upchi_K\geq1\mbox{ and }\mathsf S^{\bs(\boldsymbol{\lambda})}(\upchi_j)=1,\,j\in\llbracket1,K\rrbracket;\\
    & \label{eq:exponential tail of Gj}\mathbb P(G_j\ge y)\le C_1e^{-c_1y}\mbox{ for all }j,y\in\mathbb N;\\
    & \label{eq:conditions for R}\mathrm R_1+\cdots+\mathrm R_{K+1}=N;\,\mathrm R_j=0\mbox{ for }j\ge K+2;\,\mathrm R_j\le G_j\mbox{ for }j\in\mathbb N.
\end{align}
Let $\mathrm{p}_j=\upchi_{K+1-j}$ for $j\in\llbracket1,K\rrbracket$, then $1\le\mathrm{p}_1< \cdots<\mathrm{p}_K\le N$ and $\mathsf S^{\bs(\boldsymbol{\lambda})}(\mathrm{p}_j)=1$, $j\in\llbracket1,K\rrbracket$. On \(\{K\ge1\}\), 
\[
    \mathrm{R}_1=N-\mathrm{p}_K,\qquad
    \mathrm{R}_j=\mathrm{p}_{K-j+2}-\mathrm{p}_{K-j+1} \mbox{  for }j\in\llbracket2,K\rrbracket,
    \qquad
    \mathrm{R}_{K+1}=\mathrm{p}_1.
\] 
For clarity, we split the rest of the proof into two steps. 
\\

\noindent\textbf{Step 1.} In this step, we conclude the proof of \eqref{eq:moment bound coex proof1} assuming the following claim:
for each $C_0\in\mathbb N$, $\rho\in(0,1)$, there exists a finite constant $C_2$ depending only on $\b,C_0,\rho$, such that for $N\in\mathbb N$, $k\in\llbracket0,N\rrbracket$, $s\in\llbracket0,N-1\rrbracket$, 
\begin{equation}\label{eq:moment bound coex proof2} 
    \mathbb E\!\left[\mathrm{R}_{A_s+1}^{C_0}\right]\leq C_2 ,
    \quad 
    \mathbb E\!\left[
        \sum_{r=2}^{K+1-A_s}
        \rho^{r-1}\mathrm{R}_{A_s+r}^{C_0}
    \right]\leq C_2 ,
    \quad 
    \mathbb E\!\left[
        \mathrm{R}_1^{C_0}+
        \sum_{r=2}^{K+1}
        \rho^{r-1}\mathrm{R}_r^{C_0}
    \right]\leq C_2 ,
\end{equation}
where $A_s=\max\{\ell\in\mathbb N:N-\upchi_{\ell}\le s\}$ for $s\in\llbracket0,N-1\rrbracket$, with the convention $A_s=0$ if the set is empty.

We begin by defining a random $\bar{\boldsymbol{\mathsf{S}}}\in\{0,1\}^{N+1}$.
On $\{K=0\}$, set $\bar{\mathsf{S}}(j)=0$, $j\in\llbracket0,N\rrbracket$. On $\{K\geq1\}$, set
\[
\bar{\mathsf{S}}(j)={\bf1}_{\{j\in\{\mathrm{p}_1,\dots,\mathrm{p}_K\}\}}\quad\mbox{ for }j\in\llbracket0,N\rrbracket.
\]
Since $\mathsf S^{\bs(\boldsymbol{\lambda})}(\mathrm{p}_{\ell})=1$, $\ell\in\llbracket1,K\rrbracket$, 
we have $\mathsf S^{\bs(\boldsymbol{\lambda})}(j)\ge \overline{\mathsf S}(j)$, $j\in\llbracket0,N\rrbracket$.
By Lemma \ref{lem:coupling birth death} \textup{(\romannum{3})}, \begin{equation}\label{eq:moment bound coex proof3} \mathscr{X}^{\boldsymbol{\mathsf{S}}^{\bs(\boldsymbol{\lambda})} }(i)\preceq \mathscr{X}^{\bar{\boldsymbol{\mathsf{S}}} }(i).\end{equation}
Let $J=\max\{\ell\in\llbracket1,K\rrbracket:\mathrm{p}_\ell<i\}$, with the convention that $J=0$ if the set is empty.
We define random variables $P_1$ and $P_2$. On the event $\{J=0\}$, set $P_1=P_2=0$. On the event $\{J\geq1\}$, set $\bar{\boldsymbol{\mathsf S}}^{(1)}\in\{0,1\}^{\mathrm{p}_J+1}$ by
\begin{equation}\label{eq:moment bound coex proof6}
\bar{\mathsf S}^{(1)}(j)={\bf1}_{\{j\in\{\mathrm{p}_1,\dots,\mathrm{p}_J\}\}} 
\quad
 \mbox{for } j\in\llbracket0,\mathrm{p}_J\rrbracket .
\end{equation}
If $i-\mathrm{p}_J-1=0$, let $\bar{\boldsymbol{\mathsf S}}^{(2)}=\emptyset$; if $i-\mathrm{p}_J-1\geq1$, define $\bar{\boldsymbol{\mathsf S}}^{(2)}\in\{0,1\}^{i-\mathrm{p}_J-1 }$ to be the all-zero sequence.
Then the \(\llbracket0,\mathrm{p}_J\rrbracket\) symbols of \(\bar{\boldsymbol{\mathsf S}}\) are \(\bar{\boldsymbol{\mathsf S}}^{(1)}\), and the \(\llbracket \mathrm{p}_J+1,i-1\rrbracket\) symbols of \(\bar{\boldsymbol{\mathsf S}}\) are those of \(\bar{\boldsymbol{\mathsf S}}^{(2)}\). 
We use the convention that $\mathscr X^{\emptyset}(0)=0$. Conditionally on $\mathcal{G}=\sigma(K,\mathrm{p}_1,\dots,\mathrm{p}_K)$, and on the event $\{J\geq1\}$, let 
\[
P_1\overset{d}{=}\mathscr{X}^{\bar{\boldsymbol{\mathsf S}}^{(1)}}(\mathrm{p}_J+1),\quad
P_2\overset{d}{=}\mathscr{X}^{\bar{\boldsymbol{\mathsf S}}^{(2)}}(i-\mathrm{p}_J-1 ) .
\] and couple these variables with $\mathscr{X}^{\bar{\boldsymbol{\mathsf S}}}(i)$ using the corollary of Lemma \ref{lem:coupling birth death} (\romannum{4}) so that
\begin{equation}\label{eq:moment bound coex proof4} 
\mathscr{X}^{\bar{\boldsymbol{\mathsf S}}}(i)
\leq
P_1+P_2\quad\mbox{ on }\{J\geq1\}.
\end{equation}
Therefore,
\begin{equation}\label{eq:moment bound coex proof5}  
\begin{split}
       &  \mathbb{E}\left[z^{\mathscr{X}^{\boldsymbol{\mathsf{S}}^{\bs(\boldsymbol{\lambda})} }(i)}\right] \leq
        \mathbb{E}\left[z^{\mathscr{X}^{\bar{\boldsymbol{\mathsf{S}}} }(i)}\right]
        \leq\mathbb{E}\left[z^{\mathscr{X}^{\bar{\boldsymbol{\mathsf{S}}} }(i)}{\bf1}_{\{J=0\}}\right]+\mathbb{E}\left[z^{P_1+P_2}{\bf1}_{\{J\geq1\}}\right]
        \\
       &    
          \leq \mathbb{E}\left[z^{\mathscr{X}^{\bar{\boldsymbol{\mathsf{S}}} }(i)}{\bf1}_{\{J=0\}}\right]
     +\mathbb{E}\left[\left(z^{2P_1}+z^{2P_2}\right){\bf1}_{\{J\geq 1\}}\right]  ,
\end{split}
\end{equation}
where the first step uses \eqref{eq:moment bound coex proof3}, the second step uses \eqref{eq:moment bound coex proof4}, and the third step uses $xy\leq x^2+y^2$.

It remains to bound the two summands on the right side of \eqref{eq:moment bound coex proof5}. 
On \(\{J=0\}\), the first \(i\) symbols of \(\overline{\boldsymbol{\mathsf S}}\) are all zero, hence \(\mathscr X^{\overline{\boldsymbol{\mathsf S}}}(i)\overset{d}{=}Q^0(i)\) where \(Q^0\) is from Lemma \ref{lem:upward jumps U chain}, therefore  for   \(C_3\in\mathbb N\) depending only on \(q,z\),
\[
    \mathbb E\!\left[
        z^{\mathscr X^{\overline{\boldsymbol{\mathsf S}}}(i)}{\bf1}_{\{J=0\}}
        \,\middle|\,\mathcal{G}\right]
    \le \left(1+ C_3  i^{C_3} \right){\bf1}_{\{J=0\}}\leq 1+C_3 (2\mathrm{R}_{A_{N-1}+1})^{C_3},
\]
where the last step uses: on \(\{J=0,K\ge1\}\), \(i\le\mathrm{p}_1=\mathrm{R}_{A_{N-1}+1}\); and on \(\{J=0,K=0\}\),   \(\mathrm{R}_{A_{N-1}+1}=R_1=N\),  \(i\le 2\mathrm{R}_{A_{N-1}+1}\). Taking $s=N-1$ in the first inequality of \eqref{eq:moment bound coex proof2}, we have $C_4<\infty$ depending only on $q,\b,z$,  
\begin{equation}\label{eq:moment bound coex proof7}  
    \mathbb E\!\left[
        z^{\mathscr X^{\overline{\boldsymbol{\mathsf S}}}(i)}{\bf1}_{\{J=0\}} 
    \right]
    \le C_4,
\end{equation}

For the part of the second summand which involves $\bar{\boldsymbol{\mathsf S}}^{(1)}$, in view of \eqref{eq:moment bound coex proof6}, the successive zero-run lengths before the ones are $g_0=\mathrm{p}_1,\, g_1=\mathrm{p}_2-\mathrm{p}_{1}-1,\dots,g_{J-1}=\mathrm{p}_J-\mathrm{p}_{J-1}-1$. A deterministic zero-run of length \(h\), started from state \(y\), has the same increment law as \(Q^y(h)\) from Lemma~\ref{lem:upward jumps U chain}, and a subsequent \(1\)-symbol subtracts an independent \(\ber(1-q)\) variable from positive states. Therefore, conditionally on \(\mathcal G\) and on \(\{J\ge1\}\), the endpoint of the sequence \(\bar{\boldsymbol{\mathsf S}}^{(1)}\) is represented by the chain \(Z\) in Lemma \ref{lem:block moment bound simpler chain arbitrary start} with initial state \(x=0\), parameter \(p=1-q\), block number \(J\), and zero-run lengths \(h_{\ell-1}=g_{\ell-1}\), \(\ell\in\llbracket1,J\rrbracket\). Applying Lemma~\ref{lem:block moment bound simpler chain arbitrary start} with \(z^2\) in place of \(z\), there exist constants
\(C_5\in\mathbb N\) and \(\rho\in(0,1)\), depending only on \(q\) and \(z\), such that
\begin{equation}\label{eq:moment bound coex proof8} 
    \mathbb{E}\left[z^{2P_1}{\bf1}_{\{J\geq 1\}}\,\middle|\,\mathcal{G}\right]
     \le
    C_5\left(
        1+\sum_{\ell=1}^{J}\rho^{J-\ell+1}g_{\ell-1}^{C_5}
    \right){\bf1}_{\{J\geq1\}}                                                    
     \le
    C_5\left(
        1+\sum_{r=2}^{J+1}\rho^{r-1}\mathrm{R}_{K-J+r}^{C_5}
    \right){\bf1}_{\{J\geq1\}},
\end{equation}
where the last step uses $g_{\ell-1}\le \mathrm{R}_{K+2-\ell}$ for $\ell\in\llbracket1,J\rrbracket$, together with the change of variables \(r=J+2-\ell\). For the part which involves $\bar{\boldsymbol{\mathsf S}}^{(2)}$, since it is an all-zero sequence, 
$\mathscr{X}^{\bar{\boldsymbol{\mathsf S}}^{(2)}}(i-\mathrm{p}_J-1 )\overset{d}{=}Q^0(i-\mathrm{p}_J-1)$.
Using Lemma \ref{lem:upward jumps U chain} with \(z^2\) in place of \(z\), we obtain \(C_6\in\mathbb N\) depending only on \(q,z\),  such that
\begin{equation}\label{eq:moment bound coex proof9} 
   \mathbb{E}\left[z^{2P_2}{\bf1}_{\{J\geq 1\}}\,\middle|\,\mathcal{G}\right]
    \leq \left(1+C_6\left(i-\mathrm{p}_J-1\right)^{C_6} \right){\bf1}_{\{J\geq1\}}\leq \left(1+C_6\mathrm{R}_{K-J+1}^{C_6}\right){\bf1}_{\{J\geq1\}},
\end{equation}
where we use $i-\mathrm{p}_J-1\leq \mathrm{R}_{K-J+1}$. 
Combining \eqref{eq:moment bound coex proof8}, \eqref{eq:moment bound coex proof9}, there exists \(C_7\in\mathbb N\) depending only on \(q,z\),
\begin{equation}\label{eq:moment bound coex proof10} 
    \mathbb{E}\left[\left(z^{2P_1}+z^{2P_2}\right){\bf1}_{\{J\geq 1\}}\,\middle|\,\mathcal{G}\right] \leq C_7\left(
        1+\mathrm{R}_{K-J+1}^{C_7}
        +
        \sum_{r=2}^{J+1}\rho^{r-1}\mathrm{R}_{K-J+r}^{C_7}
    \right){\bf1}_{\{J\geq1\}}.
\end{equation}
If \(i\le N\), then \(K-J=A_{N-i}\): indeed, \(J\) is the number of indices
\(\ell\in\llbracket1,K\rrbracket\) such that \(\mathrm p_\ell<i\), so
\(K-J\) is the largest \(a\in\mathbb N\) such that \(\upchi_a\ge i\), with
value \(0\) if no such \(a\) exists. This is equivalently the largest
\(a\in\mathbb N\) such that \(N-\upchi_a\le N-i\), with the same convention.
Hence the expectation of the right side of \eqref{eq:moment bound coex proof10}
is uniformly bounded by the first and second inequality in
\eqref{eq:moment bound coex proof2}. If \(i=N+1\), then \(K-J=0\), and the same
conclusion follows from the third inequality in \eqref{eq:moment bound coex proof2}.
Therefore, we obtain $C_8<\infty$ depending only on $q,\b,z$, such that
\begin{equation}\label{eq:moment bound coex proof11}  
    \mathbb{E}\left[\left(z^{2P_1}+z^{2P_2}\right){\bf1}_{\{J\geq 1\}}\right] \leq C_8.
\end{equation}
Plugging \eqref{eq:moment bound coex proof7} and \eqref{eq:moment bound coex proof11} into \eqref{eq:moment bound coex proof5}, we obtain
\[
 \mathbb{E}\left[z^{\mathscr{X}^{\boldsymbol{\mathsf{S}}^{\bs(\boldsymbol{\lambda})} }(i)}\right]
\leq C_4+C_8\quad\mbox{ for all }N\in\mathbb{N},\, k\in\llbracket0,N\rrbracket,  \, i\in\llbracket1,N+1\rrbracket,
\]
where the right side is a finite constant depending only on $q,\b$, and $z$. We conclude the proof of \eqref{eq:moment bound coex proof1}. \\

\noindent\textbf{Step 2.} In this step, we prove the claim \eqref{eq:moment bound coex proof2}, which we restate as follows: 
for each $C_0\in\mathbb N$, $\rho\in(0,1)$, there exists a finite constant $C_2$ depending only on $\b,C_0,\rho$, such that for $N\in\mathbb N$, $k\in\llbracket0,N\rrbracket$, $s\in\llbracket0,N-1\rrbracket$, 
\begin{equation}\label{eq:moment bound coex proof2 again} 
    \mathbb E\!\left[\mathrm{R}_{A_s+1}^{C_0}\right]\leq C_2 ,
    \quad 
    \mathbb E\!\left[
        \sum_{r=2}^{K+1-A_s}
        \rho^{r-1}\mathrm{R}_{A_s+r}^{C_0}
    \right]\leq C_2 ,
    \quad 
    \mathbb E\!\left[
        \mathrm{R}_1^{C_0}+
        \sum_{r=2}^{K+1}
        \rho^{r-1}\mathrm{R}_r^{C_0}
    \right]\leq C_2 ,
\end{equation}
where $A_s=\max\{\ell\in\mathbb N:N-\upchi_{\ell}\le s\}$ for $s\in\llbracket0,N-1\rrbracket$, with the convention $A_s=0$ if the set is empty. 

Set
\(\mathcal F_a=\sigma(\upchi_1,\ldots,\upchi_a)\) for \(a\in\mathbb Z_{\ge0}\), with
\(\mathcal F_0\) trivial.
For the first inequality in \eqref{eq:moment bound coex proof2 again}, for \(y\in\mathbb N\),  
\begin{equation}\label{eq:moment bound coex proof13}  
\begin{aligned}
     \mathbb P(\mathrm{R}_{A_s+1}\ge y)
    &\le
    \sum_{a=0}^{\infty}\sum_{\ell=0}^{\infty}
     \mathbb P \!\left(
        N-\upchi_a=s-\ell,\ a\le K,\ \mathrm{R}_{a+1}\ge y,\ N-\upchi_{a+1}>s
    \right)                                                   \\
    &\le
    \sum_{a=0}^{\infty}\sum_{\ell=0}^{\infty}
     \mathbb P \!\left(
         N-\upchi_a=s-\ell,\ a\le K,\ G_{a+1}\ge \max(y,\ell+1)
    \right)\\
    &\leq C_1\sum_{\ell=0}^{\infty}
    e^{-c_1\max (y,\ell+1)}
    \sum_{a=0}^{\infty}
     \mathbb P (N-\upchi_a=s-\ell,\ a\le K)             \\
    &\le
    2C_1\sum_{\ell=0}^{\infty}
    e^{-c_1\max(y,\ell+1)}
    \le C_8e^{-c_8y},
\end{aligned}
\end{equation}
for constants \(C_8,c_8>0\) depending only on \(\b\). 
Here, the first step follows from decomposition, the second step follows from \(\{N-\upchi_a=s-\ell,\,N-\upchi_{a+1}>s\}\subseteq \{G_{a+1}\geq \ell+1\}\) because of \eqref{eq:conditions for R}, and the third step uses the exponential tail estimate of $G_{a+1}$ and that \(\{N-\upchi_a=s-\ell,\ a\le K\}\in\mathcal F_a\) while \(G_{a+1}\) is independent of $\mathcal F_a$, the fourth step follows from \(\sum_{a=0}^{\infty}{\bf1}_{\{N-\upchi_a=s-\ell,\ a\le K\}}\le2\) for each fixed \(\ell\), since $N=\upchi_0\geq \upchi_1>\dots>\upchi_K$. 
The first inequality in \eqref{eq:moment bound coex proof2 again} then follows from the exponential tails \eqref{eq:moment bound coex proof13}.
For the second inequality in \eqref{eq:moment bound coex proof2 again},  
\[ 
     \mathbb E \!\left[
        \sum_{r=2}^{K+1-A_s}
        \rho^{r-1}\mathrm{R}_{A_s+r}^{C_0}
    \right]
    \le
    \sum_{r=2}^{\infty}\rho^{r-1}
    \sum_{a=0}^{\infty}
     \mathbb E \!\left[
        {\bf1}_{\{A_s=a\}}G_{a+r}^{C_0}
    \right]
    =
    \sum_{r=2}^{\infty}\rho^{r-1}\mathbb E \left[G_2^{C_0}\right]
     \leq C_2,
\]
where the first step uses \eqref{eq:conditions for R}, and the second step uses $\{A_s=a\}=\{N-\upchi_a\le s<N-\upchi_{a+1}\}\in\mathcal F_{a+1}$ while \(G_{a+r}\) is independent of \(\mathcal F_{a+1}\) and has the same law as \(G_2\) for \(r\ge2\).
For the third inequality in \eqref{eq:moment bound coex proof2 again},  
\[
\mathbb E\!\left[
        \mathrm{R}_1^{C_0}+
        \sum_{r=2}^{K+1}
        \rho^{r-1}\mathrm{R}_r^{C_0}
    \right]
\le
\mathbb E[G_1^{C_0}]
+
\sum_{r=2}^{\infty}\rho^{r-1}\mathbb E[G_2^{C_0}]
\leq C_2,
\]
where we use \eqref{eq:conditions for R}, the finite \(C_0\)-th moment of \(G_1\),
and that \(G_j\), \(j\geq2\), are i.i.d. and subexponential. 
\end{proof}

\subsection{Stability of the lower compositions}

Recall the composition $\bs(\bg)=(\sigma_0,\dots,\sigma_r)$ associated to a path $\bg$ given by \eqref{eq:Signature}, and  let $\sigma_i=0$ for all $i>r$. We control now the change of the composition at levels $0$ to $\kappa n$ for some sufficiently small constant $\kappa>0$ when we condition on a non-negligible event $A$ at distance at least $n$ away from the minimum of the path.

\begin{lemma}\label{lem:stability low compositions coex}
    Assume $\b>1$.
    For \(k\in\llbracket0,N\rrbracket\), write $\mathbb P_{N,k}^0(\,\cdot\,)=\mathbb P^{(0,\b,\b)}(\,\cdot\,\mid T_N=k)$ and $\mathbb E_{N,k}^0$ to denote the expectation under $\mathbb P_{N,k}^0$, where we recall that $\mathbb{P}^{(0,\b,\b)}$ and $T_N$ are as in \eqref{eq:PN00 and TN}.   
For \(k,m\in\llbracket0,N\rrbracket\), let
\begin{equation}\label{eq:Ianddelta}
    \begin{split}
         \mathsf I_{N,k,m}& =  
    \{i\in\llbracket0,N-1\rrbracket: |i-k|\ge m\},\\
     \Delta_{N,k,m}(\bg)&=\left(\gamma_{i+1}-\gamma_i\right)_{i\in\mathsf I_{N,k,m}}
    \in \{-1,0,1\}^{\mathsf I_{N,k,m}},
    \end{split}
\end{equation} 
where \(\{-1,0,1\}^{\emptyset}\) is interpreted as a one-point space. Then there exist constants \(\kappa,\delta,C,c>0\), depending only on \(\b\), such that for every \(N\ge1\), every \(k,m\in\llbracket0,N\rrbracket\) with \(m\ge1\), and every event
\[
    A\in\sigma(\Delta_{N,k,m})
    \quad\text{satisfying}\quad
    \mathbb P_{N,k}^0(A)\ge e^{-\delta m},
\]
we have
\[
    \TV{
    \mathbb P_{N,k}^0
    \left(
        (\sigma_0,\ldots,\sigma_{\lfloor\kappa m\rfloor})\in\cdot
        \,\middle|\,A
    \right)
    -
    \mathbb P_{N,k}^0
    \left(
        (\sigma_0,\ldots,\sigma_{\lfloor\kappa m\rfloor})\in\cdot
    \right)
    }
    \le Ce^{-cm}.
\]
\end{lemma}
\begin{proof}
For $k\in\llbracket0,N\rrbracket$, consider the bijection $\Phi:\lrw_k\times\lrw_{N-k}\to\lrw_N$ given by $\gamma_{k-i}=-L_k+L_i$, $i\in\llbracket0,k\rrbracket$ and $\gamma_{k+i}=-L_k+R_i$, $i\in\llbracket0,N-k\rrbracket$. Then the inverse of $\Phi$ is given by $L_j=\gamma_{k-j}-\gamma_k$, $j\in\llbracket0,k\rrbracket$ and $R_j=\gamma_{k+j}-\gamma_k$, $j\in\llbracket0,N-k\rrbracket$. 
Define \(\mathrm P_{N,k}^0\) to be the measure on $\lrw_k\times\lrw_{N-k}$ induced by $\Phi$ from the measure \(\mathbb P_{N,k}^0\) on $\lrw_N$. By Lemma \ref{lem:decomposition TASEP}, \(\mathrm P_{N,k}^0\) is the law of $(\boldsymbol{L}, \boldsymbol{R})$, where $\boldsymbol{R}=(R_j)_{j\in\llbracket0,N-k\rrbracket}$ is a \(\nu_\b\)-random walk started from \(0\), conditioned to stay non-negative up to time \(N-k\); and $\boldsymbol{L}=(L_j)_{j\in\llbracket0,k\rrbracket}$ is a \(\nu_\b\)-random walk started from \(0\), conditioned on staying strictly positive during $\llbracket1,k\rrbracket$ steps if $k\geq1$, $L_0=0$ if $k=0$; and $\boldsymbol{L}$ and $\boldsymbol{R}$ are independent.
We now define a measure \({\mathrm P}_{N,k,m}^0\) on \(\lrw_k\times\lrw_{N-k}\). First sample
\begin{equation}\label{eq:centralpieces}
    (L_i)_{ i\in\llbracket0, \min(k,m-1)\rrbracket}
    \quad\text{and}\quad
    (R_i)_{i\in\llbracket0, \min(N-k,m)\rrbracket}
\end{equation}
with their joint law under \(\mathrm P_{N,k}^0\), then conditionally on these sampled pieces, if \(k>m-1\), extend the left piece from
\(L_{m-1}\) to \(L_k\) using independent \(\nu_\b\)-increments; and if \(N-k>m\), extend the right piece from
\(R_{m}\) to \(R_{N-k}\) using independent \(\nu_\b\)-increments.  The two extensions are independent of each
other and independent of the sampled pieces. Let ${\mathsf P}_{N,k,m}^0$ be the measure on $\lrw_N$ induced from \({\mathrm P}_{N,k,m}^0\) via the bijection.

Fix $\kappa=(\b-1)/(40(\b+1))$. 
Recall that for $h\in\mathbb Z_{\ge0}$, $\sigma_h =\#\{j\in\llbracket0,N\rrbracket: \gamma_j=\min(\bg)+h\}$.
Let
\[
\vartheta_h =
    \#\left\{
        j\in
        \llbracket \max(0,k-\left\lfloor m/2\right\rfloor),\,\min(N,k+\left\lfloor m/2\right\rfloor)\rrbracket:
        \gamma_j=\gamma_k+h
    \right\}\quad\mbox{for }h\in\mathbb{Z}_{\geq0},
\]
and define the event 
\[
    D_m =
    \left\{
        (\sigma_0,\ldots,\sigma_{\lfloor\kappa m\rfloor})
        \ne
        (\vartheta_0,\ldots,\vartheta_{\lfloor\kappa m\rfloor})
    \right\}.
\]
For clarity, we split the rest of the proof into three steps. 
\\

\noindent\textbf{Step 1.} In this step, we conclude the proof of the lemma assuming the following claim: there exist constants $C_1,c_1,C_2,c_2>0$ depending only on $\b$, such that for each \(N\in\mathbb N\), \(k,m\in\llbracket0,N\rrbracket\) with \(m\ge1\),
\begin{equation}\label{eq:stability proof1}
    \mathbb P_{N,k}^0(D_m)\leq C_1e^{-c_1m}\quad\mbox{and}\quad\TV{\mathbb P_{N,k}^0-{\mathsf P}_{N,k,m}^0}\leq C_2e^{-c_2m}.
\end{equation}
Fix an arbitrary $B\subseteq \mathbb Z_{\ge0}^{\lfloor\kappa m\rfloor+1}$. Let 
\[
    E_B
    =
    \{\left(\sigma_0,\ldots,\sigma_{\lfloor\kappa m\rfloor}\right)\in B\}
    \quad\mbox{and}\quad
    F_B
    =
    \{\left(\vartheta_0,\ldots,\vartheta_{\lfloor\kappa m\rfloor}\right)\in B\}.
\]
On \(D_m^{\complement}\), the events \(E_B\) and \(F_B\) agree.  Hence, for every event \(A\) with \(\mathbb P_{N,k}^0(A)>0\),
\begin{equation}\label{eq:stability proof2}
    \left|
        \mathbb P_{N,k}^0(E_B\mid A)-\mathbb P_{N,k}^0(E_B)
    \right|
     \le
    \left|
        \mathbb P_{N,k}^0(F_B\mid A)
        -
        \mathbb P_{N,k}^0(F_B)
    \right|   +
    \mathbb P_{N,k}^0(D_m\mid A)
    +
    \mathbb P_{N,k}^0(D_m). 
\end{equation}
The event \(F_B\) depends only on $(L_i)_{ i\in\llbracket0, \min(k,\lfloor m/2\rfloor)\rrbracket}$ and  $(R_i)_{ i\in\llbracket0, \min(N-k,\lfloor m/2\rfloor)\rrbracket}$.
Indeed, for $h\in\llbracket0,\lfloor\kappa m\rfloor\rrbracket$, 
\[
    \vartheta_h
    =
    \#\{i\in\llbracket0,\min(N-k,\lfloor m/2\rfloor)\rrbracket:R_i=h\}
    +
    \#\{i\in\llbracket1,\min(k,\lfloor m/2\rfloor)\rrbracket:L_i=h\},
\] 
where the second set is empty if \(k=0\).  For \(m\ge2\), \(\lfloor m/2\rfloor\le  m-1\) and \(\lfloor m/2\rfloor\le m\).  For
\(m=1\), \(\lfloor m/2\rfloor=0\).  Thus \(F_B\) is measurable with respect to the central pieces \eqref{eq:centralpieces} sampled before the independent extensions in the definition of \({\mathrm P}_{N,k,m}^0\).
On the other hand, \(A\in\sigma(\Delta_{N,k,m})\) depends only on the outside increments 
\[
    (R_{s+1}-R_s)_{s\in\llbracket m,N-k-1\rrbracket} \quad\mbox{and}\quad (L_{s-1}-L_s)_{s\in\llbracket m,k\rrbracket},
\]
with empty intervals contributing no variables. Since these increments are functions only of the independent extensions under
\({\mathsf P}_{N,k,m}^0\), we know that \(A\) and \(F_B\) are independent under \({\mathsf P}_{N,k,m}^0\).  Hence,
\[
    {\mathsf P}_{N,k,m}^0(F_B\cap A)={\mathsf P}_{N,k,m}^0(F_B)\,
    {\mathsf P}_{N,k,m}^0(A)
    =
    \mathbb P_{N,k}^0(F_B)\,
    {\mathsf P}_{N,k,m}^0(A),
\] and therefore,
\begin{equation}\label{eq:stability proof3}
    \begin{split}
        &\left|
        \mathbb P_{N,k}^0(F_B\mid A)
        -
        \mathbb P_{N,k}^0(F_B)
    \right|= 
    \left|
        \mathbb P_{N,k}^0(F_B\cap A)
        -
        \mathbb P_{N,k}^0(F_B)\mathbb P_{N,k}^0(A)
    \right| \cdot \mathbb P_{N,k}^0(A)^{-1}\\
    & \le2\TV{
            \mathbb P_{N,k}^0
            -
            {\mathsf P}_{N,k,m}^0
        } \cdot \mathbb P_{N,k}^0(A)^{-1}
    \le
    2C_2 e^{-c_2m}\cdot \mathbb P_{N,k}^0(A)^{-1}.
    \end{split}
\end{equation}
Let $\delta= \min (c_1,c_2)/2$. 
If \(\mathbb P_{N,k}^0(A)\ge e^{-\delta m}\), then
 $\mathbb P_{N,k}^0(D_m\mid A)
    \le
    \mathbb P_{N,k}^0(A)^{-1}\mathbb P_{N,k}^0(D_m)
    \le
    C_1e^{-(c_1-\delta)m}$.
Plugging this bound,  the first inequality in \eqref{eq:stability proof1}, and \eqref{eq:stability proof3} together into \eqref{eq:stability proof2}, we obtain 
\[ 
    \left|
        \mathbb P_{N,k}^0(E_B\mid A)-\mathbb P_{N,k}^0(E_B)
    \right|
     \le
    2C_2 e^{-(c_2-\delta)m}
    +
    C_1 e^{-(c_1-\delta)m}
    +
    C_1 e^{-c_1m}   
     \le
    C e^{-cm}, 
\]
for constants \(C,c>0\) depending only on \(\b\).  Taking the supremum over all  $B\subseteq \mathbb Z_{\ge0}^{\lfloor\kappa m\rfloor+1}$ concludes the proof.\\

\noindent\textbf{Step 2.} In this step, we prove the first inequality in \eqref{eq:stability proof1}: there exist $C_1,c_1>0$ depending only on $\b$,  
\begin{equation}\label{eq:stability proof4}
    \mathbb P_{N,k}^0(D_m)\leq C_1e^{-c_1m}\quad\mbox{for each }N\geq m\geq1.
\end{equation}
On \(\{T_N=k\}\), we have \(\gamma_k=\min(\bg)\).  Hence \(D_m\) can occur only if $\bg$ visits one of the levels
$\llbracket\gamma_k,\gamma_k+\lfloor\kappa m\rfloor\rrbracket$ outside the window $\llbracket \max(0,k-\lfloor m/2 \rfloor),\,\min(N,k+\lfloor m/2 \rfloor)\rrbracket$. Therefore, $D_m\subseteq D_m^{\mathrm L}\cup D_m^{\mathrm R}$, where
\begin{equation*}
    \begin{split}
        D_m^{\mathrm L}
    &=
    \left\{
        \mbox{there exists } i\in\llbracket \lfloor m/2 \rfloor+1,k\rrbracket\mbox{ satisfying }
        L_i\le \lfloor\kappa m\rfloor
    \right\} ,\\
    D_m^{\mathrm R}
    &=
    \left\{\mbox{there exists }  i\in\llbracket \lfloor m/2 \rfloor+1,N-k\rrbracket\mbox{ satisfying }
        R_i\le \lfloor\kappa m\rfloor
    \right\}.
    \end{split}
\end{equation*}
Assume \(m\ge4\). If \(N-k\le \left\lfloor m/2\right\rfloor\), then \(D_m^{\mathrm R}=\emptyset\). If \(N-k>\left\lfloor m/2\right\rfloor\), then, conditionally on \(R_{\left\lfloor m/2\right\rfloor}=x\), the future
\((R_{\left\lfloor m/2\right\rfloor+j})_{j\in\llbracket0, N-k-\left\lfloor m/2\right\rfloor\rrbracket}\) is an ordinary \(\nu_\b\)-walk started from \(x\), conditioned to stay non-negative for the remaining time interval. By Lemma \ref{lem:basic hitting monotone coupling}, on \(\{x>\left\lfloor\kappa m\right\rfloor\}\), we have $\mathbb P_{N,k}^0(D_m^{\mathrm R}\, |\, R_{\left\lfloor m/2\right\rfloor}=x)
    \le \b^{-2(x-\left\lfloor\kappa m\right\rfloor)}$,
where the right side is the probability that an unconditioned \(\nu_\b\)-walk started from \(x\) ever hits \(\left\lfloor\kappa m\right\rfloor\). Also, since \(\kappa=(\b-1)/(40(\b+1))\), we have \(2\kappa m\le (\b-1)/(4(\b+1)) \cdot\left\lfloor m/2\right\rfloor\)  for all \(m\ge4\), and Lemma \ref{lem:basic hitting monotone coupling} gives
\[
    \mathbb P_{N,k}^0(R_{\left\lfloor m/2\right\rfloor}\le 2\kappa m)
    \le
    \mathbb P_{N,k}^0(R_{\left\lfloor m/2\right\rfloor}\le (\b-1)/(4(\b+1)) \cdot\left\lfloor m/2\right\rfloor)
    \le C_{3}e^{-c_{3}m}
\]
for constants \(C_{3},c_{3}>0\) depending only on \(\b\).  Hence
\begin{equation*} 
    \begin{split} 
        \mathbb P_{N,k}^0\left(D_m^{\mathrm R}\right)
    \le
    \mathbb P_{N,k}^0\left(R_{\left\lfloor m/2\right\rfloor}\le 2\kappa m\right)
    +
    \mathbb E_{N,k}^0\!\left[
        {\bf1}_{\{R_{\left\lfloor m/2\right\rfloor}>2\kappa m\}}
        \b^{-2\left(R_{\left\lfloor m/2\right\rfloor}-\left\lfloor\kappa m\right\rfloor\right)}
    \right]    \le
    C_{3}e^{-c_{3}m}
    +
    \b^{-2\kappa m} .
    \end{split}
\end{equation*}  
The left piece is treated in the same way with an additional shift. We have
    $\mathbb P_{N,k}^0(D_m^{\mathrm L})
    \le C_4e^{-c_4m}+
    \b^{-2\kappa m}$
for some \(C_4,c_4>0\) depending only on \(\b\).  Combining the left and right bounds and enlarging the constants to absorb the \(m\in\{1,2,3\}\) cases, we conclude the proof of \eqref{eq:stability proof4}.\\ 

\noindent\textbf{Step 3.} In this step, we prove the second inequality in \eqref{eq:stability proof1}: there exist $C_2,c_2>0$ depending only on $\b$,  
\begin{equation}\label{eq:stability proof6}
    \TV{\mathbb P_{N,k}^0-{\mathsf P}_{N,k,m}^0}\leq C_2e^{-c_2m}\quad\mbox{for each }N\geq m\geq1,\,k\in\llbracket0,N\rrbracket.
\end{equation}
In view of the bijection $\Phi$ as in the first paragraph of this proof, it suffices to prove the same estimate on $\lrw_k\times\lrw_{N-k}$. 
Let $t\in\mathbb{Z}_{\geq0}$ and $(S_n)_{n\in\llbracket0,t\rrbracket}$ be an ordinary \(\nu_\b\)-random walk started from \(S_0=x\), and denote its law by $\mathbb{P}_x$. Then 
Lemma \ref{lem:basic hitting monotone coupling} yields $\mathbb{P}_x(S_n\geq0,\,n\in\llbracket0,t\rrbracket)\ge1-\b^{-2(x+1)}$, and hence 
\begin{equation}\label{eq:stability proof7}
    \TV{\mathbb{P}_x(\,\cdot\,|S_n\geq0,\,n\in\llbracket0,t\rrbracket)-\mathbb{P}_x}
    \le C_5\b^{-2x}\quad\mbox{for }x\in\mathbb{Z}_{\geq0} ,
\end{equation}
where $C_5>0$ depending only on $\b$. By a similar argument, 
\begin{equation}\label{eq:stability proof8}
    \TV{
        \mathbb{P}_x(\,\cdot\,\mid S_n\geq1,\,n\in\llbracket1,t\rrbracket)
        -
        \mathbb{P}_x
    }
    \le C_6\b^{-2x}\quad\mbox{for }x\in\mathbb{N} ,
\end{equation}
where $C_6>0$ depending only on $\b$.
Now disintegrate both \(\mathrm P_{N,k}^0\) and \( {\mathrm P}_{N,k,m}^0\) with respect to the sampled central pieces.   The central marginal is the same, and given these central pieces, the existing left and right tails under \(\mathrm P_{N,k}^0\) are independent conditioned walks, while the existing left and right tails under \( {\mathrm P}_{N,k,m}^0\) are independent unconditioned walks. 
By Lemma \ref{lem:basic hitting monotone coupling}, when \(N-k>m\),
\begin{equation}\label{eq:stability proof9}
    \mathbb E_{N,k}^0\left[\b^{-2R_m}\right]
    \le
    \mathbb P_{N,k}^0(R_m\le (\b-1)/(4(\b+1)) m)+\b^{- (\b-1)/(2(\b+1)) m}
    \le C_7e^{-c_7m}
\end{equation}
for constants \(C_7,c_7>0\) depending only on \(\b\). Similarly, when \(k>m-1\),  
\begin{equation}\label{eq:stability proof10}
    \mathbb E_{N,k}^0\left[\b^{-2L_{m-1}}\right] 
    \le C_8e^{-c_8m},
\end{equation}
for constants \(C_8,c_8>0\) depending only on \(\b\). 
Combining \eqref{eq:stability proof7}, \eqref{eq:stability proof8}, \eqref{eq:stability proof9}, \eqref{eq:stability proof10},  we obtain
\[
    \TV{
        \mathbb P_{N,k}^0
        -
         {\mathsf P}_{N,k,m}^0
    }
    \le
    C_5\mathbb E_{N,k}^0\left[\b^{-2R_m}\right]\cdot{\bf1}_{\{m<N-k\}}+C_6\mathbb E_{N,k}^0\left[\b^{-2L_{m-1}}\right]\cdot{\bf1}_{\{2\leq m\leq k\}} 
    +
    {\bf1}_{\{m=1\}} \leq C_2e^{-c_2m},
\]
which concludes the proof of \eqref{eq:stability proof6}.
\end{proof}

\subsection{Total variation distance of the tail kernels}

In the following, our goal is to show that that the birth-death chain in the probabilistic characterization of the Bryc polynomial does  up to an exponentially small error in $m$ not change at location $N$ when we condition on a typical location after $N-m$ steps for some $m \in \llbracket N \rrbracket$. To this end, we require the following definition of the birth-death chain with a fixed starting location at step $N-m$.

\begin{definition}\label{def:chain tail}
Let $q\in(0,1)$, $N\in\mathbb{N}$, and $\boldsymbol{\mathsf{S}}=(\mathsf{S}(0),\dots,\mathsf{S}(N))\in\{0,1\}^{N+1}$ be a $0$-$1$ sequence. 
For $m\in\llbracket0,N-1\rrbracket$, denote $\boldsymbol{\mathsf{S}}_{N,m}=(\mathsf{S}(N-m),\dots,\mathsf{S}(N))\in\{0,1\}^{m+1}$. 
Assume $x\in\llbracket0, N-m-1 \rrbracket$. Consider the Markov chain $\{X^{\boldsymbol{\mathsf{S}},x}_{N,m}(i)\}_{i=N-m}^{N+1}$ defined by 
$X^{\boldsymbol{\mathsf{S}},x}_{N,m}(N-m)=x$ and for $i\in\llbracket N-m,N\rrbracket$ and $j\in\llbracket0,i-1\rrbracket$,
\begin{enumerate}
    \item [(\romannum{1})] if $\mathsf{S}(i)=1$, we let $X^{\boldsymbol{\mathsf{S}},x}_{N,m}(i+1)-X^{\boldsymbol{\mathsf{S}},x}_{N,m}(i)\in\{-1,0\}$ with 
    \begin{equation}\label{eq:etget11}
    \mathbb{P}\left(X^{\boldsymbol{\mathsf{S}},x}_{N,m}(i+1)=j-1 \Big{|}X^{\boldsymbol{\mathsf{S}},x}_{N,m}(i)=j\right)=\frac{1-q^{j}}{1-q^{i+1}},
    \end{equation}
    \item [(\romannum{2})] if $\mathsf{S}(i)=0$, we let $X^{\boldsymbol{\mathsf{S}},x}_{N,m}(i+1)-X^{\boldsymbol{\mathsf{S}},x}_{N,m}(i)\in\{0,1\}$ with 
    \begin{equation}\label{eq:etget21}
    \mathbb{P}\left(X^{\boldsymbol{\mathsf{S}},x}_{N,m}(i+1)=j+1 \Big{|}X^{\boldsymbol{\mathsf{S}},x}_{N,m}(i)=j\right)=\frac{q^{j+1}-q^{i+1}}{1-q^{i+1}}.
    \end{equation}
\end{enumerate} 
The right sides of \eqref{eq:etget11} and \eqref{eq:etget21} lie in $[0,1]$ since $j\in\llbracket0,i-1\rrbracket$ and $q\in[0,1)$. 
By an induction argument, we have $X^{\boldsymbol{\mathsf{S}},x}_{N,m}(i)\in\llbracket0,i-1\rrbracket$ for all $i\in\llbracket N-m,N+1\rrbracket$, and hence this Markov chain is well-defined.  We denote by $K_{N,m}(\boldsymbol{\mathsf{S}};x,\cdot)$ the law of $X^{\boldsymbol{\mathsf{S}},x}_{N,m}(N+1)$ started at $X^{\boldsymbol{\mathsf{S}},x}_{N,m}(N-m)=x$.
\end{definition}
We have the following comparison between the law of $K_{N,m}(\boldsymbol{\mathsf{S}};x,\cdot)$ and $K_{N,m}(\boldsymbol{\mathsf{S}};0,\cdot)$.

\begin{proposition}\label{prop:expectation of TV distance coex}
Assume \(q\in(0,1)\) and \(\b>1\). For \(N\in\mathbb N\) and \(\bg\in\lrw_N\), let \(\bs(\bg)\) be the composition associated to \(\bg\) as in \eqref{eq:Signature}, and let
\(\boldsymbol{\mathsf S}^{\bs(\bg)}\in\{0,1\}^{N+1}\) be the associated \(0\)-\(1\) sequence from Definition \ref{def:associated 01 seq}.  For \(m\in\llbracket0,N-1\rrbracket\), let
\(K_{N,m}(\boldsymbol{\mathsf S}^{\bs(\bg)};x,\cdot)\), \(x\in\llbracket0,N-m-1\rrbracket\), be the transition law from Definition \ref{def:chain tail}.  For \(k\in\llbracket0,N\rrbracket\), write $\mathbb P_{N,k}^0(\,\cdot\,)=\mathbb P^{(0,\b,\b)}(\,\cdot\,\mid T_N=k)$, and  $\mathbb E_{N,k}^0$ for the corresponding expectation, where we recall that $\mathbb{P}^{(0,\b,\b)}$ and $T_N$ are as in \eqref{eq:PN00 and TN}.  Then there exist constants \(\varepsilon,C,c>0\), depending only on \(q\) and \(\b\), such that, for every \(N\ge1\), \(k\in\llbracket0,N\rrbracket\), and \(m\in\llbracket1,N-1\rrbracket\), denote $h_{N,m}=\min\{\lfloor\varepsilon m\rfloor,N-m-1\}$, then
\[
\mathbb E_{N,k}^0\!\left[
\sup_{x\in\llbracket0,h_{N,m}\rrbracket}
\TV{
K_{N,m}\left(\boldsymbol{\mathsf{S}}^{\bs(\bg)};x,\cdot\right)
-
K_{N,m}\left(\boldsymbol{\mathsf{S}}^{\bs(\bg)};0,\cdot\right)
}
\right]
\leq Ce^{-cm}.
\] 
\end{proposition}

 In order to show Proposition~\ref{prop:expectation of TV distance coex}, we will import the following two results from \cite{SchmidYangShock} on comparing $K_{N,m}(\boldsymbol{\mathsf{S}};x,\cdot)$ to the chain $\mathscr{X}$. 

\begin{lemma}[c.f.~Lemma 5.9 in \cite{SchmidYangShock}]\label{lem:bound the TV distance by hitting}
Let $q\in(0,1)$, $N\in\mathbb{N}$, and $m\in\llbracket0,N-1\rrbracket$. Let $\boldsymbol{\mathsf{S}}\in\{0,1\}^{N+1}$ be a $0$-$1$ sequence. Let  $\boldsymbol{\mathsf{S}}_{N,m}\in\{0,1\}^{m+1}$ and $K_{N,m}(\boldsymbol{\mathsf{S}};x,\cdot)$ for $x\in\llbracket0, N-m-1 \rrbracket$ be as in Definition \ref{def:chain tail}. Let the Markov chain $\mathscr{X}^{\boldsymbol{\mathsf{S}}_{N,m},x}=\{\mathscr{X}^{\boldsymbol{\mathsf{S}}_{N,m},x}(i)\}_{i=0}^{m+1}$ be as in Definition \ref{def:simpler chain}. Then for each $h\in\llbracket0,N-m-1\rrbracket$,
\begin{equation*}
    \sup_{x\in\llbracket0, h\rrbracket}
\TV{
K_{N,m}\left(\boldsymbol{\mathsf{S}} ;x,\cdot\right)
-
K_{N,m}\left(\boldsymbol{\mathsf{S}} ;0,\cdot\right)
} 
\leq
\mathbb{P}\left( \mathscr{X}^{\boldsymbol{\mathsf{S}}_{N,m},h}(i)>0\mbox{ for all }i\in\llbracket0,m+1\rrbracket \right).
\end{equation*} 
\end{lemma}

\begin{lemma}[c.f.~Lemma 5.10 in \cite{SchmidYangShock}]\label{lem:simpler chain hitting 0}
 Let \(q\in(0,1)\), $C_0,c_0>0$ and \(H\) be a \(\mathbb N\)-valued random
variable, satisfying $\mathbb{P}(H\geq y)\leq C_0e^{-c_0y}$ for all $y\in\mathbb{R}_{\geq0}$.
Let \(H_1,H_2\ldots\) be i.i.d. copies of \(H\). For $\ell\in\mathbb{N}$, consider the $0$-$1$ sequence $\boldsymbol{\mathsf S}^{H,\ell}\in\{0,1\}^{H_1+\dots+H_{\ell}}$ given by 
\[
{\mathsf S}^{H,\ell}(i)=1\quad\mbox{if and only if } i=H_1+\dots+H_j-1\mbox{ for some }j\in\llbracket1,\ell\rrbracket,
\]
Consider the Markov chain $\mathscr{X}^{\boldsymbol{\mathsf S}^{H,\ell},x}=\{\mathscr{X}^{\boldsymbol{\mathsf S}^{H,\ell},x}(i)\}_{i=0}^{H_1+\dots+H_{\ell}}$ started at $x\in\mathbb{Z}_{\geq0}$ as in Definition \ref{def:simpler chain}, sampled using transition randomness independent of  \(H_1,H_2\ldots\). Then there exist constants $C,c>0$ and $v>1$
depending only on \(q\) and on the law of \(H\), such that
\begin{equation}\label{eq:nryrbsbtn4y}
\mathbb{P}\left(\mathscr{X}^{\boldsymbol{\mathsf S}^{H,\ell},x}(i)>0\mbox{ for all }i\in\llbracket0,H_1+\dots+H_{\ell}\rrbracket\right)
\leq C e^{-c\ell}v^x\quad \mbox{for all }\ell\in\mathbb N,\ x\in\mathbb Z_{\ge0}.
\end{equation}
\end{lemma}

Apart from Lemma~\ref{lem:bound the TV distance by hitting} and Lemma~\ref{lem:simpler chain hitting 0}, we will also use the following comparison for the chain $\mathscr{X}$ for different $0$-$1$ sequences.

\begin{lemma}\label{lem:inserting zeros comparison}
Assume \(q\in(0,1)\). Let \(\ell\in\mathbb N\). For
\(a_1,\ldots,a_\ell\in\mathbb N\), define
\(\boldsymbol{\mathsf S}^{a_1,\ldots,a_\ell}
\in\{0,1\}^{a_1+\cdots+a_\ell}\) by
\[
\mathsf S^{a_1,\ldots,a_\ell}(i)=1
\quad\mbox{if and only if}\quad
i=a_1+\cdots+a_j-1
\mbox{ for some }j\in\llbracket1,\ell\rrbracket .
\] 
Let \(x,x'\in\mathbb Z_{\geq0}\) satisfy \(x\le x'\). Let
\(h_i,h_i'\in\mathbb N\), \(i\in\llbracket1,\ell\rrbracket\), satisfy
\(h_i\le h_i'\) for every \(i\in\llbracket1,\ell\rrbracket\). Then
\[
\mathbb P \left(
\mathscr X^{\boldsymbol{\mathsf S}^{h_1,\ldots,h_{\ell}},x}(i)>0
\mbox{ for all } i\in\llbracket0,h_1+\cdots+h_{\ell}\rrbracket
\right)
\le
\mathbb P \left(
\mathscr X^{\boldsymbol{\mathsf S}^{h_1',\ldots,h'_{\ell}},x'}(i)>0
\mbox{ for all } i\in\llbracket0,h'_1+\cdots+h'_{\ell}\rrbracket
\right).
\] 
\end{lemma}
\begin{proof}
    For $s\in\{0,1\}$, $u\in(0,1)$, and $j\in\mathbb{Z}_{\geq0}$, let 
    \[
    \Psi^s(j,u)=j-{\bf1}_{s=1}{\bf1}_{\{j\geq1,u\leq 1-q\}}+{\bf1}_{s=0}{\bf1}_{\{u\leq q^{j+1}\}}.
    \]
    Note that if $U\sim\mathrm{Unif}(0,1)$ and $s\in\{0,1\}$, then $\Psi^s(j,U)$ has the transition law specified in Definition \ref{def:simpler chain} with current symbol $s$. Denote $M=h_1+\cdots+h_{\ell}$ and $M'=h'_1+\cdots+h'_{\ell}$, and set $\boldsymbol{\mathsf S}=\boldsymbol{\mathsf S}^{h_1,\ldots,h_\ell}$,  $\boldsymbol{\mathsf S}'=\boldsymbol{\mathsf S}^{h'_1,\ldots,h'_\ell}$.
    Let $H_0=H'_0=0$, $H_r=h_1+\cdots+h_r$, and $H'_r=h'_1+\cdots+h'_r$ for $r\in\llbracket1,\ell\rrbracket$. Define $\iota:\llbracket0,M-1\rrbracket\to\llbracket0,M'-1\rrbracket$ as follows. If $i=H_{r-1}+t$ with $r\in\llbracket1,\ell\rrbracket$ and $0\le t<h_r-1$, set $\iota(i)=H'_{r-1}+t$; if $i=H_r-1$, set $\iota(i)=H'_r-1$. Since \(h_r\le h'_r\) for every \(r\), this is an increasing injection satisfying
    \[
    \mathsf S'(\iota(i))=\mathsf S(i)\quad\mbox{for all }i\in\llbracket0,M-1\rrbracket,
    \qquad
    \mathsf S'(j)=0\quad\mbox{for all }j\in\llbracket0,M'-1\rrbracket\setminus\iota(\llbracket0,M-1\rrbracket).
    \]
    Let $U_0,\dots,U_{M'-1}$ be i.i.d. $\mathrm{Unif}(0,1)$ random variables. We define $Y$ and $Y'$ by $Y(0)=x$, $Y'(0)=x'$,
\[
Y(i+1)=\Psi^{\mathsf S(i)}\left(Y(i),U_{\iota(i)}\right)\mbox{ for }i\in\llbracket0,M-1\rrbracket,\quad\mbox{ and }\quad 
Y'(j+1)=\Psi^{\mathsf S'(j)}(Y'(j),U_j)\mbox{ for }j\in\llbracket0,M'-1\rrbracket.
\]
Then $Y\overset{d}{=}\mathscr X^{\boldsymbol{\mathsf S}^{h_1,\ldots,h_{\ell}},x}$ and $Y'\overset{d}{=}\mathscr X^{\boldsymbol{\mathsf S}^{h_1',\ldots,h'_{\ell}},x'}$. 
For each fixed $s\in\{0,1\}$ and $u\in(0,1)$, the map $j\mapsto\Psi^s(j,u)$ is nondecreasing on $\mathbb Z_{\geq0}$.
Also $\Psi^0(j,u)\ge j$ for all $j\in\mathbb Z_{\ge0}$ and $u\in(0,1)$.
By induction, we can show that $Y(i+1)\le Y'(\iota(i)+1)$ for all $i\in\llbracket0,M-1\rrbracket$, and we conclude the proof. 
\end{proof}

We have now all tools in order to show Proposition \ref{prop:expectation of TV distance coex}.

\begin{proof}[Proof of Proposition \ref{prop:expectation of TV distance coex}]
Fix \(N\ge1\), \(k\in\llbracket0,N\rrbracket\), \(m\in\llbracket1,N-1\rrbracket\).  If \(N=1\), there is no such \(m\).  In the rest of the proof, we assume \(N\ge2\). Recall   $h_{N,m}=\min (\lfloor\varepsilon m\rfloor,N-m-1)$, where $\varepsilon>0$ will be chosen later. Let
\[
F_{N,m}(\bg)
=
\sup_{x\in\llbracket0,h_{N,m}\rrbracket}
\TV{
K_{N,m}\left(\boldsymbol{\mathsf{S}}^{\bs(\bg)};x,\cdot\right)
-
K_{N,m}\left(\boldsymbol{\mathsf{S}}^{\bs(\bg)};0,\cdot\right)
} 
\quad\mbox{ for } \bg\in\lrw_N .
\] 
This is a nonnegative function on $\lrw_N$.
Let \(\boldsymbol Y=\{Y_i\}_{i=-\infty}^{\infty}\) be a bi-infinite \(V\)-shaped \(\nu_\b\)-random walk started at \(Y_0=0\), and define $\boldsymbol{\lambda}\in\lrw_N$ by $\lambda_i=Y_{i-k}-Y_{-k}$ for $i\in\llbracket0,N\rrbracket$. 
By Lemma \ref{lem:comparision},  
\begin{equation}\label{eq:tail-TV-proof-comparison}
\mathbb E_{N,k}^0\!\left[F_{N,m}(\bg)\right]
\le
(1-\b^{-2})^{-2}\,
\mathbb E\!\left[F_{N,m}(\boldsymbol\lambda)\right].
\end{equation}
To analyze the path $\boldsymbol{\lambda}$, by Lemma \ref{lem:regeneration lambda} we choose \(\mathbb Z_{\ge0}\)-valued variables \(\upchi_j\) and \(\mathbb N\)-valued variables \(G_j\) for \(j\in\mathbb N\), and constants $C_1,c_1>0$ depending only on $\b$, satisfying the following conditions. The variables \(G_j\), \(j\in\mathbb N\), are independent, \(G_j\), \(j\ge2\), are i.i.d., and for \(j\geq2\), \(G_j\) is independent of \(\sigma(\upchi_1,\ldots,\upchi_{j-1})\). Let \(\upchi_0=N\); $K=\max\{j\in\mathbb Z_{\geq0}:\upchi_j>0\}<\infty$ almost surely; and $\mathrm R_j=\upchi_{j-1}-\upchi_j\in\mathbb Z_{\ge0}$ for \(j\in\mathbb N\). We have 
\begin{align}
    & \label{eq:monotone upchi indices 1 again}\mbox{On }\{K\ge1\}, \mbox{ }N\geq \upchi_1>\cdots>\upchi_K\geq1\mbox{ and }\mathsf S^{\bs(\boldsymbol{\lambda})}(\upchi_j)=1,\,j\in\llbracket1,K\rrbracket;\\
    & \label{eq:exponential tail of Gj again}\mathbb P(G_j\ge y)\le C_1e^{-c_1y}\mbox{ for all }j,y\in\mathbb N;\\
    & \label{eq:conditions for R again}\mathrm R_1+\cdots+\mathrm R_{K+1}=N;\,\mathrm R_j=0\mbox{ for }j\ge K+2;\,\mathrm R_j\le G_j\mbox{ for }j\in\mathbb N.
\end{align}
Let $A=\max\left\{a\in\llbracket0,K\rrbracket:\mathrm R_1+\cdots+\mathrm R_a\le m\right\}$, where the empty sum is \(0\).  Since $m\le N-1<N=\mathrm R_1+\cdots+\mathrm R_{K+1}$, we have \(A\in\llbracket0,K\rrbracket\). Note that, on the event $\{A\geq1\}$, the  $\upchi_j$s lying in the tail interval \(\llbracket N-m,N\rrbracket\) are exactly
$\upchi_A< \dots<\upchi_1$.
On  $\{A\geq1\}$, we define
\begin{equation}\label{eq:netuyu}
    D_1=m+1-(\mathrm R_1+\cdots+\mathrm R_A) 
    \quad\mbox{ and }\quad
    D_j=\mathrm R_{A-j+2}\quad\mbox{ for } j\in\llbracket2,A\rrbracket .
\end{equation}
Then each \(D_j\in\mathbb N\), 
\begin{align}
    &\label{eq:distance upchi}\upchi_A-(N-m)=D_1-1,\quad \upchi_{A-1}-\upchi_A=D_2,\quad\dots\quad  \upchi_1-\upchi_2=D_A,\\
    &\label{eq:tail-TV-proof-D-bound}D_j\le \mathrm R_{A-j+2}\le G_{A-j+2} 
    \quad\mbox{ for } j\in\llbracket1,A\rrbracket.
\end{align} 
Indeed, the second inequality follows from \eqref{eq:conditions for R again}, and the first inequality for $j\in\llbracket2,A\rrbracket$ follows from \eqref{eq:netuyu}. When $j=1$, the definition of $A$ gives $\mathrm R_1+\cdots+\mathrm R_{A+1}>m$, and hence $D_1\leq \mathrm R_{A+1}$.   
By \eqref{eq:distance upchi} and \eqref{eq:monotone upchi indices 1 again}, on the event  $\{A\geq1\}$, $\mathsf S_{N,m}^{\bs(\boldsymbol\lambda)}(D_1+\cdots+D_j-1)=1$
 for every $j\in\llbracket1,A\rrbracket$, and hence 
\begin{equation}\label{eq:netuyrthrtu}
    \mathsf S_{N,m}^{\bs(\boldsymbol\lambda)}(i)\ge \mathsf S^{D_1,\ldots,D_A}(i) 
    \quad\mbox{ for } i\in\llbracket0,D_1+\cdots+D_A-1\rrbracket .
\end{equation}
Therefore, on $\{A\geq1\}$,
\begin{equation}\label{eq:ntdynerynj}
    \begin{split}
 F_{N,m}(\boldsymbol\lambda)
&\le
\mathbb P \left(
\mathscr X^{\boldsymbol{\mathsf S}^{\bs(\boldsymbol\lambda)}_{N,m},h_{N,m}}(i)>0
\mbox{ for all }i\in\llbracket0,m+1\rrbracket
\,\middle|\,\boldsymbol Y
\right)\\
&\leq \mathbb P \left(
\mathscr X^{\boldsymbol{\mathsf S}^{D_1,\ldots,D_A},h_{N,m}}(i)>0
\mbox{ for all }i\in\llbracket0,D_1+\cdots+D_A\rrbracket
\,\middle|\,\boldsymbol Y
\right),
    \end{split}
\end{equation}
where the first step follows from Lemma \ref{lem:bound the TV distance by hitting}; for the second step, we first restrict the event on the first line to the time interval \(\llbracket0,D_1+\cdots+D_A\rrbracket\), and then apply Lemma \ref{lem:coupling birth death} \textup{(\romannum{3})}, using \eqref{eq:netuyrthrtu}.

By \eqref{eq:exponential tail of Gj again}, after increasing \(C_1\) if necessary, we may assume \(C_1e^{-c_1}\ge1\). Let \(H\) be the \(\mathbb N\)-valued random variable defined by
$\mathbb P(H\ge r)=\min (1,C_1e^{-c_1r} )$ for $r\in\mathbb N$.
Then \(H\) depends only on \(\b\) and stochastically dominates every \(G_j\), \(j\in\mathbb N\). Let $H_1,H_2,\dots$ be i.i.d. copies of $H$. For \(a\in\mathbb N\) and \(d_1,\ldots,d_a\in\mathbb N\), write
\[
    \Phi_a(d_1,\ldots,d_a)
    =
    \mathbb P \left(
    \mathscr X^{\boldsymbol{\mathsf S}^{d_1,\ldots,d_a},h_{N,m}}(i)>0
    \mbox{ for all }i\in\llbracket0,d_1+\cdots+d_a\rrbracket
    \right).
\]
By Lemma \ref{lem:inserting zeros comparison}, \(\Phi_a\) is coordinatewise nondecreasing. Since \(G_2,\ldots,G_{a+1}\) are independent and each is stochastically dominated by \(H\), the law of \((G_{a+1},G_a,\ldots,G_2)\) is stochastically dominated, for coordinatewise nondecreasing functions, by the law of \((H_1,\ldots,H_a)\). Hence, using \eqref{eq:tail-TV-proof-D-bound} on \(\{A=a\}\),
\[
\begin{split}
        &\mathbb E\!\left[
{\bf1}_{\{A=a\}}
\mathbb P \left(
\mathscr X^{\boldsymbol{\mathsf S}^{D_1,\ldots,D_a},h_{N,m}}(i)>0
\mbox{ for all }i\in\llbracket0,D_1+\cdots+D_a\rrbracket
\,\middle|\,\boldsymbol Y
\right)
\right]        \\
&=
\mathbb E\!\left[{\bf1}_{\{A=a\}}\Phi_a(D_1,\ldots,D_a)\right]
\le
\mathbb E\!\left[\Phi_a(G_{a+1},G_a,\ldots,G_2)\right]
\le
\mathbb E\!\left[\Phi_a(H_1,\ldots,H_a)\right].
\end{split}
\]
Lemma \ref{lem:simpler chain hitting 0} gives constants $C_2,c_2>0$ and $v>1$ depending only on $q$ and $\b$, such that for every $a\in\mathbb N$,
\begin{equation}\label{eq:hwtrhwyhjwn}
    \begin{split}
        &\mathbb E\!\left[
{\bf1}_{\{A=a\}}
\mathbb P \left(
\mathscr X^{\boldsymbol{\mathsf S}^{D_1,\ldots,D_a},h_{N,m}}(i)>0
\mbox{ for all }i\in\llbracket0,D_1+\cdots+D_a\rrbracket
\,\middle|\,\boldsymbol Y
\right)
\right]        \\
&\leq 
\mathbb P\left(
\mathscr X^{\boldsymbol{\mathsf S}^{H_1,\dots,H_a},h_{N,m}}(i)>0
\mbox{ for all }i\in\llbracket0,H_1+\cdots+H_a\rrbracket
\right)
\le
C_2e^{-c_2a}v^{h_{N,m}}.
    \end{split}
\end{equation}

We now choose the constant $\varepsilon>0$ in the statement of the lemma. First choose $\theta>0$ such that $\mathbb E[e^{\theta H}]<\infty$. 
Next choose \(\eta\in(0,1)\) and then \(\varepsilon>0\) such that 
\begin{equation}\label{eq:etavarepsilon}
    \eta\log \mathbb E\left[e^{\theta H}\right]\le \theta/2,\qquad \varepsilon\log v \le  \eta c_2/4. 
\end{equation} 
These choices depend only on \(q\) and \(\b\). Note that, when $\lfloor\eta m\rfloor\geq1$, 
\begin{equation}\label{eq:Alesthanetam}
    \mathbb P(A<\lfloor\eta m\rfloor)\leq \mathbb P\left(G_1+\cdots+G_{\lfloor\eta m\rfloor}>m\right)     \leq e^{-\theta m}\prod_{j=1}^{\lfloor\eta m\rfloor}\mathbb E\left[e^{\theta G_j}\right] \leq e^{-\theta m}\mathbb E\left[e^{\theta H}\right]^{\lfloor\eta m\rfloor}\leq e^{-\theta m/2},
\end{equation}
where the first step uses \(\{A<\lfloor\eta m\rfloor\}\subseteq \{\mathrm R_1+\cdots+\mathrm R_{\lfloor\eta m\rfloor}>m\}\) and \eqref{eq:conditions for R again}, the second step uses Markov inequality and the independence of $G_1, G_2,\dots$, the third step uses that $H$ stochastically dominates every $G_j$, and the last step uses \eqref{eq:etavarepsilon}.
By Lemma \ref{lem:bound the TV distance by hitting}, \(F_{N,m}\le1\). Hence, when \(\lfloor\eta m\rfloor\ge1\), \eqref{eq:ntdynerynj} gives
\[
\begin{split}
\mathbb E\!\left[F_{N,m}(\boldsymbol\lambda)\right]
&\le
\mathbb P(A<\lfloor\eta m\rfloor)
+
\sum_{a=\lfloor\eta m\rfloor}^{\infty}
\mathbb E\!\left[
{\bf1}_{\{A=a\}}
\mathbb P \left(
\mathscr X^{\boldsymbol{\mathsf S}^{D_1,\ldots,D_a},h_{N,m}}(i)>0
\mbox{ for }i\in\llbracket0,D_1+\cdots+D_a\rrbracket
\,\middle|\,\boldsymbol Y
\right)
\right].
\end{split}
\]
Combining this bound with \eqref{eq:tail-TV-proof-comparison}, \eqref{eq:hwtrhwyhjwn}, and \eqref{eq:Alesthanetam}, we get
\begin{equation}
    (1-\b^{-2})^{ 2}\cdot\mathbb E_{N,k}^0\!\left[F_{N,m}(\bg)\right]
\le
 e^{-\theta m/2}
+
\sum_{a=\lfloor\eta m\rfloor}^{\infty}C_2e^{-c_2a}v^{h_{N,m}} 
\leq 
 e^{-\theta m/2}
+
\frac{C_2e^{c_2}}{1-e^{-c_2}}
    e^{- 3\eta m c_2/4}  ,
\end{equation}
where the second step uses \(h_{N,m}\le  \varepsilon m\), \(\lfloor\eta m\rfloor\ge \eta m-1\), and \(\varepsilon\log v\le \eta c_2/4\).
Thus there exist constants \(C,c>0\), depending only on \(q\) and \(\b\), such that
 $\mathbb E_{N,k}^0\!\left[F_{N,m}(\bg)\right]\le Ce^{-cm}$
when \(\lfloor\eta m\rfloor\ge1\). For the remaining finitely many \(m\ge1\) with \(\lfloor\eta m\rfloor=0\), since \(F_{N,m}\le1\), increasing \(C\) if necessary gives the same bound.  
\end{proof}

\subsection{Stability of the endpoint of the birth-death chain}

Using the previous bounds on the measure $K_{N,\ell}\left(\boldsymbol{\mathsf S}^{\bs(\bg)};x,\cdot\right)$ from Proposition \ref{prop:expectation of TV distance coex}, we can now control the endpoint of the birth-chain $X$ when conditioning on non-negligible events $A$ that depend only on $\Delta_{N,k,m}(\bg)$, and  for $\bg\in\lrw_N$ given by \eqref{eq:Ianddelta}.

\begin{lemma}\label{lem:conditioning does not change much endpoint}
Assume \(q\in(0,1)\) and \(\b>1\). For $N\in\mathbb N$ and \(k\in\llbracket0,N\rrbracket\), let $\mathbb P_{N,k}^0(\,\cdot\,)=\mathbb P^{(0,\b,\b)}(\,\cdot\,\mid T_N=k)$, where we recall that $\mathbb{P}^{(0,\b,\b)}$ and $T_N$ are as in \eqref{eq:PN00 and TN}.  
For \(k,m\in\llbracket0,N\rrbracket\), let $\mathsf I_{N,k,m}$ and $\Delta_{N,k,m}(\bg)$ for $\bg\in\lrw_N$ be given by \eqref{eq:Ianddelta}. Then there exist constants \(\delta,C,c>0\), depending only on $q$ and $\b$, such that the following holds. For every \(N\in\mathbb N\), every \(k\in\llbracket0,N\rrbracket\), every \(m\in\llbracket1,N\rrbracket\), and every
\[
A\in\sigma(\Delta_{N,k,m})
\quad\mbox{ satisfying }\quad
\mathbb P_{N,k}^{0}(A)\ge e^{-\delta m},
\] 
one has
\begin{equation*} 
\TV{
\mathbb P_{N,k}^{0}
\left(
    X^{\boldsymbol{\mathsf S}^{\bs(\bg)}}(N+1)\in\cdot
    \,\middle|\,A
\right)
-
\mathbb P_{N,k}^{0}
\left(
    X^{\boldsymbol{\mathsf S}^{\bs(\bg)}}(N+1)\in\cdot
\right)
}
\le
C e^{-c m}.
\end{equation*}
Here, in both probabilities, one first samples \(\bg\in\lrw_N\) under the indicated law, forms the composition \(\bs(\bg)\) as in \eqref{eq:Signature}, forms the associated \(0\)-\(1\) sequence \(\boldsymbol{\mathsf S}^{\bs(\bg)}\) as in Definition \ref{def:associated 01 seq}, and then samples the Markov chain \(X^{\boldsymbol{\mathsf S}^{\bs(\bg)}}\) from Definition \ref{def:characterization measure} using transition randomness independent of \(\bg\).
\end{lemma} 
\begin{proof}
In this proof, we will write $\mathbb{P}^0=\mathbb P_{N,k}^{0}$ and $\mathbb{E}^0=\mathbb E_{N,k}^{0}$.
Let $C_1,c_1,\kappa_1,\delta_1\in(0,\infty)$ be the constants from Lemma \ref{lem:stability low compositions coex}, and let $\varepsilon_2,C_2,c_2\in(0,\infty)$ be the constants from Proposition \ref{prop:expectation of TV distance coex}. These constants depend only on $q$ and $\b$. Let
\[
    \eta=\min\left(\kappa_1/4,1/(4(1+\varepsilon_{2}))\right),\qquad      
     \delta =\min\left( \delta_{1}, c_{2}\eta/4, \varepsilon_{2}\eta\log 2/8 \right).
\]
We choose \(m_0\in\mathbb N\), depending only on $q$ and $\b$, such that, when \(m\ge m_0\) and \(N\ge m\), \begin{equation}\label{eq:suff-coex-ell admissible}
    1\le \ell=\lfloor\eta m\rfloor\le N-1,\qquad
    \ell+1\le \lfloor\kappa_1 m\rfloor,\qquad
    \ell\ge \eta m/2,\qquad
    \lfloor\varepsilon_{2}\ell\rfloor\ge  \varepsilon_{2}\ell/2,\qquad
    \lfloor\varepsilon_{2}\ell\rfloor\le N-\ell-1 .
\end{equation}
We first prove the lemma for \(m\ge m_0\). Since \(m\le N\), \eqref{eq:suff-coex-ell admissible} holds for this part of the proof. Let $M\in(0,\infty)$ be the constant from Proposition \ref{prop:moment bound coex} with \(z=2\), depending only on $q$ and $\b$, such that
\begin{equation*} 
    \sup_{N\ge1}\sup_{k\in\llbracket0,N\rrbracket}
    \sup_{i\in\llbracket0,N+1\rrbracket}
    \mathbb E^{0}\!\left[
        \mathbb E\!\left[
            2^{X^{\boldsymbol{\mathsf S}^{\bs(\bg)}}(i)}
            \,\middle|\,\bg
        \right]
    \right]
    \le M .
\end{equation*}
By Markov's inequality,
\begin{equation}\label{eq:suff-coex-intermediate tail} 
 \mathbb E^{0}\!\left[
        \mathbb P\left(
            X^{\boldsymbol{\mathsf S}^{\bs(\bg)}}(N-\ell)>\lfloor\varepsilon_{2}\ell\rfloor
            \,\middle|\,\bg
        \right)
    \right]      \le
    2^{-\lfloor\varepsilon_{2}\ell\rfloor}
    \mathbb E^{0}\!\left[
        \mathbb E\left[
            2^{X^{\boldsymbol{\mathsf S}^{\bs(\bg)}}(N-\ell)}
            \,\middle|\,\bg
        \right]
    \right]
    \le
    M \,2^{-\lfloor\varepsilon_{2}\ell\rfloor}. 
\end{equation}
Denote
\[
    F_{\ell}(\bg)
     =
    \sup_{x\in\llbracket0,\lfloor\varepsilon_{2}\ell\rfloor\rrbracket}
    \TV{
        K_{N,\ell}\left(\boldsymbol{\mathsf S}^{\bs(\bg)};x,\cdot\right)
        -
        K_{N,\ell}\left(\boldsymbol{\mathsf S}^{\bs(\bg)};0,\cdot\right)
    }\quad\mbox{ for }\bg\in\lrw_N.
\]
Since \(\lfloor\varepsilon_{2}\ell\rfloor\le N-\ell-1\), Proposition \ref{prop:expectation of TV distance coex} with \(\ell\) in place of \(m\) gives
\begin{equation}\label{eq:suff-coex-F average}
    \mathbb E^{0}[F_{\ell}(\bg)]
    \le
    C_{2}e^{-c_{2}\ell}.
\end{equation}
Define two probability distributions on $\mathbb{Z}_{\geq0}$
\[
\mu^{\bg}=\mathbb{P}\left(X^{\boldsymbol{\mathsf S}^{\bs(\bg)}}(N+1)\in\,\cdot\,\right)\quad\mbox{and}\quad
\nu^{\bg} =K_{N,\ell}\left(\boldsymbol{\mathsf S}^{\bs(\bg)};0,\cdot\right),\quad\mbox{ for fixed }\bg\in\lrw_N.
\]
By the Markov property of \(X^{\boldsymbol{\mathsf S}^{\bs(\bg)}}\) at time \(N-\ell\),
\begin{equation}\label{eq:suff-coex-mu-nu-fixed path}
    \TV{\mu^{\bg}-\nu^{\bg}}
    \le
    F_{\ell}(\bg)
    +
    \mathbb P\left(
        X^{\boldsymbol{\mathsf S}^{\bs(\bg)}}(N-\ell)>\lfloor\varepsilon_{2}\ell\rfloor
        \,\middle|\,\bg
    \right)\quad
\mbox{ for each }\bg\in\lrw_N.
\end{equation}
When \(\mathbb P^{0}(A)\ge e^{-\delta m}\), averaging \eqref{eq:suff-coex-mu-nu-fixed path} under \(\mathbb P_{N,k}^{0}(\cdot\,|A)\) and \(\mathbb P_{N,k}^{0}\), respectively, combining with \eqref{eq:suff-coex-intermediate tail} and \eqref{eq:suff-coex-F average},  
\begin{align}
\TV{
    \mathbb E^{0}[\mu^{\bg}\mid A]
    -
    \mathbb E^{0}[\nu^{\bg}\mid A]
}
&\le
e^{\delta m}
\left(
    C_{2}e^{-c_{2}\ell}
    +
    M \,2^{-\lfloor\varepsilon_{2}\ell\rfloor}
\right),
\label{eq:suff-coex-conditioned-kernel error}
\\
\TV{
    \mathbb E^{0}[\mu^{\bg}]
    -
    \mathbb E^{0}[\nu^{\bg}]
}
&\le
C_{2}e^{-c_{2}\ell}
+
M \,2^{-\lfloor\varepsilon_{2}\ell\rfloor}.
\label{eq:suff-coex-unconditioned-kernel error}
\end{align}
Extend $\bs(\bg)=(\sigma_0,\ldots,\sigma_r)$ by setting \(\sigma_j=0\) for \(j>r\). We claim that the tail sequence
\[
\boldsymbol{\mathsf S}_{N,\ell}=\left(
\mathsf S^{\bs(\bg)}(N-\ell),\dots,\mathsf S^{\bs(\bg)}(N)
\right)\in\{0,1\}^{\ell+1}
\]
is a deterministic function of \((\sigma_0,\ldots,\sigma_{\lfloor\kappa_1 m\rfloor})\).  Indeed, 
\[
\mathsf S^{\bs(\bg)}(N-h)=1\quad\mbox{ if and only if }\sigma_0+\cdots+\sigma_{j-1}=h+1\mbox{ for some }j\in\llbracket1,r\rrbracket,
\quad\mbox{ for }h\in\llbracket0,\ell\rrbracket.
\]
If $\sigma_0+\cdots+\sigma_{j-1}=h+1$, then \(j\le h+1\le \ell+1\le\lfloor\kappa_1 m\rfloor \), since the nonzero composition parts are in $\mathbb{N}$. This proves the claim.
Consequently, for every \(B\subseteq\mathbb Z_{\ge0}\), there is a function \(f_B:\mathbb Z_{\ge0}^{\lfloor\kappa_1 m\rfloor+1}\to[0,1]\) such that
\[
    \nu^{\bg}(B)
    =
    f_B\left(\sigma_0,\ldots,\sigma_{\lfloor\kappa_1 m\rfloor}\right).
\] 
For $A\in\sigma(\Delta_{N,k,m})$ satisfying  $\mathbb P^{0}(A)\ge e^{-\delta m}$, since \(\delta \le\delta_{1}\), Lemma \ref{lem:stability low compositions coex} and the dual characterization of total variation distance applied to the \([0,1]\)-valued function \(f_B\) give
\[
\left|
    \mathbb E^{0}[\nu^{\bg}(B)\mid A]
    -
    \mathbb E^{0}[\nu^{\bg}(B)]
\right|
\le
C_{1}e^{-c_{1}m}\quad\mbox{ for each }B\subseteq\mathbb Z_{\ge0}.
\]
Taking the supremum over \(B\) gives
\begin{equation}\label{eq:suff-coex-nu-decorrelation}
\TV{
    \mathbb E^{0}[\nu^{\bg}\mid A]
    -
    \mathbb E^{0}[\nu^{\bg}]
}
\le
C_{1}e^{-c_{1}m}.
\end{equation}
Combining \eqref{eq:suff-coex-conditioned-kernel error}, \eqref{eq:suff-coex-unconditioned-kernel error}, and \eqref{eq:suff-coex-nu-decorrelation}, and then using \eqref{eq:suff-coex-ell admissible} and the definition of \(\delta\), yields
\[
\TV{
\mathbb P^{0}
\left(
    X^{\boldsymbol{\mathsf S}^{\bs(\bg)}}(N+1)\in\cdot
    \,\middle|\,A
\right)
-
\mathbb P^{0}
\left(
    X^{\boldsymbol{\mathsf S}^{\bs(\bg)}}(N+1)\in\cdot
\right)
}
\le
C e^{-cm}
\]
for all \(m\ge m_0\), with \(C,c\) depending only on \(q,\b\). Increasing \(C\) covers all \(m\in\llbracket1,m_0-1\rrbracket\), because total variation distance is at most \(1\). 
\end{proof}

\subsection{Proof of Proposition \ref{prop:comparison path q earlier}}\label{subsec:proof of prop comparison of path}
We have now all tools in order to prove our main approximation result Proposition \ref{prop:comparison path q earlier} on the path measure. For clarity, we restate it below. We recall that 
    \begin{equation}\label{eq:PNbb again again} 
    \mathbb{P}^{(q,\b,\b)}(\bg)=\frac{1}{Z_N^{(q,\b,\b)}} 2^{H(\bg)}\b^{\gamma_N-2\min(\bg)}  W_{\bs(\bg)}\left(\b^2\right)\quad\mbox{ for }\bg\in\lrw_N.
    \end{equation}
\begin{proposition}\label{prop:comparison path q}
    Assume \(q\in(0,1)\) and \(\b>1\). For $N\in\mathbb N$, \(k\in\llbracket0,N\rrbracket\), let $\mathbb P_{N,k}^q(\,\cdot\,)=\mathbb P^{(q,\b,\b)}(\,\cdot\,\mid T_N=k)$
    and
    $\mathbb P_{N,k}^0(\,\cdot\,)=\mathbb P^{(0,\b,\b)}(\,\cdot\,\mid T_N=k)$.
    Then both conditional laws are well-defined.
Moreover, for \(k,m\in\llbracket0,N\rrbracket\), let $\mathsf I_{N,k,m}$ and $\Delta_{N,k,m}(\bg)$ for $\bg\in\lrw_N$ be given by \eqref{eq:Ianddelta}, then there exist constants \(C,c\in(0,\infty)\), depending only on \(q\) and \(\b\), such that, for every \(N\ge1\) and every \(k,m\in\llbracket0,N\rrbracket\),
\[ 
    \TV{
    \mathbb P_{N,k}^q
    \left(
        \Delta_{N,k,m}(\bg)\in\cdot\, \right)
    -
    \mathbb P_{N,k}^0
    \left(
        \Delta_{N,k,m}(\bg)\in\cdot\, \right)
    }
    \le
    Ce^{-cm}.
\]
\end{proposition}
\begin{proof}
    For each $k\in\llbracket0,N\rrbracket$, by Lemma \ref{lem:decomposition TASEP} applied with \(\a=\b\), we have $\mathbb P^{(0,\b,\b)}(T_N=k)>0$. For each \(\bg\in\lrw_N\), by Theorem \ref{thm:prob interpretation of Bryc polynomials}, we know that
    $W_{\bs(\bg)}(\b^2)=\mathbb E\, [\b^{2X^{\boldsymbol{\mathsf S}^{\bs(\bg)}}(N+1)}  ]\ge 1$ since \(X^{\boldsymbol{\mathsf S}^{\bs(\bg)}}(N+1)\ge0\). Hence \(\mathbb P^{(q,\b,\b)}(T_N=k)>0\) in view of \eqref{eq:PNbb again again}, and the conditional law \(\mathbb P_{N,k}^{q}\) is well-defined.
    Note that  
\begin{equation}\label{eq:suff-coex-conditional-density}
    \mathbb P_{N,k}^{q}(B)
    = 
        \frac{\mathbb E_{N,k}^{0}\!\left[{\bf1}_{B}W_{\bs(\bg)}\left(\b^2\right)\right] }{
        \mathbb E_{N,k}^{0}\!\left[W_{\bs(\bg)}\left(\b^2\right)\right]}
     \quad\mbox{ for every event }B\subseteq\lrw_N.
\end{equation}

    In the rest of the proof, we will write $\mathbb{P}^q=\mathbb P_{N,k}^{q}$, $\mathbb{E}^q=\mathbb E_{N,k}^{q}$, $\mathbb{P}^0=\mathbb P_{N,k}^{0}$, and $\mathbb{E}^0=\mathbb E_{N,k}^{0}$. 
    By the corollary in Proposition \ref{prop:moment bound coex} with $m=2$, \(z=\b^2\) and $m=1$, \(z=\b^4\), there exists $M\in(0,\infty)$ depending only on $q,\b$, 
    \begin{equation}\label{eq:suff-coex-W L2}
    \mathbb E^{0}\left[W_{\bs(\bg)}\left(\b^2\right)^2\right]\le M \quad\mbox{ and }\quad \mathbb E^{0}\left[W_{\bs(\bg)}\left(\b^4\right) \right]\le M .
\end{equation}
Let the constants \(\delta_1,C_1,c_1\in(0,\infty)\) be from Lemma \ref{lem:conditioning does not change much endpoint}, which depend only on $q$ and $\b$.
We now prove the claim that there exist constants $C_2,c_2\in(0,\infty)$ depending only on $q$ and $\b$, such that for every $N\in\mathbb N$, \(k\in\llbracket0,N\rrbracket\), \(m\in\llbracket1,N-1\rrbracket\), 
and $A\in\sigma(\Delta_{N,k,m})$ with $\mathbb P^{0}(A)>0$,
\begin{equation}\label{eq:suff-coex-covariance bound}
    \mathbb P^{0}(A)
    \left|
        \mathbb E^{0}\left[W_{\bs(\bg)}\left(\b^2\right)\middle| A\right]
        -
        \mathbb E^{0}\left[W_{\bs(\bg)}\left(\b^2\right)\right]
    \right|
    \le
    C_2e^{-c_2 m}.
\end{equation} 
We split the proof of \eqref{eq:suff-coex-covariance bound} into two cases. When $\mathbb{P}^0(A)< e^{-\delta_1 m}$, by \eqref{eq:suff-coex-W L2} and the Cauchy--Schwarz inequality,
\begin{equation}\label{eq:suff-coex-covariance bound proof1}
    \begin{split}
        &\mathbb P^{0}(A)
    \left|
        \mathbb E^{0}\left[W_{\bs(\bg)}\left(\b^2\right)\middle| A\right]
        -
        \mathbb E^{0}\left[W_{\bs(\bg)}\left(\b^2\right)\right]
    \right|\leq  
        \mathbb E^{0}\left[W_{\bs(\bg)}\left(\b^2\right) {\bf1}_A\right]
        +
        \mathbb P^{0}(A) \cdot \mathbb E^{0}\left[W_{\bs(\bg)}\left(\b^2\right)\right] \\
        &\leq  \mathbb E^{0}\left[W_{\bs(\bg)}\left(\b^2\right)^2\right]^{1/2}\mathbb P^{0}(A)^{1/2}
    +
    \mathbb P^{0}(A) \cdot \mathbb E^{0}\left[W_{\bs(\bg)}\left(\b^2\right)\right]\le
    2M^{1/2}e^{-\delta_1 m/2}.
    \end{split}
\end{equation} 
When $\mathbb{P}^0(A)\geq e^{-\delta_1 m}$, Lemma \ref{lem:conditioning does not change much endpoint} and the coupling characterization of total variation distance give a coupling of endpoint random variables \(X_A\) and \(X\), whose respective marginal laws are
\begin{equation*} 
\mathbb P^{0}
\left(
    X^{\boldsymbol{\mathsf S}^{\bs(\bg)}}(N+1)\in\cdot
    \,\middle|\,A
\right)
\quad \mbox{ and }\quad
\mathbb P^{0}
\left(
    X^{\boldsymbol{\mathsf S}^{\bs(\bg)}}(N+1)\in\cdot
\right),
\end{equation*}
satisfying 
\begin{equation}\label{eq:rthrth}
    \mathbb{P}(X_A\neq X)\leq C_1e^{-c_1m}.
\end{equation}
In view of \eqref{eq:suff-coex-W L2},
\begin{equation}\label{eq:trhtyjnwrymn}
     \mathbb E\left[\b^{4X}\right]
    =
    \mathbb E^{0}\!\left[W_{\bs(\bg)}\left(\b^4\right)\right]
    \le M\quad\mbox{ and }\quad
     \mathbb E\left[\b^{4X_A}\right]
    =
    \mathbb E^{0}\!\left[W_{\bs(\bg)}\left(\b^4\right)\middle| A\right] 
    \le M/\mathbb{P}^0(A).
\end{equation}
Therefore, combining \eqref{eq:suff-coex-W L2}, \eqref{eq:rthrth}, \eqref{eq:trhtyjnwrymn}, by Cauchy's inequality,
\begin{equation}\label{eq:suff-coex-covariance bound proof2}
    \begin{split}
        &\mathbb P^{0}(A)
    \left|
        \mathbb E^{0}\left[W_{\bs(\bg)}\left(\b^2\right)\middle| A\right]
        -
        \mathbb E^{0}\left[W_{\bs(\bg)}\left(\b^2\right)\right]
    \right|=\mathbb P^{0}(A)\left|
        \mathbb E\left[\b^{2X_A}\right]-\mathbb E\left[\b^{2X}\right]
    \right|  \\
    &\le
    \mathbb P^{0}(A)\,
    \mathbb E\!\left[
        \left|\b^{2X_A}-\b^{2X}\right|{\bf1}_{\{X_A\ne X\}}
    \right]   \le
    \mathbb P^{0}(A)\,
    \mathbb E\!\left[
        \left|\b^{2X_A}-\b^{2X}\right|^2
    \right]^{1/2}
    \mathbb P(X_A\ne X)^{1/2}   \\
       &\le
    \mathbb P^{0}(A)\,
    \left(
        2\mathbb E\left[\b^{4X_A}\right]
        +
        2\mathbb E\left[\b^{4X}\right]
    \right)^{1/2}
    C_1^{1/2}e^{-c_1m/2}\le
    2M^{1/2}C_1^{1/2}e^{-c_1m/2} .
    \end{split}
\end{equation}
Combining the bounds \eqref{eq:suff-coex-covariance bound proof1} and \eqref{eq:suff-coex-covariance bound proof2} in the two cases, we conclude the proof of the claim \eqref{eq:suff-coex-covariance bound}. 

We return to the proof of the proposition. When $m\in\llbracket1,N-1\rrbracket$, using \eqref{eq:suff-coex-conditional-density}, \eqref{eq:suff-coex-covariance bound}, and \(\mathbb E^{0}[W_{\bs(\bg)} (\b^2 )]\ge1\), we obtain, for every
\(A\in\sigma(\Delta_{N,k,m})\) with \(\mathbb P^{0}(A)>0\),
\[ 
    \left|
        \mathbb P^{q}(A)-\mathbb P^{0}(A)
    \right|
     =
    \frac{\mathbb P^{0}(A)}{\mathbb E^{0}\left[W_{\bs(\bg)}\left(\b^2\right)\right]}
    \left|
        \mathbb E^{0}\left[W_{\bs(\bg)}\left(\b^2\right)  \big{|} A\right]
        -
        \mathbb E^{0}\left[W_{\bs(\bg)}\left(\b^2\right)\right]
    \right|                                                
     \le
    C_2e^{-c_2m}. 
\]
If \(\mathbb P^{0}(A)=0\), then \eqref{eq:suff-coex-conditional-density} gives \(\mathbb P^{q}(A)=0\).  Taking the supremum over \(A\in\sigma(\Delta_{N,k,m})\) concludes the proof for \(m\in\llbracket1,N-1\rrbracket\). If \(m=0\), or if \(m=N=1\), since the total variation distance is at most \(1\), the desired bound holds after increasing \(C\).   If \(m=N\) and \(N\ge2\), in view of the fact that $\mathsf I_{N,k,N}\subseteq \mathsf I_{N,k,N-1}$, using the already proved bound with $m=N-1$ and increasing \(C\) further proves the bound for \(m=N\).
\end{proof}
  
\section{Macroscopic uniformity of the location $T_N$}\label{sec:macroscopicTN}

For the approximation of the path measure in Section~\ref{sec:PathMeasure}, we conditioned on the path location $T_N$, where we recall that $T_N  =\min\{i\in\llbracket0,N\rrbracket:\gamma_i=\min(\bg)\} $. In the following, our goal is to approximate the location $T_N$ under the path measure
    \begin{equation}\label{eq:PNbb again again again} 
    \mathbb{P}^{(q,\b,\b)}(\bg)=\frac{1}{Z_N^{(q,\b,\b)}} 2^{H(\bg)}\b^{\gamma_N-2\min(\bg)}  W_{\bs(\bg)}\left(\b^2\right)\quad\mbox{ for }\bg\in\lrw_N.
    \end{equation}
More precisely, we show the following two propositions, which provide upper and lower bounds of order $\frac{1}{N}$ on the total variation distance of $T_N$ from a uniform location on $\llbracket 0,N \rrbracket$. 
\begin{proposition}\label{prop:macroscopic uniformity TN upper}
    Assume \(q\in[0,1)\) and \(\b>1\).  Then there exists a constant $C\in(0,\infty)$ depending only on $q$ and $\b$, such that for every $N\in\mathbb Z_{\geq0}$, we have 
    \[
    \TV{
        \mathbb P^{(q,\b,\b)}(T_N\in\cdot)
        -
        \mathrm{Unif}_{\llbracket0,N\rrbracket}
    }
    \leq
    \frac{C}{N+1}.
\]
\end{proposition}

\begin{proposition}\label{prop:macroscopic uniformity TN lower}
    Assume \(q\in[0,1)\) and \(\b>1\).  Then there exists a constant $c\in(0,\infty)$ depending only on $q$ and $\b$, such that for every $N\in\mathbb N$, we have 
    \[
    \TV{
        \mathbb P^{(q,\b,\b)}(T_N\in\cdot)
        -
        \mathrm{Unif}_{\llbracket0,N\rrbracket}
    }
    \geq
    \frac{c}{N+1}.
\]
\end{proposition}

\subsection{Estimating the Radon–Nikodym derivatives}

Our key input to showing Propositions~\ref{prop:macroscopic uniformity TN upper} and~\ref{prop:macroscopic uniformity TN lower} is a sharp approximation of the expectation of the Bryc polynomial when the  location $T_N$ is bounded away from the endpoints of the segment $\llbracket 1,N \rrbracket$.

\begin{proposition}
    \label{prop:constant Lambda q b}
    Assume \(q\in(0,1)\) and \(\b>1\). For \(N\in\mathbb N\) and \(k\in\llbracket0,N\rrbracket\), we denote by $\mathbb{P}_{N,k}^0$ the measure \(\mathbb P^{(0,\b,\b)}\) conditioned on $T_N=k$, and write \(\mathbb E_{N,k}^0\) for the corresponding expectation. 
Then there exist constants $\Lambda_{q,\b}\in[1,\infty)$ and $C,c>0$ depending only on $q$ and $\b$ such that, for every $N\in\mathbb N$ and $k\in\llbracket1,N\rrbracket$,
\begin{equation*} 
 \left|\mathbb E^0_{N,k}\!\left[
        W_{\bs(\bg)}(\b^2)
    \right]-\Lambda_{q,\b}\right|
    \le
    Ce^{-c\min(k-1,N-k)}  .
\end{equation*} 
\end{proposition} 
 
In order to show Proposition~\ref{prop:constant Lambda q b}, we require some setup. The following lemma provides a simple comparison between the endpoint measures $K_{N,m}(\boldsymbol{\mathsf S};x,\cdot)$ from Definition \ref{def:chain tail} for different choices of $\mathsf S$.

\begin{lemma}\label{lem:coupling of chain tails}
    Let $q\in(0,1)$, $N\in\mathbb{N}$, and $\boldsymbol{\mathsf{S}} \in\{0,1\}^{N+1}$. 
For every $m\in\llbracket0,N-1\rrbracket$ and every $x,x'\in\llbracket0, N-m-1 \rrbracket$ satisfying $x\geq x'$, the probability measure $K_{N,m}(\boldsymbol{\mathsf S};x,\cdot)$ on $\mathbb Z_{\geq0}$ stochastically dominates the probability measure $K_{N,m}(\boldsymbol{\mathsf S};x',\cdot)$ on $\mathbb Z_{\geq0}$, where these kernels are from Definition \ref{def:chain tail}.
\end{lemma}
\begin{proof}
     For \(i\in\llbracket N-m,N\rrbracket\), \(s\in\{0,1\}\), and \(u\in(0,1)\), define the update map
\[
\Phi_i^s(j,u)
=
j-{\bf1}_{\{s=1\}}
  {\bf1}_{\left\{u\le \frac{1-q^j}{1-q^{i+1}}\right\}}
+{\bf1}_{\{s=0\}}
  {\bf1}_{\left\{u\le \frac{q^{j+1}-q^{i+1}}{1-q^{i+1}}\right\}},
\]
for \(j\in\llbracket0,i-1\rrbracket\).  One can observe that \(\Phi_i^s(j,u)\le \Phi_i^s(j',u)\) if \(0\le j\le j'\le i-1\).
Coupling the two tail birth-death chains with the same uniform random variables gives the stochastic domination.
\end{proof}

The next lemma provides a sharp approximation of the expectation of the Bryc polynomial in terms of the kernel $K_{N,m}(\boldsymbol{\mathsf S};x,\cdot)$, which will serve as a starting point in order to show Proposition~\ref{prop:constant Lambda q b}.

\begin{lemma}\label{lem:sum of K}
        Assume \(q\in(0,1)\) and \(\b>1\). For \(N\in\mathbb N\) and each $k\in\llbracket0,N\rrbracket$, we denote by $\mathbb{P}_{N,k}^0$ the measure \(\mathbb P^{(0,\b,\b)}\) conditioned on $T_N=k$, and write \(\mathbb E_{N,k}^0\) for the corresponding expectation. 
    Let \(K_{N,\ell}(\boldsymbol{\mathsf S};x,\cdot)\) be the tail kernel from Definition \ref{def:chain tail}. For $\ell\in\mathbb N$ and \(N\in\mathbb N\) with \(N\ge2\ell+2\), let  
    \[
    \mathsf{K}_{N,\ell}(\bg)=\sum_{y\ge0}
        \b^{2y}
        K_{N,\ell}
        \left(
            \boldsymbol{\mathsf S}^{\bs(\bg)};0,\{y\}
        \right)\quad\mbox{ for }\bg\in\lrw_N.
\]
Then there exist constants \(C,c\in(0,\infty)\), depending only on \(q,\b\), such that for $\ell\in\mathbb N$, \(N\in\mathbb N\) with \(N\ge2\ell+2\),
\begin{equation*} 
    \sup_{k\in\llbracket0,N\rrbracket}
    \mathbb E^0_{N,k}
    \left[
        \left|W_{\bs(\bg)}\left(\b^2\right)-\mathsf{K}_{N,\ell}(\bg) \right|\right]
    \le
    C e^{-c\ell}.
\end{equation*}
\end{lemma}
\begin{proof}
By Proposition \ref{prop:moment bound coex}, there exists a constant $M\in(0,\infty)$ depending only on $q$ and $\b$, such that 
\begin{equation} \label{eq:proof sum of K moment bound}
\mathbb E^0_{N,k}
    \left[
        \mathbb E\!\left[
            \b^{4X^{\boldsymbol{\mathsf S}^{\bs(\bg)}}(i)}
            \Big{|}\bg
        \right]  \right]
    \le M
    \quad\mbox{and } 
      \mathbb E_{N,k}^0\left[W_{\bs(\bg)}\left(\b^4\right) \right]\le M, 
    \quad
    \text{for }N\in\mathbb N,\ k\in\llbracket0,N\rrbracket,\,i\in\llbracket0,N+1\rrbracket .
\end{equation}
Let \(\varepsilon_1,C_1,c_1\in(0,\infty)\) be the constants from Proposition \ref{prop:expectation of TV distance coex}.  Decrease \(\varepsilon_1\), if necessary, so that \(\varepsilon_1\in(0,1]\).   Since \(N\ge2\ell+2\), we have
 $\lfloor\varepsilon_1\ell\rfloor\le \ell \le N-\ell-2$.
 As in the proof of Lemma \ref{lem:conditioning does not change much endpoint}, define two measures
\[
\mu^{\bg}=\mathbb{P}\left(X^{\boldsymbol{\mathsf S}^{\bs(\bg)}}(N+1)\in\,\cdot\,\right)\quad\mbox{and}\quad
\nu^{\bg} =K_{N,\ell}\left(\boldsymbol{\mathsf S}^{\bs(\bg)};0,\cdot\right),\quad\mbox{ for fixed }\bg\in\lrw_N.
\]

We  claim that there exist constants \(C_2,c_2\in(0,\infty)\), depending only on $q$ and $\b$, such that
\begin{equation}\label{eq:TN-unif-TV-tail-average fixed}
    \mathbb E^0_{N,k}
    \left[
        \TV{\mu^{\bg}-\nu^{\bg}} \right]
    \le
    C_2 e^{-c_2\ell}\quad\mbox{ for all }N\geq2\ell+2,\,k\in\llbracket0,N\rrbracket.
\end{equation}
Indeed, by the Markov property of \(X^{\boldsymbol{\mathsf S}^{\bs(\bg)}}\) at time \(N-\ell\), for each $\bg\in\lrw_N$,
\begin{equation*} 
    \TV{\mu^{\bg}-\nu^{\bg}}
    \le
    \sup_{x\in\llbracket0,\lfloor\varepsilon_1\ell\rfloor\rrbracket}
    \TV{
        K_{N,\ell}
        \left(
            \boldsymbol{\mathsf S}^{\bs(\bg)};x,\cdot
        \right)
        -
        K_{N,\ell}
        \left(
            \boldsymbol{\mathsf S}^{\bs(\bg)};0,\cdot
        \right)
    }   
    +
    \mathbb P\left(
        X^{\boldsymbol{\mathsf S}^{\bs(\bg)}}(N-\ell)>\lfloor\varepsilon_1 \ell\rfloor
        \,\middle|\,\bg
    \right) .
\end{equation*}
We take the expectation $\mathbb E^0_{N,k}$ on both sides. By Proposition \ref{prop:expectation of TV distance coex}, the expectation of the first summand on the right side is bounded by $C_1e^{-c_1\ell}$. For the second summand, Markov's inequality and \eqref{eq:proof sum of K moment bound} give
\[
\begin{aligned}
    &\mathbb E^0_{N,k}
    \left[
        \mathbb P\!\left(
            X^{\boldsymbol{\mathsf S}^{\bs(\bg)}}(N-\ell)>\lfloor\varepsilon_1 \ell\rfloor
            \,\middle|\,\bg
        \right) \right]                                     
     \le
    \b^{-4\lfloor\varepsilon_1 \ell\rfloor}\cdot
    \mathbb E^0_{N,k}
    \left[
        \mathbb E\!\left[
            \b^{4X^{\boldsymbol{\mathsf S}^{\bs(\bg)}}(N-\ell)}
            \,\middle|\,\bg
        \right]  \right]
    \le
    M \cdot\b^{-4\lfloor\varepsilon_1 \ell\rfloor} .
\end{aligned}
\]
Hence
\[
 \mathbb E^0_{N,k}
    \left[
        \TV{\mu^{\bg}-\nu^{\bg}} \right]\leq C_1e^{-c_1\ell}+ M \cdot\b^{-4\lfloor\varepsilon_1 \ell\rfloor} 
    \le
    C_2 e^{-c_2\ell}
\]
for some \(C_2,c_2\in(0,\infty)\), depending only on $q$ and $\b$, and we conclude the claim \eqref{eq:TN-unif-TV-tail-average fixed}.

We next claim that, for any increasing function $g(y)$ of $y\in\mathbb{Z}_{\geq0}$,
\begin{equation}\label{eq:nu bounded by mu}
    \nu^{\bg}(g)\leq \mu^{\bg}(g)\quad\mbox{ for fixed }\bg\in\lrw_N.
\end{equation}
Indeed, by Lemma \ref{lem:coupling of chain tails}, since $g$ is increasing, $\nu^{\bg}(g)\leq K_{N,\ell}  (\boldsymbol{\mathsf S}^{\bs(\bg)};x,g)$ for $x\in\llbracket0,N-\ell-1\rrbracket$. Averaging this inequality with respect to the conditional law of
\(X^{\boldsymbol{\mathsf S}^{\bs(\bg)}}(N-\ell)\) gives \eqref{eq:nu bounded by mu}. Taking $g(y)=\b^{4y}$, by \eqref{eq:proof sum of K moment bound},
\begin{equation}\label{eq:grtwhw}
        \mathbb E^0_{N,k}
    \left[
        \mu^{\bg}\left(g(\cdot)\right)+\nu^{\bg}\left(g(\cdot)\right) \right]
    \le 2M \quad\mbox{ for } k\in\llbracket0,N\rrbracket.
\end{equation}
For probability measures \(\rho,\rho'\) on \(\mathbb Z_{\ge0}\) and \(f:\mathbb{Z}_{\geq0}\to\mathbb R_{\geq0}\) satisfying
\(\sum_{y\geq0}f(y)^2(\rho(y)+\rho'(y))<\infty\), by Cauchy--Schwarz inequality,
\begin{equation}\label{eq:jntyujgdvbsdtb}
    \left|
        \sum_{y\geq0} f(y)\rho(y)-\sum_y f(y)\rho'(y)
    \right|
    \le
    \left(
       \sum_{y\geq0}  f(y)^2\left(\rho(y)+\rho'(y)\right)
    \right)^{1/2}
    \left(2\TV{\rho-\rho'}\right)^{1/2}.
\end{equation}
By Theorem \ref{thm:prob interpretation of Bryc polynomials} and the definitions of \(\mu^{\bg}\), \(\nu^{\bg}\), and \(\mathsf{K}_{N,\ell}\),
we have $W_{\bs(\bg)}(\b^2)=\mu^{\bg}(\b^{2\cdot})$ and  $\mathsf{K}_{N,\ell}(\bg)=\nu^{\bg}(\b^{2\cdot})$. Using \eqref{eq:jntyujgdvbsdtb} with \(\rho=\mu^{\bg}\), \(\rho'=\nu^{\bg}\), $f(y)=\b^{2y}$,     then using Cauchy--Schwarz again for the average over \(\bg\),
\[
\begin{aligned}
     \mathbb E^0_{N,k}
    \left[
        \left|W_{\bs(\bg)} \left(\b^2 \right)-\mathsf{K}_{N,\ell}(\bg)\right|  \right]                                            
     &\le
    \left(
        2\,
        \mathbb E^0_{N,k}
        \left[
            \TV{\mu^{\bg}-\nu^{\bg}} \right]
    \right)^{1/2} 
          \left(\mathbb E^0_{N,k}
       \left[
            \mu^{\bg}(\b^{4\cdot})+\nu^{\bg}(\b^{4\cdot})  \right]\right)^{1/2} \\
    &\leq  \left(2C_2 e^{-c_2\ell}\right)^{1/2} (2M)^{1/2}\leq C e^{-c\ell},
\end{aligned}
\]
for some $C,c\in(0,\infty)$ depending only on $q$ and $\b$, where we used \eqref{eq:TN-unif-TV-tail-average fixed} and \eqref{eq:grtwhw}. 
\end{proof}

 Similar to Definition~\ref{def:simpler chain}, we require the following auxiliary birth-death chain for the proof of Proposition~\ref{prop:constant Lambda q b}.  

\begin{definition} \label{def:limit chain}
Assume $q\in(0,1)$, $n\in\mathbb N$, and $\boldsymbol{\mathsf S}=(\mathsf S(0),\dots,\mathsf S(n))\in\{0,1\}^{n+1}$. We define a time-inhomogeneous Markov chain $\mathfrak{X}^{\boldsymbol{\mathsf{S}}}=\{\mathfrak{X}^{\boldsymbol{\mathsf{S}}}(i)\}_{i=0}^{n+1}$ indexed by this $0$-$1$ sequence $\boldsymbol{\mathsf{S}}$ through $\mathfrak{X}^{\boldsymbol{\mathsf{S}}}(0)=0$ and that for each $i\in\llbracket0,n\rrbracket$ and $j\in\mathbb{Z}_{\geq0}$,
\begin{enumerate}
    \item [(\romannum{1})] if $\mathsf{S}(i)=1$, we let $\mathfrak{X}^{\boldsymbol{\mathsf{S}}}(i+1)-\mathfrak{X}^{\boldsymbol{\mathsf{S}}}(i)\in\{-1,0\}$ with 
    \begin{equation}\label{eq:etget13}
    \mathbb{P}\left(\mathfrak{X}^{\boldsymbol{\mathsf{S}}}(i+1)=j-1 \Big{|}\mathfrak{X}^{\boldsymbol{\mathsf{S}}}(i)=j\right)= 1-q^{j} ,
    \end{equation}
    \item [(\romannum{2})] if $\mathsf{S}(i)=0$, we let $\mathfrak{X}^{\boldsymbol{\mathsf{S}}}(i+1)-\mathfrak{X}^{\boldsymbol{\mathsf{S}}}(i)\in\{0,1\}$ with 
    \begin{equation}\label{eq:etget23}
    \mathbb{P}\left(\mathfrak{X}^{\boldsymbol{\mathsf{S}}}(i+1)=j+1 \Big{|}\mathfrak{X}^{\boldsymbol{\mathsf{S}}}(i)=j\right)= q^{j+1} .
    \end{equation}
\end{enumerate} 
We mention that the right sides of \eqref{eq:etget13} and \eqref{eq:etget23} lie in $[0,1]$ since $j\in\mathbb{Z}_{\geq0}$ and $q\in(0,1)$. 
By an induction argument, $\mathfrak{X}^{\boldsymbol{\mathsf{S}}}(i)\in\llbracket0,i \rrbracket$ for all $i\in\llbracket0,n+1\rrbracket$, and hence this Markov chain is well-defined.
\end{definition}
\begin{definition}\label{def:01 seq u}
Assume $n\in\mathbb N$ and  \(\mathbf u=(u_0,\ldots,u_{n+1})\in\mathbb Z_{\ge0}^{n+2}\). We define a $0$-$1$ sequence \(\boldsymbol{\mathsf S}^{\mathbf u}=(\mathsf S^{\mathbf u}(0),\ldots,\mathsf S^{\mathbf u}(n))\) by, for each $i\in\llbracket0,n\rrbracket$,
\[
    \mathsf S^{\mathbf u}(i)=1
    \quad\mbox{ if and only if }\quad 
    u_j>0
    \mbox{ and }
    u_0+\cdots+u_{j-1}=n+1 -i \mbox{ for some }j\in\llbracket1,n+1\rrbracket .
\]
\end{definition}

The following lemma relates the kernel $\mathsf{K}_{N,\ell}(\bg)$ to the moment generating function of $\mathfrak{X}^{\boldsymbol{\mathsf S}^{\mathbf u}}(\ell+1)$.

\begin{lemma}\label{lem:convergence to limiting chain}
    Assume \(q\in(0,1)\), \(\b>1\).    For $\ell\in\mathbb N$ and \(N\in\mathbb N\) satisfying \(N\ge2\ell+2\), recall 
    \[
    \mathsf{K}_{N,\ell}(\bg)=\sum_{y\ge0}
        \b^{2y}
        K_{N,\ell}
        \left(
            \boldsymbol{\mathsf S}^{\bs(\bg)};0,\{y\}
        \right)\quad\mbox{ for }\bg\in\lrw_N.
\] 
Here, \(K_{N,\ell}(\boldsymbol{\mathsf S};x,\cdot)\) is the tail kernel from Definition \ref{def:chain tail}.
We define
\[
G_\ell(\mathbf u)=\mathbb E\left[\b^{2\mathfrak{X}^{\boldsymbol{\mathsf S}^{\mathbf u}}(\ell+1)}\right]\quad
\mbox{ for each }\mathbf u=(u_0,\ldots,u_{\ell+1})\in\mathbb{Z}_{\geq0}^{\ell+2},
\]
where the $0$-$1$ sequence $\boldsymbol{\mathsf S}^{\mathbf u}$ and the Markov chain $\mathfrak{X}^{\boldsymbol{\mathsf S}^{\mathbf u}}$ are from Definitions \ref{def:01 seq u} and \ref{def:limit chain}. 
For each \(\bg\in\lrw_N\), let \(\bs(\bg)=(\sigma_0,\dots,\sigma_r)\) be the composition associated to \(\bg\) as in \eqref{eq:Signature}, and extend it by  $\sigma_h=0$ for $h>r$.
Then 
\[ 
    \sup_{\bg\in\lrw_N}
    \left|
        \mathsf{K}_{N,\ell}(\bg)
        -
        G_\ell(\sigma_0,\ldots,\sigma_{\ell+1} )
    \right|
    \leq      \b^{2(\ell+1)}(\ell+1)\cdot \frac{q^{N-\ell+1}}{1-q^{N-\ell+1}} .  
\]
\end{lemma}
\begin{proof} Since $\mathfrak{X}^{\boldsymbol{\mathsf S}^{\mathbf u}}(\ell+1)\in\llbracket0,\ell+1\rrbracket$, 
$G_\ell(\mathbf u)\in[1,\b^{2(\ell+1)}]$.
For fixed \(\ell \), set \(\mathbf u_{\bg}=(\sigma_0,\ldots,\sigma_{\ell+1})\). We claim that 
\begin{equation}\label{eq:ytuhrjnilo.,}
    \mathsf S^{\bs(\bg)}(N-\ell+a)
    =
    \mathsf S^{\mathbf u_{\bg}}(a)
    \quad \mbox{ for }a\in\llbracket0,\ell\rrbracket .
\end{equation}
Indeed, by Definition \ref{def:01 seq u}, it suffices to show that, for \(a\in\llbracket0,\ell\rrbracket\),
\[
\mathsf S^{\bs(\bg)}(N-\ell+a)=1\quad\mbox{ if and only if }\quad 
\sigma_h>0  \mbox{ and } \sigma_0+\cdots+\sigma_{h-1}=\ell-a+1    \mbox{ for some }h\in\llbracket1,\ell+1\rrbracket.
\]
By  Definition \ref{def:associated 01 seq}, the one-symbols of \(\boldsymbol{\mathsf S}^{\bs(\bg)}\) occur exactly at 
$i=\sum_{m=h}^{r}\sigma_m
    =
    N+1-\sum_{m=0}^{h-1}\sigma_m$ for some  $h\in\llbracket1,r\rrbracket$.
Thus \(i=N-\ell+a\) exactly when \(\sum_{m=0}^{h-1}\sigma_m=\ell-a+1\), and such an \(h\) must lie in \(\llbracket1,\ell+1\rrbracket\).  We conclude the claim \eqref{eq:ytuhrjnilo.,}. 
Let \(X^{\boldsymbol{\mathsf S}^{\bs(\bg)},0}_{N,\ell}\) be the tail chain from Definition \ref{def:chain tail}, and let 
\[
\Theta_{N,\ell}^{\bg}(a)
    =
    X^{\boldsymbol{\mathsf S}^{\bs(\bg)},0}_{N,\ell}(N-\ell+a)\quad\mbox{ and }\quad
    \Xi^{\bg}(a)=\mathfrak X^{\boldsymbol{\mathsf S}^{\mathbf u_{\bg}}}(a)\quad\mbox{ for }a\in\llbracket0,\ell+1\rrbracket.
\] 
Then, by the definitions of \(\mathsf{K}_{N,\ell}\), \(K_{N,\ell}\), and \(G_\ell\),
\[
    \mathsf{K}_{N,\ell}(\bg)=\mathbb E\!\left[\b^{2\Theta_{N,\ell}^{\bg}(\ell+1)}\right],
    \qquad
    G_\ell(\sigma_0,\ldots,\sigma_{\ell+1})
    =
    \mathbb E\!\left[\b^{2\Xi^{\bg}(\ell+1)}\right].
\]
All states reachable from \(0\) during \(\ell+1\) updates belong to \(\llbracket0,\ell+1\rrbracket\).  For
\(a\in\llbracket0,\ell\rrbracket\), \(j\in\llbracket0,\ell+1\rrbracket\), and \(i=N-\ell+a\), the assumption \(N\ge2\ell+2\) implies \(j\le i-1\), and
\[
    \left|
    \frac{1-q^j}{1-q^{i+1}}-(1-q^j)
    \right|
    \le
    \frac{q^{i+1}}{1-q^{i+1}},
    \qquad
    \left|
    \frac{q^{j+1}-q^{i+1}}{1-q^{i+1}}-q^{j+1}
    \right|
    \le
    \frac{q^{i+1}}{1-q^{i+1}}.
\]
Hence both one-step discrepancies are at most
 $\delta_N = q^{N-\ell+1}/(1-q^{N-\ell+1})$.
Couple \(\Theta_{N,\ell}^{\bg}\) and \(\Xi^{\bg}\) by using the same independent uniform variable at each of the \(\ell+1\) updates.  On the event that the chains agree up to time \(a\), the conditional probability of a first disagreement at the next update is at most \(\delta_N\). Therefore
\[
    \mathbb P\!\left(
        \Theta_{N,\ell}^{\bg}(\ell+1)\ne \Xi^{\bg}(\ell+1)
    \right)
    \le
    (\ell+1)\delta_N .
\]
Since both endpoints lie in \(\llbracket0,\ell+1\rrbracket\),
\[
\begin{aligned}
    \left|
        \mathsf{K}_{N,\ell}(\bg)
        -
        G_\ell(\sigma_0,\ldots,\sigma_{\ell+1})
    \right|
    &\le
    \b^{2(\ell+1)}
    \mathbb P\!\left(
        \Theta_{N,\ell}^{\bg}(\ell+1)\ne \Xi^{\bg}(\ell+1)
    \right)   
     \le
    \b^{2(\ell+1)}(\ell+1)\delta_N .
\end{aligned}
\]
The last bound is independent of \(\bg\).  Taking the supremum over \(\bg\in\lrw_N\) concludes the proof.
\end{proof}

\begin{lemma}\label{lem: first part proof limit of low level composition}
    Assume $\b>1$. For each $N\in\mathbb N$ and each $k\in\llbracket0,N\rrbracket$, we denote by $\mathbb{P}_{N,k}^0$ the measure \(\mathbb P^{(0,\b,\b)}\) conditioned on $T_N=k$. 
    For \(\bg\in\lrw_N\), let \(\bs(\bg)=(\sigma_0,\dots,\sigma_r)\) be the composition associated to \(\bg\) as in \eqref{eq:Signature}, and extend it by $\sigma_h=0$ for $h>r$. Then 
    \begin{equation}\label{eq:limit of low level composition0}
    \sigma_h =
    \#\{j\in\llbracket0,N\rrbracket:
        \gamma_j=\min(\bg)+h\}
    \quad \mbox{ for } h\in\mathbb{Z}_{\geq0}.
    \end{equation}
    Let \(\boldsymbol{Y}=(Y_i)_{i\in\mathbb Z}\) be the bi-infinite \(V\)-shaped \(\nu_\b\)-random walk from Definition \ref{def:V shaped RW}.  Define
\[
    \vartheta_h
    =
    \#\{j\in\mathbb Z:Y_j=h\}
    \quad \mbox{ for } h\in\mathbb Z_{\ge0}.
\]
These variables are finite almost surely.  Moreover, for $h\in\mathbb Z_{\geq0}$, $t\in\mathbb N$, we define the truncated compositions
\begin{equation*} 
\sigma_h^{[t]}(\bg) =    \#\{i\in\llbracket0,N\rrbracket:
        |i-k|\le t,\ \gamma_i=\gamma_k+h\} \quad\mbox{ and }\quad \vartheta_h^{[t]}=\#\{i\in\mathbb Z: |i|\le t,\ Y_i=h\} .
\end{equation*}
Then there exist $C,c>0$, depending only on $\b$, such that, for
$N\in\mathbb N$, $k\in\llbracket0,N\rrbracket$, \(d\in\mathbb Z_{\ge0}\), and $t\in\mathbb N$, 
\begin{equation}  \label{eq:comparing compositions with truncated new}
         \mathbb P^{0}_{N,k}
    \left(
        (\sigma_0,\ldots,\sigma_d)
        \ne
        \left(\sigma_0^{[t]},\ldots,\sigma_d^{[t]}\right) \right)
     \le C e^{Cd} e^{-c t},\qquad  
         \mathbb P 
    \left(
        (\vartheta_0,\ldots,\vartheta_d)
        \ne
        \left(\vartheta_0^{[t]},\ldots,\vartheta_d^{[t]}\right) \right)
     \le C e^{Cd} e^{-c t}. 
\end{equation}
\end{lemma}
\begin{proof}
        The first assertion \eqref{eq:limit of low level composition0} follows from \eqref{eq:Signature} and the extension $\sigma_h=0$ for $h>r$. For the second assertion, both one-sided conditioned paths $R_i=Y_i$ and $L_i=Y_{-i}$ for $i\in\mathbb Z_{\geq0}$ escape linearly to $\infty$ almost surely, hence each fixed level is visited finitely many times. Therefore all \(\vartheta_h\), \(h\in\mathbb Z_{\ge0}\) are finite almost surely. 

By Lemma \ref{lem:decomposition TASEP}, under the measure $\mathbb P^{0}_{N,k}$, the right path $R_i^{N,k}=\gamma_{k+i}-\gamma_k$, $i\in\llbracket0,N-k\rrbracket$ is a \(\nu_\b\)-walk started from \(0\), conditioned to stay nonnegative up to time \(N-k\). If \(k\ge1\), then \(\gamma_{k-1}-\gamma_k=1\),
and the left path $Z_i^{N,k}=\gamma_{k-i-1}-\gamma_k-1$, $i\in\llbracket0,k-1\rrbracket$ is a \(\nu_\b\)-walk started from \(0\), conditioned to stay nonnegative up to time \(k-1\). The left and right paths are independent.

	We start to prove the first bound in \eqref{eq:comparing compositions with truncated new}.
	Under \(\mathbb P^0_{N,k}\), one has \(T_N=k\), and hence \(\gamma_k=\min(\bg)\). Thus \(\sigma_h^{[t]}(\bg)\le \sigma_h\) for every \(h\in\llbracket0,d\rrbracket\). We have
\begin{equation*}
    \begin{split}
        &\left\{\sigma_h\neq \sigma_h^{[t]}\mbox{ for some }h\in\llbracket0,d\rrbracket\right\}
 \subseteq\left\{\gamma_i-\gamma_k\in\llbracket0,d\rrbracket\mbox{ for some }i\in\llbracket0,N\rrbracket,\,|i-k|>t\right\}\\
&=\left\{R_i^{N,k}\in\llbracket0,d\rrbracket \mbox{ for some }  i\in\llbracket t+1,N-k\rrbracket \right\} \cup
\left\{Z_i^{N,k}\in\llbracket-1,d-1\rrbracket \mbox{ for some }  i\in\llbracket t,k-1\rrbracket\right\},
    \end{split}
\end{equation*} 
where the second event is interpreted as empty when \(k=0\).
In the rest of the proof, fix $\kappa= (\b-1)/(2(\b+1))$.
We first bound the probability of the first event on the right side, which concerns the right path. Let \(S\) be an ordinary \(\nu_\b\)-random walk started from \(0\), and let
$A_m=\{S_i\ge0\text{ for }i\in\llbracket0,m\rrbracket\}$.  By Lemma \ref{lem:basic hitting monotone coupling}, there exist  
\(C_1,c_1>0\), depending only on \(\b\), such that
\[
    \mathbb P(S_t\le \kappa t\mid A_m)\le C_1e^{-c_1t}\quad \mbox{ for } m\ge t .
\]
For each $m\geq t$ and \(x>d\), the conditional future after time \(t\), given \(S_t=x\) and \(A_m\), is the \(\nu_\b\)-random walk from \(x\) conditioned to stay nonnegative up to time \(m-t\). Let $B$ denote the event that the unconditioned walk ever hits $d$. By Lemma \ref{lem:basic hitting monotone coupling},
\[ 
    \mathbb P_x\left( S_i\le d\mbox{ for some }i\in\llbracket0,m-t\rrbracket
        \,\middle|\,A_{m-t}
    \right)
     \le
    \frac{
        \mathbb P_x(B)
    }{
        \mathbb P_x(A_{m-t})
    }                                                   
     \le
    \frac{\b^{-2(x-d)}}{1-\b^{-2}} . 
\] 
When $\kappa t>d$ and $m\geq t$, combining the last two displays, and distinguishing whether $S_t\le \kappa$ or $S_t> \kappa$ for the second inequality,
\[
\begin{aligned}
    \mathbb P\left( S_i\le d\mbox{ for some }i\in\llbracket t,m\rrbracket
        \,\middle|\,A_m
    \right)
    &\le
    \mathbb P(S_t\le \kappa t\mid A_m)                                      
    +
     \b^{2d} \left(1-\b^{-2} \right)^{-1}
    \mathbb E\left[\b^{-2S_t}\,\middle|\,A_m\right]                    \\
    &\le
    C_1e^{-c_1t}
    +
    \b^{2d} \left(1-\b^{-2} \right)^{-1}\left(
        \mathbb P(S_t\le \kappa t\mid A_m)+\b^{-2\kappa t}
    \right)                                                            
     \le C_2e^{C_2d}e^{-c_2t},
\end{aligned}
\]
for some $C_2,c_2>0$ depending only on $\b$. If $\kappa t\le d$, the same bound holds after increasing \(C_2\), since the probability is at most \(1\) and \(1\le e^{(c_2/\kappa)d}e^{-c_2t}\). If \(N-k<t\), the event concerning the right path is empty; otherwise, taking \(m=N-k\) in the last display shows that its probability is bounded by \(C_2e^{C_2d}e^{-c_2t}\).
Regarding the event concerning the left path, if \(k-1<t\) or \(d=0\), the event is empty. If \(d\ge1\) and \(k-1\ge t\), using the right side estimate just proved with \(d-1\) in place of \(d\), and enlarging constants to cover \(t=1\), gives the same bound with \(C_2e^{C_2d}e^{-c_2t}\). Taking a union bound proves the first inequality in \eqref{eq:comparing compositions with truncated new}.  

Regarding the second inequality in \eqref{eq:comparing compositions with truncated new}, note that the right path is the law of a \(\nu_\b\)-walk started from \(0\) conditioned on $A_\infty=\{S_i\ge0\text{ for all }i\ge0\}$. Repeating the preceding right side argument with \(A_m\) replaced by \(A_\infty\), since we still have \(\mathbb P_x(A_\infty)\ge1-\b^{-2}\), gives the bound \(C_3 e^{C_3d}e^{-c_3t}\) for the right path of the \(V\)-walk, with \(C_3,c_3>0\) depending only on \(\b\).  For the left path, if \(d=0\) the relevant event is empty, because \(Y_{-1-i}-1\ge0\). If \(d\ge1\), the shifted process \(Y_{-1-i}-1\), \(i\ge0\), is a \(\nu_\b\)-walk started from \(0\) conditioned to stay nonnegative, so the same argument, with \(d-1\) in place of \(d\), gives the bound for the left path. A union bound gives the second inequality in \eqref{eq:comparing compositions with truncated new}.
\end{proof}

The following lemma is a standard result for biased random walks. 
\begin{lemma}\label{lem:technical lem RW nonnegative}
      Assume \(\b>1\).  Let \(S\) be a \(\nu_\b\)-random walk started from $S_0=0$.  Let $A_n=\{S_i\ge0\text{ for }i\in\llbracket0,n\rrbracket\}$ and $A_\infty=\{S_i\ge0\text{ for all }i\ge0\}$. Then there exist constants \(C,c>0\), depending only on \(\b\), such that for each $n\in\mathbb Z_{\geq0}$ and $m\in\llbracket0,n\rrbracket$,
    \[
\TV{
\mathbb P\left((S_0,\ldots,S_m)\in\,\cdot\,\middle|\,A_n\right)
-
\mathbb P\left((S_0,\ldots,S_m)\in\,\cdot\,\middle|\,A_\infty\right)
}
\le Ce^{-c(n-m)} .
\]
\end{lemma}
\begin{proof}
Let $T=\inf\{j\ge0:S_j<0\}$, 
so that $A_n=\{T>n\}$ and $A_\infty=\{T=\infty\}$. Using a biased Gambler's ruin estimate, we see that $P(A_\infty)\ge 1-\b^{-2}>0$. 
Now set
\[
\lambda=\mathbb E[\b^{-\xi}]
       =\frac{4}{\b+\b^{-1}+2}\in(0,1).
\]
By Markov's inequality and independence of the increments,
\[
\mathbb P(S_j<0)
\le \mathbb E[\b^{-S_j}]
=\lambda^j.
\]
Consequently,
\[
\mathbb P(A_n\setminus A_\infty)
=\mathbb P(n<T<\infty)
\le \sum_{j=n+1}^{\infty}\mathbb P(S_j<0)
\le \frac{\lambda^{n+1}}{1-\lambda}.
\]
Now because  $A_\infty\subseteq A_n$, we get that 
\begin{align*}
\TV{
\mathbb P\left((S_0,\ldots,S_m)\in\,\cdot\,\middle|\,A_n\right)
-
\mathbb P\left((S_0,\ldots,S_m)\in\,\cdot\,\middle|\,A_\infty\right)
}\le
\frac{\mathbb P(A_n\setminus A_\infty)}
     {\mathbb P(A_n)}
\le
\frac{\lambda^{n+1}}
     {(1-\lambda)(1-\b^{-2})}
\le Ce^{-c(n-m)},
\end{align*}
where $c=-\log\lambda>0$ and $C>0$ depends only on $\b$.
\end{proof}

Next, we compare the path measure conditioned on a fixed location $T_N=k$ to a bi-infinite $V$-shaped \(\nu_\b\)-random walk.

\begin{lemma}\label{lem: second part proof limit of low level composition}
     Assume $\b>1$. For $N\in\mathbb N$ and  $k\in\llbracket0,N\rrbracket$,   denote by $\mathbb{P}_{N,k}^0$ the measure \(\mathbb P^{(0,\b,\b)}\) conditioned on $T_N=k$. Let \(\boldsymbol{Y}=(Y_i)_{i\in\mathbb Z}\) be the bi-infinite \(V\)-shaped \(\nu_\b\)-random walk from Definition \ref{def:V shaped RW}.  
     Then there exist $C,c>0$, depending only on $\b$, such that, for
$N\in\mathbb N$, $k\in\llbracket0,N\rrbracket$,  $t\in\mathbb N$ with $\min(k-1,N-k)\geq2t+2$,
\[
\TV{
\mathbb P_{N,k}^0
\left((\gamma_{k+i}-\gamma_k)_{|i|\le t}\in\cdot\right)
-
\mathbb P
\left((Y_i)_{|i|\le t}\in\cdot\right)
}
\leq
Ce^{-c(\min(k-1,N-k)-t)}. 
\]
\end{lemma}
\begin{proof}
    By Lemma \ref{lem:decomposition TASEP}, under the measure $\mathbb P^{0}_{N,k}$, the right path $R_i^{N,k}=\gamma_{k+i}-\gamma_k$, $i\in\llbracket0,N-k\rrbracket$ is a \(\nu_\b\)-walk started from \(0\), conditioned to stay nonnegative up to time \(N-k\). If \(k\ge1\), then \(\gamma_{k-1}-\gamma_k=1\),
and the left path $Z_i^{N,k}=\gamma_{k-i-1}-\gamma_k-1$, $i\in\llbracket0,k-1\rrbracket$ is a \(\nu_\b\)-walk started from \(0\), conditioned to stay nonnegative up to time \(k-1\). The left and right paths are independent.
For the limiting paths, by Definition \ref{def:V shaped RW}, $R_i=Y_i$ and $Z_i=Y_{-1-i}-1$, $i\in\mathbb Z_{\geq0}$ are independent \(\nu_\b\)-walks started from \(0\) and conditioned to stay nonnegative.

Apply Lemma \ref{lem:technical lem RW nonnegative} with \(m=t\) and \(n=N-k\) to the right path \(R_i^{N,k}\), and with \(m=t-1\) and \(n=k-1\) to the left path \(Z_i^{N,k}\). Since \(\min\left( k-1 ,\,N-k\right)\ge2t+2\), both applications are valid. We have
\begin{equation}\label{eq:nbr8loylty}
    \begin{split}
        &\TV{
\mathbb P^0_{N,k}\left(
(R_0^{N,k},\ldots,R_t^{N,k},Z_0^{N,k},\ldots,Z_{t-1}^{N,k})\in\,\cdot\,
\right)
-
\mathbb P\left(
(R_0,\ldots,R_t,Z_0,\ldots,Z_{t-1})\in\,\cdot\,
\right)
}\\
&\le  Ce^{-c(N-k-t)}+Ce^{-c(k-t)} \le 2 Ce^{-c(\min\left( k-1 ,\,N-k\right)-t)} .
    \end{split}
\end{equation}
Let \(F\) be the deterministic map defined by
\[
F(r_0,\ldots,r_t,z_0,\ldots,z_{t-1})=(w_i)_{|i|\le t},
\qquad
w_i=r_i\ \text{ for } i\in\llbracket0,t\rrbracket,\quad
w_{-j}=z_{j-1}+1\ \text{ for }j\in\llbracket1,t\rrbracket .
\]
Then
\[
F(R_0^{N,k},\ldots,R_t^{N,k},Z_0^{N,k},\ldots,Z_{t-1}^{N,k})
=
(\gamma_{k+i}-\gamma_k)_{|i|\le t},\qquad
F(R_0,\ldots,R_t,Z_0,\ldots,Z_{t-1})=(Y_i)_{|i|\le t}.
\] 
Applying this deterministic map to the total variation bound \eqref{eq:nbr8loylty} concludes the proof.
\end{proof}

 For the next statement, we require in the following only the special case where $\Uppsi=1$. The case of general $\Uppsi$ is required in the upcoming section.

\begin{lemma}\label{lem:limit of low level composition}
    Assume $\b>1$. We use the same notation as in Lemma \ref{lem: first part proof limit of low level composition}. 
        Then there exist constants $C,c\in(0,\infty)$ depending only on $\b$, such that for every \(N\in\mathbb N\), \(k\in\llbracket0,N\rrbracket\), \(d\in\mathbb Z_{\ge0}\), and \(t\in\mathbb N\) satisfying $\min\left(k-1,\,N-k\right)\geq 2t+2$, and for every bounded function \(G:\mathbb Z_{\ge0}^{d+1}\to\mathbb R\)
    and every $\Uppsi:\mathbb{Z}^{2t+1}\to[0,1]$,
\begin{equation*}
    \begin{split}
        &\left|
        \mathbb E^0_{N,k}
        \left[
            \Uppsi\left((\gamma_{k+i}-\gamma_k)_{|i|\le t}\right) G(\sigma_0,\ldots,\sigma_d) \right]
        -
        \mathbb E
        \left[
            \Uppsi\left((Y_i)_{|i|\le t}\right) G(\vartheta_0,\ldots,\vartheta_d)
        \right]
    \right|\\
    &\leq C\|G\|_\infty
    \left(
        e^{Cd}e^{-ct}
        + e^{-c\left(\min\left( k-1 ,\,N-k\right)-t\right)}
    \right).
    \end{split}
\end{equation*} 
\end{lemma} 
\begin{proof}
    By Lemma \ref{lem: first part proof limit of low level composition}, there exist $C_1,c_1>0$ depending only on $\b$, such that for
$N\in\mathbb N$, $k\in\llbracket0,N\rrbracket$,   $t\in\mathbb N$, 
\begin{equation}\label{eq:bsrythnw0}
         \mathbb P^{0}_{N,k}
    \left(
        (\sigma_0,\ldots,\sigma_d)
        \ne
        \left(\sigma_0^{[t]},\ldots,\sigma_d^{[t]}\right) \right)
     \le C_1 e^{C_1d} e^{-c_1 t},\qquad  
         \mathbb P 
    \left(
        (\vartheta_0,\ldots,\vartheta_d)
        \ne
        \left(\vartheta_0^{[t]},\ldots,\vartheta_d^{[t]}\right) \right)
     \le C_1 e^{C_1d} e^{-c_1 t}. 
\end{equation}
By Lemma \ref{lem: second part proof limit of low level composition}, there exist $C_2,c_2>0$ depending only on $\b$, such that for $t\in\mathbb N$ with $\min(k-1,N-k)\geq2t+2$,
\begin{equation}\label{eq:bsrythnw}
\TV{
\mathbb P_{N,k}^0
\left((\gamma_{k+i}-\gamma_k)_{|i|\le t}\in\cdot\right)
-
\mathbb P
\left((Y_i)_{|i|\le t}\in\cdot\right)
}
\leq
C_2e^{-c_2(\min(k-1,N-k)-t)}. 
\end{equation}
We define, for $(x_{-t},\ldots,x_t)\in\mathbb Z^{2t+1}$,
\[
    H_{d,t}(x_{-t},\ldots,x_t)
    =
    \left(
        \#\{i\in\llbracket -t,t\rrbracket:x_i=h\}
    \right)_{h\in\llbracket0,d\rrbracket} .
\]
Since \(\min(k-1,N-k)\ge 2t+2\), all indices \(k+i\), \(|i|\le t\), lie in \(\llbracket0,N\rrbracket\), and hence
\[
    H_{d,t}\left((\gamma_{k+i}-\gamma_k)_{|i|\le t}\right)
    =
    \left(\sigma_0^{[t]},\ldots,\sigma_d^{[t]}\right) 
    \quad\mbox{ and }\quad
    H_{d,t}\left((Y_i)_{|i|\le t}\right)
    =
    \left(\vartheta_0^{[t]},\ldots,\vartheta_d^{[t]}\right).
\]
Therefore, using \( \Uppsi\in[0,1]\),
\begin{equation*}
    \begin{split}
        &\left|
        \mathbb E^0_{N,k}
        \left[
            \Uppsi\left((\gamma_{k+i}-\gamma_k)_{|i|\le t}\right) G(\sigma_0,\ldots,\sigma_d) \right]
        -
        \mathbb E
        \left[
            \Uppsi\left((Y_i)_{|i|\le t}\right) G(\vartheta_0,\ldots,\vartheta_d)
        \right]
    \right|                                                    \\
    &\le
    2\|G\|_\infty
    \mathbb P^0_{N,k}
    \left(
        (\sigma_0,\ldots,\sigma_d)
        \ne
        (\sigma_0^{[t]},\ldots,\sigma_d^{[t]})
    \right)                       +          2\|G\|_\infty
    \mathbb P
    \left(
        (\vartheta_0,\ldots,\vartheta_d)
        \ne
        (\vartheta_0^{[t]},\ldots,\vartheta_d^{[t]})
    \right)                       \\
    &\quad +\left|
        \mathbb E^0_{N,k}
        \left[
            \Uppsi\left((\gamma_{k+i}-\gamma_k)_{|i|\le t}\right)
            G\left(H_{d,t}\left((\gamma_{k+i}-\gamma_k)_{|i|\le t}\right)\right)
        \right]
        -
        \mathbb E
        \left[
            \Uppsi\left((Y_i)_{|i|\le t}\right)
            G\left(H_{d,t}\left((Y_i)_{|i|\le t}\right)\right)
        \right]
    \right|                           .
    \end{split}
\end{equation*}
Note that the function $x\mapsto \Uppsi(x)G(H_{d,t}(x))$ has sup norm at most \(\|G\|_\infty\).  Hence, by \eqref{eq:bsrythnw}, the last term on the right side above is at most
 $2C_2\|G\|_\infty
    e^{-c_2(\min(k-1,N-k)-t)}$.
Combining this with \eqref{eq:bsrythnw0},
\[
\begin{aligned}
&\left|
        \mathbb E^0_{N,k}
        \left[
            \Uppsi\left((\gamma_{k+i}-\gamma_k)_{|i|\le t}\right) G(\sigma_0,\ldots,\sigma_d) \right]
        -
        \mathbb E
        \left[
            \Uppsi\left((Y_i)_{|i|\le t}\right) G(\vartheta_0,\ldots,\vartheta_d)
        \right]
    \right|                                                    \\
&\le
    4C_1\|G\|_\infty e^{C_1d}e^{-c_1t}
    +
    2C_2\|G\|_\infty
    e^{-c_2(\min(k-1,N-k)-t)} .
\end{aligned}
\]
Choosing \(c\le \min(c_1,c_2)\) and then \(C\) large enough, depending only on \(\b\), we conclude the proof.
\end{proof}

We have now all tools in order to show the approximation result in Proposition~\ref{prop:constant Lambda q b}.

\begin{proof}[Proof of Proposition \ref{prop:constant Lambda q b}]
We want to show that there exist constants $\Lambda_{q,\b}\in[1,\infty)$ and $C,c>0$ depending only on $q$ and $\b$ such that, for every $N\in\mathbb N$ and $k\in\llbracket1,N\rrbracket$,
\begin{equation}\label{eq:constant Lambda q b again}
 \left|\mathbb E^0_{N,k}\!\left[
        W_{\bs(\bg)}(\b^2)
    \right]-\Lambda_{q,\b}\right|
    \le
    Ce^{-c\min(k-1,N-k)}  .
\end{equation}  
By Proposition \ref{prop:moment bound coex} applied with \(z=\b^2\), and by Theorem \ref{thm:prob interpretation of Bryc polynomials} for the lower bound, there exists a constant $M\in(0,\infty)$ depending only on $q$ and $\b$, such that 
\begin{equation} \label{eq:tyhjhneyjne}
    1\le \mathbb E_{N,k}^0
    \left[
        W_{\bs(\bg)}\left(\b^2\right) \right]\le M 
    \qquad
    \text{for all }N\in\mathbb N,\ k\in\llbracket0,N\rrbracket .
\end{equation}
As in Lemma \ref{lem:sum of K}, let \(K_{N,\ell}(\boldsymbol{\mathsf S};x,\cdot)\) be the tail kernel from Definition \ref{def:chain tail}. For $\ell\in\mathbb N$, \(N\ge2\ell+2\), let  
    \[
    \mathsf{K}_{N,\ell}(\bg)=\sum_{y\ge0}
        \b^{2y}
        K_{N,\ell}
        \left(
            \boldsymbol{\mathsf S}^{\bs(\bg)};0,\{y\}
        \right)\quad\mbox{ for }\bg\in\lrw_N.
\]
By Lemma \ref{lem:sum of K}, there exist \(C_1,c_1\in(0,\infty)\), depending only on \(q,\b\), such that for each $\ell\in\mathbb N$ and \(N\ge2\ell+2\),
\begin{equation} \label{eq:mryuik1}
    \sup_{k\in\llbracket0,N\rrbracket}
    \mathbb E^0_{N,k}
    \left[
        \left|W_{\bs(\bg)}\left(\b^2\right)-\mathsf{K}_{N,\ell}(\bg) \right|\right]
    \le
    C_1 e^{-c_1\ell}.
\end{equation}
As in Lemma \ref{lem:convergence to limiting chain}, there is a bounded function $G_\ell(\mathbf u)=\mathbb E [\b^{2\mathfrak{X}^{\boldsymbol{\mathsf S}^{\mathbf u}}(\ell+1)} ]\in[1, \b^{2(\ell+1)}]$ for $\mathbf u\in\mathbb{Z}_{\geq0}^{\ell+2}$,
where the $0$-$1$ sequence $\boldsymbol{\mathsf S}^{\mathbf u}$ and the chain $\mathfrak{X}^{\boldsymbol{\mathsf S}^{\mathbf u}}$ are from Definitions \ref{def:01 seq u} and \ref{def:limit chain}. 
For \(\bg\in\lrw_N\), let \(\bs(\bg)=(\sigma_0,\dots,\sigma_r)\) be the composition associated to \(\bg\) as in \eqref{eq:Signature}, and extend it by  $\sigma_h=0$ for $h>r$. By Lemma \ref{lem:convergence to limiting chain},  
\begin{equation} \label{eq:mryuik2} 
    \sup_{\bg\in\lrw_N}
    \left|
        \mathsf{K}_{N,\ell}(\bg)
        -
        G_\ell(\sigma_0,\ldots,\sigma_{\ell+1} )
    \right|
    \leq      \b^{2(\ell+1)}(\ell+1)\cdot \frac{q^{N-\ell+1}}{1-q^{N-\ell+1}} .   
\end{equation}
Let \(\boldsymbol{Y}=(Y_i)_{i\in\mathbb Z}\) be the bi-infinite \(V\)-shaped \(\nu_\b\)-random walk from Definition \ref{def:V shaped RW}, and let $  \vartheta_h
    =
    \#\{j\in\mathbb Z:Y_j=h\}$ for $h\in\mathbb Z_{\ge0}$.  By Lemma \ref{lem:limit of low level composition}  with  \(d=\ell+1\),   \(G=G_{\ell}\), and $\Uppsi\equiv1$, there exist constants $C_2,c_2\in(0,\infty)$ depending only on $\b$, such that for each $t\in\mathbb N$ satisfying $\min\left(k-1,\,N-k\right)\geq 2t+2$,  
\begin{equation} \label{eq:mryuik3} 
\left|
        \mathbb E^0_{N,k}
        \left[
            G_\ell(\sigma_0,\ldots,\sigma_{\ell+1}) \right]
        -
        \mathbb E
        \left[
            G_\ell(\vartheta_0,\ldots,\vartheta_{\ell+1})
        \right]
    \right|\leq C_2\|G_\ell\|_\infty
    \left(
        e^{C_2(\ell+1)}e^{-c_2t}
        +
        e^{C_2t}e^{-c_2\left(\min\left( k-1 ,\,N-k\right)-t\right)}
    \right). 
\end{equation}
We now define
\[
    a_\ell
    = \mathbb E 
    \left[
        G_\ell(\vartheta_0,\ldots,\vartheta_{\ell+1})
    \right]
    \quad\mbox{ for all } \ell\in\mathbb N .
\] 
For  \(\ell\in\mathbb N\), \(N\ge 2\ell+2\), \(k\in\llbracket1,N\rrbracket\), and \(t\in\mathbb N\) satisfying \(\min(k-1,N-k)\geq 2t+2\), by the triangle inequality,
\begin{multline*} 
    \left|
        \mathbb E_{N,k}^0
    \left[
        W_{\bs(\bg)}\left(\b^2\right) \right]-a_\ell
    \right|                                                     
 \le 
\mathbb E^0_{N,k}
    \left[
        \left|W_{\bs(\bg)}\left(\b^2\right)-\mathsf{K}_{N,\ell}(\bg) \right|\right] \\
        +
\sup_{\bg\in\lrw_N}
    \left|
        \mathsf{K}_{N,\ell}(\bg)
        -
        G_\ell(\sigma_0,\ldots,\sigma_{\ell+1})
    \right|                                                         
 + 
    \left|
        \mathbb E^0_{N,k}
        \left[
            G_{\ell}(\sigma_0,\ldots,\sigma_{\ell+1}) \right]
        -
        a_\ell
    \right|.
\end{multline*} 
By \eqref{eq:mryuik1}, \eqref{eq:mryuik2}, \eqref{eq:mryuik3}, and \(\|G_\ell\|_\infty\le \b^{2(\ell+1)}\), we obtain, under the same hypotheses,
\begin{equation}\label{eq:TN-unif-bulk-aell-fixed} 
    \begin{split}
        &\left|
        \mathbb E_{N,k}^0
    \left[
        W_{\bs(\bg)}\left(\b^2\right) \right]-a_\ell
    \right|\\
    &\le
    C_1 e^{-c_1\ell}
    +    \b^{2(\ell+1)}(\ell+1)\cdot \frac{q^{N-\ell+1}}{1-q^{N-\ell+1}}
    + C_2\b^{2(\ell+1)}
    \left(
        e^{C_2(\ell+1)}e^{-c_2t}
        +
        e^{C_2t}e^{-c_2\left(\min(k-1,N-k)-t\right)}
    \right). 
    \end{split}  
\end{equation} 
Choose \( \theta\in(0,1/2)\) satisfying \(C_2\theta<c_2(1-\theta)\).    In \eqref{eq:TN-unif-bulk-aell-fixed}, take \(N=2n+1\), \(k=n+1\), and \(t=\lfloor \theta n\rfloor\), and then let \(n\to\infty\).  For each fixed \(\ell\in\mathbb N\), all hypotheses of \eqref{eq:TN-unif-bulk-aell-fixed} hold for all sufficiently large \(n\), and all terms on the right side except \(C_1e^{-c_1\ell}\) tend to \(0\).   Write   $F_n=
    \mathbb E^0_{2n+1,n+1} [
        W_{\bs(\bg)} (\b^2 ) ]$. Then
        \begin{equation} \label{eq:CauchyF}
            \limsup_{n\to\infty}|F_n-a_\ell|
    \le C_1e^{-c_1\ell}
    \quad\mbox{ for each fixed }\ell\in\mathbb N .
        \end{equation}
Applying this with \(\ell,\ell'\in\mathbb N\), and using the triangle inequality, 
\begin{equation}\label{eq: qethbtqt} 
    |a_\ell-a_{\ell'}|
    \le
    \limsup_{n\to\infty}\left(|a_\ell-F_n|+|F_n-a_{\ell'}|\right)
    \le C_1e^{-c_1\ell}+C_1e^{-c_1\ell'} .
\end{equation} 
Hence \((a_\ell)_{\ell\ge1}\) is a Cauchy sequence.  Let
\[
    \Lambda_{q,\b}=\lim_{\ell\to\infty}a_\ell .
\]
Letting \(\ell'\to\infty\) in \eqref{eq: qethbtqt} gives
\begin{equation}\label{eq:aell-to-Lambda-rate}
    |a_\ell-\Lambda_{q,\b}|\le C_1e^{-c_1\ell}.
\end{equation}
Since \(F_n\in[1,M]\) by \eqref{eq:tyhjhneyjne}, the same $\limsup$ bound  as in \eqref{eq:CauchyF} gives
 $$ \operatorname{dist}(a_\ell,[1,M])
    \le
    \limsup_{n\to\infty}|a_\ell-F_n|
    \le C_1e^{-c_1\ell}.$$
Therefore $\Lambda_{q,\b}\in[1,M]\subset[1,\infty)$.
Denote $m=\min(k-1,N-k)$.  We now choose \( \eta\in(0,1)\) such that
\[
    (2\log \b+C_2)\eta<c_2\theta,\qquad
    2\eta\log \b<c_2(1-\theta)-C_2\theta,\qquad
    2\eta\log \b<(2-\eta)\log(1/q),
\]
 and let 
\[
\delta=
\min\left(
c_1\eta,\,
c_2\theta-(2\log\b+C_2)\eta,\,
c_2(1-\theta)-C_2\theta-2\eta\log\b,\,
(2-\eta)\log(1/q)-2\eta\log\b
\right)>0 .
\]
We then choose \(m_0\in\mathbb N\), depending only on \(q,\b\), such that if \(m\ge m_0\), \(\ell=\lfloor \eta m\rfloor\),   \(t=\lfloor \theta m\rfloor\), then \(\ell,t\in\mathbb N\), \(N\ge2\ell+2\), and \(m\ge2t+2\). 
 Since \(k-1\ge m\) and \(N-k\ge m\), we have \(N\ge 2m+1\).  Moreover,
\begin{align*}
    &\ell\le \eta m,\qquad \ell+1\le \eta m+1,\qquad
t\le \theta m,\qquad t\ge \theta m-1,\qquad
m-t\ge (1-\theta)m,\\
&  N-\ell+1\ge (2-\eta)m+2,\qquad N-\ell+1\ge (2-\eta)m+2,\qquad 1-q^{N-\ell+1}\ge 1-q.
\end{align*} 
By the choice of $\theta,\eta,\delta$, there exists $C_3>0$ depending only on $q,\b$, such that, when \(m\ge m_0\),
\begin{align*}
    \b^{2(\ell+1)}(\ell+1)\frac{q^{N-\ell+1}}{1-q^{N-\ell+1}}
&\le
C_3(m+1)
e^{
-\left((2-\eta)\log(1/q)-2\eta\log\b\right)m }
\le C_3(m+1)e^{-\delta m},\\
\b^{2(\ell+1)}e^{C_2(\ell+1)}e^{-c_2t}
&\le
C_3 e^{
-\left(c_2\theta-(2\log\b+C_2)\eta\right)m }
\le C_3e^{-\delta m},\\
\b^{2(\ell+1)}e^{C_2t}e^{-c_2(m-t)}
&\le
C_3e^{
-\left(c_2(1-\theta)-C_2\theta-2\eta\log\b\right)m }
\le C_3e^{-\delta m}.
\end{align*}
By \eqref{eq:TN-unif-bulk-aell-fixed} and  \eqref{eq:aell-to-Lambda-rate}, whenever \(m\ge m_0\),
\[
    \left|
        \mathbb E_{N,k}^0
    \left[
        W_{\bs(\bg)}\left(\b^2\right) \right]-\Lambda_{q,\b}
    \right|
    \le 2C_1e^{-c_1\ell}+ C_3(m+1+2C_2)e^{-\delta m}\leq C_4e^{-c_4 m},
\] 
for some constants $C_4,c_4>0$ depending only on $q,\b$, since \(\ell=\lfloor \eta m\rfloor\) and \(\delta\le c_1\eta\).
When \(m<m_0\), the bound follows from \eqref{eq:tyhjhneyjne} and \(\Lambda_{q,\b}\in[1,M]\), after possibly increasing \(C_4\).  This proves \eqref{eq:constant Lambda q b again}.
\end{proof}

\subsection{Upper bound on the distance from the uniform distribution: Proof of Proposition \ref{prop:macroscopic uniformity TN upper}} 

We first prove the upper bound on the total variation distance between the law of $T_N$ and the uniform distribution on the interval $\llbracket 0,N\rrbracket$ in the special case when $q=0$.
 
\begin{lemma}\label{lem:macroscopic uniformity q0}
    Assume \(\b>1\).  Then there exists a constant $C\in(0,\infty)$ depending only on $\b$, such that  
    \[  \TV{
        \mathbb P^{(0,\b,\b)}(T_N\in\cdot)
        -
        \mathrm{Unif}_{\llbracket0,N\rrbracket}
    }
    \leq
    \frac{C}{N+1}\quad\mbox{ for every }N\in\mathbb N.
\] 
\end{lemma}
\begin{proof}
    Taking $\a=\b$ in Lemma \ref{lem:decomposition TASEP} and  $w=(\b+\b^{-1}+2)/4$, we obtain 
\begin{equation} \label{eq:jdtyjykmt}
    \mathbb{P}^{(0,\b,\b)}(T_N = m) = \begin{cases}
\displaystyle D^{-1}\cdot    \b/(4w) \cdot   p_{m-1}p_{N-m}  & \quad\mbox{ if }
m\in\llbracket1,N\rrbracket,\\ 
\displaystyle D^{-1}\cdot  p_N  &\quad\mbox{ if } m = 0.
\end{cases}
\end{equation}
Here, for $\ell\in\mathbb{Z}_{\geq0}$, we denote by $p_{\ell}$ the probability that a $\nu_{\b}$-random walk started at $0$ does not reach level $-1$ during the first $\ell$ steps, and 
\begin{equation}\label{eq:D}
    D = \sum_{m=1}^{N}   \b/(4w) \cdot  p_{m-1}p_{N-m} +  p_N.\end{equation}

We first prove the claim that there exist constants \(C_1,c_1\in(0,\infty)\), depending only on \(\b\), such that
\begin{equation}\label{eq:p ell exponential convergence}
    0\leq p_\ell-p_{\infty}\leq C_1e^{-c_1\ell}
    \quad\mbox{ for all }
    \ell\in\mathbb Z_{\geq0},\quad\mbox{ where } p_{\infty}=1-\b^{-2}.
\end{equation}
Let \((S_i)_{i\geq0}\) be the \(\nu_\b\)-random walk started at \(S_0=0\), and let $\tau_{-1}=\inf\{i\geq0:S_i=-1\}$.
By the first statement in Lemma \ref{lem:basic hitting monotone coupling} with $h=1$,   we obtain
 $p_\ell=\mathbb P_0(\tau_{-1}>\ell)\downarrow \mathbb P_0(\tau_{-1}=\infty)=p_{\infty}$.
Hence, by the last statement in Lemma \ref{lem:basic hitting monotone coupling}  with $\kappa=(\b-1)/(2(\b+1))$,
\[
p_\ell-p_{\infty}= \mathbb P_0(\ell< \tau_{-1}<\infty)\leq 
\sum_{k=\ell+1}^{\infty}\mathbb P_0(S_k=-1)\leq \sum_{k=\ell+1}^{\infty}\mathbb P_0(S_k\leq \kappa k)\leq \sum_{k=\ell+1}^{\infty} C_2e^{-c_2k} ,
\]
for some $C_2,c_2\in(0,\infty)$ depending only on $\b$. We conclude the claim \eqref{eq:p ell exponential convergence}. Hence, for $m\in\llbracket1,N\rrbracket$,
\begin{equation}\label{eqw:hrbsth}
    \left| p_{m-1}p_{N-m}-p_{\infty}^2 \right|\leq 
        |p_{m-1}-p_\infty|p_{N-m}
        +
        p_\infty|p_{N-m}-p_\infty|\leq C_1e^{-c_1(m-1)}+C_1e^{-c_1(N-m)}.
\end{equation}
Combining  \eqref{eq:D}, \eqref{eq:p ell exponential convergence}, and \eqref{eqw:hrbsth},
\begin{equation}\label{eq:fysrhth}
    \begin{split}
        &\left|D-(N+1) \cdot\b/(4w) \cdot p_{\infty}^2\right|\leq\sum_{m=1}^{N}  \b/(4w) \cdot\left| p_{m-1}p_{N-m}-p_{\infty}^2 \right|
        +\left| p_N- \b/(4w) \cdot p_{\infty}^2\right| \\
    &\leq  \b/(4w) \cdot\sum_{m=1}^{N}\left(C_1e^{-c_1(m-1)}+C_1e^{-c_1(N-m)} 
    \right)       +1+ \b/(4w)     \leq C_3,
    \end{split}
\end{equation}
where $C_3\in(0,\infty)$ depends only on $\b$. 
Hence, combining the last displays, for large enough $N$ and $m\in\llbracket1,N\rrbracket$,
\begin{equation}\label{eq:nmdftyfukmed}
    \begin{split}
        &\left|\mathbb P^{(0,\b,\b)}(T_N=m)-\frac{1}{N+1}\right|=\left| \frac{    (\b/(4w)) \cdot   p_{m-1}p_{N-m}}{D}-\frac{1}{N+1}\right|\\
        &=\frac{1}{(N+1)D}\cdot \left|D-(N+1)\frac{\b}{4w}p_{m-1}p_{N-m}\right|
        \leq \frac{1}{ D}\left(\frac{C_3}{N+1}+ \frac{\b}{4w} \left| p_{m-1}p_{N-m}-p_{\infty}^2 \right|\right)\\
        &\leq \frac{1}{ D}\left(\frac{C_3}{N+1}+ \frac{\b}{4w} \left(C_1e^{-c_1(m-1)}+C_1e^{-c_1(N-m)}\right)\right).
    \end{split}
\end{equation}
When $m=0$, we have 
\begin{equation}\label{eq:myudkuek}
\begin{split}
    &\left|\mathbb P^{(0,\b,\b)}(T_N=0)-\frac{1}{N+1}\right|=\left|\frac{p_N}{D}-\frac{1}{N+1}\right|\leq \frac{1}{D}+\frac{1}{N+1}
    \end{split}
\end{equation}
By \eqref{eq:fysrhth}, there is \(C_4\in(0,\infty)\), depending only on \(\b\), such that
\(D\leq C_4(N+1)\) for every \(N\in\mathbb Z_{\geq0}\); hence \(1/(N+1)\leq C_4/D\).
Summing \eqref{eq:nmdftyfukmed} over $m\in\llbracket1,N\rrbracket$  and combining with \eqref{eq:myudkuek}, for large enough $N$,
\[
\TV{
        \mathbb P^{(0,\b,\b)}(T_N\in\cdot)
        -
        \mathrm{Unif}_{\llbracket0,N\rrbracket}
    }
    \leq C_5/D\leq C/(N+1),
\]
where $C_5,C\in(0,\infty)$ depend only on $\b$. The last inequality follows from \eqref{eq:fysrhth}. Increasing \(C\), if necessary, covers the remaining finitely many \(N\in\mathbb Z_{\geq0}\), since the total variation distance is at most \(1\). 
\end{proof}

We now prove the desired upper bound on the total variation distance between the law of $T_N$ and the uniform distribution on the interval $\llbracket 0,N\rrbracket$ for general $q \in [0,1)$.

\begin{proof}[Proof of Proposition \ref{prop:macroscopic uniformity TN upper}]
For simplicity, we will write $\mathbb{P}^0=\mathbb{P}^{(0,\b,\b)}$, $\mathbb{P}^q=\mathbb{P}^{(q,\b,\b)}$, $\mathbb{E}^0=\mathbb{E}^{(0,\b,\b)}$,  and $\mathbb{E}^q=\mathbb{E}^{(q,\b,\b)}$ in this proof.
When $N=0$, the total variation distance is $0$. We next assume $N\geq1$. If $q=0$, the result follows from Proposition \ref{lem:macroscopic uniformity q0}. In the rest of the proof, we fix \(q\in(0,1)\). By the result for $\mathbb{P}^0$,
\begin{equation}\label{eq:q0TN}
    \TV{
        \mathbb P^{0}(T_N\in\cdot)
        -
        \mathrm{Unif}_{\llbracket0,N\rrbracket}
    }
    \leq C_1/(N+1)\quad\mbox{ for all }N\in\mathbb N,
\end{equation}
for some $C_1>0$ depending only on $q,\b$. 
% By \eqref{eq:PNbb again again again},
% \begin{equation}\label{eq:TN-upper-RN}
%     \mathbb P^q(B)
%     =
%     \frac{\mathbb E^0\left[{\bf1}_B W_{\bs(\bg)}(\b^2)\right]}
%     {\mathbb E^0\left[W_{\bs(\bg)}(\b^2)\right]}\quad\mbox{ for every event }B\subseteq\lrw_N.
% \end{equation} 
By Lemma \ref{lem:decomposition TASEP} with \(\a=\b\), 
\(\mathbb P^0(T_N=k)>0\) for every \(k\in\llbracket0,N\rrbracket\).
Let 
\[
p_{N,k}=\mathbb P^0(T_N=k)\quad\mbox{ and }\quad A_{N,k}=\mathbb E^0\left[W_{\bs(\bg)}(\b^2)\mid T_N=k\right]\quad\mbox{ for }k\in\llbracket0,N\rrbracket.
\] By \eqref{eq:q0TN},
\begin{equation}\label{eq:TN-upper-p-pointwise}
    p_{N,k}
    \le
   (C_1+1)/(N+1) \quad
    \mbox{for every }N\in\mathbb N,\ k\in\llbracket0,N\rrbracket .
\end{equation}
By Theorem \ref{thm:prob interpretation of Bryc polynomials} and Definition \ref{def:characterization measure}, \(W_{\bs(\bg)}(\b^2)\ge1\).  By Proposition \ref{prop:moment bound coex} with \(z=\b^2\) and \(m=1\), there exists    \(M\in(0,\infty)\)  depending only on \(q,\b\), such that
\begin{equation}\label{eq:TN-upper-W-moment}
    1\le A_{N,k}\le M
    \quad
    \mbox{for every }N\in\mathbb N,\ k\in\llbracket0,N\rrbracket .
\end{equation}
By Proposition \ref{prop:constant Lambda q b}, there exist \(\Lambda_{q,\b}\in[1,\infty)\) and \(C_2,c_2\in(0,\infty)\) depending only
on \(q,\b\), such that
\begin{equation}\label{eq:TN-upper-A-Lambda}
    |A_{N,k}-\Lambda_{q,\b}|
    \le
    C_2 e^{-c_2\min(k-1,N-k)}\quad
    \mbox{for every }N\in\mathbb N,\ k\in\llbracket1,N\rrbracket .
\end{equation} 
By \eqref{eq:TN-upper-p-pointwise}, \eqref{eq:TN-upper-W-moment}, and \eqref{eq:TN-upper-A-Lambda},
\begin{equation}\label{eq:TN-upper-average-A-Lambda}
\begin{aligned}
    \sum_{k=0}^N p_{N,k}|A_{N,k}-\Lambda_{q,\b}|
    &\le
    p_{N,0}(M+\Lambda_{q,\b})
    +
    C_2\sum_{k=1}^N p_{N,k}e^{-c_2\min(k-1,N-k)}                                      \\
    &\le
    \frac{(C_1+1)(M+\Lambda_{q,\b})}{N+1}
    +
    \frac{(C_1+1)C_2}{N+1}
    \sum_{k=1}^N e^{-c_2\min(k-1,N-k)}\le
    \frac{C_3}{N+1},
\end{aligned}
\end{equation}
for some constant \(C_3\in(0,\infty)\) depending only on \(q,\b\). Here, in the last step we used $\sum_{k=1}^N e^{-c_2\min(k-1,N-k)}
    \le
    2\sum_{r=0}^{\infty}e^{-c_2r}
    =
     2/(1-e^{-c_2})$. Note that 
\begin{equation}\label{eq:TN-upper-mu-formula}
1\leq \mathbb E^0\left[W_{\bs(\bg)}(\b^2) \right]=D_N=\sum_{j=0}^N p_{N,j}A_{N,j},\qquad 
   \mathbb P^q(T_N=k)
    = p_{N,k}A_{N,k}/D_N.
\end{equation}
Combining with \(\sum_{k=0}^N p_{N,k}=1\) and \eqref{eq:TN-upper-average-A-Lambda}, we also have 
 \begin{equation}\label{eq:TN-upper-D-Lambda}
    |D_N-\Lambda_{q,\b}|
    \le
    \sum_{k=0}^N p_{N,k}|A_{N,k}-\Lambda_{q,\b}|
    \le C_3/(N+1).
\end{equation} 
By \eqref{eq:TN-upper-mu-formula},  \eqref{eq:TN-upper-average-A-Lambda}, and \eqref{eq:TN-upper-D-Lambda},
\[
\begin{aligned} 
        &\TV{ \mathbb P^{q}(T_N\in\cdot)- \mathbb P^{0}(T_N\in\cdot)}\le\sum_{k=0}^N\left|\mathbb P^{q}(T_N=k)-\mathbb P^{0}(T_N=k)\right|   =   
    \sum_{k=0}^N
    p_{N,k}\left|\frac{A_{N,k}}{D_N}-1\right|                         \\
    &=
    \frac{1}{D_N}
    \sum_{k=0}^N p_{N,k}|A_{N,k}-D_N|                                   \le
    \sum_{k=0}^N p_{N,k}|A_{N,k}-\Lambda_{q,\b}|
    +
    |D_N-\Lambda_{q,\b}|                                                \le
    2C_3/(N+1).
\end{aligned}
\] 
Combining with \eqref{eq:q0TN}, by the triangle inequality, 
  $\TV{
        \mathbb P^q(T_N\in\cdot)
        -
        \mathrm{Unif}_{\llbracket0,N\rrbracket}
    }
                                            \le
    (2C_3+C_1)/(N+1)$. 
\end{proof}

\subsection{Lower bound on the distance from the uniform distribution: Proof of Proposition \ref{prop:macroscopic uniformity TN lower}}

We will now establish the corresponding lower bound on the total variation distance between the law of $T_N$ and the uniform distribution on the interval $\llbracket 0,N\rrbracket$.

\begin{proof}[Proof of Proposition \ref{prop:macroscopic uniformity TN lower}] 
We write $\mathbb P^0=\mathbb P^{(0,\b,\b)}$ and $\mathbb P^q=\mathbb P^{(q,\b,\b)}$. We split the proof into two steps.\\

\noindent\textbf{Step 1. } In this step, we conclude the proof of the proposition assuming the following claim
\begin{equation}\label{eq:claimheryhjtnwsyn}
    \TV{\mathbb P^q(T_N\in\cdot)-\mathrm{Unif}_{\llbracket0,N\rrbracket}}\geq \mathbb P^q(T_N=0,\ \gamma_1=0)/2.
\end{equation}
By \eqref{eq:PNbb again again again},
\begin{equation} \label{eq:ytnedyjnetyjn}
    \mathbb P^q(B)
    =
    \frac{\mathbb E^0\left[{\bf1}_B W_{\bs(\bg)}(\b^2)\right]}
    {\mathbb E^0\left[W_{\bs(\bg)}(\b^2)\right]}\quad\mbox{ for every event }B\subseteq\lrw_N.
\end{equation}
By Theorem \ref{thm:prob interpretation of Bryc polynomials}, $W_{\bs(\bg)}(\b^2)\ge1$ for every $\bg\in\lrw_N$.  
When \(q\in(0,1)\), by Proposition \ref{prop:moment bound coex} with \(z=\b^2\) and \(m=1\); and when $q=0$, by \(W_{\bs(\bg)}(\b^2)=1\), there exists \(M\in(0,\infty)\), depending only on \(q,\b\), such that
\begin{equation*} 
    \mathbb E^0\left[W_{\bs(\bg)}(\b^2)\mid T_N=k\right]\le M
    \quad
    \mbox{for every }N\in\mathbb N,\ k\in\llbracket0,N\rrbracket .
\end{equation*}
Hence \(\mathbb E^0[W_{\bs(\bg)}(\b^2)]\le M\). Combining this upper bound, the lower bound \(W_{\bs(\bg)}(\b^2)\ge1\), in view of \eqref{eq:ytnedyjnetyjn},
\begin{equation}\label{eq:q-vs-zero-lower-TN-converse}
    \mathbb P^q(B)
    \ge 
    \mathbb P^0(B)/M 
    \quad\mbox{ for every event }B\subseteq\lrw_N.
\end{equation}  
For $\ell\in\mathbb{Z}_{\geq0}$, we denote by $p_{\ell}$ the probability that a $\nu_{\b}$-random walk started at $0$ does not reach level $-1$ during the first $\ell$ steps, set \(p_\infty=1-\b^{-2}\), and set  $w=(\b+\b^{-1}+2)/4$.
Taking $\a=\b$ in Lemma \ref{lem:decomposition TASEP},  
\begin{equation*}
     \mathbb{P}^{0}(T_N = 0) =\frac{p_N}{\sum_{m=1}^{N}   \b/(4w) \cdot  p_{m-1}p_{N-m} +  p_N} \ge
    \frac{p_\infty}{N+1},
\end{equation*}
where we used $\b/(4w)<1$ and, by Lemma \ref{lem:basic hitting monotone coupling}, \(p_\infty\leq p_{\ell}\leq1\) for all $\ell\in\mathbb{Z}_{\geq0}$. 
By Lemma \ref{lem:decomposition TASEP}, under \(\mathbb P^0(\,\cdot\,|T_N=0)\), the path \((\gamma_i-\gamma_0)_{0\le i\le N}\) has the law of the \(\nu_\b\)-random walk started from \(0\), conditioned on staying nonnegative up to time $N$.
Thus, using the Markov property at time \(1\),
\[ 
    \mathbb P^0(\gamma_1=0\mid T_N=0)
     =
    \mathbb P_0(S_1=0\mid S_i\ge0\text{ for }i\in\llbracket0,N\rrbracket)                        
     = \nu_\b(0)\,p_{N-1}/p_N                              
     \ge
    \nu_\b(0)\,p_\infty , 
\]
where \(p_0=1\) when \(N=1\).  Taking $B=\{\gamma_1=0, T_N=0\}$ in \eqref{eq:q-vs-zero-lower-TN-converse} and combining with the last two displays,
 \[
 \mathbb P^q(T_N=0,\ \gamma_1=0)\geq    \nu_\b(0)\,p_\infty^2/(M(N+1)).
 \]
Combining with \eqref{eq:claimheryhjtnwsyn}, we conclude the proof.\\

\noindent\textbf{Step 2. } In this step, we prove the claim \eqref{eq:claimheryhjtnwsyn}.  Define
\(\mathrm{inv}:\lrw_N\to\lrw_N\) by \[(\mathrm{inv}(\bg))_i=\gamma_{N-i}-\gamma_N \quad\mbox{ for } i\in\llbracket0,N\rrbracket.\] 
Note that $\bs(\mathrm{inv}(\bg))=\bs(\bg)$, $H(\mathrm{inv}(\bg))=H(\bg)$, and
\[
    \min(\mathrm{inv}(\bg))=\min(\bg)-\gamma_N,\qquad
    (\mathrm{inv}(\bg))_N-2\min(\mathrm{inv}(\bg))
    =-\gamma_N-2(\min(\bg)-\gamma_N)
    =\gamma_N-2\min(\bg).
\]
Thus \eqref{eq:PNbb again again again} gives
$\mathbb P^q(\mathrm{inv}(\bg))=\mathbb P^q(\bg)$ for $\bg\in\lrw_N$.  
Let
\[
    Q_N(\bg)
    =
    \max\{i\in\llbracket0,N\rrbracket:\gamma_i=\min(\bg)\}\quad\mbox{ for }\bg\in\lrw_N,
\]   
then  $N-Q_N(\bg)=T_N(\mathrm{inv}(\bg))$ for $\bg\in\lrw_N$. Hence
$\mathbb P^q(T_N=N)=\mathbb P^q(Q_N=0)$.
Therefore, 
\begin{equation}\label{eq:TN-TV-by-endpoint-gap}
\begin{aligned}
    \TV{\mathbb P^{q}(T_N\in\cdot)-\mathrm{Unif}_{\llbracket0,N\rrbracket}}
    &\ge
    \max\left(\left|\mathbb P^{q}(T_N=0)-(N+1)^{-1}\right|,\ \left|\mathbb P^{q}(T_N=N)-(N+1)^{-1}\right|\right)       \\
    &\ge
    \frac12
    \left|
        \left(\mathbb P^{q}(T_N=0)-(N+1)^{-1}\right)-\left(\mathbb P^{q}(Q_N=0)-(N+1)^{-1}\right)
    \right|                                                  \\
    &=
    \frac12\,
    \mathbb P^q\left(T_N=0,\ Q_N>0\right)\geq\frac{1}{2}\mathbb P^q(T_N=0,\ \gamma_1=0),
\end{aligned}
\end{equation}
where we used \(\{Q_N=0\}\subseteq\{T_N=0\}\) and \(\{T_N=0,\gamma_1=0\}\subseteq\{T_N=0,Q_N>0\}\).
We conclude   \eqref{eq:claimheryhjtnwsyn}.
\end{proof}

\section{Approximating the stationary distribution of the open ASEP at the coexistence line: Proof of Theorem \ref{thm:ApproxCoex}}\label{sec: proof of main thm}

In this section, we establish our main total variation approximation result Theorem \ref{thm:ApproxCoex} on the stationary distribution of the open ASEP on the coexistence line, which, for clarity, we restate below again as Theorem~\ref{thm:ApproxCoex again}.
Recall that we write \(\ber_I(p)\) for the Bernoulli product measure on \(\{0,1\}^I\) with parameter \(p\).

\begin{theorem}\label{thm:ApproxCoex again}
 Assume $q\in[0,1)$ and $\b>1$. Then there exist constants $C,C',c>0$ depending only on $q$ and $\b$, such that for each $N\geq m\geq1$ and $k\in\llbracket0,N\rrbracket$, 
\begin{equation*}
        \TV{
    \mathbb{P} 
    \left(
        (\tau_i)_{i\in \llbracket1,k-m\rrbracket\cup \llbracket k+m,N\rrbracket}\in\cdot
        \,\middle|\,\mathsf{M}_N=k
    \right)
    -
    \ber_{\llbracket1,k-m\rrbracket} \left(\frac{1}{1+\b}\right)
    \otimes
    \ber_{\llbracket k+m,N\rrbracket} \left(\frac{\b}{1+\b}\right) 
}
\le
Ce^{-cm},
\end{equation*} as well as
\begin{equation*}
\qquad\frac{C'}{N}\leq  \TV{
        \mathbb{P}  (\mathsf M_N\in\cdot)
        -
        \mathrm{Unif}_{\llbracket0,N\rrbracket}
    }
    \leq \frac{C}{N} , 
\end{equation*} 
where we write $\mathrm{Unif}_{\llbracket0,N\rrbracket}$ to denote the uniform distribution on the lattice $\llbracket0,N\rrbracket$. In particular, as $N\to\infty$, $\mathsf M_N/N$ converge weakly to the uniform random variable on the interval $[0,1]$. 
\end{theorem}

In order to show Theorem \ref{thm:ApproxCoex again}, we will work with the extended state space used in the two-layer representation for the stationary distribution of the open ASEP from Theorem~\ref{thm:characterization stationary measure ASEP}.
Assume $q\in[0,1)$ and $\b>1$ and recall that
\begin{equation}\label{eq:PNbb}
\mathbb{P}^{(q,\b,\b)}(\bg)=\frac{1}{Z_N^{(q,\b,\b)}} 2^{H(\bg)}\b^{\gamma_N-2\min(\bg)}  W_{\bs(\bg)}\left(\b^2\right)\quad\mbox{ for }\bg\in\lrw_N.
\end{equation}
We consider the enlarged probability space with probability measure denoted by $\mathbb{P}^{(q,\b,\b)}_{N,\uptau}$, which supports $\bg\sim\mathbb{P}^{(q,\b,\b)}$ together with i.i.d. Bernoulli variables $\kappa_i$, $i\in\llbracket1,N\rrbracket$ with mean $1/2$ that are independent of $\bg$. 
By Theorem \ref{thm:characterization stationary measure ASEP}, the random variable $\tau=(\tau_1,\dots,\tau_N)\in\{0,1\}^N$ on this enlarged probability space, given by   
\begin{equation}\label{eq:occupation coex}
    \tau_i=\begin{cases}
        1 & \mbox{ if } \gamma_i-\gamma_{i-1}=1,\\
        \kappa_i & \mbox{ if } \gamma_i-\gamma_{i-1}=0,\\ 
        0 & \mbox{ if } \gamma_i-\gamma_{i-1}=-1,
    \end{cases}
\end{equation}
is distributed as $\mu_N^{(q,\b,\b)}$. Here, $\mu_N^{(q,\b,\b)}$ is the stationary measure of the open ASEP on the lattice $\llbracket1,N\rrbracket$ with jump rates $\alpha,\beta,\gamma,\delta$ and $q$ satisfying \eqref{eq:conditions open ASEP} and Liggett's condition \eqref{eq:Liggett's condition} on the coexistence line $\a=\b>1$.
On this enlarged probability space, we define observables 
\begin{equation}\label{eq:observables TN MN}
    \begin{split}
            T_N &=\min\{i\in\llbracket0,N\rrbracket:\gamma_i=\min(\bg)\}  ,   \\
     \mathsf{M}_N &= \min\left\{
        j\in\llbracket0,N\rrbracket:
        s_j (\tau)=\min_{\ell\in\llbracket0,N\rrbracket}s_{\ell} (\tau)
    \right\}  , 
    \end{split}
\end{equation}
where $\tau\in\{0,1\}^N$ is as in \eqref{eq:occupation coex}, $s_0(\tau)=0$, and $s_j(\tau)=\sum_{i=1}^j\left(2\tau_i -1\right)$ for $j\in\llbracket1,N\rrbracket$. 

We have the following strategy to prove Theorem \ref{thm:ApproxCoex again}. In Section~\ref{sec:Approx1}, we establish the approximation result from Theorem \ref{thm:ApproxCoex again} when we condition on the location of $T_N$ instead of $\mathsf M_N$.
We then show in Section~\ref{sec:ApproxMid} that the location $\mathsf M_N$ is close to the position $T_N$, and that the approximation from Theorem \ref{thm:ApproxCoex again} can be achieved when jointly conditioning on $\mathsf M_N$ and $T_N$.
In Sections~\ref{sec:Approx3} and~\ref{sec:Approx4}, we provide upper and lower bounds on the total variation distance between $\mathsf M_N$  and the uniform distribution on $\llbracket 0,N \rrbracket$, which will follow from combining the upper bounds from Proposition \ref{prop:part of the thm upper bound} with the lower bounds from Proposition \ref{prop:part of the thm lower bound}.
 In Section~\ref{sec:Final}, we combine the above results in order to conclude Theorem~\ref{thm:ApproxCoex again}.

\subsection{Approximation by Bernoulli product measures}\label{sec:Approx1}

In this subsection, we prove the following approximation result when conditioning on the location $T_N$. To simplify notation, we denote by $\mathbf{P}_{N,k}^q$ the measure $\mathbb{P}^{(q,\b,\b)}_{N,\uptau}$ conditioned on $T_N=k$.

\begin{proposition}\label{prop:first part proof Thm}
    Assume $q\in[0,1)$ and $\b>1$.  Then there exist constants $C,c>0$ depending only on $q$ and $\b$, such that for each $N\geq m\geq1$ and $k\in\llbracket0,N\rrbracket$, 
\begin{equation}\label{eq:estimate by Bernoulli product}
\TV{
    \mathbf{P}_{N,k}^q
    \left(
        (\tau_i)_{i\in \llbracket1,k-m\rrbracket\cup \llbracket k+m,N\rrbracket}\in\cdot 
    \right)
    -
    \ber_{\llbracket1,k-m\rrbracket} \left(\frac{1}{1+\b}\right)
    \otimes
    \ber_{\llbracket k+m,N\rrbracket} \left(\frac{\b}{1+\b}\right) 
}
\le
Ce^{-cm}.
\end{equation}
\end{proposition}

The next lemma on the law of a $\nu_{\b}$-random walk is taken from \cite{SchmidYangShock}. 
\begin{lemma}[c.f. Lemma 5.3 in \cite{SchmidYangShock}]\label{lem:appro conditioned random walk}
Assume $\b>1$. Let $\{Y_i\}_{i=0}^N$ be a $\nu_{\b}$-random walk with $N$ steps started at $Y_0=0$, conditioned to stay non-negative. 
Then there exist constants $C,c>0$ depending only on $\b$, such that for each $N\in\mathbb N$ and $m\in\llbracket0,N\rrbracket$, 
    \begin{equation*}
        \TV{ \mathbb{P}((Y_{i}-Y_{i-1})_{i\in\llbracket m+1,N\rrbracket} \in \cdot\, ) -\nu_{\b}^{\otimes (N-m)}} \leq Ce^{-cm}.
    \end{equation*} 
\end{lemma}
 Let us stress that strictly speaking, Lemma 5.3 in \cite{SchmidYangShock} asserts that $m \geq 1$. However, the case $m=0$ can be covered by increasing the constant $C>0$ if necessary.

\begin{proof}[Proof of Proposition \ref{prop:first part proof Thm}]
 In addition to the measure $\mathbf{P}_{N,k}^q$, we will also consider the measure \(\mathbb P_{N,k}^q\) on the path space which we obtain by \(\mathbb P^{(q,\b,\b)}\) from \eqref{eq:PNab} when conditioning on $T_N=k$. 
If $q=0$, Lemma \ref{lem:decomposition TASEP} with \(\a=\b\) gives
\(\mathbb P^{(0,\b,\b)}(T_N=k)>0\) for  \(k\in\llbracket0,N\rrbracket\). If \(q\in(0,1)\), Proposition \ref{prop:comparison path q} gives
\(\mathbb P^{(q,\b,\b)}(T_N=k)>0\) for  \(k\in\llbracket0,N\rrbracket\). Since \(T_N\) is a function only of \(\bg\), we know that $\mathbb P^{(q,\b,\b)}_{N,\uptau}(T_N=k)=\mathbb P^{(q,\b,\b)}(T_N=k)>0$, and hence  \(\mathbf P_{N,k}^q\) is well-defined. Since \(\{T_N=k\}\) is measurable with respect to \(\bg\), under \(\mathbf P_{N,k}^q\), the variables \((\kappa_i)_{i=1}^N\) remain i.i.d. \(\ber(1/2)\), independent of \(\bg\). Define  
\[
\Upphi:\{-1,0,1\}\times\{0,1\}\to\{0,1\},\qquad \Upphi(u,\kappa)={\bf1}_{u=1}+\kappa{\bf1}_{u=0}.
\]
Note that \(\tau_i=\Upphi(\gamma_i-\gamma_{i-1},\kappa_i)\) for $i\in\llbracket1,N\rrbracket$.
Observe that if $\updelta\sim \nu_{\b}$ and $\kappa\sim\ber(1/2)$ are independent variables, then $\Upphi(\updelta,\kappa)\sim \ber(\b/(1+\b))$ and $\Upphi(-\updelta,\kappa)\sim \ber(1/(1+\b))$.

We first prove the estimate \eqref{eq:estimate by Bernoulli product} for $q=0$. Let \(C_1,c_1>0\) be the constants from Lemma \ref{lem:appro conditioned random walk}, which depend only on $\b$. For the right block, if \(\llbracket k+m,N\rrbracket\ne\emptyset\), then by Lemma \ref{lem:decomposition TASEP} we know that, under \(\mathbb P_{N,k}^0\),  $R_j=\gamma_{k+j}-\gamma_k$, $j\in\llbracket0,N-k\rrbracket$ is a \(\nu_\b\)-random walk started from \(0\), conditioned to stay nonnegative up to time \(N-k\). Using Lemma \ref{lem:appro conditioned random walk} with length $N-k$ and parameter $m-1$,  
\[
\TV{
\mathbb P_{N,k}^0
\left(
    (R_j-R_{j-1})_{j\in\llbracket m,N-k\rrbracket}\in\cdot
\right)
-
\nu_\b^{\otimes(N-k-m+1)}
}
\le C_1e^{-c_1(m-1)}.
\]
For \(j\in\llbracket m,N-k\rrbracket\), \(R_j-R_{j-1}=\gamma_{k+j}-\gamma_{k+j-1}\). Applying the coordinatewise map \(\Upphi\) with independent variables $\kappa_i\sim\ber(1/2)$, $i\in\llbracket k+m,N\rrbracket$, gives
\begin{equation}\label{eq:coex-right-bernoulli-q0-fixed new}
\TV{
\mathbf P_{N,k}^0\left((\tau_i)_{i\in \llbracket k+m,N\rrbracket}\in\cdot\right)
-
\ber_{\llbracket k+m,N\rrbracket}\left(\frac{\b}{1+\b}\right)
}
\le C_1e^{-c_1(m-1)}.
\end{equation} 
For the left block, if $\llbracket1,k-m\rrbracket\neq\emptyset$, then by Lemma \ref{lem:decomposition TASEP} we know that, under \(\mathbb P_{N,k}^0\),  $Z_j=\gamma_{k-1-j}-\gamma_{k-1}$, $j\in\llbracket0,k-1\rrbracket$ is a \(\nu_\b\)-random walk started from \(0\),
conditioned to stay nonnegative up to time \(k-1\). Using Lemma \ref{lem:appro conditioned random walk} with length $k-1$ and parameter $m-1$,  
\[
\TV{
\mathbb P_{N,k}^0
\left(
    (Z_j-Z_{j-1})_{j\in\llbracket m,k-1\rrbracket}\in\cdot
\right)
-
\nu_\b^{\otimes(k -m )}
}
\le C_1e^{-c_1(m-1)}.
\]
For \(i=k-j\) with \(j\in\llbracket m,k-1\rrbracket\), we have
\(\gamma_i-\gamma_{i-1}=-(Z_j-Z_{j-1})\). Applying the coordinatewise map \((u,\kappa)\mapsto\Upphi(-u,\kappa)\) with independent variables $\kappa_i\sim\ber(1/2)$, $i\in\llbracket1,k-m\rrbracket$, gives
\begin{equation}\label{eq:coex-left-bernoulli-q0-fixed new}
\TV{
\mathbf P_{N,k}^0\left((\tau_i)_{i\in \llbracket1,k-m\rrbracket}\in\cdot\right)
-
\ber_{\llbracket1,k-m\rrbracket}\left(\frac{1}{1+\b}\right)
}
\le C_1e^{-c_1(m-1)}.
\end{equation}
When both left and right blocks are nonempty, Lemma \ref{lem:decomposition TASEP} gives the independence of the left and right paths under \(\mathbb P_{N,k}^0\). The auxiliary Bernoulli variables on disjoint sets of sites are independent of the paths and of one another. When one of the two blocks is empty, the corresponding marginal law and the corresponding Bernoulli product measure are both the unit mass on the one-point space.  Combining \eqref{eq:coex-right-bernoulli-q0-fixed new} and \eqref{eq:coex-left-bernoulli-q0-fixed new},  
\begin{equation}\label{eq:dbhrythetyjn}
\TV{
\mathbf P_{N,k}^0
\left((\tau_i)_{i\in \llbracket1,k-m\rrbracket\cup \llbracket k+m,N\rrbracket}\in\cdot\right)
-
\ber_{\llbracket1,k-m\rrbracket}\left(\frac{1}{1+\b}\right)
\otimes
\ber_{\llbracket k+m,N\rrbracket}\left(\frac{\b}{1+\b}\right)
}
\le 2C_1e^{-c_1(m-1)}.
\end{equation}
This proves \eqref{eq:estimate by Bernoulli product} in the case $q=0$.  
In the rest of the proof, we assume $q\in(0,1)$. Let \(C_2,c_2>0\) be the constants from Proposition \ref{prop:comparison path q}.
Recall from \eqref{eq:Ianddelta} that, for \(k,m\in\llbracket0,N\rrbracket\),
\[
\mathsf I_{N,k,m}  =  
    \{i\in\llbracket0,N-1\rrbracket: |i-k|\ge m\}\qquad\mbox{ and }\qquad\Delta_{N,k,m}(\bg) =\left(\gamma_{i+1}-\gamma_i\right)_{i\in\mathsf I_{N,k,m}}
    \in \{-1,0,1\}^{\mathsf I_{N,k,m}}.
\]
Denote $I=\llbracket1,k-m\rrbracket\cup \llbracket k+m,N\rrbracket$.
One can check   $\{i-1:i\in I\}\subseteq\mathsf I_{N,k,m-1}$. Thus every increment \(\gamma_i-\gamma_{i-1}\), \(i\in I\), is a coordinate projection of \(\Delta_{N,k,m-1}(\bg)\). The vector \((\tau_i)_{i\in I}\) is obtained from these increments by adjoining the independent Bernoulli variables \(\kappa_i\sim\ber(1/2)\), \(i\in I\), and applying \(\Upphi\).  By Proposition \ref{prop:comparison path q},
\begin{equation*} 
\begin{split}
&\TV{
\mathbf P_{N,k}^q\left((\tau_i)_{i\in I}\in\cdot\right)
-
\mathbf P_{N,k}^0\left((\tau_i)_{i\in I}\in\cdot\right)
} \\&\leq \TV{
\mathbb P_{N,k}^q\left(\Delta_{N,k,m-1}(\bg)\in\cdot\right)
-
\mathbb P_{N,k}^0\left(\Delta_{N,k,m-1}(\bg)\in\cdot\right)
}
\le C_2e^{-c_2(m-1)} .
\end{split}
\end{equation*}
Combining with \eqref{eq:dbhrythetyjn}, we conclude the proof of the estimate  \eqref{eq:estimate by Bernoulli product} for $q\in(0,1)$. 
\end{proof}

\subsection{Local approximation for the conditioned path measure} \label{sec:ApproxMid}

In the following, 
let $\mathbf{E}_{N,k}^q$ denote expectation under $\mathbf{P}_{N,k}^q$. Moreover, we set $\rho_-=\frac{1}{1+\b}$ and $\rho_+=\frac{\b}{1+\b}$, and write
\begin{equation}\label{eq:CoreDefs}
    \begin{split}
   I_{k,m}&=\llbracket 1,k-m \rrbracket\cup \llbracket k+m,N \rrbracket,\qquad X_{k,m}=(\tau_i)_{i\in I_{k,m}},\\
B_{k,m}&=\operatorname{Ber}_{\llbracket 1,k-m \rrbracket}(\rho_-)
          \otimes\operatorname{Ber}_{\llbracket k+m,N \rrbracket}(\rho_+),   
    \end{split}
\end{equation}
and set $W(\gamma)=W_{\sigma(\gamma)}(b^2)$ and $Z_k= \mathbf{E}_{N,k}^0 W$,
with the convention that $W\equiv1$ for $q=0$.
Note that from \eqref{eq:suff-coex-conditional-density} and Proposition~\ref{prop:moment bound coex}, we see that
\begin{equation}\label{asepm:eq:RN}
\frac{\mathrm d \mathbf{P}_{N,k}^q}{\mathrm d \mathbf{P}_{N,k}^0}=\frac W{Z_k}
\end{equation}
and  $Z_k \in [1,C_W]$ for some constant $C_W>0$ uniformly in $N,k$. For fixed $k\in \llbracket 0,N \rrbracket$ and $r \in \llbracket 1,N \rrbracket$, let $a=\max(0,k-r)$ and $z=\min(N,k+r)$. The goal of the following lemma is to construct a  probability measure $\mathbf{Q}_{k,r}$ on the same state space as $\mathbf{P}_{N,k}^q$, such that particles have a law close to $\mathbf{P}_{N,k}^q$ on $\llbracket a+1,z \rrbracket$ while the particles on $\llbracket 1,a \rrbracket\cup \llbracket z+1,N \rrbracket$ have a product law with parameter $\frac1{1+\b}$ on the left and $\frac\b{1+\b}$ on the right, and are independent of all path increments and Bernoulli random variables on $ \llbracket a+1,z \rrbracket$.

\begin{lemma}\label{asepm:lem:local}
There exists constants $c,C>0$ such that for every $N\ge1$, $k\in \llbracket 0,N \rrbracket$ and $r \in \llbracket 1,N \rrbracket$, we find a measure $\mathbf{Q}_{k,r}$ so that
\begin{equation}\label{asepm:eq:localTV}
\left\|\mathbf{P}_{N,k}^q-\mathbf{Q}_{k,r}\right\|_{\mathrm{TV}}\le Ce^{-cr}
\end{equation} while $\mathbf{Q}_{k,r}$ is on $\llbracket 1,a \rrbracket$ a Bernoulli product measure with parameter $\frac1{1+\b}$,  and on $\llbracket z+1,N \rrbracket$ a Bernoulli product measure with parameter$\frac\b{1+\b}$.
\end{lemma}

\begin{proof}
Define a path measure $\mathbf{R}_k$ by taking independent path increments, with $\gamma_0=0$, and laws
\begin{equation}\label{asepm:eq:reference}
\gamma_i-\gamma_{i-1}\sim
\begin{cases}
\nu_{1/b},&i<k,\\
-1,&i=k,\quad k\ge1,\\
\nu_b,&i>k,
\end{cases}
\end{equation}
 where we recall
\[
\nu_b(1)=\frac b{b+b^{-1}+2},\quad
\nu_b(0)=\frac2{b+b^{-1}+2},\quad
\nu_b(-1)=\frac{b^{-1}}{b+b^{-1}+2}.
\]
Consider the event
\[
A=\{\gamma_i>\gamma_k\text{ for }i<k,\quad
       \gamma_i\ge\gamma_k\text{ for }i>k\}.
\]
and note that by Lemma~\ref{lem:decomposition TASEP} and Lemma~\ref{lem:basic hitting monotone coupling},
\begin{equation}\label{asepm:eq:conditionA}
\mathbf{P}_{N,k}^0=\mathbf{R}_k(\,\cdot\mid A) , \qquad \mathbf{R}_k(A) \geq (1-\b^{-2})^2>0. 
\end{equation}
 Similarly, consider the events
\begin{equation}\label{asepm:eq:Ar}
A_r=\{\gamma_i>\gamma_k\text{ for }a\le i<k,\quad
        \gamma_i\ge\gamma_k\text{ for }k<i\le z\}
\end{equation}
and note that $A\subseteq A_r$. Set
\[
\mathcal F_{k,r}
=\sigma(\gamma_i-\gamma_{i-1},\kappa_i:a+1\le i\le z)
\]
and note that the events $A_r$ are measurable with respect to $\mathcal F_{k,r}$. Moreover, we find constants $c,C>0$ such that 
\begin{equation}\label{asepm:eq:survivalerror}
\mathbf{R}_k(A_r\setminus A)
\le Ce^{-cr}
\end{equation} for all $r\geq 1$. Building on the tools developed in Section~\ref{sec:macroscopicTN}, we will now control the effect of the Bryc polynomial.

Assume in the following that $q\in(0,1)$. For $h\ge0$, define
\[
\sigma_h^{[r]}=\#\{i\in\llbracket a,z \rrbracket:\gamma_i=\gamma_k+h\}.
\]
These are $\mathcal F_{k,r}$-measurable and coincide with the truncated counts in Lemma~\ref{lem: first part proof limit of low level composition}. For $\ell=\lfloor\varepsilon r\rfloor$
with $\varepsilon>0$ sufficiently small, we set
\begin{equation}\label{asepm:eq:H}
H_{k,r}=G_\ell(\sigma_0^{[r]},\ldots,\sigma_{\ell+1}^{[r]}), 
\end{equation}
recalling from Definitions~\ref{def:limit chain} and \ref{def:01 seq u} together with Lemma~\ref{lem:convergence to limiting chain} the quantity $G_\ell$ for tuples of nonnegative integers.  Applying now Lemma~\ref{lem:sum of K} to approximate $W$ by $K_{N,\ell}$, Lemma~\ref{lem:convergence to limiting chain} to approximate $K_{N,\ell}$ by $G_\ell(\sigma_0,\ldots,\sigma_{\ell+1})$, and Lemma~\ref{lem: first part proof limit of low level composition} to truncate those occupation counts to $\sigma_h^{[r]}$, together with the fact that $\sup_{\bg}(|H_{k,r}(\bg)|,|W(\bg)|)\leq \b^{2(\ell+1)}$, we see that
\begin{align}\label{asepm:eq:weightapprox}
\mathbf E_{N,k}^0[|W-H_{k,r}|]
\le  C_1e^{-c_1\ell}+\b^{2(\ell+1)}(\ell+1)
\frac{q^{N-\ell+1}}{1-q^{N-\ell+1}}+C_2\b^{2(\ell+1)}e^{C_2(\ell+1)-c_2r}
\end{align} for some constants $c_1,C_1,c_2,C_2>0$.
For the last term, the two evaluations of $G_\ell$ agree whenever all the counts through level $\ell+1$ agree, and we use that their difference is bounded by $b^{2(\ell+1)}$, otherwise.
Choosing now $\varepsilon>0$ small enough, all terms in \eqref{asepm:eq:weightapprox} decay exponentially in $r$, and therefore
\begin{equation}\label{asepm:eq:weightL1}
\mathbf E_{N,k}^0[|W-H_{k,r}|]\le Ce^{-cr}
\end{equation} for all $r$ large enough with constants $c,C>0$.

Observe that from \eqref{asepm:eq:RN} and \eqref{asepm:eq:conditionA}, we get that 
\begin{equation}\label{asepm:eq:RP}
\frac{\mathrm d \mathbf P_{N,k}^q}{\mathrm d \mathbf R_k}
=\frac{W\mathbf1_A}{ \mathbf R_k(A)Z_k} 
\end{equation}
and we can define the measure $\mathbf Q_{k,r}$ by the relation
\begin{equation}\label{asepm:eq:Qdef}
 \frac{\mathrm d \mathbf Q_{k,r}}{\mathrm d \mathbf{R}_k}
=\frac{H_{k,r}\mathbf1_{A_r}}{\mathbf E_{\mathbf R_k}[H_{k,r}\mathbf1_{A_r}]}.
\end{equation}
Note that the density in \eqref{asepm:eq:Qdef} is $\mathcal F_{k,r}$-measurable, and has the correct marginals outside of $\llbracket a+1,z\rrbracket$. 
Comparing the unnormalized densities, we obtain
\begin{align}\label{asepm:eq:densityerror}
\mathbf E_{\mathbf R_k}[|W\mathbf1_A-H_{k,r}\mathbf1_{A_r}|]
= \mathbf R_k(A)\mathbf E_{N,k}^0[|W-H_{k,r}|]
 +\mathbf E_{\mathbf R_k}[H_{k,r}\mathbf1_{A_r\setminus A}] \leq Ce^{-cr},
\end{align}
where the last step uses \eqref{asepm:eq:weightL1}. Now using the fact that for nonnegative densities $f,g$, we have
\begin{equation}\label{asepm:eq:normalization}
\left\|\frac{f\,\mathrm d \mathbf R_k}{\mathbf E_{\mathbf R_k}[f]}-\frac{g\,\mathrm d \mathbf R_k}{\mathbf E_{\mathbf R_k}[g]}\right\|_{\mathrm{TV}}
\le\frac{\mathbf E_{\mathbf R_k}[|f-g|]}{\mathbf E_{\mathbf R_k}[f]}, 
\end{equation}
combining \eqref{asepm:eq:conditionA}, \eqref{asepm:eq:densityerror} and \eqref{asepm:eq:normalization} shows \eqref{asepm:eq:localTV} for all $r$ large enough as desired.
\end{proof}

As a direct consequence, we see that the locations $T_N$ and $\mathsf M_N$ must be close with high probability.

\begin{corollary}
\label{cor:particle-minimum-localization}
Assume $q\in[0,1)$ and $\b>1$.
There exist constants $C,c>0$, depending only on $q,\b$,
such that, for every $N\ge m\ge1$ and $k\in \llbracket 0,N \rrbracket$,
\[
\mathbf{P}_{N,k}^q(|\mathsf M_N-k|\ge m)\le Ce^{-cm}.
\]
\end{corollary}

\begin{proof}
In the following, set
\[
r=\lfloor m/2\rfloor,\qquad d=m-r
\] and recall $a=\max(0,k-r), z=\min(N,k+r)$.
By Lemma~\ref{asepm:lem:local}, there exists a probability
measure $\mathbf Q_{k,r}$ such that
\begin{equation}\label{eq:localization-law-comparison}
\|\mathbf{P}_{N,k}^q-\mathbf Q_{k,r}\|_{\mathrm{TV}}
\le C_0e^{-c_0r}, 
\end{equation}
where under $\mathbf Q_{k,r}$, the particles on $\llbracket 1,a \rrbracket$ are independent
Bernoulli variables with parameter $1/(1+\b)$, and 
on $\llbracket z+1,N \rrbracket$ independent Bernoulli variables with
parameter $\b/(1+\b)$.
Write
\begin{equation}\label{eq:PathUsed}
    s_j=\sum_{i=1}^j(2\tau_i-1),\qquad s_0=0,
\end{equation}
and note that we have
\[
\{\mathsf M_N\le k-m\}
\subseteq
\bigcup_{\ell=d}^{a}
\{s_a-s_{a-\ell}\ge0\}.
\]
Now using that $(s_j)_{j \in \llbracket 1,N\rrbracket}$ performs outside of the interval $\llbracket a+1,z \rrbracket$ biased random walks, we find a constant $c \in (0,1)$ so that
\begin{equation}\label{eq:localization-leftt-tail}
\mathbf Q_{k,r}(s_a-s_{a-\ell}\ge0)
\le
c^\ell
\end{equation} for all $\ell \geq d$,  and hence
\begin{equation}\label{eq:localization-left-tail}
\mathbf Q_{k,r}(M_N\le k-m)
\le
\tilde{c}^\ell 
\end{equation} for some constant $\tilde{c}>0$. 
A similar argument holds for $k+m>N$, and we get that
\begin{equation}\label{eq:localization-right-tail}
\mathbf Q_{k,r}(M_N\ge k+m)
\le
\tilde{c}^\ell 
\end{equation} for all $\ell \geq d$.
Thus, combining \eqref{eq:localization-law-comparison},
\eqref{eq:localization-left-tail}, and
\eqref{eq:localization-right-tail}, we conclude.
\end{proof}

Next, we argue that the measure $\mathbf P_{N,k}^q$ together with the location of $\mathsf M_N$ can be jointly approximated by a Bernoulli product measure and the marginal law of $\mathsf M_N$. To this end, we set
\begin{equation}\label{eq:Collection2}
    D_k:= \mathsf M_N-k, \qquad \mathbf K_k := \mathbf{P}_{N,k}^q( D_k \in \cdot )
\end{equation} as the variable $\mathsf M_N$ shifted by $k$, and its law under conditioning on $\{ T_N=k \}$.

\begin{proposition}\label{asepm:prop:joint}
There exist constants $c,C>0$ such that uniformly in $N\ge m\ge1$ and $k\in \llbracket 0,N \rrbracket$,
\begin{equation}\label{asepm:eq:joint}
\left\|\mathbf{P}_{N,k}^q\big( (X_{k,m},D_k) \in \cdot \big)-B_{k,m}\otimes \mathbf K_k\right\|_{\mathrm{TV}}\le Ce^{-cm}.
\end{equation}
\end{proposition}
\begin{proof}
Set $r=m/2$ and consider the interval $ \llbracket a,z \rrbracket$ from Lemma~\ref{asepm:lem:local}. Recall \eqref{eq:PathUsed}, and let $D_{k,r}$ be the location relative to $k$ of the first minimum of
\[
(s_i(\tau)-s_k(\tau))_{i\in \llbracket a,z \rrbracket}.
\]
Note that  $D_{k,r}$ is $\mathcal F_{k,r}$-measurable. If $|M_N-k|\le r$, the global first minimum lies in this window and is also its first minimum. Thus
\[
\{D_{k,r}\ne D_k\}\subseteq\{|D_k|>r\}, 
\]
and so by Corollary~\ref{cor:particle-minimum-localization},
\begin{equation}\label{asepm:eq:Dlocal}
\mathbf{P}_{N,k}^q(D_{k,r}\ne D_k)\le Ce^{-cr}.
\end{equation}
Since $r<m$, all coordinates of $X_{k,m}$ lie outside the window $\llbracket a+1,z\rrbracket$. Now for the measure $\mathbf Q_{k,r}$ the independence assertion in its construction in Lemma~\ref{asepm:lem:local} gives
\begin{equation}\label{asepm:eq:Qproduct}
\mathbf Q_{k,r}\big((X_{k,m},D_{k,r}) \in \cdot \big)
=B_{k,m}\otimes  \mathbf Q_{k,r}(D_{k,r} \in \cdot ).
\end{equation}
By \eqref{asepm:eq:localTV}, \eqref{asepm:eq:Dlocal}, and the contraction property of the total variation distance under measurable maps,  we see that
\begin{equation}\label{asepm:eq:jointintermediate}
\left\|\mathbf{P}_{N,k}^q\big((X_{k,m},D_k) \in \cdot \big)
-B_{k,m}\otimes\mathbf Q_{k,r}(D_{k,r} \in \cdot )\right\|_{\mathrm{TV}}\le Ce^{-cr}.
\end{equation}
Taking marginals gives us that
\[
\left\|\mathbf K_k-\mathbf Q_{k,r}(D_{k,r} \in \cdot )\right\|_{\mathrm{TV}}\le Ce^{-cr}, 
\]
and so by the triangle inequality for the total variation distance, we obtain \eqref{asepm:eq:joint}.
\end{proof}

Next, we record upper and lower bounds on the probability of the location $T_N$ to attain a given value $k$, as well as a lower bound on the probability for $\mathsf M_N$ to attain a given value $j$, which we will use later on. 

\begin{lemma}\label{asepm:lem:pointwise} There exist constants $c,C>0$ such that uniformly in $N\ge1$ and $k,j\in \llbracket 0,N \rrbracket$,
\begin{equation}\label{asepm:eq:pointwise}
\frac c{N+1}\le\mathbb{P}^{(q,\b,\b)}_{N,\uptau}(T_N=k)\le\frac C{N+1},\qquad
\mathbb{P}^{(q,\b,\b)}_{N,\uptau}(\mathsf M_N=j)\ge\frac c{N+1}.
\end{equation}
\end{lemma}
\begin{proof}
Let $p_n$ be the probability that a $\nu_b$-walk stays nonnegative for $n$ steps, and note that 
\begin{equation}
    0 < \inf_{n \in \mathbb N} p_n \leq \sup_{n \in \mathbb N} p_n <1 . 
\end{equation}
Hence, we see from \eqref{eq:jdtyjykmt} that
\begin{equation}\label{asepm:eq:Tzero}
\frac{c}{N}\le\mathbb{P}^{(0,\b,\b)}_{N,\uptau}(T_N=k)
\le\frac{C}{N}.
\end{equation} for some constants $c,C>0$.
Recalling $Z_k$ from \eqref{asepm:eq:RN}, we see by a change of measure for $q \in (0,1)$ that
\[
\mathbb{P}^{(q,\b,\b)}_{N,\uptau}(T_N=k)
=\frac{\mathbb{P}^{(0,\b,\b)}_{N,\uptau}(T_N=k)Z_k}
{\sum_{h=0}^N\mathbb{P}^{(0,\b,\b)}_{N,\uptau}(T_N=h)Z_h}.
\]
Since by Proposition~\ref{prop:moment bound coex}, we have $Z_k \in [1,C_W]$ for some constant $C_W>0$ uniformly in $k \in \llbracket 1,N \rrbracket$, we obtain the first pair of bounds in \eqref{asepm:eq:pointwise}.

Next, we give a lower bound on $\mathbb{P}^{(q,\b,\b)}_{N,\uptau}(M_N=k)$. Recall the measure $\mathbf{R}_k$ from \eqref{asepm:eq:reference}, and let $E$ denote the event that a bi-infinite path $\bg$ stays non-negative. Note that by the transience of a biased random walk, we find a constant $c>0$, depending only on $q$ and $\b$, such that uniformly in $k\in \mathbb Z$, we have that
\begin{equation}
    \mathbf{R}_k(E)  \geq c . 
\end{equation}
From the definition of the measure $\mathbf{P}_{N,k}^0(M_N=k)$ together with  \eqref{asepm:eq:conditionA}, we see that
\begin{equation}\label{asepm:eq:diagonalzero}
\mathbf{P}_{N,k}^0(M_N=k)\ge\frac{\mathbf R_k(E)}{\mathbf R_k(A)} \geq c_1 >0
\end{equation} for some constant $c_1>0$.
Using \eqref{asepm:eq:RN} and that $Z_k \in [1,C_W]$ for some constant $C_W>0$ uniformly in $k \in \llbracket 1,N \rrbracket$, we get that
\begin{equation}\label{asepm:eq:diagonalq}
\mathbf{P}_{N,k}^q(M_N=k)
=\frac{\mathbf{E}_{N,k}^0[W\mathbf1_{\{M_N=k\}}]}{Z_k} >0 . 
\end{equation} 
Combining this with the lower bound for $\mathbb{P}^{(q,\b,\b)}_{N,\uptau}(T_N=k)$ yields
\[
\mathbb{P}^{(q,\b,\b)}_{N,\uptau}(M_N=k)
\ge\mathbb{P}^{(q,\b,\b)}_{N,\uptau}(T_N=k)\mathbf{P}_{N,k}^q(M_N=k)
\ge\frac c{N+1}
\]
for all $N$ large enough and some constant $c>0$, allowing us to conclude.
\end{proof}

\subsection{Upper bound on the total variation distance between $\mathsf M_N$ and the uniform distribution}\label{sec:Approx3}

Next, we establish an upper bound on the total variation distance between the law of the location $\mathsf M_N$ under the measure $\mathbb{P}^{(q,\b,\b)}_{N,\uptau}$ and the uniform distribution on $\llbracket 0,N \rrbracket$. 

\begin{proposition}\label{prop:part of the thm upper bound}
    Assume $q\in[0,1)$ and $\b>1$.   Then there exists a constant $C>0$ depending only on $q$ and $\b$, such that for each $N\in\mathbb N$,
    \[
    \TV{
        \mathbb{P}^{(q,\b,\b)}_{N,\uptau} (\mathsf M_N\in\cdot)
        -
        \mathrm{Unif}_{\llbracket0,N\rrbracket}
    }
    \leq \frac{C}{N}.
    \]
\end{proposition}

As a first step, we require the following technical lemma on the auxiliary chain $\mathfrak{X}$ from Definitions \ref{def:01 seq u} and \ref{def:limit chain}.

\begin{lemma}\label{lem:first technical finite-window limits in the bulk}
    Assume $q\in(0,1)$ and $\b>1$. As in Lemma \ref{lem:convergence to limiting chain}, for $\ell\in\mathbb N$, define the function
\[
G_\ell(\mathbf u)=\mathbb E\left[\b^{2\mathfrak{X}^{\boldsymbol{\mathsf S}^{\mathbf u}}(\ell+1)}\right]\quad
\mbox{ for each }\mathbf u=(u_0,\ldots,u_{\ell+1})\in\mathbb{Z}_{\geq0}^{\ell+2},
\]
where the $0$-$1$ sequence $\boldsymbol{\mathsf S}^{\mathbf u}$ and the Markov chain $\mathfrak{X}^{\boldsymbol{\mathsf S}^{\mathbf u}}$ are from Definitions \ref{def:01 seq u} and \ref{def:limit chain}.
Let \(\boldsymbol Y=(Y_i)_{i\in\mathbb Z}\) be the bi-infinite \(V\)-shaped \(\nu_\b\)-random walk from Definition \ref{def:V shaped RW}, and let
\[
    \vartheta_h=\#\{j\in\mathbb Z:Y_j=h\}
    \quad\mbox{ for } h\in\mathbb Z_{\ge0}.
\]
Recall from Lemma \ref{lem: first part proof limit of low level composition} that these variables are finite almost surely.
Suppose  $\Uppsi:\mathbb{Z}^{2m+1}\to[0,1]$ is a function. 
For \(\ell\ge m\geq1\), define
\[
    \Uptheta_{m,\ell} 
    =
    \mathbb E
    \left[\Uppsi\left((Y_i)_{|i|\leq m}\right)
        G_\ell(\vartheta_0,\ldots,\vartheta_{\ell+1})
    \right].
\]
Then as $\ell\to\infty$, the limit of the sequence $\Uptheta_{m,\ell }$ exists for the fixed function \(\Uppsi\), which we denote by $\Uptheta_{m }$. 
Moreover, there exist constants $C,c>0$ depending only on $q$ and $\b$, such that 
\[
    \left|\Uptheta_{m,\ell }-\Uptheta_{m }\right|
    \leq
    Ce^{-c\ell} 
    \quad\mbox{ for all } \ell\ge m\geq1. 
\]
\end{lemma}
\begin{proof}
    By Definition \ref{def:limit chain}, for every \(0\)-\(1\) sequence of length \(\ell+1\), $\mathfrak X^{\boldsymbol{\mathsf S}}(\ell+1)\in\llbracket0,\ell+1\rrbracket$. Hence 
    \begin{equation}\label{eq:first-technical-G-bound new}
    1\le G_\ell(\mathbf u)\le \b^{2(\ell+1)},
    \qquad
    \mathbf u\in\mathbb Z_{\ge0}^{\ell+2}.
\end{equation}
Therefore, \(\Uptheta_{m,\ell}\) is finite for every \(\ell\ge m\geq1\).  For \(n\ge m\), define
\[ 
    \Xi_n
    =
    \mathbb E^0_{2n+1,n+1}
    \left[
        \Uppsi\left((\gamma_{n+1+i}-\gamma_{n+1})_{|i|\le m}\right)
        W_{\bs(\bg)}(\b^2)
    \right],
\]
where \(\bs(\bg)\) is the composition associated to \(\bg\) by \eqref{eq:Signature}. We first prove the claim that, for every fixed \(\ell\ge m\),
\begin{equation}\label{eq:first-technical-limsup-Fn new}
    \limsup_{n\to\infty}
    \left| \Xi_n-\Uptheta_{m,\ell}\right|
    \le
    C_1 e^{-c_1\ell}
\end{equation}
for some constants \(C_1,c_1\in(0,\infty)\) depending only on \(q\) and \(\b\). Indeed, for all sufficiently large \(n\), we have
\(2n+1\ge2\ell+2\), \(\lfloor n/4\rfloor\ge m\), and \(n\ge2\lfloor n/4\rfloor+2\). For those \(n\), write
\[
    \mathsf K_{2n+1,\ell}(\bg)
    =
    \sum_{y\ge0}
        \b^{2y}
        K_{2n+1,\ell}
        \left(
            \boldsymbol{\mathsf S}^{\bs(\bg)};0,\{y\}
        \right),
\]
as in Lemma \ref{lem:sum of K}. Since \(0\le \Uppsi\le1\), Lemma \ref{lem:sum of K} gives \(C_2,c_2\in(0,\infty)\), depending only on \(q\) and \(\b\), such that
\begin{equation}\label{eq:first-technical-sumK-step new} 
 \left|
\Xi_n
-
\mathbb E^0_{2n+1,n+1}
\left[
    \Uppsi\left((\gamma_{n+1+i}-\gamma_{n+1})_{|i|\le m}\right)
    \mathsf K_{2n+1,\ell}(\bg)
\right]
\right|                                                         
\le
\mathbb E^0_{2n+1,n+1}
\left[
    \left|W_{\bs(\bg)}(\b^2)-\mathsf K_{2n+1,\ell}(\bg)\right|
\right]
\le
C_2e^{-c_2\ell}. 
\end{equation}
Let \(\bs(\bg)=(\sigma_0,\ldots,\sigma_r)\), and extend it by \(\sigma_h=0\) for \(h>r\). By Lemma \ref{lem:convergence to limiting chain},
\[
\sup_{\bg\in\lrw_{2n+1}}
\left|
    \mathsf K_{2n+1,\ell}(\bg)
    -
    G_\ell(\sigma_0,\ldots,\sigma_{\ell+1})
\right|
\le
\b^{2(\ell+1)}(\ell+1)
\frac{q^{2n -\ell+2}}{1-q^{2n -\ell+2}}.
\]
Hence
\begin{equation}\label{eq:first-technical-K-to-G-step new}
\begin{aligned}
&\Big|
\mathbb E^0_{2n+1,n+1}
\left[
    \Uppsi\left((\gamma_{n+1+i}-\gamma_{n+1})_{|i|\le m}\right)
    \mathsf K_{2n+1,\ell}(\bg)
\right] 
-
\mathbb E^0_{2n+1,n+1}
\left[
    \Uppsi\left((\gamma_{n+1+i}-\gamma_{n+1})_{|i|\le m}\right)
    G_\ell(\sigma_0,\ldots,\sigma_{\ell+1})
\right]
\Big|          \\
&\le
\b^{2(\ell+1)}(\ell+1)
\frac{q^{2n+1-\ell+1}}{1-q^{2n+1-\ell+1}}
\to 0
\qquad\mbox{as }n\to\infty.
\end{aligned}
\end{equation} 
Finally,  we apply Lemma \ref{lem:limit of low level composition} with \(N=2n+1\), \(k=n+1\), \(d=\ell+1\), \(t=\lfloor n/4\rfloor\), \(G=G_\ell\), and \(\Uppsi=\Uppsi_n'\), where $\Uppsi_n' ((x_i)_{|i|\le \lfloor n/4\rfloor} )=
    \Uppsi ((x_i)_{|i|\le m} ) $ for each  $(x_i)_{|i|\le \lfloor n/4\rfloor}\in\mathbb Z^{2\lfloor n/4\rfloor+1}$.
Combining with \eqref{eq:first-technical-G-bound new}, there exist constants \(C_3,c_3\in(0,\infty)\), depending only on \(q\) and \(\b\), such that
\begin{equation}\label{eq:first-technical-low-level-step new}
\begin{aligned}
&\left|
\mathbb E^0_{2n+1,n+1}
\left[
    \Uppsi\left((\gamma_{n+1+i}-\gamma_{n+1})_{|i|\le m}\right)
    G_\ell(\sigma_0,\ldots,\sigma_{\ell+1})
\right]
-
\Uptheta_{m,\ell}
\right|                                                        \\
& 
\le
C_3\b^{2(\ell+1)}
\left(
    e^{C_3(\ell+1)}e^{-c_3\lfloor n/4\rfloor}
    +
    e^{-c_3(n-\lfloor n/4\rfloor)}
\right)
\to 0
\qquad\mbox{as }n\to\infty.
\end{aligned}
\end{equation} 
Combining \eqref{eq:first-technical-sumK-step new},
\eqref{eq:first-technical-K-to-G-step new}, and
\eqref{eq:first-technical-low-level-step new}, we obtain \eqref{eq:first-technical-limsup-Fn new}. 

We return to the proof of the lemma. Let \(\ell,\ell'\ge m\). Then, by \eqref{eq:first-technical-limsup-Fn new},
\[
\begin{aligned}
\left|\Uptheta_{m,\ell}-\Uptheta_{m,\ell'}\right|
&\le
\limsup_{n\to\infty}
\left(
    \left|\Uptheta_{m,\ell}-\Xi_n\right|
    +
    \left|\Xi_n-\Uptheta_{m,\ell'}\right|
\right)         \le
C_1e^{-c_1\ell}+C_1e^{-c_1\ell'}.
\end{aligned}
\]
Therefore \((\Uptheta_{m,\ell})_{\ell\ge m}\) is a Cauchy sequence in \(\mathbb R\), and the limit $\Uptheta_m=\lim_{\ell\to\infty}\Uptheta_{m,\ell}$
exists. Letting \(\ell'\to\infty\)  gives
\[
    \left|\Uptheta_{m,\ell}-\Uptheta_m\right|
    \le
    C_1e^{-c_1\ell} 
    \quad\mbox{ for } \ell\ge m,
\] which concludes the proof. 
\end{proof}

Next, we provide an upper bound on the total variation distance  between the open ASEP stationary measure and a suitable sequence of probability measures $\mathfrak p_m$ in a radius of $m$ around the location $T_N$.

\begin{lemma}\label{lem:finite-window limits in the bulk}
    Assume $q\in[0,1)$ and $\b>1$.  Then there exist a sequence of probability measures $\mathfrak p_m$ on $\{-1,1\}^{\llbracket -m+1,m\rrbracket}$ for $m\in\mathbb N$ and constants $\varepsilon\in(0,1/4)$ and $C,c>0$ depending only on $q$ and $\b$, such that for each $N\in\mathbb N$, $k\in\llbracket0,N\rrbracket$, and $m\in\llbracket1,\lfloor\varepsilon\min(k-1,N-k)\rfloor\rrbracket$,
    \begin{equation}\label{eq:goal finite-window limits in the bulk}
    \TV{ \mathbf{P}_{N,k}^q\left((2\tau_{k+i}-1)_{i\in\llbracket -m+1,m\rrbracket}\in\,\cdot\,\right) -\mathfrak p_m  }\leq 
    C\left(e^{-cm}+e^{-c\min(k-1,N-k)}\right).
    \end{equation}
\end{lemma} 
\begin{proof}
Throughout this proof, we denote $D=\min(k-1,N-k)$ for simplicity.
Recall from Proposition~\ref{prop:first part proof Thm} that the conditional measures \(\mathbf P_{N,k}^q\) are well-defined. Again, we write
 $\mathbb P_{N,k}^q(\,\cdot\,)=\mathbb P^{(q,\b,\b)}(\,\cdot\,\mid T_N=k)$
for the corresponding conditional law of \(\bg\in\lrw_N\), and  \(\mathbb E_{N,k}^q\) for expectation with respect to this law.

For an event $A\subseteq\{-1,1\}^{\llbracket -m+1,m\rrbracket}$, we define $\Phi_A^m:\mathbb Z^{2m+1}\to[0,1]$ as follows.
For each $(y_i)_{|i|\leq m}$, if 
$y_i-y_{i-1}\notin\{-1,0,1\}$ for some $i\in\llbracket -m+1,m\rrbracket$, we let $\Phi_A^m((y_i)_{|i|\leq m})=0$. Otherwise, let 
\[
\Phi_A^m\left((y_i)_{|i|\leq m}\right)=\mathbb{P}\left((y_i-y_{i-1}+\upsigma_i{\bf1}_{y_i-y_{i-1}=0} )_{i\in\llbracket -m+1,m\rrbracket}\in A\right)
\]
where $\upsigma_i$, $i\in\mathbb Z$  are i.i.d. symmetric \(\{-1,1\}\)-valued variables.
Since \(T_N\) is a function of \(\bg\), under \(\mathbf P_{N,k}^q\) the variables \((\kappa_i)_{i=1}^N\) remain i.i.d. \(\ber(1/2)\) and independent of \(\bg\). Therefore, for $ m\in\llbracket1,D\rrbracket$, 
\begin{equation*}
    \mathbb E_{N,k}^q\!\left[\Phi_A^m\left((\gamma_{k+i}-\gamma_k)_{|i|\leq m}\right)\right]
    =\mathbf{P}_{N,k}^q
    \left(
        (2\tau_{k+i}-1)_{i\in \llbracket -m+1,m\rrbracket}\in A
    \right).
\end{equation*} 
Let \(\boldsymbol Y=(Y_i)_{i\in\mathbb Z}\) be the bi-infinite \(V\)-shaped \(\nu_\b\)-random walk from Definition \ref{def:V shaped RW}, and let
\[
    \vartheta_h=\#\{j\in\mathbb Z:Y_j=h\}
    \quad\mbox{ for } h\in\mathbb Z_{\ge0}.
\]
For \(t\ge m\), define $\Phi_{A,t}^m ((x_i)_{|i|\le t} )
    =
    \Phi_A^m ((x_i)_{|i|\le m} )$ for each  $(x_i)_{|i|\le t}\in\mathbb Z^{2t+1}$.
By Lemma \ref{lem:limit of low level composition} with $\Uppsi=\Phi_{A,t}^m$, there exist constants \(C_1,c_1\in(0,\infty)\), depending only on \(\b\), such that
for every \(d\in\mathbb Z_{\ge0}\) and $m, t\in\mathbb N$ satisfying \(m\leq t\) and \(D=\min(k-1,N-k)\geq2t+2\), and every bounded function \(G:\mathbb Z_{\ge0}^{d+1}\to\mathbb R\), 
\begin{equation}\label{eq:ntyrjhrfnbsg}
    \begin{split}
        &\left|
    \mathbb E_{N,k}^0
    \left[
        \Phi_A^m\left((\gamma_{k+i}-\gamma_k)_{|i|\leq m}\right)G(\sigma_0,\ldots,\sigma_d)
    \right]
    -
    \mathbb E
    \left[
        \Phi_A^m\left((Y_i)_{|i|\leq m}\right)G(\vartheta_0,\ldots,\vartheta_d)
    \right]
\right|      \\
&\leq C_1\|G\|_\infty
\left(
    e^{C_1d}e^{-c_1t}
    +
    e^{-c_1(D-t)}
\right),
    \end{split}
\end{equation} 
where we recall that \(\bs(\bg)=(\sigma_0,\ldots,\sigma_r)\) is extended by \(\sigma_h=0\) for \(h>r\).

We now construct the probability measures \(\mathfrak p_m\). First suppose \(q=0\).  For \(m\in\mathbb N\), define
\[
    \mathfrak p_m(A)
    =
    \mathbb E\left[\Phi_A^m\left((Y_i)_{|i|\le m}\right)\right]
    \quad\mbox{ for }
    A\subseteq\{-1,1\}^{\llbracket -m+1,m\rrbracket}.
\] 
Since, for each sample of \(\boldsymbol Y\),  
\(A\mapsto \Phi_A^m((Y_i)_{|i|\le m})\) is a probability measure on
\(\{-1,1\}^{\llbracket -m+1,m\rrbracket}\), \(\mathfrak p_m\) is a probability measure.
Choose $\varepsilon=\frac{1}{8}$. If $1\leq m\leq \varepsilon D$, set \(t=\lfloor D/4\rfloor\). Then \(D\ge8\), and hence \(m\le t\), \(D\ge2t+2\), \(t\ge D/8\), and \(D-t\ge3D/4\). Taking \(d=0\) and \(G\equiv1\) in \eqref{eq:ntyrjhrfnbsg} gives, for every  $A\subseteq\{-1,1\}^{\llbracket -m+1,m\rrbracket}$, 
\[
\left|
    \mathbf P_{N,k}^0
    \left(
        (2\tau_{k+i}-1)_{i\in \llbracket -m+1,m\rrbracket}\in A
    \right)
    -
    \mathfrak p_m(A)
\right|
\leq
C_1\left(e^{-c_1t}+e^{-c_1(D-t)}\right)
\leq
C_2e^{-c_2D} ,
\]
where $C_2,c_2>0$ depend only on $q$ and $\b$.
Taking the supremum over \(A\) proves the estimate \eqref{eq:goal finite-window limits in the bulk} for \(q=0\).

In the rest of the proof, we assume $q\in(0,1)$. As in Lemma \ref{lem:convergence to limiting chain}, define the function
\[
G_\ell(\mathbf u)=\mathbb E\left[\b^{2\mathfrak{X}^{\boldsymbol{\mathsf S}^{\mathbf u}}(\ell+1)}\right]\quad
\mbox{ for each }\mathbf u=(u_0,\ldots,u_{\ell+1})\in\mathbb{Z}_{\geq0}^{\ell+2},
\]
where the $0$-$1$ sequence $\boldsymbol{\mathsf S}^{\mathbf u}$ and the Markov chain $\mathfrak{X}^{\boldsymbol{\mathsf S}^{\mathbf u}}$ are from Definitions \ref{def:01 seq u} and \ref{def:limit chain}. Then 
\begin{equation}
    1\leq G_\ell(\mathbf u)\leq \b^{2(\ell+1)}\quad
\mbox{ for each }\mathbf u \in\mathbb{Z}_{\geq0}^{\ell+2}.
\end{equation}
For \(m\in\mathbb N\), \(A\subseteq\{-1,1\}^{\llbracket -m+1,m\rrbracket}\), and \(\ell\ge m\), define
\[
    \Theta_{m,\ell}(A)
    =
    \mathbb E
    \left[
        \Phi_A^m\left((Y_i)_{|i|\leq m}\right)
        G_\ell(\vartheta_0,\ldots,\vartheta_{\ell+1})
    \right].
\]
By Lemma \ref{lem:first technical finite-window limits in the bulk} with $\Uppsi=\Phi_A^m$, the limit
\[
    \Theta_{m}(A)=\lim_{\ell\to\infty}\Theta_{m,\ell}(A)
\]
exists, and  there exist constants $C_3,c_3>0$ depending only on $q$ and $\b$, such that for all $A$
\begin{equation}\label{eq:claim Cauchy}
    \left|\Theta_{m,\ell}(A)-\Theta_{m}(A)\right|
    \leq
    C_3e^{-c_3\ell} 
    \qquad \mbox{ for}\quad\ell\ge m. 
\end{equation}

Let $E=\{-1,1\}^{\llbracket-m+1,m\rrbracket}$ be the full event.  
We next claim that
\begin{equation}\label{eq:full event}
\Theta_{m}\left(E\right)= \Lambda_{q,\b}\in[1,\infty),
\end{equation}
where the right side is the constant from Proposition \ref{prop:constant Lambda q b}. Indeed,
\(\Phi_E^m((x_i)_{|i|\le m})=1\) if
\(x_i-x_{i-1}\in\{-1,0,1\}\) for all
\(i\in\llbracket-m+1,m\rrbracket\).   Hence
$\Theta_{m,\ell}\left(E\right)=\mathbb E\left[ G_\ell(\vartheta_0,\ldots,\vartheta_{\ell+1}) \right]$.
The estimate \eqref{eq:first-technical-limsup-Fn new} in the proof of Lemma~\ref{lem:first technical finite-window limits in the bulk}, applied with \(\Uppsi=\Phi_{E}^m\), gives constants \(C_4,c_4>0\), depending only on \(q,\b\), such that, for every \(\ell\ge m\),
\[
    \limsup_{n\to\infty}
    \left|\mathbb E_{2n+1,n+1}^0\left[W_{\bs(\bg)}\left(\b^2\right)\right]-\Theta_{m,\ell}(E)\right|
    \le C_4e^{-c_4\ell}.
\]
By Proposition \ref{prop:constant Lambda q b}, \(\mathbb E_{2n+1,n+1}^0[W_{\bs(\bg)}(\b^2)]\to\Lambda_{q,\b}\). Combining this with \eqref{eq:claim Cauchy}, we get, for every \(\ell\ge m\),
\[
\left|\Theta_m(E)-\Lambda_{q,\b}\right|
\le
C_3e^{-c_3\ell}+C_4e^{-c_4\ell}.
\]
Letting \(\ell\to\infty\) proves \eqref{eq:full event}. We now define 
\[
    \mathfrak p_m(A)
    = \Theta_{m}(A) \cdot\Lambda_{q,\b}^{-1} 
    \qquad\mbox{ for }  A\subseteq\{-1,1\}^{\llbracket-m+1,m\rrbracket}.
\]
For each \(\ell\), the map \(A\mapsto\Theta_{m,\ell}(A)\) is a finite measure on
\(\{-1,1\}^{\llbracket-m+1,m\rrbracket}\). Hence
\(A\mapsto\Theta_m(A)\) is also a finite measure, and \eqref{eq:full event}
shows that its total mass is \(\Lambda_{q,\b}\). Hence \(\mathfrak p_m\) is a probability measure on \(\{-1,1\}^{\llbracket-m+1,m\rrbracket}\).

It remains to prove the estimate \eqref{eq:goal finite-window limits in the bulk}. Choose  \(\varepsilon\in(0,1/8]\), depending only
on \(q\) and \(\b\), such that 
\begin{equation}\label{eq:ujmnteuy}
    2\varepsilon(2\log\b+C_1)<c_1/8,
    \qquad
    4\varepsilon\log\b<c_1/2,
    \qquad
    \left( 2/\varepsilon-1\right)\log(1/q)>4\log\b . 
\end{equation}
In the rest of the proof, we fix \(N\in\mathbb N\), \(k\in\llbracket0,N\rrbracket\), and
\(m\in\llbracket1,\lfloor\varepsilon D\rfloor\rrbracket\), where we recall \(D=\min(k-1,N-k)\). Set \(t=\lfloor D/4\rfloor\). 
We have \(D\ge 8\), and hence \(m\le t\), \(D\ge2t+2\), and \(N\ge2D+1\ge2m+2\).  

We claim that there exist constants \(C_5,c_5>0\), depending only on \(q,\b\), such that
\begin{equation}\label{eq:rwy.i} 
         \b^{2(m+1)}(m+1)
\frac{q^{N-m+1}}{1-q^{N-m+1}}
 \leq C_5e^{-c_5m} ,  \qquad
 \b^{2(m+1)}
\left(
    e^{C_1(m+1)}e^{-c_1t}
    +
    e^{-c_1(D-t)}
\right)
 \leq
C_5e^{-c_5D}. 
\end{equation} 
Indeed, the first estimate follows from $N-m+1\ge 2D+2-m\ge (2/\varepsilon-1)m+2 $, \(1-q^{N-m+1}\ge1-q\), \((2/\varepsilon-1)\log(1/q)>4\log\b\), and absorbing the factor \(m+1\) into the exponential; and the second estimate follows from the fact that the two summands are bounded, up to constants depending only on \(q,\b\), respectively by $\exp( (2\varepsilon(2\log\b+C_1)-c_1/4 )D )$ and $\exp ( (4\varepsilon\log\b-3c_1/4 )D )$, which decay exponentially in \(D\) by \eqref{eq:ujmnteuy}.

For \(A\subseteq\{-1,1\}^{\llbracket-m+1,m\rrbracket}\), define
\[
    \Upgamma_{N,k}(A)
    =
    \mathbb E_{N,k}^0
    \left[
        \Phi_A^m\left((\gamma_{k+i}-\gamma_k)_{|i|\leq m}\right)
        W_{\bs(\bg)}(\b^2)
    \right].
\]
Using Lemma \ref{lem:sum of K} with \(\ell=m\), Lemma
\ref{lem:convergence to limiting chain} with \(\ell=m\), and
\eqref{eq:ntyrjhrfnbsg} with \(d=m+1\), \(G=G_m\), and $t=\lfloor D/4\rfloor$,
together with \(1\le G_m\le \b^{2(m+1)}\) and \eqref{eq:rwy.i}, we obtain
\begin{equation}\label{eq:tagged N1}
    \begin{split}
        &\left|
    \Upgamma_{N,k}(A)-\Theta_{m,m}(A)
\right|\leq \mathbb E_{N,k}^0
\left[
    \left|
        W_{\bs(\bg)}(\b^2)-\mathsf K_{N,m}(\bg)
    \right|
\right] +\sup_{\bg\in\lrw_N}
\left|
    \mathsf K_{N,m}(\bg)
    -
    G_m(\sigma_0,\ldots,\sigma_{m+1})
\right| \\
&\qquad+\left|
\mathbb E_{N,k}^0
\left[
    \Phi_A^m\left((\gamma_{k+i}-\gamma_k)_{|i|\leq m}\right)
    G_m(\sigma_0,\ldots,\sigma_{m+1})
\right]
-
\Theta_{m,m}(A)
\right| \leq
C_6\left(e^{-c_6m}+e^{-c_6D}\right),
    \end{split}
\end{equation}
where $C_6,c_6>0$ are constants depending only on $q$ and $\b$.
Combining with \eqref{eq:claim Cauchy},
\begin{equation}\label{eq:tagged N2}
    \left|\Upgamma_{N,k}(A)-\Theta_m(A)\right|
    \leq
    C_7\left(e^{-c_7m}+e^{-c_7D}\right),
\end{equation}
for $C_7,c_7>0$ depending only on $q$ and $\b$.
Write
 $Z_{N,k}=\mathbb E_{N,k}^0[W_{\bs(\bg)}(\b^2)]$.
Theorem \ref{thm:prob interpretation of Bryc polynomials} gives \(Z_{N,k}\ge1\),  and
\[
    \mathbb P_{N,k}^q(B)
    =
   \mathbb E_{N,k}^0\left[{\bf1}_B W_{\bs(\bg)}(\b^2)\right]\cdot Z_{N,k}^{-1}
    \qquad\mbox{for every event }B\subseteq\lrw_N ,
\]
in view of \eqref{eq:PNbb}.
Therefore, using the definition of \(\Phi_A^m\), 
\[
\begin{aligned}
    \mathbf P_{N,k}^q
    \left(
        (2\tau_{k+i}-1)_{i\in \llbracket-m+1,m\rrbracket}\in A
    \right)
    &=
    \mathbb E_{N,k}^q
    \left[
        \Phi_A^m\left((\gamma_{k+i}-\gamma_k)_{|i|\le m}\right)
    \right]   =
    \Upgamma_{N,k}(A) \cdot Z_{N,k}^{-1}  .
\end{aligned}
\]
Since \(A\mapsto\Theta_m(A)\) is a finite measure with total mass \(\Lambda_{q,\b}\), we have
\(0\le \Theta_m(A)\le \Lambda_{q,\b}\). Hence
\begin{equation}
    \begin{split}
        &\left|
    \frac{\Upgamma_{N,k}(A)}{Z_{N,k}}
    -
    \frac{\Theta_m(A)}{\Lambda_{q,\b}}
\right|\le
\frac{|\Upgamma_{N,k}(A)-\Theta_m(A)|}{Z_{N,k}}
+
\Theta_m(A)
\left|
    \frac1{Z_{N,k}}-\frac1{\Lambda_{q,\b}}
\right| \\
&  \le
|\Upgamma_{N,k}(A)-\Theta_m(A)|
+
|Z_{N,k}-\Lambda_{q,\b}|.
    \end{split}
\end{equation}  
Since \(Z_{N,k}=\Upgamma_{N,k}(E)\) and
\(\Theta_m(E)=\Lambda_{q,\b}\), applying \eqref{eq:tagged N2} first to \(A\) and then to \(E\) gives 
\[
\left|
    \mathbf P_{N,k}^q
    \left(
        (2\tau_{k+i}-1)_{i\in \llbracket-m+1,m\rrbracket}\in A
    \right)
    -
    \mathfrak p_m(A)
\right|
\le
C_8\left(e^{-c_8m}+e^{-c_8D}\right),
\]
for constants $C_8,c_8>0$ depending only on $q$ and $\b$.
Taking the supremum over all events \(A\subseteq\{-1,1\}^{\llbracket -m+1,m\rrbracket}\) gives the desired total variation bound \eqref{eq:goal finite-window limits in the bulk} for \(q\in(0,1)\). 
\end{proof}

 Next, we construct a good approximation $\mathscr p$ for the law of $\mathsf M_N$ with respect to $\mathbf{P}_{N,k}^q$.

\begin{lemma}\label{lem:main claim in proof upper}
    Assume $q\in[0,1)$ and $\b>1$. For $N\in\mathbb N$ and $k\in\llbracket0,N\rrbracket$,, recall that $\mathbf{P}_{N,k}^q$ is the measure $\mathbb{P}^{(q,\b,\b)}_{N,\uptau}$ conditioned on $T_N=k$. Then there exists a unique probability measure $\mathscr p$ on $\mathbb Z$ and  $C,c>0$ depending only on $q$ and $\b$, such that for each $N\in\mathbb N$, $k\in\llbracket0,N\rrbracket$,
    \begin{equation}\label{eq:goal main claim in proof upper}
     \TV{ \mathbf{P}_{N,k}^q\left(\mathsf M_N-k\in\,\cdot\,\right) -\mathscr p  }
    \le
    C\left(e^{-c k}+e^{-c(N-k)}\right).
     \end{equation}
    Moreover, 
    \[
        \mathscr p\left(\mathbb Z\setminus \llbracket -m,m\rrbracket\right)
        \le Ce^{-cm}
        \qquad\mbox{for every }m\in\mathbb N .
    \]
\end{lemma}
\begin{proof}
    For \(m\in\mathbb N\) and
\(x=(x_i)_{i\in\llbracket -m+1,m\rrbracket}\in
\{-1,1\}^{\llbracket -m+1,m\rrbracket}\), define a path
\(S(x)=(S_j(x))_{j\in\llbracket -m,m\rrbracket}\)
by
\[
S_j (x)=-{\bf1}_{\{j\in\llbracket-m,-1\rrbracket\}}\cdot\sum_{i=j+1}^{0}x_i+{\bf1}_{\{j\in\llbracket1,m\rrbracket\}}\cdot\sum_{i=1}^{j}x_i\qquad\mbox{ for } j\in\llbracket -m,m\rrbracket.
\]
We define
\[
    \mathtt{M}_m(x)
    =
    \min\left(
        j\in\llbracket -m,m\rrbracket:
        S_j(x)
        =
        \min_{\ell\in\llbracket -m,m\rrbracket}
         S_\ell (x)
    \right).
\]
Let \(\varepsilon \in(0,1/4)\), \(C_1,c_1>0\), and \((\mathfrak p_m)_{m\in\mathbb N}\) be the objects supplied by Lemma \ref{lem:finite-window limits in the bulk}. Throughout this proof, we denote $D=\min(k-1,N-k)$ for simplicity. We now define a probability measure \(\mathfrak q_m\) on \(\mathbb Z\) by $\mathfrak q_m(B)
    =
    \mathfrak p_m (\mathtt{M}_m^{-1}(B) )$ for $B\subseteq\mathbb Z$, i.e.,  \(\mathfrak q_m\) is the law of \(\mathtt{M}_m\) when the
increment vector has law \(\mathfrak p_m\). We claim that there exist constants \(C_2,c_2>0\), depending only on \(q\) and \(\b\), such that, when \(m\in\llbracket1,\lfloor \varepsilon D\rfloor\rrbracket\),
\begin{equation}\label{eq:main-claim-local-qm-approx new}
    \TV{
        \mathbf P_{N,k}^q(\mathsf M_N-k\in\cdot)-\mathfrak q_m
    }
    \le
    C_2\left(e^{-c_2m}+e^{-c_2D }\right).
\end{equation}
Indeed, for such $m$, we have $\{k+i: i\in\llbracket -m+1,m\rrbracket\}\subseteq \llbracket1,N\rrbracket$. Denote 
\[
    X_{N,k}^{m}=(2\tau_{k+i}-1)_{i\in\llbracket -m+1,m\rrbracket} 
    \qquad\mbox{ and }\qquad
    \mathtt{M}_{N,k}^{m}=\mathtt{M}_m\left(X_{N,k}^{m}\right).
\]
Note that, for every \(j\in\llbracket -m,m\rrbracket\),  
 $S_j (X_{N,k}^{m})= s_{k+j}(\tau)-s_k(\tau)$.
Hence, on the event \(\{|\mathsf M_N-k|\le m\}\), the point \(\mathsf M_N-k\) is the first minimum of \((S_j (X_{N,k}^{m}))_{j\in\llbracket -m,m\rrbracket}\), and therefore
\[
    \left\{|\mathsf M_N-k|\le m\right\}
    \subseteq
    \left\{\mathtt{M}_{N,k}^{m}=\mathsf M_N-k\right\}.
\]
By Corollary~\ref{cor:particle-minimum-localization} with some constants $C_3,c_3>0$, we get
\begin{equation}\label{eq:mryu9o}
\begin{aligned}
    &\TV{
        \mathbf P_{N,k}^q(\mathsf M_N-k\in\cdot)
        -
        \mathbf P_{N,k}^q(\mathtt{M}_{N,k}^{m}\in\cdot)
    }                                                \\
    & \le
    2\mathbf P_{N,k}^q\left(\mathtt{M}_{N,k}^{m}\ne \mathsf M_N-k\right)
    \le
    2\mathbf P_{N,k}^q(|\mathsf M_N-k|>m)
    \le
    2C_3e^{-c_3m}.
\end{aligned}
\end{equation}
Next, for every \(B\subseteq\mathbb Z\),
\[
\begin{aligned}
&\left|
    \mathbf P_{N,k}^q\left(\mathtt{M}_{N,k}^{m}\in B\right)-\mathfrak q_m(B)
 \right|                                                      =
\left|
    \mathbf P_{N,k}^q\left(X_{N,k}^{m}\in \mathtt{M}_m^{-1}(B)\right)
    -
    \mathfrak p_m\left(\mathtt{M}_m^{-1}(B)\right)
\right|                                                      \le
\TV{
    \mathbf P_{N,k}^q(X_{N,k}^{m}\in\cdot)-\mathfrak p_m
}.
\end{aligned}
\]
Taking the supremum over \(B\subseteq\mathbb Z\) and using Lemma \ref{lem:finite-window limits in the bulk}, we obtain
\[
    \TV{
        \mathbf P_{N,k}^q\left(\mathtt{M}_{N,k}^{m}\in\cdot\right)-\mathfrak q_m
    }
    \le
    C_1\left(e^{-c_1m}+e^{-c_1D}\right).
\]
Combining this with \eqref{eq:mryu9o}, we conclude the proof of the claim \eqref{eq:main-claim-local-qm-approx new}.

We next prove that \((\mathfrak q_m)_{m\in\mathbb N}\) is a Cauchy sequence with respect to the space of probability measures on $\mathbb{Z}$ equipped with the total variation distance.
Fix \(n\ge m\geq1\). For every \(K\in\mathbb N\) such that \(n\le \lfloor\varepsilon K\rfloor\), apply \eqref{eq:main-claim-local-qm-approx new} with \(N=2K+1\) and \(k=K+1\), with \(m\) and \(n\).  Then \(D=K\), and 
\[
\begin{aligned}
    \TV{\mathfrak q_m-\mathfrak q_n}
    &\le
    \TV{
        \mathfrak q_m-
        \mathbf P_{2K+1,K+1}^q(\mathsf M_{2K+1}-(K+1)\in\cdot)
    }                                                        +
    \TV{
        \mathbf P_{2K+1,K+1}^q(\mathsf M_{2K+1}-(K+1)\in\cdot)
        -\mathfrak q_n
    }                                                        \\
    &\le
    C_2\left(e^{-c_2m}+e^{-c_2n}+2e^{-c_2K}\right).
\end{aligned}
\]
Since the left side above does not depend on \(K\), letting \(K\to\infty\) yields
\begin{equation}\label{eq:hrytwjh3}
    \TV{\mathfrak q_m-\mathfrak q_n}
    \le
    2C_2e^{-c_2m} 
    \qquad\mbox{ for each } n\ge m.
\end{equation}
Thus the sequence \((\mathfrak q_m)_{m\in\mathbb N}\) is a Cauchy sequence in the space of probability measure on $\mathbb{Z}$ equipped with the total variation distance. Since \(\mathbb Z\) is countable, total variation convergence is equivalent to convergence of the point masses in \(\ell^1(\mathbb Z)\). The space \(\ell^1(\mathbb Z)\) is complete, 
hence there exists a probability measure \(\mathscr p\) on \(\mathbb Z\) such that
\[
    \lim_{m\to\infty}\TV{\mathfrak q_m-\mathscr p}=0.
\]
Letting \(n\to\infty\) in \eqref{eq:hrytwjh3} and combining with the above equation,
\begin{equation}\label{eq:main-claim-qm-to-p new}
    \TV{\mathfrak q_m-\mathscr p}
    \le 2C_2e^{-c_2m} 
   \qquad\mbox{ for each } m\in\mathbb N .
\end{equation} 
Since \(\mathfrak q_m\) is supported on \(\llbracket -m,m\rrbracket\), this also implies
\begin{equation}\label{eq:main-claim-tail-p new}
    \mathscr p\left(\mathbb Z\setminus\llbracket -m,m\rrbracket\right)
    \le
    \TV{\mathscr p-\mathfrak q_m}
    \le
    2C_2e^{-c_2m} 
    \qquad\mbox{ for each } m\in\mathbb N .
\end{equation}

We next prove \eqref{eq:goal main claim in proof upper}. When $D \geq\lceil 2/\varepsilon \rceil+1$, we let $m=\lfloor\varepsilon D \rfloor$, so \(m\in\llbracket1,\lfloor\varepsilon D\rfloor\rrbracket\). By \eqref{eq:main-claim-local-qm-approx new} and \eqref{eq:main-claim-qm-to-p new},
\[
\begin{aligned}
    \TV{
        \mathbf P_{N,k}^q(\mathsf M_N-k\in\cdot)-\mathscr p
    }
    &\le
    \TV{
        \mathbf P_{N,k}^q(\mathsf M_N-k\in\cdot)-\mathfrak q_m
    }
    +
    \TV{\mathfrak q_m-\mathscr p}                         \\
    &\le
    C_2\left(e^{-c_2m}+e^{-c_2D}\right)
    +
    2C_2e^{-c_2m}                                  \le
    C_4e^{-c_4D },
\end{aligned}
\]
for constants \(C_4,c_4>0\) depending only on \(q\) and \(\b\). Since \(D=\min(k-1,N-k)\),
 $e^{-c_4D}
    \le
    e^{c_4}e^{-c_4k}+e^{-c_4(N-k)} $, and we conclude the estimate \eqref{eq:goal main claim in proof upper}.  For the finitely many cases $D\leq \lceil 2/\varepsilon \rceil $,   $e^{-c_4k}+e^{-c_4(N-k)}\ge e^{-c_4(\lceil 2/\varepsilon \rceil+2)}$, and the total variation distance is bounded by $1$, so they can be absorbed into the constants.  

Finally, we prove uniqueness of the measure \(\mathscr p\). Suppose that \(\mathscr p'\) is another probability measure on \(\mathbb Z\) satisfying the  estimate \eqref{eq:goal main claim in proof upper} of the lemma, possibly with different constants. Taking \(N=2L\) and \(k=L\), 
\[
\begin{aligned}
    \TV{\mathscr p- \mathscr p'}
    &\le
    \TV{
        \mathscr p-
        \mathbf P_{2L,L}^q(\mathsf M_{2L}-L\in\cdot)
    }
    +
    \TV{
        \mathbf P_{2L,L}^q(\mathsf M_{2L}-L\in\cdot)
        -
         \mathscr p'
    }                                                       \\
    &\le
    C\left(e^{-cL}+e^{-cL}\right)
    +
      C'\left(e^{- c' L}+e^{- c' L}\right).
\end{aligned}
\]
Letting \(L\to\infty\) gives \(\TV{\mathscr p- \mathscr p'}=0\). Hence \(\mathscr p= \mathscr p'\), completing the proof.
\end{proof}

\begin{proof}[Proof of Proposition \ref{prop:part of the thm upper bound}]
We regard probability measures on $\llbracket0,N\rrbracket$ as probability measures on $\mathbb Z$. Let
    \[
        u_N=\mathrm{Unif}_{\llbracket0,N\rrbracket},\qquad
        \upmu_N^{T}
        =
        \mathbb{P}^{(q,\b,\b)}_{N,\uptau}(T_N\in\cdot),
        \qquad
        \upmu_N^{\mathsf M}
        =
        \mathbb{P}^{(q,\b,\b)}_{N,\uptau}(\mathsf M_N\in\cdot).
    \]
    Since $T_N$ is a function only of the path $\bg$, its law under the enlarged measure $\mathbb{P}^{(q,\b,\b)}_{N,\uptau}$ is the same as its law under $\mathbb P^{(q,\b,\b)}$. By Proposition   \ref{prop:macroscopic uniformity TN upper}, there exists a constant
    $C_1\in(0,\infty)$, depending only on $q$ and $\b$, such that
    \begin{equation}\label{eq:MN-upper-TN-unif}
        \TV{\upmu_N^{T}-u_N}
        \le \ C_1/N \quad\mbox{ for each }N\in\mathbb N.
    \end{equation}
    In particular, 
    \begin{equation}\label{eq:MN-upper-pointwise-TN}
        \upmu_N^{T}(\{k\})
        \le
        u_N(\{k\})+\TV{\upmu_N^{T}-u_N}
        \le \left(1+C_1\right)/N\quad\mbox{ for each }N\in\mathbb N,\,k\in\llbracket0,N\rrbracket.
    \end{equation}

By Lemma \ref{lem:main claim in proof upper}, there exist a probability measure $\mathscr p$ on $\mathbb Z$ and constants $C_2,c_2\in(0,\infty)$ depending only on $q$ and $\b$, such that
 \begin{equation}\label{eq:MN-upper-conditional-main}
        \TV{
            \mathbf P_{N,k}^q(\mathsf M_N-k\in\cdot)-\mathscr p
        }
        \le
        C_2\left(e^{-c_2k}+e^{-c_2(N-k)}\right)\quad\mbox{ for each }N\in\mathbb N,\,k\in\llbracket0,N\rrbracket,
    \end{equation}    
and 
    \begin{equation}\label{eq:MN-upper-tail-p}
        \mathscr p\left(\mathbb Z\setminus\llbracket-m,m\rrbracket\right)
        \le C_2e^{-c_2m}\quad\mbox{ for each }m\in\mathbb N.
    \end{equation}
Note that $\mathscr p$ has finite first absolute moment. Indeed,  
    \begin{equation}\label{eq:MN-upper-first-moment-p}
      \sum_{r\in\mathbb Z}|r|\mathscr p(\{r\}) 
        \le
        1+\sum_{m=1}^{\infty}C_2e^{-c_2m}=C_3
        <\infty .
    \end{equation}

For a probability measure $\lambda$ on $\llbracket0,N\rrbracket$, define the convolution $\lambda*\mathscr p$ on $\mathbb Z$ by
    \[
        (\lambda*\mathscr p)(A)
        =
        \sum_{k=0}^{N}\lambda(\{k\})\,\mathscr p(A-k) 
        \qquad\mbox{ for each }
        A\subseteq\mathbb Z,
    \]
    where we set $A-k:=\{r\in\mathbb Z:k+r\in A\}$.
    We first compare $\upmu_N^{\mathsf M}$ with $\upmu_N^{ T}*\mathscr p$. For $k\in\llbracket0,N\rrbracket$, denote
    \[
        K_{N,k}(B)
        =
        \mathbf P_{N,k}^q(\mathsf M_N-k\in B) 
        \qquad\mbox{ for each } B\subseteq\mathbb Z.
    \]
    Then
    \[
        \upmu_N^{\mathsf M}(A)
        =
        \sum_{k=0}^{N}
        \upmu_N^{ T}(\{k\})\,K_{N,k}(A-k)\qquad\mbox{ for each }
        A\subseteq\mathbb Z.
    \]
Therefore, using \eqref{eq:MN-upper-conditional-main} and \eqref{eq:MN-upper-pointwise-TN},
\begin{equation}\label{eq:MN-upper-first-comparison}
    \begin{split}
         \TV{\upmu_N^{\mathsf M}-\upmu_N^{ T}*\mathscr p}
        &\le
        \sum_{k=0}^{N}
        \upmu_N^{ T}(\{k\})
        \TV{K_{N,k}-\mathscr p}                                    \le
        C_2
        \sum_{k=0}^{N}
        \upmu_N^{ T}(\{k\})
        \left(e^{-c_2k}+e^{-c_2(N-k)}\right)\\
        &\leq  C_2(1+C_{1}) N^{-1}
        \sum_{k=0}^{N}
        \left(e^{-c_2k}+e^{-c_2(N-k)}\right)                      \le
        2C_2(1+C_{1}) \left(1-e^{-c_2}\right)^{-1}  N^{-1}.
    \end{split}
\end{equation}

Next we compare $\upmu_N^{ T}*\mathscr p$ with $u_N*\mathscr p$. The map which sends a point $k$ to the probability measure $\mathscr p(\cdot-k)$ is a Markov kernel from  $\llbracket0,N\rrbracket$ to $\mathbb Z$.  Hence
    \begin{equation}\label{eq:MN-upper-second-comparison}
        \TV{\upmu_N^{ T}*\mathscr p-u_N*\mathscr p}
        \le
        \TV{\upmu_N^{ T}-u_N}
        \le C_{1}/N .
    \end{equation}

We now compare $u_N*\mathscr p$ with $u_N$. For $r\in\mathbb Z$, let $u_N^{(r)}$ be the law of $U+r$, where $U$ has law $u_N$. Then
        $u_N*\mathscr p=\sum_{r\in\mathbb Z}\mathscr p(\{r\})u_N^{(r)}$.
    By convexity of total variation distance,
    \[
        \TV{u_N*\mathscr p-u_N}
        \le
        \sum_{r\in\mathbb Z}
        \mathscr p(\{r\})\TV{u_N^{(r)}-u_N}.
    \]
    For fixed $r$, the measures $u_N^{(r)}$ and $u_N$ are uniform measures, with mass $(N+1)^{-1}$ at each point, on the intervals
    $\llbracket r,N+r\rrbracket$ and $\llbracket0,N\rrbracket$ respectively.   Therefore
    $ \TV{u_N^{(r)}-u_N} \le 2|r|/(N+1)$.
  Combining with \eqref{eq:MN-upper-first-moment-p}, 
    \begin{equation}\label{eq:MN-upper-third-comparison}
        \TV{u_N*\mathscr p-u_N}
        \le
         2 (N+1)^{-1} 
        \sum_{r\in\mathbb Z}|r|\mathscr p(\{r\})
         \le
        2C_3(N+1)^{-1} .
    \end{equation}
    
   Finally, combining  \eqref{eq:MN-upper-first-comparison},   \eqref{eq:MN-upper-second-comparison}, and    \eqref{eq:MN-upper-third-comparison},
    \[
    \begin{aligned}
        \TV{\upmu_N^{\mathsf M}-u_N}
        &\le
        \TV{\upmu_N^{\mathsf M}-\upmu_N^{ T}*\mathscr p}
        +
        \TV{\upmu_N^{ T}*\mathscr p-u_N*\mathscr p}
        +
        \TV{u_N*\mathscr p-u_N}                                      \\
        &\le 
        \left(
             2C_2(1+C_{1}) \left(1-e^{-c_2} \right)^{-1}
            +C_{1}
            +2C_3
        \right) N^{-1}.
    \end{aligned}
    \]
    The quantity in parentheses is finite and depends only on $q$ and $\b$. We conclude the proof.
\end{proof}

\subsection{Upper bound on the total variation distance between $\mathsf M_N$ and the uniform distribution}\label{sec:Approx4}

We start with a general lower bound on random walks, which we then turn into a lower bound on the total variation distance between $\mathsf M_N$ and the uniform distribution on $\llbracket 0,N\rrbracket$.

\begin{lemma}\label{lem:technical lem in lower bound}
   Let $\b>1$.  Let $(\updelta_i,\upupsilon_i)$, $i\in\mathbb N$ be i.i.d.  variables such that \(\updelta_i\sim\nu_\b\), and, conditionally on \(\updelta_i\), we set
   \[
   \upupsilon_i=\updelta_i+{\bf1}_{\updelta_i=0}\upsigma_i ,
   \]
   where $\upsigma_i$, $i\in\mathbb N$  are i.i.d. symmetric \(\{-1,1\}\)-valued random variables, independent of $\updelta_i$,   $i\in\mathbb N$.
   Then there exists a sufficiently large $N$, depending only on $\b$, such that  
\begin{equation}\label{eq:hnrstyhj}
    \mathbb P\left(
        N+\updelta_1+\dots+\updelta_t\geq0
        \text{ and }
       N+\upupsilon_1+\dots+\upupsilon_t\geq0
        \text{ for all }t\in\mathbb{Z}_{\geq0}
    \right)
    \ge \frac{1}{2} .
\end{equation}
\end{lemma}
\begin{proof}
    We have $ \mathbb E[\updelta_1]=\mathbb E[\upupsilon_1]=(\b-1)/(\b+1)$. For \(Y\in\{\updelta_1,\upupsilon_1\}\), the function \(\theta\mapsto\mathbb E[e^{-\theta Y}]\) is differentiable at \(0\), has value \(1\) at \(0\), and has derivative \(-\mathbb E[Y]<0\) at \(0\). Therefore there exists \(\theta >0\) such that $\rho =\mathbb E[e^{-\theta  \updelta_1}]<1$ and $\rho'=\mathbb E[e^{-\theta \upupsilon_1}]<1$. 
For \(N\in\mathbb N\), Markov's inequality and a union bound give 
\[ 
    \mathbb P\left(
        \mbox{there exists } t\in\mathbb N\mbox{ such that }
        N+\updelta_1+\dots+\updelta_t<0
    \right)
     \le
    \sum_{t\ge1}
    e^{-\theta  N}\rho^t
    =
    e^{-\theta N} \rho/(1-\rho), 
\]
and a corresponding bound where the $\updelta_i$ on the left side are replaced by $\upupsilon_i$, and $\rho$ on the right side is replaced by  $\rho'$.
Choose \(N\in\mathbb N\) such that
 $e^{-\theta N}\rho/(1-\rho)
    +
    e^{-\theta N}\rho'/(1-\rho')
    \le 1/2$.
Then \eqref{eq:hnrstyhj} holds for this $N$.
\end{proof}

\begin{proposition}\label{prop:part of the thm lower bound}
    Assume $q\in[0,1)$ and $\b>1$.  Then there exists a constant $C>0$ depending only on $q$ and $\b$, such that for each $N\geq2$,
    \[
    \TV{
        \mathbb{P}^{(q,\b,\b)}_{N,\uptau} (\mathsf M_N\in\cdot)
        -
        \mathrm{Unif}_{\llbracket0,N\rrbracket}
    }
    \geq \frac{C}{N}.
    \]
\end{proposition}
\begin{proof}
    We recall that, for
\(\tau=(\tau_1,\ldots,\tau_N)\in\{0,1\}^N\),  
\begin{equation}\label{eq:occupation walk}
    s_0(\tau)=0,\qquad
    s_j(\tau)= (2\tau_1-1) +\dots+(2\tau_j-1) 
    \quad \mbox{ for }j\in\llbracket1,N\rrbracket,
\end{equation}
and
 $\mathsf M_N
    =
    \min (
        j\in\llbracket0,N\rrbracket:
        s_j(\tau)=\min_{\ell\in\llbracket0,N\rrbracket}s_\ell(\tau) )$.
Define 
\begin{equation}\label{eq:ANBN}
    A_N=\{\tau:s_j(\tau)\ge0\text{ for all }j\in\llbracket0,N\rrbracket\}\qquad\mbox{ and }\qquad
    B_N=\{\tau:s_j(\tau)>0\text{ for all }j\in\llbracket1,N\rrbracket\}.
\end{equation}
In this proof, we write $\mathbb P^{0}_{N,\uptau}=\mathbb P^{(0,\b,\b)}_{N,\uptau}$,  $\mathbb P^{0} =\mathbb P^{(0,\b,\b)}$, $\mathbb P^{q}_{N,\uptau}=\mathbb P^{(q,\b,\b)}_{N,\uptau}$, and $\mathbb P^{q} =\mathbb P^{(q,\b,\b)}$.
For clarity, we split the rest of the proof into two steps.\\

\noindent\textbf{Step 1. } In this step, we prove the proposition assuming the following claim: there exists a constant $c_1>0$ depending only on $\b$, such that 
\begin{equation}\label{eq:MN-tie-lower-target new}
    \mathbb P^{0}_{N,\uptau}(A_N\setminus B_N)
    \ge c_1/N 
    \qquad\mbox{ for all } N\ge2 .
\end{equation}
We first prove the corresponding bound for general $q\in[0,1)$.
Indeed, in view of \eqref{eq:PNbb}, 
\begin{equation} \label{eq:ytnedyjnetyjn again}
    \mathbb{P}^{q}_{N,\uptau}(B)
    =
    \frac{\mathbb E^{0}_{N,\uptau}\left[{\bf1}_B W_{\bs(\bg)}(\b^2)\right]}
    {\mathbb E^{0}\left[W_{\bs(\bg)}(\b^2)\right]}\quad\mbox{ for every event }B .
\end{equation}
By Theorem \ref{thm:prob interpretation of Bryc polynomials}, $W_{\bs(\bg)}(\b^2)\ge1$ for every $\bg\in\lrw_N$.  
When \(q\in(0,1)\), by Proposition \ref{prop:moment bound coex} with \(z=\b^2\) and \(m=1\); and when $q=0$, by \(W_{\bs(\bg)}(\b^2)=1\), there exists \(M\in(0,\infty)\), depending only on \(q,\b\), such that
\begin{equation*} 
    \mathbb E^{0}\left[W_{\bs(\bg)}(\b^2)\mid T_N=k\right]\le M
    \quad
    \mbox{for every }N\in\mathbb N,\ k\in\llbracket0,N\rrbracket .
\end{equation*}
Hence \(\mathbb E^{0}[W_{\bs(\bg)}(\b^2)]\le M\). Combining this upper bound, the lower bound \(W_{\bs(\bg)}(\b^2)\ge1\), by \eqref{eq:ytnedyjnetyjn again},
\begin{equation*} 
    \mathbb{P}^{q}_{N,\uptau}(B)
    \ge 
    \mathbb{P}^{0}_{N,\uptau}(B)/M 
    \quad\mbox{ for every event }B .
\end{equation*}  
Using this inequality for the event $A_N\setminus B_N$ and combining with \eqref{eq:MN-tie-lower-target new},
\begin{equation}\label{eq:MN-tie-lower-target new q}
    \mathbb P^{q}_{N,\uptau}(A_N\setminus B_N)
    \ge C_1/(MN) 
    \qquad\mbox{ for all } N\ge2 .
\end{equation}
To lower bound the total variation distance, we recall the map   in Step 2 of the proof of Proposition \ref{prop:macroscopic uniformity TN lower},  
\[
\mathrm{inv}:\lrw_N\to\lrw_N,\qquad (\mathrm{inv}(\bg))_i=\gamma_{N-i}-\gamma_N\quad \mbox{ for }i\in\llbracket0,N\rrbracket.
\]
As in the proof of Proposition \ref{prop:macroscopic uniformity TN lower},  we know that
\begin{equation}\label{eq:tdujkmdfytk}\mathbb P^{q}(\mathrm{inv}(\bg))=\mathbb P^{q}(\bg) \quad\mbox{ for } \bg\in\lrw_N .\end{equation}
We now extend this map to the enlarged space of random walk paths and Bernoulli random variables. Define 
\[
    (\mathrm{inv}(\kappa))_i=1-\kappa_{N+1-i}\quad \mbox{ for }i\in\llbracket1,N\rrbracket .
\]
 Checking the cases in \eqref{eq:occupation coex}, we know that
\begin{equation}\label{eq:tdujkmdfytk2}
    (\mathrm{inv}(\tau))_i=1-\tau_{N+1-i}\quad \mbox{ for }i\in\llbracket1,N\rrbracket .
\end{equation}
By \eqref{eq:tdujkmdfytk} and the fact that the product Bernoulli\((1/2)\) law of $\kappa$ is invariant under the inversion map,
we know that this map preserves the measure \(\mathbb P^{q}_{N,\uptau}\) on the enlarged space.
Observe from  \eqref{eq:ANBN} that $A_N=\{\mathsf M_N=0\}$. Since $s_j(\mathrm{inv}(\tau))  = s_{N-j}(\tau)-s_N(\tau)$ by \eqref{eq:tdujkmdfytk2}, we know that
\(\mathrm{inv}(\tau)\in B_N\) if and only if $s_i(\tau)>s_N(\tau)$ for all $i\in\llbracket0,N-1\rrbracket$, which is equivalent to \(\mathsf M_N=N\). 
Hence  
$\mathbb P^{q}_{N,\uptau}(\mathsf M_N=N)
    =
    \mathbb P^{q}_{N,\uptau}(B_N)$, and therefore,
 \[
 \mathbb P^{q}_{N,\uptau}(\mathsf M_N=0)
    -
    \mathbb P^{q}_{N,\uptau}(\mathsf M_N=N)      =
    \mathbb P^{q}_{N,\uptau}(A_N\setminus B_N).
 \]
By the above display, we have \[
\begin{aligned}
    &\TV{\mathbb{P}^{q}_{N,\uptau} (\mathsf M_N\in\cdot)-\mathrm{Unif}_{\llbracket0,N\rrbracket}}
     \ge
    \max\left(\left|\mathbb{P}^{q}_{N,\uptau} (\mathsf M_N=0)-(N+1)^{-1}\right|,\left|\mathbb{P}^{q}_{N,\uptau} (\mathsf M_N=N)-(N+1)^{-1}\right|\right)       \\
    &\ge
    \frac12\left|
        \left(\mathbb{P}^{q}_{N,\uptau} (\mathsf M_N=0)-(N+1)^{-1}\right)-\left(\mathbb{P}^{q}_{N,\uptau} (\mathsf M_N=N)-(N+1)^{-1}\right)
    \right|                                          =
    \frac12\,
    \mathbb P^{q}_{N,\uptau}(A_N\setminus B_N).
\end{aligned}
\]
Combining with \eqref{eq:MN-tie-lower-target new q}, we conclude the proof.\\

\noindent\textbf{Step 2. } In this step, we prove the claim \eqref{eq:MN-tie-lower-target new}, which we restate as follows: there exists a constant $c_1>0$ depending only on $\b$, such that 
\begin{equation}\label{eq:MN-tie-lower-target new new}
    \mathbb P^{0}_{N,\uptau}(A_N\setminus B_N)
    \ge c_1/N 
    \qquad\mbox{ for all } N\ge2 .
\end{equation}
We begin by proving the following claim
\begin{equation}\label{eq:claim limit TN0}
    \lim_{N\to\infty}
    (N+1) \cdot 
        \mathbb P^{0}(T_N=0)  
    = (\b+1)/(\b-1).
\end{equation}
Indeed, this follows from the techniques developed in the proof of Lemma \ref{lem:macroscopic uniformity q0}. Recall from \eqref{eq:jdtyjykmt}  that $\mathbb P^{0}(T_N=0)=p_N/D$, where   for $\ell\in\mathbb{Z}_{\geq0}$, $p_{\ell}$ denotes the probability that a $\nu_{\b}$-random walk started at $0$ does not reach level $-1$ during the first $\ell$ steps,  $w=(\b+\b^{-1}+2)/4$, and 
 $    D = \sum_{m=1}^{N}   \b/(4w) \cdot  p_{m-1}p_{N-m} +  p_N $. 
Recall from \eqref{eq:p ell exponential convergence} and \eqref{eq:fysrhth} that there exist constants \(C_2,c_2,C_3,c_3\in(0,\infty)\), depending only on \(\b\), such that
\begin{equation*} 
    0\leq p_\ell-p_{\infty}\leq C_2e^{-c_2\ell}
    \quad\mbox{ for all }
    \ell\in\mathbb Z_{\geq0},\qquad \qquad\qquad \left|D-(N+1) \cdot\b/(4w) \cdot p_{\infty}^2\right|    \leq C_3, 
\end{equation*} 
where $p_{\infty}=1-\b^{-2}$. Taking $\ell=N$ we know  $\lim_{N\to\infty} (N+1)p_N/D= (\b+1)/(\b-1)$, and hence  \eqref{eq:claim limit TN0}.

Let $N_0\in\mathbb N$ be the constant chosen from Lemma \ref{lem:technical lem in lower bound}, which depends only on $\b$. For \(N\ge N_0+2\), define an event \(E_N\) on the enlarged space as the event that all of the following hold:
\begin{align*}
    & 0=\gamma_0=\gamma_1=\gamma_2,\qquad 
    \kappa_1=1,\mbox{ }
    \kappa_2=0,\qquad 
    \gamma_i-\gamma_{i-1}=1\quad\mbox{for }i\in\llbracket3,N_0+2\rrbracket,\\
    &    N_0+\sum_{i=N_0+3}^t  (\gamma_i-\gamma_{i-1})\ge0 
    \quad\mbox{and}\quad
    N_0+\sum_{i=N_0+3}^t  (2\tau_i-1)\ge0 \qquad\mbox{ for }t\in\llbracket N_0+2,N \rrbracket.
\end{align*} 
Then $E_N\subseteq A_N\setminus B_N$ and $E_N\subseteq\{T_N=0\}$. Indeed, on $E_N$, the occupation walk \eqref{eq:occupation walk} satisfies $s_0=0$, $s_1=1$, $s_2=0$, then increases by \(1\) during the next \(N_0\) steps, and after time \(N_0+2\) remains nonnegative; similarly, the path \(\bg\) remains nonnegative and \(\gamma_0=0\).

We next compute the conditional probability $\mathbb P^{0}_{N,\uptau} (E_N\, |\,T_N=0 )$. By Lemma \ref{lem:decomposition TASEP}, under \(\mathbb P^{0}(\,\cdot\,|T_N=0)\), the path \((\gamma_i-\gamma_0)_{i\in\llbracket0,N\rrbracket}\) is a \(\nu_\b\)-random walk started at \(0\), conditioned to stay nonnegative up to time \(N\).  Since \(T_N\) depends only on the path, the $\kappa_i$, $i\in\llbracket1,N\rrbracket$ remain i.i.d. \(\ber(1/2)\) and
independent of the path.  Under the unconditioned \(\nu_\b\)-walk together with the independent i.i.d. $\kappa_i$, the prescribed first \(N_0+2\) path increments and $\kappa_1,\kappa_2$ have probability $\nu_\b(0)^2\,\nu_\b(1)^{N_0}\,2^{-2}$.
After these prescribed steps, the future pairs $(\gamma_{j+N_0+2}-\gamma_{j+N_0+1},2\tau_{j+N_0+2}-1)$, $j\in\mathbb N$  have the same joint law as $(\updelta_j,\upupsilon_j)$, $j\in\mathbb N$ from Lemma \ref{lem:technical lem in lower bound}, independently of the prescribed initial segment.  Therefore, by Lemma \ref{lem:technical lem in lower bound}, for every
\(N\ge N_0+2\),
\begin{equation*}
    \begin{split}
        & \mathbb P^{0}_{N,\uptau}
    \left(E_N\,\middle|\,T_N=0\right)\\&=p_N^{-1}\cdot\nu_\b(0)^2\,\nu_\b(1)^{N_0}\,2^{-2}
        \,
        \mathbb P\left(
            N_0+  \updelta_1+\dots+\updelta_t\ge0,\
            N_0+  \upupsilon_1+\dots+\upupsilon_t\ge0
            \text{ for all }t\in\llbracket0,N-N_0-2\rrbracket
        \right)                                            \\
        &\ge 
        p_N^{-1}\cdot\nu_\b(0)^2\,\nu_\b(1)^{N_0} \,2^{-2}\cdot 1/2\ge 1/8\cdot\nu_\b(0)^2\,\nu_\b(1)^{N_0} =C_4,
    \end{split}
\end{equation*} 
where $C_4>0$ depends only on $\b$. Consequently, for every \(N\ge N_0+2\),
\begin{equation*} 
    \mathbb P^{0}_{N,\uptau}(A_N\setminus B_N)
    \ge C_4\cdot
    \mathbb P^{0}(T_N=0).
\end{equation*}
Combining with \eqref{eq:claim limit TN0}, there exist constants $N_1\in\mathbb N$ and $C_5>0$ depending only on $\b$, such that 
\[
\mathbb P^{0}_{N,\uptau}(A_N\setminus B_N)\geq C_5/N\quad\mbox{ for }N\geq N_1.
\]
To conclude the proof of the claim \eqref{eq:MN-tie-lower-target new new}, it suffices to handle the finitely many values $N\in\llbracket2,N_1-1\rrbracket$. 
We next show that $\mathbb P^{0}_{N,\uptau}(A_N\setminus B_N)>0$ for all $N\geq2$. Indeed, for every \(N\ge2\), the marginal law of $\tau$ under \(\mathbb P^{0}_{N,\uptau}\) has full support on \(\{0,1\}^N\). To prove this, fix \(\tau\in\{0,1\}^N\) and choose the path \(\bg \in\lrw_N\) with $\gamma_i-\gamma_{i-1}=2\tau_i-1$ for $i\in\llbracket1,N\rrbracket$. Then $\mathbb{P}^{0}(\bg)>0$, and conditionally on this path the configuration \(\tau\) is produced with probability one, since $\bg$ has no horizontal steps. Observe that the configuration $(1,0,1,1,\ldots,1)$ belongs to \(A_N\setminus B_N\). Hence $\mathbb P^{0}_{N,\uptau}(A_N\setminus B_N)>0$ for all $N\geq2$, and choosing now $C_1$ sufficiently small, we conclude.  
\end{proof}

\subsection{Proof of Theorem~\ref{thm:ApproxCoex again}}\label{sec:Final}

We now have all tools to prove Theorem~\ref{thm:ApproxCoex again} on approximating the stationary distribution of the open ASEP on the coexistence line.

\begin{proof}[Proof of Theorem~\ref{thm:ApproxCoex again}]
Fix $j\in \llbracket 0,N \rrbracket$, recall \eqref{eq:CoreDefs}, and abbreviate
\[
F=X_{j,m}=(\tau_i)_{i \in I_{j,m}},\quad B=B_{j,m},\quad
\mu_k=\mathbb{P}^{(q,\b,\b)}_{N,\uptau}(T_N=k),\quad w_j=\mathbb{P}^{(q,\b,\b)}_{N,\uptau}(\mathsf M_N=j), \quad D_k=\mathsf M_N-k.
\]
Set $h=m/3$. For any event $E$ in the state space of $F$,
\begin{align}\label{asepm:eq:mixture}
\mathbb{P}^{(q,\b,\b)}_{N,\uptau}(F\in E,M_N=j)-w_jB(E)=\sum_{k=0}^N\mu_k
\bigl[\mathbf{P}_{N,k}^q(F\in E,M_N=j)-\mathbf{P}_{N,k}^q(M_N=j)B(E)\bigr].
\end{align}
We split now the sum according to $|k-j|\le h$ and $|k-j|>h$. 
Note that if $|k-j|\le h$, then we get
\begin{equation}\label{asepm:eq:projection}
I_{j,m}\subseteq I_{k,h}.
\end{equation}
Hence, applying Proposition~\ref{asepm:prop:joint} with radius $h$, projected onto $F$, and taking the displacement coordinate equal to $j-k$, we see that for some constants $c_1,C_1>0$
\begin{equation}\label{asepm:eq:nearpoint}
\left|\mathbf{P}_{N,k}^q(F\in E,M_N=j)-\mathbf{P}_{N,k}^q(M_N=j)B(E)\right|
\le C_1e^{-c_1h}.
\end{equation}
Since $\mu_k\le C/(N+1)$, the contribution of coordinates  $|k-j|\le h$ to \eqref{asepm:eq:mixture} is at most $(N-1)^{-1}C_2(2h+1)e^{-c_2h}$ for some constants $C_2,c_2>0$.

Now consider the case where $|k-j|>h$. Note that the difference between the summands in \eqref{asepm:eq:mixture} is at most $\mathbf{P}_{N,k}^q(M_N=j)$. Hence, Corollary~\ref{cor:particle-minimum-localization} gives
\[
\mathbf{P}_{N,k}^q(M_N=j)
\le \mathbf{P}_{N,k}^q(|M_N-k|\ge|j-k|)
\le C_3e^{-c_3|j-k|} , 
\]
and so the total contribution is bounded from above by
\begin{align}\label{asepm:eq:far}
\sum_{|k-j|>h}\mu_k\mathbf{P}_{N,k}^q(M_N=j)
\le\frac {C_4}{N+1}\sum_{|k-j|>h}e^{-c_4|j-k|}\le\frac{C_5e^{-c_5h}}{N+1}.
\end{align} for some constants $c_3,C_3,c_4,C_4,c_5,C_5>0$. 
Combining now the above estimates,  we get that 
\[
\left|\mathbb{P}^{(q,\b,\b)}_{N,\uptau}(F\in E,M_N=j)-w_jB(E)\right|
\le\frac{C_6(h+1)e^{-c_6h}}{N+1}.
\]
uniformly in $E$ with some constants $c_6,C_6>0$. Recall from Lemma~\ref{asepm:lem:pointwise} that $w_j\ge c/(N+1)$. We now divide by $w_j$ and take the supremum over $E$ to see that
\begin{equation}\label{asepm:eq:final}
\left\|\mathbb{P}^{(q,\b,\b)}_{N,\uptau}(F \in \cdot \mid M_N=j)-B\right\|_{\mathrm{TV}}
\le C_7(h+1)e^{-c_7h}\le C_8e^{-c_8m}.
\end{equation}
for all $h$ large enough and constants $c_7,C_7,c_8,C_8>0$. This yields the desired Bernoulli approximation when conditioning on the event $M_N=j$.
The upper and lower total variation bounds for $M_N$ follow now from
Propositions \ref{prop:part of the thm upper bound} and \ref{prop:part of the thm lower bound}. 
The convergence of $M_N/N$ to the uniform measure on $[0,1]$
follows from their total variation approximation. This completes the proof.
\end{proof}

\section{Fluctuations and correlations of the stationary measure of the open ASEP at the coexistence line}\label{sec:Applications}

In the following, we establish the results on fluctuations and correlations of the stationary measure of the open ASEP on the coexistence line.

\subsection{Moderate deviation estimates: Proof of Theorem \ref{thm:moderate deviation Coex}}

We start with the moderate deviation bounds from Theorem \ref{thm:moderate deviation Coex}, which we state below again for clarity. 

\begin{theorem}\label{thm:moderate deviation Coex again}
Consider the open ASEP on $\llbracket1,N\rrbracket$ with jump rates $\alpha,\beta,\gamma,\delta$ and $q$ satisfying \eqref{eq:conditions open ASEP}. Assume Liggett's condition \eqref{eq:Liggett's condition} and that the open ASEP is on the coexistence line $\a=\b>1$, where $\a=(1-\alpha)/\alpha$ and $\b=(1-\beta)/\beta$. We denote the stationary measure by $\mu_N^{(q,\b,\b)}$. Let $\tau =(\tau_1,\ldots,\tau_N)\sim\mu_N^{(q,\b,\b)}$ and $\mathsf{M}_N =\mathsf{M} (\tau)$ be given by \eqref{eq:Mtau}. 
Denote $\rho_-=1/(1+\b)$ and $\rho_+=\b/(1+\b)$. 
There exist constants \(C,c\in(0,\infty)\), depending only on \(q\) and \(\b\), such that, for every \(N\in\mathbb N\), every nonempty set \(I\subseteq\llbracket1,N\rrbracket\), and every \(x\geq0\),
\[
\mu_N^{(q,\b,\b)}\left(
\left|
\sum_{i\in I}\tau_i
-\rho_-\left|I\cap\llbracket1,\mathsf M_N\rrbracket\right|
-\rho_+\left|I\cap\llbracket\mathsf M_N+1,N\rrbracket\right|
\right|
\geq x\sqrt{|I|}
\right)
\leq Ce^{-cx^2}.
\]
\end{theorem}
\begin{proof}
    By Theorem \ref{thm:ApproxCoex}, $\tau\sim\mu_N^{(q,\b,\b)}$ can be coupled with a random variable $T_N\in\llbracket0,N\rrbracket$, such that 
\begin{align}  
                &\TV{
    \mathbb{P} 
    \left(
        (\tau_i)_{i\in \llbracket1,k-m\rrbracket\cup \llbracket k+m,N\rrbracket}\in\cdot
        \,\middle|\,T_N=k
    \right)
    -
    \ber_{\llbracket1,k-m\rrbracket} \left(\rho_-\right)
    \otimes
    \ber_{\llbracket k+m,N\rrbracket} \left(\rho_+\right) 
}
\le C_1 e^{-c_1m},\label{eq:approx in proof of moderate}\\
&\mathbb{P} \left(|T_N- \mathsf{M}_N |\ge m\right)\le C_1 e^{-c_1m}, \label{eq: dist in proof moderate}
\end{align}
for $N\geq m\geq1$, $k\in\llbracket0,N\rrbracket$, 
where $C_1,c_1\in(0,\infty)$ depend only on $q$ and $\b$. 
For \(k\in\llbracket0,N\rrbracket\), \(i\in\llbracket1,N\rrbracket\), let
\[
p_i(k)=\rho_-{\bf1}_{\{i\leq k\}}+\rho_+ {\bf1}_{\{i > k\}}.
\]
Then for each nonempty set \(I\subseteq\llbracket1,N\rrbracket\) and $k,\ell\in\llbracket0,N\rrbracket$,
\begin{align}
 \label{eq:PIk}  & P_I(k)=  \sum_{i\in I}p_i(k)
=\rho_-|I\cap\llbracket1,k\rrbracket|
+\rho_+|I\cap\llbracket k+1,N\rrbracket|, \\
\label{eq:coex profile Lipschitz}&\left|P_I(k)-P_I(\ell)\right|
\leq   |k-\ell|.
\end{align}
Choose $c=\min(1/16, c_1/32)$, $C=2C_1+6$. It suffices to prove that, for every nonempty \(I\subseteq\llbracket1,N\rrbracket\), \(x\geq0\),
\begin{equation}\label{eq:coex auxiliary intrinsic}
\mathbb P\left(
\left|\sum_{i\in I}\tau_i-P_I(\mathsf M_N)\right|
\geq x\sqrt{L}
\right)
\leq Ce^{-cx^2},
\end{equation} 
where $L=|I|$.
When $x\leq 4$,  we have $1\leq Ce^{-cx^2}$ and \eqref{eq:coex auxiliary intrinsic} holds. When $x\geq\sqrt{L}$, the event on the left side of \eqref{eq:coex auxiliary intrinsic} is empty since 
$ |\sum_{i\in I}\tau_i-P_I(\mathsf M_N) |
\leq\sum_{i\in I}|\tau_i-p_i(\mathsf M_N)|<L$. We next assume $4<x<\sqrt{L}$ and denote $n=\lfloor x^2/16\rfloor$.
For each $k\in\llbracket0,N\rrbracket$ and $m\in\llbracket1,N\rrbracket$, we denote
$E_{k,m}=\llbracket1,k-m\rrbracket\cup\llbracket k+m,N\rrbracket$.
By contraction of total variation under projections, using \eqref{eq:approx in proof of moderate} with $m=n$, 
\begin{equation}\label{eq:in proof projected TV}
\TV{
\mathbb{P} 
    \left(
        (\tau_i)_{i\in I\cap E_{k,n}}\in\cdot
        \,\middle|\,T_N=k
    \right) 
-
\ber_{I\cap \llbracket1,k-n\rrbracket}(\rho_-)
\otimes
\ber_{I\cap \llbracket k+n,N\rrbracket}(\rho_+)
}
\leq C_1e^{-c_1 n}.
\end{equation}
We next claim that
\begin{equation}\label{eq:in proof inclusion}
\left\{\left|\sum_{i\in I}\tau_i-P_I(\mathsf M_N)\right|
\geq x\sqrt{L}\right\}\cap\{|T_N-\mathsf M_N|<n\}
\subseteq
\left\{
\left|\sum_{i\in I\cap E_{T_N,n}   }(\tau_i-p_i(T_N))\right|
\geq x\sqrt{L}-3n
\right\}.
\end{equation}
Indeed, this follows from
\[
\sum_{i\in I}\tau_i-P_I(\mathsf M_N)
=
\sum_{i\in I\cap E_{T_N,n} }(\tau_i-p_i(T_N))
+\sum_{i\in I\setminus  E_{T_N,n} }(\tau_i-p_i(T_N))
+P_I(T_N)-P_I(\mathsf M_N),
\]\
in view of
 $|\sum_{i\in I\setminus E_{T_N,n}}(\tau_i-p_i(T_N))|
\leq |I\setminus E_{T_N,n}|<2n$,
and the fact that, on \(\{|T_N-\mathsf M_N|<n\}\), we have
$|P_I(T_N)-P_I(\mathsf M_N)|
\leq |T_N-\mathsf M_N|<n$ by \eqref{eq:coex profile Lipschitz}.

For \(k\in\llbracket0,N\rrbracket\), we have $|I\cap E_{k,n}|\in\mathbb N$ since 
$|I\cap E_{k,n}|=L-|I\setminus E_{k,n}|
>L-2n
\geq L- x^2/8
>0$.
Under the product measure in \eqref{eq:in proof projected TV}, the coordinates indexed by \(I\cap E_{k,n}\) are independent Bernoulli variables with means
\(p_i(k)\). Moreover,
 $x\sqrt L-3n
>0$, and so by \eqref{eq:in proof projected TV} and \cite[Lemma 7.2]{SchmidYangShock},
\[
\begin{split}
 \mathbb P\left(
\left.
\left|\sum_{i\in I\cap E_{k,n}}(\tau_i-p_i(k))\right|
\geq x\sqrt L-3n
\,\right|\,T_N=k
\right) 
&\leq
2 e^{  - 2\left(x\sqrt L-3n\right)^2/|I\cap E_{k,n}| }
+C_1e^{-c_1n } \leq
 e^{- c_2 x^2}
+C_2e^{-c_3n } 
\end{split}
\] for some constants $c_1,C_1,c_2,C_2,c_3>0$, 
where we used $|I\cap E_{k,n}|\leq L$ for the second inequality.  
The displayed bound is uniform in \(k\). Using \eqref{eq:in proof inclusion} on \(\{|T_N-\mathsf M_N|<n\}\) and 
\eqref{eq: dist in proof moderate} on its complement,  and that $n \geq  x^2/32$ as $x>4$, we obtain
\[
\mathbb P\left(\left|\sum_{i\in I}\tau_i-P_I(\mathsf M_N)\right|
\geq x\sqrt L\right)\leq  C_3e^{-c_4x^2}
\]  
for some constants $C_3,c_4>0$, allowing us to conclude the proof of \eqref{eq:coex auxiliary intrinsic}.
\end{proof}

\subsection{Second-order limit fluctuations: Proof of Theorem \ref{thm:ApproxCoex Fluctuations} }

We will now show the fluctuation limit theorem for the number of particles under the stationary open ASEP on the coexistence line. 

\begin{theorem} 
\label{thm:ApproxCoex Fluctuations again}
Consider the open ASEP on $\llbracket1,N\rrbracket$ with jump rates $\alpha,\beta,\gamma,\delta$ and $q$ satisfying \eqref{eq:conditions open ASEP}. Assume Liggett's condition \eqref{eq:Liggett's condition} and that the open ASEP is on the coexistence line $\a=\b>1$. Let  $\mathsf{M} =\mathsf{M} (\tau)$ be given by \eqref{eq:Mtau} and set
\[
h(x)=\sum_{i=1}^{\lfloor Nx\rfloor}\tau_i\quad\mbox{ for }x\in[0,1],
\]
where $\tau =(\tau_1,\ldots,\tau_N)\sim\mu_N^{(q,\b,\b)}$.
Let \(U\sim\mathrm{Unif}[0,1]\), and let
\(B=\{B_x\}_{0\le x\le1}\) and \(B'=\{B'_x\}_{0\le x\le1}\) be two
standard Brownian motions, with \(U,B,B'\) independent.  Then, as $N\to\infty$, 
\[
\left(
    \frac{ \mathsf{M}  }{N},
    \left\{
    \frac{1}{\sqrt N}\left(
        h(x) 
        -\rho_-\left(\lfloor Nx\rfloor\wedge  \mathsf{M}  \right)
        -\rho_+\left(\lfloor Nx\rfloor-  \mathsf{M}  \right)_+\right)
    \right\}_{0\le x\le1}
\right)
\Rightarrow
\left(
    U,
    \left\{
        \sigma
        \left(B_{x\wedge U}+B'_{(x-U)_+}\right)
    \right\}_{0\le x\le1}
\right),
\] 
where $\rho_-=1/(1+\b)$, $\rho_+=\b/(1+\b)$, and $\sigma= \sqrt{\b}/(1+\b)$. Here, the convergence is in the space \([0,1]\times D[0,1]\), where \(D[0,1]\) is equipped with the uniform topology induced by the supremum norm $\|f\|_\infty=\sup_{0\le t\le1}|f(t)|$.
\end{theorem} 
\begin{proof}
In view of Theorem \ref{thm:ApproxCoex}, after enlarging the probability space, there exist independent random variables
\[
\varsigma_i^-\sim\ber(\rho_-), \,i\in\llbracket1,N\rrbracket,\qquad\qquad\varsigma_i^+\sim\ber(\rho_+), \,i\in\llbracket1,N\rrbracket,
\]
independent of \(T_N\), and constants $C,c\in(0,\infty)$ depending only on $q$ and $\b$, such that 
\begin{equation}\label{eq:EN in proof flux}
    \mathbb P\left(E_N^{\complement}\right)\le Ce^{-c\lfloor N^{1/4}\rfloor},\qquad\mbox{where }
E_N=
\left\{
\begin{array}{ll}
\tau_i=\varsigma_i^-\quad &\mbox{ for }
1\le i\le T_N-\lfloor N^{1/4}\rfloor,\\ 
\tau_i=\varsigma_{i-T_N}^+\quad &\mbox{ for }
T_N+\lfloor N^{1/4}\rfloor\le i\le N.
\end{array}
\right\}
\end{equation} Throughout the proof, write
\(\mathsf M_N=\mathsf M(\tau)=\mathsf M\).
Define, for $x\in[0,1]$,
\begin{align*}
    Z_N(x)&=\frac{1}{\sqrt N}
\left(
\sum_{i=1}^{\lfloor Nx\rfloor}\tau_i
-\rho_-\left(\lfloor Nx\rfloor\wedge\mathsf M_N\right)
-\rho_+\left(\lfloor Nx\rfloor-\mathsf M_N\right)_+
\right),\\
X_N(x)&=\frac{1}{\sqrt N}
\left(
\sum_{i=1}^{\lfloor Nx\rfloor\wedge T_N}\left(\varsigma_i^--\rho_-\right)
+
\sum_{j=1}^{(\lfloor Nx\rfloor-T_N)_+}\left(\varsigma_j^+-\rho_+\right)
\right),\\
R_N^\pm(x)&=
\frac{1}{\sqrt N}
\sum_{i=1}^{\lfloor Nx\rfloor}(\varsigma_i^\pm-\rho_\pm),
\end{align*} 
and define \(S_N^\pm(x)\), $x\in[0,1]$ to be the linear interpolations of $R_N^\pm(x)$. 
Define 
\[
\Psi:[0,1]\times C[0,1]^2\to C[0,1],\qquad\Psi(u,f,g)(x)=f(x\wedge u)+g((x-u)_+).
\] 
Note that the map $\Psi$ is continuous in the supremum norms.  Indeed, if
\((u_n,f_n,g_n)\to(u,f,g)\), then
\[ 
\|\Psi(u_n,f_n,g_n)-\Psi(u,f,g)\|_\infty
\le 
\|f_n-f\|_\infty+\|g_n-g\|_\infty +\omega_f(\lvert u_n-u\rvert)+\omega_g(\lvert u_n-u\rvert)
\to0, 
\]
where \(\omega_f\) and \(\omega_g\) are the moduli of continuity.
Donsker's invariance principle 
and \(T_N/N\Rightarrow U\) from Theorem~\ref{thm:ApproxCoex} give
$ (T_N/N,S_N^-,S_N^+ )
\Rightarrow (U,\sigma B,\sigma B' )$ in $[0,1]\times C[0,1]^2$.
By the continuous mapping theorem,  
\[
\left(
\frac{T_N}{N},
\Psi\left(
\frac{T_N}{N}, S_N^-,  S_N^+
\right)
\right)
\Rightarrow
\left(
U,
\left\{
\sigma\left(B_{x\wedge U}+B'_{(x-U)_+}\right)
\right\}_{0\le x\le1}
\right)\qquad\mbox{ in }[0,1]\times C[0,1].
\]  To prove the theorem, it suffices to show that 
\begin{equation}\label{eq:flux suffices to show}
\left|
\frac{\mathsf M_N}{N}-\frac{T_N}{N}
\right|
+
\left\|
Z_N-
\Psi\left(
\frac{T_N}{N}, S_N^-, S_N^+
\right)
\right\|_\infty
\xrightarrow{\mathbb P}0,
\end{equation}
Note that 
\begin{equation*} 
\|Z_N-X_N\|_\infty
\le
N^{-1/2}\left(2\lfloor N^{1/4}\rfloor-1+
\lvert\rho_+-\rho_-\rvert\,\lvert\mathsf M_N-T_N\rvert\right)  
\qquad\text{ on }E_N.
\end{equation*}
On the other hand, since
$\|R_N^\pm- S_N^\pm\|_\infty\le 1/\sqrt N$ and 
$X_N(x)
=
R_N^- (x\wedge (T_N/N) )
+
R_N^+ ( (x- (T_N/N)  )_+ )$, 
\begin{equation*} 
\left\|
X_N-
\Psi\left(
\frac{T_N}{N}, S_N^-, S_N^+
\right)
\right\|_\infty
\le\frac{2}{\sqrt N}.
\end{equation*}
Combining the last two displays and \eqref{eq:EN in proof flux}, together with Theorem \ref{thm:ApproxCoex}, we conclude \eqref{eq:flux suffices to show}.
\end{proof}

\subsection{Limits of the correlation functions: Proof of Theorem \ref{thm:asymptotics of correlation functions}}

As a third application of our characterization and approximation results for the stationary distribution of the open ASEP on the coexistence line, we show the following statement on its correlations.

\begin{theorem}\label{thm:asymptotics of correlation functions again}
Consider the open ASEP on the lattice $\llbracket1,N\rrbracket$ with jump rates $\alpha,\beta,\gamma,\delta$ and $q$ satisfying \eqref{eq:conditions open ASEP}. Assume Liggett's condition \eqref{eq:Liggett's condition} holds and that the open ASEP is on the coexistence line $\a=\b>1$, where $\a=(1-\alpha)/\alpha$ and $\b=(1-\beta)/\beta$. We denote the stationary measure of this model by $\mu_N^{(q,\b,\b)}$, which is a probability measure on the space $\{0,1\}^N$. Suppose $\tau^{(N)}=(\tau_1^{(N)},\dots,\tau_N^{(N)})\sim\mu_N^{(q,\b,\b)}$.

Fix $m\in\mathbb N$ and \(0\le x_1\le\cdots\le x_m\le1\).
Suppose that for each $N\geq m$, there exist integers  
\(j_{1,N},\ldots,j_{m,N}\) satisfying $1\le j_{1,N}<\cdots<j_{m,N}\le N$ and 
$\lim_{N\to\infty}j_{r,N}/N =x_r$ for each $r\in\llbracket1,m\rrbracket$. 
Then
\[
\lim_{N\to\infty}
\mathbb E 
\left[
    \prod_{r=1}^m \tau_{j_{r,N}}^{(N)}
\right]
=
\sum_{\ell=0}^m
    (x_{\ell+1}-x_\ell)\rho_-^\ell\rho_+^{m-\ell},
\]
where \(x_0=0\), \(x_{m+1}=1\), $\rho_-=1/(1+\b)$, and $\rho_+=\b/(1+\b)$.
\end{theorem}
\begin{proof}
    We work on the enlarged probability space as in Theorem \ref{thm:ApproxCoex}. Define 
    \begin{equation}\label{eq:correlation bad set cardinality}
{B_N}
=
\left\{
k\in\llbracket0,N\rrbracket:
\min_{1\leq r\leq m}
\left|k-j_{r,N}\right|< \lfloor\sqrt N \rfloor
\right\},\qquad\mbox{ then } 
|{B_N}|\leq m\left(2\lfloor\sqrt N \rfloor-1\right).
\end{equation}
For \(k\in\llbracket0,N\rrbracket\), define  $L_N(k)
=
\#\{
r\in\llbracket1,m\rrbracket:
j_{r,N}\leq k \}$ and  
$F_N(k)
=
\rho_-^{L_N(k)}
\rho_+^{m-L_N(k)}$.
By the total-variation estimate in Theorem \ref{thm:ApproxCoex}, there exist constants \(C,c>0\) depending
only on \(q\) and \(\b\), such that
\begin{equation*} 
\left|
\mathbb E\left[
\prod_{r=1}^m\tau_{j_{r,N} }^{(N)}
\right]
-
\mathbb E\left[F_N(T_N)\right]
\right|
\leq
Ce^{-c\lfloor\sqrt N\rfloor}
+
\mathbb P(T_N\in{B_N}).
\end{equation*}
By Theorem \ref{thm:ApproxCoex} and \eqref{eq:correlation bad set cardinality}, 
\begin{equation*} 
\mathbb P(T_N\in{B_N})
\leq
\mathrm{Unif}_{\llbracket0,N\rrbracket}({B_N})
+
\left\|
\mathbb P(T_N\in\,\cdot\,)-\mathrm{Unif}_{\llbracket0,N\rrbracket}
\right\|_{\mathrm{TV}} \leq
\frac{m(2\lfloor\sqrt N\rfloor-1)}{N+1}
+
\frac{C}{N}. 
\end{equation*}
Moreover, since \(0\leq F_N\leq1\), we claim that
\begin{equation*} 
\left|
\mathbb E[F_N(T_N)]
-
\frac{1}{N+1}\sum_{k=0}^N F_N(k)
\right|
\leq
\frac{C}{N}.
\end{equation*}
Indeed, for each probability measure \(\nu\) on \(\llbracket0,N\rrbracket\), we have
$\mathbb E_\nu[F_N]
=
\int_0^1
\nu (\{k:F_N(k)>t\} )\,dt$.
Taking the difference of this identity for the distributions \(\mathbb P(T_N\in\,\cdot\,)\) and \(\mathrm{Unif}_{\llbracket0,N\rrbracket}\),  we conclude this claim.

Let $j_{0,N}=0$ and $j_{m+1,N}=N+1$. For   \(\ell\in\llbracket0,m\rrbracket\),   $L_N(k)=\ell$ if and only if  $j_{\ell,N}
\leq k
<
j_{\ell+1,N}$. Hence
\begin{equation*} 
\frac{1}{N+1}\sum_{k=0}^N F_N(k)
=
\sum_{\ell=0}^m
\frac{
j_{\ell+1,N} -j_{\ell,N} 
}{N+1}
\rho_-^\ell\rho_+^{m-\ell}.
\end{equation*}

Combining the last four displayed equations, we obtain 
\[
\begin{split}
&\left|
\mathbb E\left[
\prod_{r=1}^m\tau_{j_{r,N} }^{(N)}
\right]
-
\sum_{\ell=0}^m
\frac{
j_{\ell+1,N} -j_{\ell,N} 
}{N+1}
\rho_-^\ell\rho_+^{m-\ell}
\right| \leq
Ce^{-c\lfloor\sqrt N\rfloor}
+
\frac{m(2\lfloor\sqrt N\rfloor-1)}{N+1}
+
\frac{2C}{N}.
\end{split}
\]
The right side tends to zero. Note that for every \(\ell\in\llbracket0,m\rrbracket\),
$ 
\lim_{N\to\infty}(j_{\ell+1,N}-j_{\ell,N})
/(N+1)
=
x_{\ell+1}-x_\ell$, and so we conclude the proof.
\end{proof}

\bibliographystyle{goodbibtexstyle.bst}
\bibliography{RWApproximation}

\end{document}